\documentclass[10pt]{article}   
\usepackage[utf8]{inputenc}   
\usepackage[left=1.8cm,right=1.8cm,top=1.8cm,bottom=1.8cm]{geometry} 
\usepackage[T1]{fontenc}                
\usepackage[english]{babel}              
\usepackage{graphicx}
\usepackage{amsmath,amsfonts,amssymb,amsthm}  
\usepackage{xspace}
\usepackage{geometry} 
\usepackage{graphicx}
\usepackage[overload]{empheq}
\usepackage{dsfont} 
\usepackage{stmaryrd}
\usepackage{tikz} 
\usepackage{minted}     
\usepackage{array}
\usepackage{booktabs}
\usepackage{mathrsfs}
\usepackage{enumerate}
\usepackage{subfig}
\usepackage{float}

\usepackage{hyperref}
\hypersetup{
colorlinks = true,
urlcolor = blue,
linkcolor = blue,
citecolor = blue,
}

\makeatletter

\newif\ifappendixtoc
\appendixtocfalse

\let\oldaddcontentsline\addcontentsline
\renewcommand{\addcontentsline}[3]{%
  \ifappendixtoc
    \def\tempa{#1}%
    \def\tempb{toc}%
    \ifx\tempa\tempb
      \oldaddcontentsline{atoc}{#2}{#3}%
    \else
      \oldaddcontentsline{#1}{#2}{#3}%
    \fi
  \else
    \oldaddcontentsline{#1}{#2}{#3}%
  \fi
}

\newcommand{\appendixtableofcontents}{%
  \section*{Appendix contents}
  \@starttoc{atoc}%
}

\makeatother

\usepackage{tikz}

\newtheorem{theorem}{Theorem}
\newtheorem{definition}{Definition}
\newtheorem{assumption}{Assumption} 
\newtheorem{proposition}{Proposition} 
\newtheorem{lemma}{Lemma}
\newtheorem{corollary}{Corollary}
\newtheorem{remark}{Remark}

\usepackage{amsmath,amssymb,euscript ,yfonts,psfrag,latexsym,dsfont,graphicx,bbm,color,amstext,wasysym,epsfig,subfig,parskip,textcomp}

\usepackage{pgfplots}
\usepackage{mathrsfs}
\usetikzlibrary{arrows}

\usepackage{natbib}
\bibpunct{(}{)}{;}{a}{,}{,}
\usepackage{hyperref}
\hypersetup{
colorlinks = true,
urlcolor = blue,
linkcolor = blue,
citecolor = blue,
}

\DeclareMathOperator*{\argmin}{arg min}

\newcommand{\R}{\mathbb{R}}
\newcommand{\E}{\mathbb{E}}
\makeatletter
\def\thm@space@setup{%
    \thm@preskip=6pt plus 2pt minus 2pt 
    \thm@postskip=1pt plus 2pt minus 2pt 
}
\makeatother   

\usepackage{titling}

\title{Lipschitz regularity  in Flow Matching and Diffusion Models:\\ sharp sampling rates and functional inequalities}
\date{}  
\usepackage{authblk}
\author{\vspace{1.1em}Arthur St\'ephanovitch}
\affil{\small CREST, ENSAE, IP Paris\\
5 av. H. Le Chatelier, 91120 Palaiseau, France,\\ \texttt{arthur.stephanovitch@ensae.fr}}
\begin{document}
\maketitle 

\begin{abstract}
\noindent Under general assumptions on the target distribution $p^\star$, we establish a sharp Lipschitz regularity theory for flow-matching vector fields and diffusion-model scores, with optimal dependence on time and dimension. As applications, we obtain Wasserstein discretization bounds for Euler-type samplers in dimension $d$: with $N$ discretization steps, the error achieves the optimal rate $\sqrt{d}/N$ up to logarithmic factors. Moreover, the constants do not deteriorate exponentially with the spatial extent of $p^\star$. We also show that the one-sided Lipschitz control yields a globally Lipschitz transport map from the standard Gaussian to $p^\star$, which implies Poincar\'e and log-Sobolev inequalities for a broad class of probability measures.
\end{abstract}

\tableofcontents
\addtocontents{toc}{\protect\setcounter{tocdepth}{3}}

\section{Introduction}
\subsection{Continuous-time generative models and Lipschitz regularity} 
Generative modeling aims at learning a  target distribution $p^\star$ on \(\mathbb{R}^d\) from data, in order to generate new samples whose law is close to $p^\star$.
A dominant paradigm is to construct a continuous-time transport from a simple reference law (typically the standard Gaussian) to the target law $p^\star$.
Two influential families of methods instantiate this principle in complementary ways:
\emph{score-based diffusion models} (SGMs) build such a transport through time-reversal of a stochastic dynamics and require the estimation of a score function, whereas \emph{flow matching} learns directly a deterministic velocity field generating a probability flow.
Both viewpoints ultimately reduce sampling to the numerical approximation of an ODE/SDE whose drift is learned from data, and the present work is concerned with quantitative regularity properties of these drifts.

\subsubsection{Generative modeling through (stochastic) differential equations}

A convenient way to describe flow-based generative models is to start from a family of
intermediate distributions $(p_t)_{t\in[0,1]}$ connecting a simple reference law $p_0$ to the target
distribution $p^\star$. In the constructions considered in this paper, these marginals are of
Gaussian-mixture type: there exist a latent variable $Z\sim \pi$ and an independent Gaussian noise
$\xi\sim \mathcal N(0,I_d)$ such that
\begin{equation}\label{eq:intro_gaussian_path}
   Y_t = m_t(Z)+\sigma_t \xi,
\qquad
p_t=\text{Law}(Y_t). 
\end{equation}
One may think of $Z$ as the latent variable that parametrizes the Gaussian components of the mixture. In the
concrete models considered later, its law $\pi$ is built from the data distribution and the
reference Gaussian law, so sampling $Z$ amounts in practice to taking a point from the dataset,
together with an auxiliary reference-noise variable.

Assuming that $t\mapsto m_t(z)$ and $t\mapsto \sigma_t$ are differentiable, one can associate to this
interpolation an explicit conditional vector field $v_t^c$ such that
\[
\partial_t Y_t = v_t^c(Z,Y_t).
\]
The corresponding marginal vector field is then obtained by conditional expectation:
\begin{equation}\label{eq:hgjkloiuhj}
v_t(x)=\E\big[v_t^c(Z,Y_t)\mid Y_t=x\big],
\end{equation}
and transports the density path $(p_t)_{t\in[0,1]}$ in the sense that it solves the continuity equation $\partial_t p_t+\nabla \cdot (v_t p_t)=0$.
Equivalently to \eqref{eq:hgjkloiuhj}, $v_t$ is the unique minimizer of the quadratic regression problem
\begin{equation}\label{eq:fhodozpzmdjfd}
v_t\in \argmin_{w:\mathbb R^d\to\mathbb R^d}
\E\big[\|w(Y_t)-v_t^c(Z,Y_t)\|^2\big].
\end{equation}
This identity is the conceptual bridge between probabilistic constructions and modern training
objectives. In practice, this regression problem is approximated empirically from samples $(Z^{(i)},\xi^{(i)},t^{(i)})$, by optimizing over a neural-network class $\mathcal V$:
\[
\hat v \approx \argmin_{w\in\mathcal V}\frac{1}{n}\sum_{i=1}^n \left\| w\bigl(t^{(i)},Y_{t^{(i)}}^{(i)}\bigr)-v_{t^{(i)}}^c\bigl(Z^{(i)},Y_{t^{(i)}}^{(i)}\bigr)\right\|^2,
\]
where $Z^{(i)}$ is constructed from the dataset and the optimization is carried out in practice by stochastic gradient descent.

Once the vector field has been learned, sampling amounts to approximating the continuous flow associated with it.
This can be viewed in two closely related ways. In the deterministic case, one considers the ODE
\[
\partial_t X_t = v_t(X_t),
\qquad
X_0\sim p_0,
\]
whose marginals follow the prescribed curve $(p_t)_{t\in[0,1]}$. This is the flow-matching point of
view.

More generally, as detailed in Section~\ref{sec:models}, the same curve of marginals may be realized as the family of time marginals of an
It\^o dynamics
\begin{equation}
    \label{eq:intro_sde_general}
    dX_t = a_t(X_t)\,dt + b_t\,dB_t,
\qquad
t\in[0,1],
\qquad
X_0\sim p_0,
\end{equation}
for a suitable drift $a_t$ being also solution to a minimization problem of the form \eqref{eq:fhodozpzmdjfd}. When $b_t\equiv 0$, this reduces to the ODE above. When $b_t\not\equiv 0$,
one obtains a stochastic sampler; in diffusion models, the relevant drift is expressed in terms of the
score $\nabla \log p_t$ and corresponds to a reverse-time SDE associated with the same family of
marginals.

Therefore, flow matching and diffusion models rely on
the same basic mechanism: a Gaussian-mixture interpolation defines a path of marginals, an
explicit conditional drift is projected onto the current state and 
used inside an ODE or an SDE sampler. The main goal of this paper is to study the regularity of these
projected drifts.

\subsubsection{The role of Lipschitz regularity}

A recurring theme in continuous-time generative modeling is that the model is defined through a
continuous flow,
\begin{equation}\label{eq:fhfkdodopdpdpd}
dX_t = a_t(X_t)\,dt + b_t\,dB_t,
\end{equation}
but used through a numerical scheme. For both flow matching and diffusion models, the practical algorithm is an Euler-type discretization driven by a learned
approximation of the true drift. This makes quantitative regularity bounds on the drift central: they are the common currency that simultaneously (i) guarantees well-posedness and stability
of the underlying bridge, (ii) controls discretization errors of samplers, and (iii) allows one to extract
analytic information on the target distribution through the induced transport map. We discuss below in detail why Lipschitz regularity is important.

\vspace{0.2cm}
\noindent\textbf{Stability of trajectories and propagation of perturbations.}
A first reason is \emph{dynamic stability}. 
A convenient quantitative notion is the \emph{one-sided Lipschitz} control $L:[0,1]\to\mathbb{R}$ such that, for all $x,y\in\mathbb{R}^d$,
\begin{equation}\label{eq:OSL}
\langle x-y,  a_t(x)-a_t(y)\rangle \le L_t\|x-y\|^2 .
\end{equation}
The key point is that
\eqref{eq:OSL} yields a sharp Grönwall-type stability inequality: if two trajectories are driven by the equation \eqref{eq:fhfkdodopdpdpd} and differ only through initial condition, then their
distance at time $t$ is bounded by the initial distance times $\exp(\int_0^tL_sds)$. This is precisely the mechanism by which one can
control how initial errors or learning errors in a drift  propagate to the terminal distribution.

\noindent\textbf{Discretization error of Euler-type samplers in Wasserstein distance.}
A second reason is that Lipschitz estimates control the local truncation error in the discretization analysis.
Writing \(X_t\) for the exact process and \(\bar X_t\) for its Euler scheme with exact Gaussian increment, one derives on each interval
\([t_k,t_{k+1}]\) the one step decomposition
\[
X_{t_{k+1}}-\bar X_{t_{k+1}}
=
X_{t_k}-\bar X_{t_k}
+\int_{t_k}^{t_{k+1}} \big(a_s(X_s)-a_{t_k}(\bar X_{t_k})\big) ds,
\]
Now, to control the local truncation error $\|a_s(X_s)-a_{t_k}(X_{t_k})\|$, one needs both global space-Lipschitz regularity
\begin{equation}\label{eq:globalLip}
\|a_t(x)-a_t(y)\|\le C_t\|x-y\| ,
\end{equation}
to control how much the drift can change when the state is perturbed, and time-regularity
\begin{equation}\label{eq:timeReg}
\|\partial_t a_t(x)\|\le M_t\big(\sqrt{d}+\|x\|\big),
\end{equation}
to control how much the drift can vary over a single time-step.
Equivalently, \eqref{eq:globalLip} and \eqref{eq:timeReg} are  exactly the
estimates that convert the pathwise coupling into a closed recursion for the mean-square error, and
therefore into a non-asymptotic bound on \(W_2(\text{Law}(X_1),\text{Law}(\bar X_1))\).

\vspace{0.2cm}
\noindent\textbf{From one-sided Lipschitz control to functional inequalities.}
A third, more structural reason is that regularity of $v_t$ controls the geometry of the induced
transport. Let $X_t(x)$ denote the flow map solving $\partial_t X_t(x)=v_t(X_t(x))$ with $X_0(x)=x$.
Differentiating in space shows that $\nabla X_t$ satisfies a linear ODE driven by $\nabla v_t$ along
the flow. As a result, an integrability bound on the maximal eigenvalue of $\nabla v_t$ yields a global
Lipschitz bound on $X_1$:
\[
\|\nabla X_t(x)\|_{\mathrm{op}} \le \exp\Big(\int_0^t \sup_{z\in\mathbb{R}^d}\lambda_{\max}(\nabla v_s(z)) ds\Big).
\]
In particular, when $X_0\sim \mathcal{N}(0,\text{Id})$, the terminal map $T:=X_1$ pushes forward the
standard Gaussian measure onto the target $p^\star$. A dimension-free Lipschitz bound on $T$ is then a
powerful analytic tool: Lipschitz changes of variables preserve functional inequalities, so one can transfer Poincaré and log-Sobolev inequalities to
targets.
This viewpoint connects the regularity theory of learned drifts/scores with classical questions in
probability theory (concentration, $\Psi$-Sobolev inequalities), and it highlights that Lipschitz control is
not merely a technical assumption for numerical schemes: it is a bridge between generative modeling and
quantitative analysis of measures.

\vspace{0.2cm}
\noindent\textbf{Scope of this paper.}
Motivated by these three roles, the goal of the present work is to establish sharp bounds on
the spatial (one-sided and two-sided) Lipschitz regularity and on the time-regularity of the canonical drift
fields arising in flow matching and diffusion models, under weak assumptions on the data distribution. The
applications developed later ($W_2$ discretization bounds and transfer
of functional inequalities) are direct consequences of these regularity
estimates and serve as illustrations of their probabilistic relevance.

\subsection{Previous works and contributions}
\subsubsection{Assumptions on the target distribution}\label{sec:saaumptionstarget}
Throughout the paper, we work with a density $p^\star$ that is $(\alpha,\beta,K)$-weakly log-concave with $\alpha,K>0$ and $\beta\in (0,1]$, in the sense of Definition~\ref{defi:weakconcave}.
\begin{definition}[Weak log-concavity]
\label{defi:weakconcave}
A probability density $p:\mathbb{R}^d\rightarrow \mathbb{R}$ is said to be $(\alpha,\beta,K)$-weakly log-concave if it is of the form 
$p(x)=\exp(-u(x)+a(x))$ and satisfies both the following assumptions:
\begin{enumerate}
    \item $u:\mathbb{R}^d\rightarrow \mathbb{R}\cup\{\infty\}$ is convex, is $C^2$ on the interior of the support and satisfies $\nabla^2 u \succeq \alpha\text{Id}$, 
    \item For all $x,y\in \mathbb{R}^d$, $|a(x)-a(y)|\leq K\|x-y\|^{\beta}$.
\end{enumerate}
\end{definition}
Informally, the probability measures in Definition~\ref{defi:weakconcave} can be thought of as sub-Gaussian densities $p^\star$ with convex support that are $\beta$-Hölder on compact sets of the support with a constant that depends on the size of the compact.

\begin{figure}[H]
  \centering
  \includegraphics[width=0.85\textwidth]{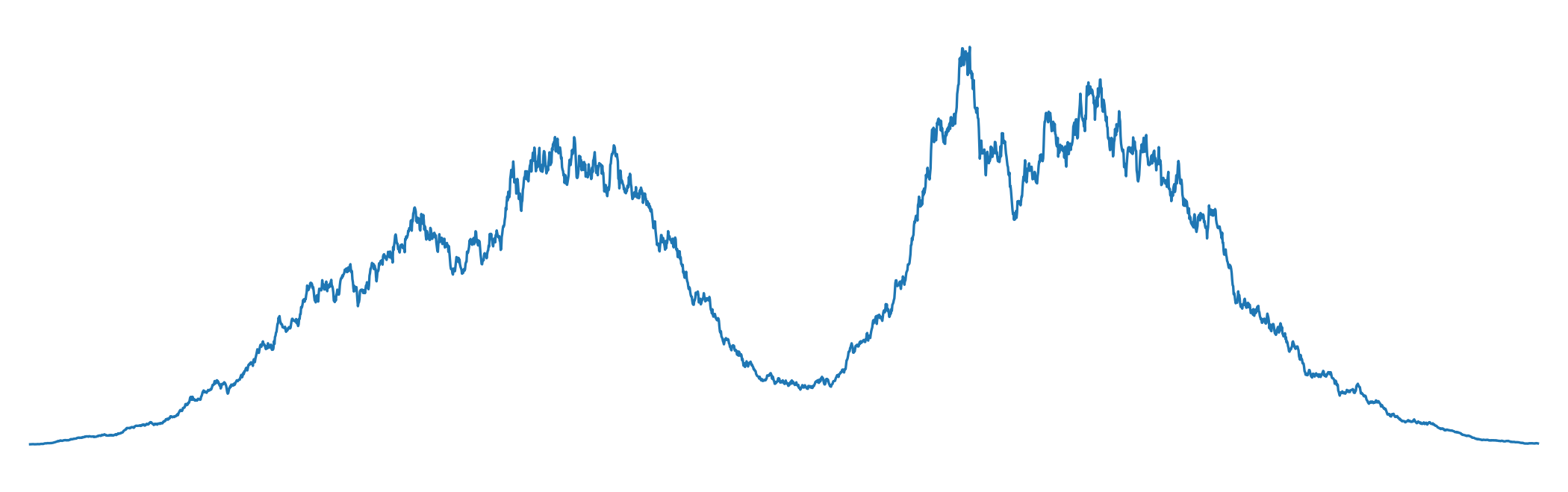}
  \caption{Illustration of Definition~\ref{defi:weakconcave} for Hölder regularity $\beta=1/2$.}
  \label{fig:beta_half_example}
\end{figure}

\paragraph{Tail assumption:  sub-Gaussianity.}
Definition~\ref{defi:weakconcave} does not impose any restriction on the support of $p^\star$ besides its convexity.
It is compatible both with \emph{full-support} densities and with \emph{compactly supported} ones.
Indeed, the curvature requirement on $u$ is only local and compact support can be encoded by taking $u=+\infty$ outside $\mathrm{supp}(p^\star)$.

In the full-support case $\mathrm{supp}(p^\star)=\mathbb{R}^d$, the decomposition enforces sub-Gaussian tails: the strongly convex part $u$ provides at least quadratic growth, while the Hölder perturbation $a$ can grow at most linearly when $\beta\le 1$. This behavior is not an artefact of the decomposition, but rather matches what is structurally compatible with the main goals of the paper. 

Indeed, if one seeks a globally Lipschitz transport map $T$ such that $T_\#\gamma_d=p^\star$,
then sub-Gaussian concentration for $p^\star$ is unavoidable:
for any $1$-Lipschitz $\varphi$, the random variable $\varphi(T(G))-\mathbb{E}[\varphi(T(G))]$
(with $G\sim\gamma_d$) is sub-Gaussian with variance proxy of order $\mathrm{Lip}(T)^2$ by Gaussian concentration.

The same tail information is also crucial when one passes from the continuous-time to time-discretized algorithms:
our results use stability estimates of the form
\[
\|X_{t_{k+1}}-\bar X_{t_{k+1}}\|
 \le\
(1+L h) \|X_{t_k}-\bar X_{t_k}\|
 + h \|\varepsilon_k\|,
\]
where $\bar X$ denotes a numerical scheme, $h$ is the time step, and $\varepsilon_k$ collects local truncation errors.
While Lipschitzness controls the propagation of perturbations, turning such  bounds into
non-asymptotic estimates requires control of moments of $X_t$.  Sub-Gaussian tails provide exactly this and yield stable constants that do not deteriorate
with~$d$.

\paragraph{Regularity assumption: local Hölder control.}
The regularity condition  on $a$ controls the local oscillations of the log-density: on any ball of radius $r$, the multiplicative weight $e^{a}$ can change by at most a factor $\exp(K(2r)^\beta)$. In particular, $a$ may encode large-scale nonconvexity as long as it does not create arbitrarily sharp fluctuations at arbitrarily small scales. Note also that we do \emph{not} assume differentiability of $a$: when $\beta=1$ this is merely a Lipschitz assumption, and when $\beta<1$ it still allows rough perturbations far from $C^1$ (see the example of Figure~\ref{fig:beta_half_example}). 

A control on $a$ is needed as quantitative bounds on $\nabla\log p_t$ can blow up if $a$ is allowed to oscillate arbitrarily fast. Indeed, without a modulus of continuity on $a$, the factor $e^{a}$ can concentrate mass into extremely thin regions which may produce very large derivatives of $\log p_t$.

Overall, Definition~\ref{defi:weakconcave} is a fairly mild structural assumption: it allows both compactly supported and full-support targets and permits rough perturbations through a mere Hölder control, while enforcing the sub-Gaussian tail behavior that is unavoidable for the sharp guarantees pursued here.

\subsubsection{Previous work}\label{sec:previouswork}
We measure generative error using the Wasserstein distance $W_2$ as it provides an
intrinsically geometric notion of discrepancy: it directly accounts for the ambient
distance in the sample space and therefore distinguishes between samples that are slightly
off the data support and samples that are far away. In contrast, total variation and
Kullback-Leibler are $f$-divergences that essentially depend on density ratios; as a result
they do not encode the underlying geometry and may be ill-defined if one wants to generalize the results to low dimensional structure like the manifold hypothesis.
We refer the reader to \cite{arsenyan2025assessing} and \cite{beyler2025convergence} for a thorough
discussion of the advantages of Wasserstein distances in the context of generative modeling.

\textbf{Sampling bounds in Wasserstein distance.}
The recent literature on non-asymptotic $W_2$ guarantees for diffusion-model samplers is now large, with results spanning a wide range of assumptions on the target distribution, the discretization grid, and the notion of learning error. Rather than attempting an exhaustive survey, we refer to Table~\ref{tab:w2_sampling_bounds} for a representative snapshot of $W_2$ sampling bounds obtained by direct stability arguments (as opposed to bounds derived by first controlling KL/TV, which typically introduces additional losses).

\begin{table}[H]
\centering
\setlength{\tabcolsep}{4pt}
\renewcommand{\arraystretch}{1.25}
\scriptsize
\begin{tabular}{
>{\centering\arraybackslash}m{2.35cm}
>{\centering\arraybackslash}m{2.15cm}
>{\centering\arraybackslash}m{3.05cm}
>{\centering\arraybackslash}m{3.05cm}
>{\centering\arraybackslash}m{2.15cm} 
>{\centering\arraybackslash}m{3.5cm}
}
\toprule
\textbf{Reference} &
\textbf{Method} &
\textbf{Assumptions on $p^\star$} &
\textbf{$W_2$ sampling rate} &
\textbf{Assumptions on the drift} &
\textbf{$\epsilon_t$} \\
\midrule
 
\cite{strasman2025noise_schedule} &
Diffusion models &
strongly log concave &
$\widetilde O \left(1/\sqrt N+\sqrt d/N+\int \epsilon_t\right)$ &
$ \text{Lip}_x,$ $\text{Lip}_t$ &
$\|a_t(\hat X_{t})-\hat a_t(\hat X_{t})\|_{L_2}$ \\

\cite{yu_yu_2025_advancing_wasserstein} & 
Diffusion models & 
strongly log-concave&
$\widetilde O \left(\sqrt{d/N}+\int \epsilon_t\right)$&
$\text{Lip}_x$, $\text{Lip}_t$ &
$\|a_t(\hat X_{t})-\hat a_t(\hat X_{t})\|_{L_2}$ \\

\cite{gao_nguyen_zhu_2025_jmlr} &
Diffusion models &
strongly log concave and  bounded log-Hessian &
$\widetilde O\left(\sqrt{d/N}+\int \epsilon_t\right)$ &
$\text{Lip}_t$ &
$\|a_t(\hat X_{t})-\hat a_t(\hat X_{t})\|_{L_2}$ \\

\cite{bruno2025wasserstein} &
Diffusion models &
weakly log concave &
$\widetilde O \left( \sqrt{d/N}+\int \epsilon_t\right)$&
$\text{Höl}_t^{1/2},  \text{Lip}_x$&
$\|a_t( X_{t}^{\text{aux}})-\hat a_t( X_{t}^{\text{aux}})\|_{L_2}$  \\

\cite{gao_zhu_2024_pflow_wasserstein} &
Probability flow ODE &
strongly log concave and bounded log-Hessian&
$\widetilde O\left(\sqrt d/N+\int \epsilon_t\right)$ &
$\text{Lip}_t$ &
$\|a_t(\hat X_{t})-\hat a_t(\hat X_{t})\|_{L_2}$ \\

\cite{kremling2025pflow_weaklogconcave} &
Probability flow ODE &
weakly log concave  and  bounded log-Hessian &
$\widetilde O\left(\sqrt d/N+\int \epsilon_t\right)$ &
$\text{Lip}_t$ &
$\|a_t(\hat X_{t})-\hat a_t(\hat X_{t})\|_{L_2}$ \\

\cite{pmlr-v267-zhou25l}&  Flow matching
&
bounded support  and  bounded log-Hessian &
$\widetilde O\left(e^d(1/N^{1/3}+\int \epsilon_t)\right)$&
$\text{Lip}_t$ &
$\|a_t( X_{t})-\hat a_t( X_{t})\|_{L_2}$
\\

\cite{benton_debortoli_doucet_2024_flowmatching} &
Flow matching &
weakly log concave  and  bounded support &
$\widetilde O \left(\sqrt d/N+N^{2}\int \epsilon_t \right)$ &
$\varnothing$ &
$\|a_t( X_{t})-\hat a_t( X_{t})\|_{L_2}$ \\

\cite{gao2024convergence} &
Flow matching &
bounded log-Hessian &
$\widetilde O\left(\sqrt{d}/N^{1/3}+\int \epsilon_t\right)$&
$\varnothing$ &
$\|a_t( X_{t})-\hat a_t( X_{t})\|_{L_2}$ 
\\ 

\cite{beyler2025convergence} &
Diffusion models  and  Probability flow ODE &
Gaussian convolution of bounded-support law &
$\widetilde O \left(d^{1/4}/N^{1/2}+\int\epsilon_t\right)$ &
$\varnothing$&
$\|a_t( X_{t})-\hat a_t( X_{t})\|_{L_2}$ \\

\cite{gentiloni_silveri_ocello_2025} &
Diffusion models &
weakly log-concave  and  one-sided log-Lipschitz &
$ \widetilde O\left(\sqrt{d/N}+\int \epsilon_t\right)$ &
$\varnothing$ &
$\|a_t(\hat X_{t})-\hat a_t(\hat X_{t})\|_{L_2}$ \\

\cite{wang2024wasserstein}& 
Diffusion models &
$C^2$ Gaussian perturbation & 
$ \widetilde O\left(\sqrt{d}/N+\int \epsilon_t\right)$ &
$\varnothing$ &
$\|a_t( X_{t})-\hat a_t( X_{t})\|_{L_2}$ \\

\cite{arsenyan2025assessing} &
Diffusion models &
weakly log concave  and  bounded support &
$\widetilde O\left(\sqrt d/N+\int \epsilon_t\right)$ &
$\varnothing$ &
$\|a_t-\hat a_t\|_{L_\infty}$ \\

This work &
Flow matching + Diffusion models &
weakly log concave &
$\widetilde O \left(\sqrt d/N+\int \epsilon_t \right)$ &
$\varnothing$ &
$\|a_t( X_{t})-\hat a_t( X_{t})\|_{L_2}$ \\

\bottomrule
\end{tabular}
\caption{\textbf{Representative $W_2$ sampling bounds.}\\
{\small For each reference we report the dominant scaling of the terminal sampling error in $W_2$ as a function of the dimension $d$, the number of discretization steps $N$, and a learning error $\epsilon$ on the drift.
Several works derive bounds involving several error terms and auxiliary parameters. Whenever needed, we optimize over these auxiliary parameters to extract a simplified rate depending only on $(N,d,\epsilon)$.
We suppress logarithmic factors by absorbing in $\widetilde O(\cdot)$. 
 In the column ``Assumption on the drift'', we write $\varnothing$ when no extra regularity assumption is imposed on the true or learned score/velocity field, or when the only requirement is that the learned drift matches the regularity of the true drift and the cited work proves this regularity for the true drift (rather than assuming it). When a condition is displayed instead of $\varnothing$ (e.g. $\mathrm{Lip}_x$, $\mathrm{Lip}_t$, \dots), it should be read as an additional assumption on the true or learned drift, with no proof that the true drift satisfies it  with the stated assumptions on the target density $p^\star$. For simplicity, we do not distinguish whether $\epsilon_t$ is integrated or evaluated on the discretization grid, nor whether the relevant error term is $\left(\int \epsilon_t^2\,dt\right)^{1/2}$ rather than $\int \epsilon_t\,dt$. Note also that the notion of weak log-concavity is not standardized and the definition varies across references; in some works it is compatible with non-$C^1$ potentials (for instance \cite{bruno2025wasserstein}, as well as our Definition~\ref{defi:weakconcave}), while in other works the same label is paired with additional differentiability assumptions.}}
\label{tab:w2_sampling_bounds}
\end{table}

A recurring feature of the $W_2$ sampling literature is to impose \emph{a priori}
regularity on the drift (e.g. global space-Lipschitz bounds,
time-Lipschitz bounds, or uniform bounds on higher derivatives).
While these conditions make stability estimates tractable, it is not clear
when they are actually satisfied, nor whether they remain meaningful at the small-noise end where
singular behavior is expected.
Having score regularity should be viewed as a nontrivial analytic property
of the exact intermediate marginals $p_t$, and therefore calls for a separate
regularity theory (which is one of the main goals of the present work).

Among the works summarized in Table~\ref{tab:w2_sampling_bounds}, \cite{arsenyan2025assessing} is particularly close to the perspective of this paper. They obtain the optimal discretization scaling $\sqrt{d}/N$ for Euler-type diffusion samplers without assuming Lipschitz regularity of the score. Their main estimates nevertheless require uniform $L^\infty$ control of the score error $\hat{s}_t - s_t$ over the sampling grid, a stronger requirement than the usual denoising score-matching risk. In addition, their analysis is specific to diffusion-model samplers and does not cover deterministic flow-matching dynamics. Another closely related work is \cite{wang2024wasserstein}, who establish $W_2$ bounds for diffusion-model samplers in a setting where the target distribution is a $C^2$ perturbation of a Gaussian measure. This specific structure allows them to isolate the Gaussian component of the score, prove exponential decay of the remaining nonlinear part, and derive sampling bounds of order $\sqrt{d}/N$. Their framework is therefore much more specific than ours or that of \cite{arsenyan2025assessing} as it allows the Gaussian component of the score to be separated explicitly and the remaining terms exhibit strong compensations.

As discussed in \cite{beyler2025convergence},
several works control the drift error evaluated along the \emph{learned} reverse trajectory
$\|s_t(\widehat X_t)-\widehat s_t(\widehat X_t)\|_{L_2}$, but this does not match the
denoising score matching objective, since the law of $\widehat X_t$ itself depends on $\widehat s_t$.
In contrast, when one compares the exact reverse dynamics $(X_t)$ and its learned counterpart, the stability estimate naturally produces the
learning term as the drift mismatch along the true process, namely
\[
\epsilon_{\mathrm{drift}}(t)\;:=\;\|s_t(X_t)-\widehat s_t(X_t)\|_{L_2}.
\]
This is the statistically natural notion of score error as it aligns with what is minimized by denoising score matching, and it cleanly
separates learning from discretization/initialization effects.

\paragraph{Score regularity in diffusion models.}
Many statistical and discretization analyses
either assume Lipschitz bounds on  the score $s_t$, for instance the polynomial-time guarantees for probability
flow samplers in \cite{chen2024probability} are obtained under the
assumption that $s_t$ is $L$-Lipschitz along the forward process (hence the regularity is
an input to the theory rather than a derived property). In contrast, recent works have started to \emph{prove} regularity estimates for
the exact score itself.

\cite{beyler2025convergence} work under a bounded-support assumption and from this alone, derive global space-Lipschitz estimates for the score function.
The main drawback is that the resulting stability constants are very large, which propagates into very conservative $W_2$ guarantees.
In this sense, while the mechanism ``bounded support $\Rightarrow$ Lipschitz score'' is explicit, the constants are not optimized for sampling complexity. The paper \cite{brigati2024heat} establishes sharp dimension-free Lipschitz bounds along the heat flow  $p_t = p^\star \ast  \gamma_t$ when $p^\star=\exp(-u+a)$ with $u$ $\alpha$-convex (allowing $\alpha<0$ provided $1+\alpha t>0$) and $a$ is supposed to be Lipschitz.
Beyond this Lipschitz-perturbative  regime, \cite{Stephanovitch2025regularity}
derives Lipschitz and higher-order regularity bounds for $s_t$ under similar assumptions but
allowing rough log-densities (e.g. log--$\beta$-Hölder with additional bounded assumptions) and provides sharp
time-dependence, but the constants depend on the dimension.

Beyond such general regularity theories, alternative bounds can be obtained under stronger structural assumptions, typically with constants tailored to the specific model class.
For Gaussian-mixture targets, \cite{LiangEtAl2025GMMsmoothness}  derives an explicit bound on the full Lipschitz constant of the score. Importantly, the
bound is independent of the number of mixture components and is presented as tight with respect to the
natural covariance regularization. However, the estimate is not uniform in space
and the bound carries a pronounced  dependence on the dimension. Assuming $p^\star= e^{-U}$ with $\nabla^2 U \leq \gamma \text{Id}$ and $U$ weakly convex (different definition than in this paper), \cite{gentiloni_silveri_ocello_2025} prove the upper bound $\sup_x \|\nabla^2\log p_t(x)\|_{\text{\text{op}}}\leq L$ with $L$ depending on the weak log-concavity profile.

\cite{wang2024wasserstein} consider targets of the form $
p^\star(x)=\exp \left(-\frac12 \|x\|_A^2 + h(x)\right),$
where the perturbation $h$ has bounded first and second derivatives. In this regime, they prove exponential decay in time of a modified score and its Jacobian. This is a powerful result, but it crucially exploits the explicit Gaussian structure of the target. In particular, it is  far from Definition~\ref{defi:weakconcave}, for which no such Gaussian decomposition is available.

\textbf{Velocity field regularity in flow matching.} Compared with diffusion models, direct regularity results for the flow-matching velocity field are still scarce. \cite{benton_debortoli_doucet_2024_flowmatching} derive an explicit Jacobian formula for the stochastic-interpolant velocity field and use it to obtain a bound on the full spatial Lipschitz constant. This is one of the first general analyses of regularity for flow-matching and covers log-concave targets and certain Gaussian-smoothed bounded laws. Compared with the present work, however, their result does not provide time-regularity estimates, does not identify a pointwise terminal-time blow-up rate, and is not dimension-free.

\cite{gao2024convergence} are among the closest predecessors to the present work, since they also derive regularity of the exact velocity field from structural assumptions on $p^\star$.  For linear-interpolation flow matching, they show a time-Lipschitz bound of order $(1-t)^{-2}$ which captures the same time scale as ours although the estimate is not sharp with respect to the dimension. For space regularity, their Jacobian estimate is of order $(1-t)^{-3}$ which is more singular than the $(1-t)^{-1}$ behaviour established here. Their assumptions are also stronger than ours: they require either strong log-concavity together with a bounded Hessian of the log-density, or bounded support together with the same $C^2$-type control on the log-density, or else a Gaussian smoothing of a compactly supported law.

Under a bounded-support assumption on the target distribution, \cite{pmlr-v267-zhou25l} obtain an explicit bound on the Lipschitz constant of the velocity field, of order $d(1-t)^{-3}$ near terminal time. Under an additional smoothness assumption on the target density, they further establish a global-in-time space-Lipschitz bound, again with polynomial dependence on the dimension. \cite{kunkel2025distribution} studies the spatial Lipschitz constant of the exact flow-matching vector field through a covariance representation and proves an integrated-in-time bound under explicit covariance asymptotics. This gives a refined structural picture of spatial regularity, and shows that some singular contribution from the noise schedule is unavoidable. However, it is not dimension-free and is verified only under strong target assumptions.

\paragraph{Transfer of Gaussian functional inequalities.}
A standard way to derive functional inequalities for a probability measure $\mu$ on $\mathbb{R}^d$
is to realize it as a Lipschitz pushforward of the standard Gaussian measure $\gamma_d$.
This approach was initiated by  \cite{caffarelli2000monotonicity} using optimal transport and then generalized by \cite{kim2012generalization} with heat-flow maps. In particular,
\cite{neeman2022lipschitz} proves existence of Lipschitz changes of variables between $\gamma_d$ and certain perturbations by exploiting this construction,
while \cite{mikulincer2024brownian} establish refined Lipschitz bounds for the corresponding flow map for (semi)-log-concave targets and Gaussian mixtures.

A major perturbative extension is obtained in \cite{fathi2024transportation}, which gives dimension-free existence of globally Lipschitz transports
from $\gamma_d$  strongly curved reference measures to log-Lipschitz perturbations.  \cite{brigati2024heat}
derive log-Hessian lower bounds along the heat equation in a related regime, yielding improved quantitative Lipschitz estimates.
The coupling-based method of \cite{ConfortiEichinger2025Coupling} provides an alternative probabilistic route to bounding the Lipschitz constant of the same flow map
and relaxes some structural assumptions in the log-Lipschitz perturbation framework,
and \cite{LopezRivera2024BakryEmery} develops a Bakry-\'Emery approach on weighted manifolds,
recovering Lipschitz transports in that setting.

In contrast with the above works which largely treat log-Lipschitz tilts, our setting allows Hölder control of the log-density (including $\beta<1$), while still producing a globally Lipschitz map and hence transferring Gaussian inequalities.

In our earlier paper~\cite{Stephanovitch2024smooth}, we studied the same diffusion transport mechanism but pursued a different goal: under the more restrictive Gaussian-tilt assumption $p^\star(x)\propto \exp \big(-\|x\|^2/2+a(x)\big)$, it establishes
higher-order  regularity of the Langevin transport map. This allows to transfer
higher-order logarithmic Sobolev inequalities but no attempt is made to obtain dimension-free estimates.
By contrast, the present work focuses on \emph{dimension-free} Lipschitz control of flow-matching/diffusion-induced transport maps under
weaker structural assumptions on the target, enabling a dimension-free transfer of Gaussian functional inequalities.

\subsubsection{Contributions}

\paragraph{Lipschitz regularity of vector fields/scores.}
Our first set of contributions is a unified, sharp regularity theory for the
canonical drift fields that appear in flow matching and score-based diffusion models. We treat, in a
single framework, the three notions of regularity that are needed for stability, discretization, and analytic
applications: \emph{one-sided} Lipschitz control, \emph{two-sided} Lipschitz control, and \emph{time}
regularity.

A central novelty is our one-sided Lipschitz theory, which is the most difficult regime and the one that truly
drives quantitative stability because it enters inside exponentials through Grönwall-type arguments. The general one-sided results are stated in an abstract form and only require a weak
log-concavity structure of the intermediate marginals and a control of the conditional
fluctuations of the path derivative. We then verify these assumptions and obtain explicit bounds for our three concrete model
classes: Lipman flow matching (Assumption~\ref{assum:lipman}), Stochastic-interpolant flow matching (Assumption~\ref{assum:sto-int}), and Diffusion Models (Assumption~\ref{assum:diffusion}). The precise definitions of these models are given in Section~\ref{sec:models}.

Informally, Corollaries~\ref{coro:lipmafm},~\ref{coro:notlipmafm} and~\ref{coro:score_osl_selfcontained} show that for $a_t$ the drift of the true sampling equation of either flow matching or (rescaled) diffusion models, $a_t$ satisfies that for all $z\in [0,1]$:
\[
\int_z^1 \sup_{x\in\mathbb{R}^d} \lambda_{\max}\big(\nabla a_t(x)\big) dt \;\le\; C,
\]
with $C>0$ independent of the dimension. This $L^1$ estimate is precisely what yields
dimension-free stability factors of the form $\exp \big(\int_0^1 \sup_x\lambda_{\max}(\nabla a_t(x))dt\big)\le e^C$,
and it is the key input behind both our discretization analysis and the Lipschitz-transport consequences developed later.

Beyond one-sided control, we establish two-sided space-Lipschitz and time-Lipschitz
estimates. In particular, Corollaries~\ref{cor:lipman-global}, \ref{coro:global-Lipschitz-v} and~\ref{coro:difflipbound} obtain 
$$\sup_x\|\nabla a_t(x)\|_{\mathrm{op}} \leq C (1-t)^{-1},$$
and Corollaries~\ref{coro:dtv_rate},~\ref{coro:timeestimatesi} and~\ref{coro:timeregudiffu} show that
\[
\|\partial_t a_t(x)\| \le \frac{C}{(1-t)^2}\bigl(\sqrt d+\|x\|\bigr).
\]These blow-up rates are sharp in time:
the $(1-t)^{-1}$ (resp. $(1-t)^{-2}$) behavior is the natural singularity created by vanishing noise in the
bridge and cannot, in general, be improved. They are also sharp in dimension in the sense that all
constants $C$ in the regularity bounds are independent of $d$.

Compared with previous works, our results (i) provide the first dimension-sharp Lipschitz regularity bounds for
flow matching drifts, and (ii) substantially relax the
assumptions used to derive score regularity in diffusion models. Much of the existing sampling literature
either assumes Lipschitz/one-sided Lipschitz properties as an input, or proves regularity under stronger
conditions. In contrast, we work under the weak log-concavity framework of
Definition~\ref{defi:weakconcave}, allowing Hölder perturbations with regularity $\beta\le 1$.

Importantly, these precise blow-up rates are exactly the regime in which the geometric time grid compensates the terminal singularity, making the discretization analysis work without any extra slack in time or dimension.

\paragraph{Sampling error.}
Our second contribution is a quantitative sampling theory in $W_2$ for ODE/SDE samplers with approximated drift.
A key point is that we do  not impose  global Lipschitz assumptions on the true drift: regularity needed for Wasserstein control is proved from our structural assumptions on the target $p^\star$ (Definition~\ref{defi:weakconcave}) via the sharp bounds of Section~\ref{sec:lipreguv}.
The only additional requirement is that the learned drift $\hat{a}_t$ inherits the same structural regularity bounds as its true counterpart $a_t$ on the sampling time interval.
This type of assumption is essentially unavoidable for the error analysis used here, where the drift mismatch is measured along the true process $X_t\sim p_t$, namely $
\epsilon_{\mathrm{drift}}(t)\;:=\;\|a_t(X_t)-\widehat a_t(X_t)\|_{L_2}.$

We first develop model-agnostic discretization  bounds:
Proposition~\ref{prop:geom_mesh_general} (ODE) and Theorem~\ref{theo:geom_mesh_generaldiff} (SDE) give general $W_2$ discretization guarantees for explicit Euler and Euler-Maruyama schemes under regularity controls of the drift with the singular behavior that naturally occurs near terminal times.

Specializing these results to the canonical generative dynamics yields sharp non-asymptotic rates.
For flow matching, Corollary~\ref{cor:lipman_sampling_error_log2_over_N} provides to our knowledge the \emph{first linear-in-$N$} Wasserstein sampling bound:
with an appropriate early-stopping time and a geometric grid, the sampling error scales as
\[
W_2\bigl(p^\star,\text{Law}(\bar X_N)\bigr) \leq C\left( \frac{\sqrt d}{N} \log^2 N + \int_0^\tau \varepsilon_{\mathrm{drift}}(t) dt\right),
\]
where $\bar X_N$ denotes the terminal iterate of the $N$-step discretized sampler. This removes the polynomial amplification in $N$ that appears in prior stability-based analyses of flow matching.
For diffusion models, Corollary~\ref{cor:diffusion_sde_sampling_logN} yields an analogous bound under the same weak log-concavity assumptions on $p^\star$:
\[
W_2\bigl(p^\star,\text{Law}(\bar X_N)\bigr) \leq C\left( \frac{\sqrt d}{N} \log^3 N + \int_0^\tau \varepsilon_{\mathrm{drift}}(t) dt\right),
\]
where the extra logarithmic loss is an artefact of using a single geometric grid over the whole sampling interval. The rate could be improved by combining a uniform grid on the initial part with a geometric refinement only near the terminal time, as in \cite{arsenyan2025assessing}.
In both cases, the leading $\sqrt d/N$ dependence is sharp up to logarithmic factors, and our assumptions on $p^\star$ are the weakest in the literature that still allow such $W_2$ guarantees (as motivated in Section~\ref{sec:saaumptionstarget}).

An important caveat in comparing Wasserstein sampling bounds is that the displayed dependence on $N$ and $d$ can hide very large structural constants. In several previous works, these constants can be exponential either explicitly in the size of the support of $p^\star$, or implicitly in weak-convexity quantities that themselves grow with the size of the effective support (where most of the mass lies); this is the case, for instance, in \cite{benton_debortoli_doucet_2024_flowmatching,arsenyan2025assessing,beyler2025convergence,gentiloni_silveri_ocello_2025,kremling2025pflow_weaklogconcave}. In our setting, the constants also depend exponentially on the weak log-concavity parameters, but this dependence is of a different nature as it does not deteriorate with the size of the effective support. Indeed, under Definition~\ref{defi:weakconcave} the effective support can be of order $\sqrt d$ (the Gaussian scale) but this size does not enter the exponential. This illustrates that the effective-support geometry of $p^\star$ does not deteriorate the constants.

\paragraph{Functional inequalities.}
Our third contribution leverages the dimension-free one-sided Lipschitz control developed in Section~\ref{sec:onesidedlipreguv} to derive dimension-free functional inequalities for the target law. This  connects the regularity theory of drifts with the problem of identifying  classes of probability measures that satisfy functional inequalities. Indeed, Poincar\'e and logarithmic Sobolev inequalities often appear as standing assumptions in probability, statistics, and machine learning, where they yield concentration, variance and mixing estimates. Therefore, proving that $p^\star$ inherits such inequalities enlarges the family of measures for which many general results become directly applicable.

Concretely, starting from the probability flow ODE / flow-matching dynamics, we consider the
associated flow map \(T(x)=X_{1}(x)\).
The key observation is that an \(L^{1}\)-in-time bound on the maximal eigenvalue
\(\lambda_{\max}(\nabla v_{t})\) integrates into a global Lipschitz bound on \(T\) by a
Grönwall argument. 

Once a dimension-free Lipschitz transport \(T\) is available, functional inequalities follow
by stability under Lipschitz change of variables. As a result, Theorem~\ref{theo:functional-ineq-flowmatching} provides our main conclusion in this direction: every
\((\alpha,\beta,K)\)-weakly log-concave probability measure $p^\star$ (Definition~\ref{defi:weakconcave} with $\alpha,K>0$ and $\beta\in (0,1]$) satisfies a \emph{dimension-free}
family of \(\Psi\)-Sobolev inequalities. In particular, \(p^\star\) satisfies
Poincar\'e and logarithmic Sobolev inequalities:
\[
\mathrm{Var}_{p^\star}(f) \leq C \int_{\mathbb{R}^{d}} |\nabla f|^{2} dp^\star,
\qquad
\mathrm{Ent}_{p^\star}(f^{2}) \leq C \int_{\mathbb{R}^{d}} |\nabla f|^{2} dp^\star,
\]
with $C>0$ depending only on \((\alpha,\beta,K)\) and not on \(d\). Compared to the results discussed in Section~\ref{sec:previouswork}, our functional-inequality consequences apply under strictly more general assumptions on the target: relative to \cite{brigati2024heat} we move from the log-Lipschitz regime \(\beta=1\) to the genuinely rough Hölder regime \(\beta\in(0,1]\), and relative to \cite{Stephanovitch2024smooth} we treat general strongly convex reference potentials \(u\) (not only Gaussian tilts) while obtaining constants that are dimension-free.

\subsection{Models and notations}
\subsubsection{Notations}
\paragraph{Linear algebra and differential operators.}
We work in $\R^d$ with inner product $\langle x,y\rangle$ and Euclidean norm $\|x\|$.
For $A\in\R^{d\times d}$, $\|A\|_{\text{op}}$ denotes the operator norm.
If $A$ is symmetric, $\lambda_{\max}(A)$ and $\lambda_{\min}(A)$ denote its largest and smallest eigenvalues.
For a general  matrix $A$, we set
$
\lambda_{\max}(A):=\lambda_{\max} \left(\tfrac{A+A^\top}{2}\right),$
so that $\lambda_{\max}(\nabla v)$ naturally controls one-sided Lipschitz constants.
We write $\text{Id}$ for the identity matrix, $\text{Tr}(A)$ for the trace, and $u\otimes v:=uv^\top$.
For $f:\R^d\to\R^k$ differentiable, $\nabla f(x)$ is its Jacobian; if $f:\R^d\to\R$ is $C^2$ then
$\nabla^2 f(x)$ is the Hessian and $\Delta f(x):=\text{Tr}(\nabla^2 f(x))$.
For a vector field $v:\R^d\to\R^d$, $\nabla \cdot v$ denotes its divergence.
Time derivatives are written $\partial_t$; for trajectories $t\mapsto X_t$ we also use $\dot X_t$.

\paragraph{Measures and Wasserstein distance.}
$\mathcal P(\R^d)$ is the set of Borel probability measures on $\R^d$, and $\mathcal P_2(\R^d)$ those with finite second moment.
For $p\in[1,\infty)$ and a random vector $X$, $\|X\|_{L^p}:=(\E\|X\|^p)^{1/p}$.
For $\mu\in\mathcal P(\R^k)$ and integrable $f:\R^k\to\R^m$, define the centered map
\[
H_f^\mu(z):=f(z)-\int f(z') d\mu(z').
\]
For real-valued $f,g$, we set
\[
\text{Var}_\mu(f):=\int (H_f^\mu(z))^2)d\mu(z),
\qquad
\text{Cov}_\mu(f,g):=\int H_f^\mu(z)H_g^\mu(z)d\mu(z).
\]
For $\mu,\nu\in\mathcal P_2(\R^d)$, the squared Wasserstein distance corresponds to 
\[
W_2^2(\mu,\nu):=\inf_{\pi\in\Pi(\mu,\nu)}\int_{\R^d\times\R^d}\|x-y\|^2 d\pi(x,y),
\]
where $\Pi(\mu,\nu)$ is the set of couplings.
For a measurable map $T$, $T_\#\mu$ denotes the pushforward measure of $\mu$ by $T$.

\paragraph{One-sided Lipschitz constants.}
For a vector field $v:\R^d\to\R^d$, we say that $v$ is one-sided Lipschitz with constant $L$ if
\[
\langle v(x)-v(y),x-y\rangle \le L\|x-y\|^2,\qquad \forall x,y\in\R^d,
\]
and we use the pointwise bound $\lambda_{\max}(\nabla v(x))$ as the infinitesimal version of this control.

\paragraph{Gaussian and SDE.}
We denote the standard Gaussian measure on $\R^d$ by $\gamma_d$.
For $m\in\R^d$ and $\sigma>0$, $\mathcal N(m,\sigma^2\text{Id})$ is the Gaussian law with mean $m$ and covariance $\sigma^2\text{Id}$,
and $\varphi^{m,\sigma}$ denotes its density.
$(W_t)_{t\ge 0}$ (or $(B_t)_{t\ge 0}$) denotes a standard $d$-dimensional Brownian motion.
We consider It\^o SDEs of the form
\[
dX_t=a_t(X_t) dt+b_t dW_t,
\]
where $b_t$ may be scalar or matrix-valued depending on the context.

\paragraph{Constants.}
$C>0$ denotes a \textbf{dimension-free} positive constant whose value may change from line to line.
We write $A\lesssim B$ if $A\le CB$ for such a constant $C$.

\subsubsection{Models}\label{sec:models}

\paragraph{Gaussian-mixture paths and ODE.}
Let $\pi\in \mathcal{P}(\mathbb{R}^d\times \mathbb{R}^d)$ be a coupling between an initial law $p_0$ and the target law $p^\star$, and let
$Z\sim\pi$ denote the associated latent variable. Let $\xi\sim\mathcal N(0,I_d)$ be independent of $Z$.
Given a map $m_t:\mathrm{Supp}(\pi)\to\R^d$ and a scalar schedule $\sigma_t\ge 0$, we consider the
Gaussian-mixture path
\begin{equation}\label{eq:bridge_general}
    X_t = m_t(Z) + \sigma_t \xi,\qquad t\in[0,1],
\end{equation}
and write $p_t$ for the density of $X_t$. For $m\in\R^d$ and $\sigma>0$, we denote by
$\varphi^{m,\sigma}$ the density of $\mathcal N(m,\sigma^2 I_d)$, so that
\[
p_t(x)=\int \varphi^{m_t(z),\sigma_t}(x) \pi(dz).
\]
For fixed realizations $(Z,\xi)=(z,\eta)$, assuming that the maps $t\mapsto m_t(z)$, $t\mapsto \sigma_t$ are differentiable and $\sigma_t>0$, we have
\[
\dot X_t(z,\eta)=\partial_t m_t(z)+\frac{\sigma_t'}{\sigma_t}\bigl(X_t(z,\eta)-m_t(z)\bigr),
\]
using that $
\eta=\frac{X_t(z,\eta)-m_t(z)}{\sigma_t}$.
This motivates the definition of the conditional velocity field
\[
v_t^{\mathrm c}(z,x):=\partial_t m_t(z)+\frac{\sigma_t'}{\sigma_t}\bigl(x-m_t(z)\bigr),\qquad \sigma_t>0,
\]
so that $\dot X_t=v_t^{\mathrm c}(Z,X_t)$ holds pathwise on the set $\{\sigma_t>0\}$.
We then define the velocity field as the conditional projection
\[
v_t(x) := \E \left[v_t^{\mathrm c}(Z,X_t)\mid X_t=x\right]
        = \E \left[\dot X_t\mid X_t=x\right],
\]
so that the curve of densities $(p_t)_{t\in[0,1]}$ satisfies the continuity equation
\begin{equation}\label{eq:continuity_equation}
\partial_t p_t+\nabla \cdot(p_t v_t)=0.
\end{equation}
A direct computation yields for every $x\in\R^d$ and $t$ such that $\sigma_t>0$,
\begin{equation}\label{eq:vf}
\begin{aligned}
v_t(x)
&=\int\Bigl(\frac{\sigma_t'}{\sigma_t}(x-m_t(z))+m_t'(z)\Bigr) 
\frac{\varphi^{m_t(z),\sigma_t}(x)}{\int \varphi^{m_t(z'),\sigma_t}(x) \pi(dz')} \pi(dz) \\
&=-\sigma_t'\sigma_t \nabla \log p_t(x) \;+\; \E \left[m_t'(Z)\mid X_t=x\right].
\end{aligned} 
\end{equation}
 The vector score function $\nabla \log p_t(x)$ captures how the smoothed density varies in space and acts as a denoising force toward regions where $p_t$ is large. The second term, $\E[m_t'(Z)\mid X_t=x]$, is a mean-velocity term: it is the posterior average of how the component means move along the interpolation.

\paragraph{Assumptions on the dynamic.}
Our main estimates  are first established abstractly under assumptions that only involve
the triplet $(m,\pi,\sigma)$, and are then specialized
to concrete generative-model constructions by verifying these abstract hypotheses in each case.

We will focus on three classical regimes, corresponding respectively to (i) deterministic flow matching in the
sense of \cite{lipman2023flow}, (ii) a variant of stochastic interpolants \citep{albergo2025stochastic}, 
and (iii) score-based diffusion models \citep{song2020score}.  In all three settings we take the reference
distribution to be the standard Gaussian $\gamma_d$ and we assume throughout that the target density
$p^\star=e^{-u+a}$ is $(\alpha,\beta,K)$-weakly log-concave in the sense of Definition~\ref{defi:weakconcave}.

We start with Lipman flow matching, where $m_t$ depends only on the target endpoint and the noise level
decreases to $0$ as $t\to 1$; the additional lower bounds on $f_t$ and $\sigma_t$ prevent degeneracies near the
endpoints and ensure that the projected drift remains well behaved.

\begin{assumption}[Lipman flow matching]\label{assum:lipman}
The vector field $v:[0,1]\times \mathbb{R}^d\rightarrow \mathbb{R}^d$ is of the form \eqref{eq:vf} with
\begin{itemize}
    \item $\pi=\gamma_d\otimes p^\star$,
    \item $p^\star=e^{-u+a}$ is $(\alpha,\beta,K)$-weakly log-concave with \(\alpha,K>0\) and \(\beta\in(0,1]\),
    \item $m_t((x_0,x_1))=f_tx_1$ with $f\in C^1_{\text{a.e.}}([0,1],[0,1])$ increasing satisfying $f_0=0=1-f_1$,
    \item $\sigma\in C^1_{\text{a.e.}}([0,1],[0,1])$ is decreasing and satisfies $\sigma_1=0=1-\sigma_0$,
    \item there exists $\gamma  \in (0,1)$ such that $\gamma  \leq \sigma_{\frac{1}{2}}\wedge f_{\frac{1}{2}}$ and $f_{1-\gamma  }^2\geq \sigma_{1-\gamma  }$.
\end{itemize}
\end{assumption}
\noindent
In Assumption~\ref{assum:lipman}, the marginal satisfies for each $t\in[0,1]$,
\[
X_t=f_tY+\sigma_t\xi,\qquad Y:=X_1\sim p^\star,\  \xi\sim\mathcal N(0,I_d) \text{ independent}.
\]
An important feature here is that the map $m_t((x_0,x_1))=f_tx_1$ \emph{does not depend on $x_0$}: the latent variable carried by the reference Gaussian is irrelevant, and the only randomness beyond $Y$ is the additive noise $\xi$. 

We now turn to the stochastic-interpolant variant, which mixes both endpoints in $m_t$ and uses a
non-monotone noise schedule that vanishes at both $t=0$ and $t=1$ (with a single peak in between); the extra
balance conditions control the interaction between the deterministic interpolation $(f_t,g_t)$ and the random
perturbation $\sigma_t$.

\begin{assumption}
[Stochastic-interpolant flow matching]\label{assum:sto-int} The vector field $v:[0,1]\times \mathbb{R}^d\rightarrow \mathbb{R}^d$ is of the form \eqref{eq:vf} with
\begin{itemize}
    \item $\pi=\gamma_d\otimes p^\star$,
    \item $p^\star=e^{-u+a}$ is $(\alpha,\beta,K)$-weakly log-concave with \(\alpha,K>0\) and \(\beta\in(0,1]\),
    \item $m_t((x_0,x_1))=g_tx_0+f_tx_1$ with $f,g\in C^1_{\text{a.e.}}([0,1],[0,1])$ increasing and decreasing respectively and satisfying $f_0=1-f_1=1-g_0=g_1=0$,
    \item $\sigma\in C^1_{\text{a.e.}}([0,1],[0,1])$ satisfies $\sigma_1=0=\sigma_0$. Furthermore, there exists $\delta \in (0,1)$ such that for all $t\in [0,\delta)$, $\sigma_t'> 0$ and for all $t\in (\delta,1)$ $\sigma_t'< 0$,
    \item There exists $q\in [0,1]$ such that $|g_t'f_t-f_t'g_t|\leq K \sigma_t^q$ and $\int_0^1 \sigma_t^{-1+q}dt\leq K$ 
    \item for all $t\in [0,1]$, $g_t+f_t\geq \gamma$ and for all $t\in (0,\delta)$ $g_t\geq \gamma$
\end{itemize}
\end{assumption}
Compared to Lipman flow matching, Assumption~\ref{assum:sto-int} allows the interpolant to retain a contribution of the variable $X_0$ (through a coefficient \(g_t\)) together with an additional injected noise \(\sigma_t\). At the level of marginals this gives
\[
X_t = f_t Y + \bar g_t \xi,
\qquad \bar g_t^2 := g_t^2+\sigma_t^2,
\qquad Y\sim p^\star,\ \xi\sim\mathcal N(0,I_d) \text{ independent},
\]
so the path remains a Gaussian mixture but with a variance schedule \(\bar g_t\) that is no longer constrained to be monotone. This seemingly mild change is precisely what complicates the one-sided Lipschitz analysis in Section~\ref{sec:onesidedlipreguv}: the sign of \(\bar g_t'\) determines whether the score contribution is dissipative or expansive, so one can no longer rely solely on the log-concave smoothing mechanism available when the noise decreases. This stochastic-interpolant model is standard in the recent generative-modeling literature and has been found to perform strongly in numerical experiments  \citep{albergo2025stochastic,boffi2025flowmapmatching}. 

Finally, we single out the rescaled diffusion-model regime, which corresponds to the Ornstein--Uhlenbeck
interpolation; we state it separately since it is a central benchmark in modern generative modeling.

\begin{assumption}[Rescaled Diffusion models]\label{assum:diffusion}
The vector field $v:[0,1]\times \mathbb{R}^d\rightarrow \mathbb{R}^d$ is of the form \eqref{eq:vf} with
\begin{itemize}
    \item $\pi=\gamma_d\otimes p^\star$,
    \item $p^\star=e^{-u+a}$ is $(\alpha,\beta,K)$-weakly log-concave with \(\alpha,K>0\) and \(\beta\in(0,1]\),
    \item $m_t((x_0,x_1))=tx_1$  and $\sigma_t=\sqrt{1-t^2}$.
\end{itemize}
\end{assumption}
Assumption~\ref{assum:diffusion} corresponds to the reversed path of the Ornstein-Uhlenbeck path that underlies a large class of score-based diffusion models \cite{song2020score}. Concretely, the choice
$m_t(x_0,x_1)=t x_1$ and $\sigma_t=\sqrt{1-t^2}$ yields the interpolation
$$X_t=tY+\sqrt{1-t^2} \xi,$$ i.e. the Ornstein-Uhlenbeck semigroup rescaled in
$t\in(0,1]$. Equation~\eqref{eq:vf} gives
$$v_t(x)=\frac{1}{t}(\nabla\log p_t(x)+x),$$
which corresponds to the \emph{probability flow ODE} \citep{chen2024probability}.
To connect Assumption~\ref{assum:diffusion} to the usual reverse-time diffusion sampler, remark that the continuity equation \eqref{eq:continuity_equation} can be transformed using the score function:
\[
\partial_t p_t+\nabla \cdot(p_t v_t)=0
\quad\Longleftrightarrow\quad
\partial_t p_t=-\nabla \cdot(p_t a_t)+\frac{b_t^2}{2} \Delta p_t,
\qquad
a_t:=v_t+\frac{b_t^2}{2} \nabla \log p_t,
\]
for any $b_t\geq 0$. In other words, the same curve $(p_t)_{t\in [0,1]}$ may be regarded as solving a Fokker-Planck equation with diffusion $\frac{b_t^2}{2}$. Consequently, $(p_t)_{t\in [0,1]}$ can be realized as the family of time-marginals of the It\^o dynamics
\[
dX_t=a_t(X_t) dt+b_t dB_t.
\]
Taking a forward time $T>0$ and a diffusion coefficient $b_t\equiv\sqrt2$, after the time reparametrization $t=e^{-s}$, this coincides with the classical time-reversal of the Ornstein-Uhlenbeck diffusion:
\[
dX_s=(X_s+2\nabla \log q_{T-s}(X_s)) ds+\sqrt{2} dB_s,
\]
with $q_s=\text{Law}(e^{-s}Y+\sqrt{1-e^{-2s}}\xi)$. 

\paragraph{Scope of the assumptions.}
Ultimately, the regularity estimates we seek are not tied to a specific training objective but to the
regularity of the intermediate mean laws $$\mu_t := (m_t)_{\#}\pi.$$
Indeed,
the bounds of Section~\ref{sec:onesidedlipreguv} are derived abstractly from two types of information:
\begin{itemize}
    \item[(i)]structural regularity of \(\mu_t\) (weak log-concavity with a \(\beta\)-Hölder perturbation,
and in some places a bounded-perturbation variant), and
\item[(ii)] a quantitative control of the fluctuations of \(\partial_t m_t(X_0,X_1)\) under the posterior\  $\text{Law}((X_0,X_1)|X_t=x)$.
\end{itemize}
In principle, this viewpoint would allow one to treat much more general couplings \(\pi\) and nonlinear
families \(m_t\); however, verifying such assumptions on \(\mu_t\) outside of structured
Gaussian-mixture models is typically hard: the density of \(\mu_t\) depends on the global
geometry of the pushforward (inverse Jacobian determinants), and obtaining dimension-free  concentration information together with a
uniform Hölder control of the log-density is delicate even when the endpoints are regular.
This is precisely why we focus on the product coupling and affine
interpolations, for which \(\mu_t\) reduces to explicit
Gaussian smoothing and rescaling operations that preserve the
weak log-concavity.
Nevertheless, the theory in Section~\ref{sec:onesidedlipreguv} is formulated only in terms of the
regularity of \((m_t)_{\#}\pi\): consequently, any alternative construction for instance
Optimal Transport-based schemes \citep{tong2024improving}, would fall within the scope of our results as soon as one can verify the same
regularity hypotheses.

\addtocontents{toc}{\protect\setcounter{tocdepth}{2}}
\section{Lipschitz estimates}\label{sec:lipreguv}
\subsection{One-sided Lipschitz regularity}\label{sec:onesidedlipreguv}
\subsubsection{General one-sided Lipschitz estimate}\label{sec:onesidedlipregugene}

In this section we derive an abstract control of the one-sided Lipschitz constant of the canonical velocity field $v_t$ from Section~\ref{sec:models} solving the continuity equation
\[
\partial_t p_t+\nabla \cdot(p_t v_t)=0,
\]
where $p_t$ is the law of
the Gaussian-mixture path
\[
X_t=m_t(Z)+\sigma_t \xi,\qquad \xi\sim \mathcal N(0,\mathrm{Id})\  \text{independent of }Z.
\]
Recall (see the discussion leading to \eqref{eq:vf}) that $v_t$ can be decomposed as
\begin{equation}\label{eq:jffdodjhdhhsdhdjddd}
v_t(x)= -\sigma_t'\sigma_t \nabla \log p_t(x)\;+\;\mathbb E \left[m_t'(Z)\mid X_t=x\right].
\end{equation}
The first term is the ``score part'' (a denoising force produced by the Gaussian smoothing),
while the second term is the posterior average of the velocity of the means $m_t(Z)$.

A convenient way to understand derivatives of conditional expectations is to consider the posterior law of $Z$ given $X_t=x$,
denoted by $\pi^{t,x}=\mathrm{Law}(Z\mid X_t=x)$. The regularization created by the Gaussian noise $\sigma_t\xi$ induces a quadratic
confinement at scale $\sigma_t^{-2}$ which combines with the curvature of the mean law $(m_t)_\#\pi$.
When $(m_t)_\#\pi$ has a strongly log-concave core with curvature $\alpha_t$, these two effects add up into an effective curvature
$\alpha_t+\sigma_t^{-2}$, suggesting that typical posterior fluctuations occur at scale
\[
(\alpha_t+\sigma_t^{-2})^{-1}\;\asymp\;\frac{\sigma_t^2}{1+\alpha_t\sigma_t^2}.
\]
This heuristic is precisely what will explain the factor $(\alpha_t\sigma_t^2+1)^{-1}$ in the bounds below.

\begin{assumption}\label{assum:1} The probability density $(m_t)_{\#}\pi:\mathbb{R}^d\rightarrow \mathbb{R}$ is $(\alpha_t,\beta,K_t)$-weakly log-concave for $\alpha_t,K_t>0$ and $\beta\in (0,1]$.
\end{assumption} 

This hypothesis supposes that $(m_t)_\#\pi$ behaves like a strongly log-concave density (with parameter $\alpha_t$), up to a $\beta$-Hölder
tilt controlled by $K_t$. The strongly convex part yields dimension-free concentration and covariance control, while the Hölder
perturbation quantifies how far we allow the law of $m_t(Z)$ to deviate from a purely log-concave reference.
As a consequence, posterior variances under $\pi^{t,x}$ are controlled at the scale
$\sigma_t^2/(\alpha_t\sigma_t^2+1)$, uniformly in dimension; this is the main mechanism behind the next results.

\begin{assumption}[Control of $m_t' $]\label{assum:varmprime}
    There exists $L:[0,1]\rightarrow \mathbb{R}_{>0}$ such that, for all $x\in\mathbb R^d$, $t\in [0,1]$ and unit $h\in\mathbb S^{d-1}$,
\[
\operatorname{Var}_{\pi^{t,x}} \big(h^\top m_t'(Z)\big)
 \le 
L_t^2 \frac{\sigma_t^2}{\alpha_t\sigma_t^2+1}.
\]
\end{assumption}

The velocity $v_t$ \eqref{eq:jffdodjhdhhsdhdjddd} contains the regression term $\mathbb E[m_t'(Z)\mid X_t=x]$. Differentiating such conditional expectations produces
conditional covariance terms, hence it is natural that a variance control appears.
Assumption~\ref{assum:varmprime} requires that, in every direction, the posterior fluctuations of the time-velocity $m_t'(Z)$
are at most of order $L_t$ times the typical posterior scale of $m_t(Z)$, namely $\sigma_t^2/(\alpha_t\sigma_t^2+1)$.
In concrete models, $m_t$ is often affine in time (or close to it), so that $m_t'$ is comparable to $m_t$; then the inequality above
is a quantitative way of stating that the mixture means do not move too erratically relative to their posterior uncertainty. Under Assumptions~\ref{assum:1} and~\ref{assum:varmprime} we obtain the following bound on the largest eigenvalue of the Jacobian of the vector field.

\begin{theorem}\label{theo:onesidedlip2}Let $v:[0,1]\times \mathbb{R}^d\rightarrow \mathbb{R}^d$ of the form \eqref{eq:vf} with $\sigma\in C^1_{\text{a.e.}}([0,1],[0,1])$ and $(m_t)_t$ satisfying Assumptions~\ref{assum:1} and~\ref{assum:varmprime}. Then for $x\in \mathbb{R}^d$ and $t\in [0,1]$ such that $\sigma_t>0$, we have for all $\xi \in \{0,1\}$,
\begin{itemize}
    \item if $\sigma_t'\leq 0$, then $$\lambda_{\max}\left(\nabla v_t(x)\right)\leq  \frac{1}{\alpha_t\sigma_t^2+1}\left(C\exp\Big(\frac{CK_t^2}{(\alpha_t+\sigma_t^{-2})^{\beta}}\Big)\Big)\Biggl( \underbrace{L_t}_{m_t' \text{ term}}-\underbrace{\sigma_t'\sigma_t^{-1}(K_t\sigma_t^\beta)^\xi}_{\text{score stability}}\Biggl)+\underbrace{\sigma_t\sigma_t'\alpha_t}_{\text{log-concave behavior}}\right),$$
    \item if $\sigma_t'\geq 0$ and $\nabla^2 u_t \preceq A_t \text{Id}$ in the decomposition $(m_t)_{\#} \pi =\exp(-u_t+a_t)$ with $u_t$ convex and $a_t$ $\beta$-Hölder (from Assumption~\ref{assum:1}), then
   $$\lambda_{\max}\left(\nabla v_t(x)\right)\leq  \frac{1}{\alpha_t\sigma_t^2+1}\left(C\exp\Big(\frac{CK_t^2}{(\alpha_t+\sigma_t^{-2})^{\beta}}\Big)\Biggl( \underbrace{L_t}_{m_t' \text{ term}}+\underbrace{\sigma_t'\sigma_t^{-1}(K_t\sigma_t^\beta)^\xi}_{\text{score stability}}\Biggl)+\underbrace{\sigma_t\sigma_t'A_t}_{\text{bounded Hessian  behavior}}\right),$$
\end{itemize}
with $C>0$ depending only on $\beta$ and not the dimension $d$.
\end{theorem}
The proof of Theorem~\ref{theo:onesidedlip2} can be found in Section~\ref{sec:prooftheo12}.
The bound has three contributions, which are highlighted in the statement.
First, the $m_t'$ term reflects the cost of transporting a mixture whose component means move with time; it is governed by $L_t$.
Second, the ``score stability'' term quantifies the sensitivity of the score contribution $-\sigma_t'\sigma_t\nabla\log p_t$,
and exhibits two regimes through $\xi\in\{0,1\}$: taking $\xi=0$ gives a robust bound, while $\xi=1$ reveals an additional gain
when the perturbation is small at the noise scale, i.e. when $K_t\sigma_t^\beta\ll 1$.
Finally, the last term is the contribution of the strongly log-concave core.
When $\sigma_t'\le 0$ (noise decreasing), $\sigma_t\sigma_t'\alpha_t\le 0$ is dissipative, reflecting contraction created by decreasing noise.
When $\sigma_t'\ge 0$ (noise increasing), this mechanism becomes potentially expansive and one needs the upper curvature bound
$\nabla^2 u_t\preceq A_t\mathrm{Id}$ to control it.

The exponential prefactor in Theorem~\ref{theo:onesidedlip2} arises from allowing an unbounded Hölder perturbation. 
If the perturbation is uniformly bounded, one can avoid this price and obtain a sharper estimate, see Theorem~\ref{theo:onesidedlip} in Section~\ref{sec:prooftheo12}.

\subsubsection{Lipman flow matching setting}\label{sec:lipmanonesided}

In this subsection we specialize the abstract one-sided Lipschitz estimate of Theorem~\ref{theo:onesidedlip2} to the Lipman flow matching regime. The condition $\sigma_0>0$ of the Lipman case is not merely technical: it prevents a small-noise singularity at early times and
keeps the conditional distributions uniformly regularized. Moreover, in the Lipman schedule the noise level
is decreasing, which places the velocity field in a dissipative regime where one-sided Lipschitz bounds follow
from lower curvature information only. This contrasts with the increasing-noise case treated in
Section~\ref{sec:stoponesidedlip}, where additional structure is needed precisely during the expansive phase.

\begin{assumption}[Behavior of $\sigma_t^2$ when $\sigma_0>0$]\label{assum:sigma} The function $\sigma$ belongs to $ C^1_{\text{a.e.}}([0,1],[0,1])$ and satisfies $\lambda(\{\sigma=0\})=0$. Furthermore,
there exists $\delta \in [0,1)$ such that for $t\in [0,\delta]$, $ |\sigma_t'|\vee \sigma_t^{-1}\le A$ and for $t\in [\delta,1]$, $\sigma_t'\leq 0.$
\end{assumption}

Assumption~\ref{assum:sigma} is designed to split the analysis into two time regimes. On a short initial window one works under uniform non-degeneracy: the noise is bounded away from $0$ and does not vary too fast. This prevents blow-ups of the terms like $1/\sigma_t$ and $|\sigma_t'|/\sigma_t$. After that, one only needs the monotone property $\sigma_t'\leq 0$ which makes the score part of the drift dissipative in the one-sided estimate and allows to control the singular behavior as $\sigma_t\to 0$ near $t=1$. We now state a general integrable one-sided Lipschitz estimate corresponding to the regime $\sigma_0>0$. 

\begin{theorem}[$\sigma_0>0$]\label{theo:onesidedlipinteg}Let $v:[0,1]\times \mathbb{R}^d\rightarrow \mathbb{R}^d$ of the form \eqref{eq:vf} with $(m,\pi,\sigma)$  satisfying Assumptions~\ref{assum:1}, \ref{assum:varmprime} and~\ref{assum:sigma}. Suppose that there exists $A>0$ such that
\begin{itemize}
    \item $\int_0^1\frac{L_t}{\alpha_t\sigma_t^2+1}\leq A $,
    \item there exists $s\in [0,1),\chi \in [0,\beta)$ such that for all $t\geq s$, $K_t\leq A\sigma_t^{-\chi}$,\item For all $t\in [0,\delta\vee s]$, $K_t^2\alpha_t^{-\beta}\leq A$.
\end{itemize}
 Then, for all $z\in [0,1]$
$$\int_z^1 \sup_{x\in   \mathbb{R}^d} \lambda_{\max}\left(\nabla v_t(x)\right)dt\leq C,$$
with $C>0$ depending only on $A,\beta,\chi,\sigma_\delta,\sigma_s$ and not the dimension $d$.  
\end{theorem}

The proof of Theorem~\ref{theo:onesidedlipinteg} can be found in Section~\ref{sec:theo:onesidedlipinteg}. The three additional conditions needed for integrability should be read as follows.
\begin{itemize}
    \item The quantity $\frac{L_t}{\alpha_t\sigma_t^2+1}$ measures the possible expansion coming from the mean-velocity term, attenuated by posterior concentration. Indeed, when either $\alpha_t$ is large (strong curvature of the log-concave core) or $\sigma_t$ is not too small, conditioning on $X_t=x$ is more stable, and the contribution of $m_t'(Z)$ to $\nabla v_t$ is correspondingly reduced.
    \item The estimate $K_t\le A\sigma_t^{-\chi}$ with $\chi<\beta$ is a compatibility condition between the vanishing-noise limit and the $\beta$-Hölder perturbation in the weak log-concavity structure. Heuristically, it ensures that at the noise scale $\sigma_t$ the Hölder perturbation becomes small, which prevents the score term from generating an exponential blow-up.
    \item The bound $K_t^2\alpha_t^{-\beta}\le A$ on the initial window $[0,\delta\vee s]$ is an early-time condition ensuring that the estimates used before entering the small-noise regime do not incur uncontrolled constants.
\end{itemize}
Applying Theorem~\ref{theo:onesidedlipinteg}  to the Lipman flow-matching parametrization, we obtain the following explicit, dimension-free integrable one-sided Lipschitz estimate.

\begin{corollary}\label{coro:lipmafm} Suppose that we are under the Lipman flow matching setting of Assumption~\ref{assum:lipman}.
 Then, for all $z\in [0,1]$,
$$\int_z^1 \sup_{x\in \mathbb{R}^d} \lambda_{\max}\left(\nabla v_t(x)\right)dt\leq C,$$
with $C>0$ depending only on $\alpha,\beta,K,\gamma,\sigma_{1-\gamma}$ and not the dimension $d$.  
\end{corollary}

The proof of Corollary~\ref{coro:lipmafm} can be found in Section~\ref{sec:coro:lipmafm}. This result shows that the assumptions of Theorem~\ref{theo:onesidedlipinteg} can be verified explicitly in the Lipman flow matching parametrization. In particular, the weak log-concavity parameters $(\alpha_t,K_t)$ and the mean-velocity scale $L_t$ become concrete functions of the schedules, and the conclusion yields an integrable one-sided Lipschitz bound independent of the dimension.

\subsubsection{Stochastic-interpolant flow matching setting}\label{sec:stoponesidedlip}

We now specialize the general one-sided Lipschitz theory of Theorem~\ref{theo:onesidedlip2} to the
stochastic-interpolant flow-matching construction.
Compared to the Lipman setting, the key difference is that the auxiliary noise level
$\sigma_t$ vanishes at both endpoints and is typically non-monotone:
one injects noise at the beginning (to randomize the interpolation) and then removes it
to reach the target distribution at terminal time.
This apparently mild change has an important consequence for stability.

The following assumption encodes the ``single bump'' behavior of the noise schedule that is standard
in stochastic interpolants.

\begin{assumption}[Behavior of 
$\sigma_t^2$ when $\sigma_0=0$]\label{assum:sigma2} $\sigma\in C^1_{\text{a.e.}}([0,1],[0,1])$ and there exists $\delta \in (0,1)$ such that for all $t\in [0,\delta)$, $\sigma_t'> 0$ and for all $t\in (\delta,1)$ $\sigma_t'< 0$ .
\end{assumption}

\noindent
Assumption~\ref{assum:sigma2} allows us to separate the dynamics into two phases:
(i) a \emph{noise-injection} phase $t\in[0,\delta)$ where the score part may be expansive, and
(ii) a \emph{denoising} phase $t\in(\delta,1)$ where the score part is dissipative.
The next theorem is the analogue of the monotone-noise case, adapted to schedules that start at $\sigma_0=0$.
Its assumptions mirror the structure of the general bound obtained earlier.

\begin{theorem}[$\sigma_0=0$]\label{theo:onesidedlipinteg0}Let $v:[0,1]\times \mathbb{R}^d\rightarrow \mathbb{R}^d$ of the form \eqref{eq:vf} with $(m,\pi,\sigma)$  satisfying Assumptions~\ref{assum:1}, \ref{assum:varmprime} and \ref{assum:sigma2}. Suppose that there exists $A>0$ such that
\begin{itemize}
    \item $\int_0^1\frac{L_t}{\alpha_t\sigma_t^2+1}\leq A $,
    \item there exists $\chi \in [0,\beta)$ such that for all $t\in [0,1]$, $K_t\leq A\sigma_t^{-\chi}$.
    \item For all $t$ such that $\sigma_t'>0$, writing $(m_t)_{\#} \pi =\exp(-u_t+a_t)$ with $u_t$ convex and $a_t$ $\beta$-Hölder (from Assumption~\ref{assum:1}), $u_t$ satisfies $\nabla^2 u_t \preceq A \text{Id}$.
\end{itemize}
 Then,
$$\int_0^1 \sup_{x\in \mathbb{R}^d} \lambda_{\max}\left(\nabla v_t(x)\right)dt\leq C,$$
with $C>0$ depending only on $A,\beta,\chi$ and not the dimension $d$.  
\end{theorem}

The proof of Theorem~\ref{theo:onesidedlipinteg0} can be found in Section~\ref{sec:theo:onesidedlipinteg0}. \noindent
Theorem~\ref{theo:onesidedlipinteg0} should be read as the analogue of Theorem~\ref{theo:onesidedlipinteg} (the case $\sigma_0>0$),
adapted to the  regime where $\sigma_0=0$ and $\sigma_t$ is typically non-monotone.
In Theorem~\ref{theo:onesidedlipinteg}, the growth control $K_t\lesssim \sigma_t^{-\chi}$ is only required for $t\ge s$:
    this matches the fact that, when $\sigma_0>0$ and $\sigma_t$ stays bounded away from $0$ initially,
    the only potentially singular regime is near $t=1$ (where $\sigma_t$ becomes small).
    Here, since $\sigma_t$ is also small near $t=0$, we need the same type of control for all times
    to prevent the Hölder perturbation constants from blowing up at either endpoint.

Additionally, in Theorem~\ref{theo:onesidedlipinteg}, one assumes $K_t^2\alpha_t^{-\beta}\le A$ on an initial interval,
    which
    keeps the exponential envelope in the general bound uniformly bounded.
    When $\sigma_0=0$, such a strategy cannot rely on $\sigma_t^{-1}$ being bounded near $t=0$.
    Instead, the new difficulty is the possible expansion created by the score term when $\sigma_t'>0$;
    the condition $\nabla^2 u_t \preceq A \mathrm{Id}$ precisely prevents this expansion from becoming too large,
    and Assumption~\ref{assum:sigma2} ensures that this ``bad sign'' only occurs up to time $\delta$. The following result shows that the assumptions of Theorem~\ref{theo:onesidedlipinteg0} can be verified explicitly in the Stochastic-interpolant flow matching parametrization.

\begin{corollary}\label{coro:notlipmafm} Suppose that we are under the Stochastic-interpolant flow matching setting of Assumption~\ref{assum:sto-int}.
 Then, for all $z\in [0,1]$
$$\int_z^1 \sup_{x\in \mathbb{R}^d} \lambda_{\max}\left(\nabla v_t(x)\right)dt\leq C,$$
with $C>0$ depending only on $\alpha,\beta,K,\gamma,\sigma_{\delta}$ and not the dimension $d$.  
\end{corollary}
The proof of Corollary~\ref{coro:notlipmafm} can be found in Section~\ref{sec:coro:notlipmafm}.

\subsubsection{Diffusion models setting}
\label{sec:score_osl}
In the diffusion-model regime, the prescribed probability path is
the (rescaled) Ornstein--Uhlenbeck interpolation
\[
X_t \;=\; tY+\sqrt{1-t^2} \xi, \qquad Y\sim p^\star,\ \xi\sim\mathcal N(0,I_d) \text{ independent}.
\]
The associated \emph{probability flow ODE} is exactly the flow-matching vector field
\(
v_t(x)=\E[\dot X_t | X_t=x].
\)
Writing the score as \(s(t,x):=\nabla\log p_t(x)\), one has the closed-form identity
\[
v_t(x)=\frac{1}{t}\bigl(x+s(t,x)\bigr).
\]
Since the path \((X_t)_{t\in[0,1]}\) is a particular case of the Lipman flow-matching setting
 with schedules \(f_t=t\) and \(\sigma_t=\sqrt{1-t^2}\),
we may directly invoke the one-sided Lipschitz estimate of Section~\ref{sec:lipmanonesided}.

\begin{corollary}\label{coro:score_osl_selfcontained}
Suppose that we are under the Diffusion models setting of Assumption~\ref{assum:diffusion} and let $s(t,x)=\nabla\log p_t(x)$ be the score.
Then there exists a constant $C>0$ independent of the dimension such that for all $z\in [0,1]$,
\[
\int_z^1 \frac{1}{t} \sup_{x\in\R^d}\lambda_{\max}\big(\text{Id}+\nabla s(t,x)\big)  dt   \le   C.
\] 
\end{corollary}

The proof of Corollary~\ref{coro:score_osl_selfcontained} can be found in Section~\ref{sec:coro:score_osl_selfcontained}.

\begin{remark}\label{rem:contractive-window}
In the diffusion-model setting, our sharp one-sided Lipschitz control implies that the
reverse drift enjoys a genuinely contractive regime over a non-negligible initial portion of the
backward dynamics. More precisely, for the reverse drift $a_t$,
one can choose an explicit $\kappa>0$
such that, for all $t\in(0,\kappa]$ and all $x,y\in\R^d$,
\[
\langle x-y,  a_t(x)-a_t(y)\rangle \le -\frac12 \|x-y\|^2 .
\]
Consequently, one
may discretize the first stage of the sampling procedure with constant step sizes on $(0,\kappa]$,
and only switch later to a geometric discretization as $t$ is close to $1$. This two-stage schedule is advocated and analyzed in \cite{arsenyan2025assessing}. We do not exploit this
additional structure here and, for simplicity, we work throughout with purely geometric step-size
discretizations.
\end{remark}

\subsection{Lipschitz regularity}\label{sec:lipschitz}
\subsubsection{General Lipschitz estimates}\label{sec:lipschitzgene}

In this section we control the \emph{two-sided} Lipschitz constant of $v_t(\cdot)$, i.e.
\[
\text{Lip}(v_t)  :=  \sup_{x\in\R^d}\|\nabla v_t(x)\|_{\text{op}}.
\]
This is stronger than the one-sided estimates of Section~\ref{sec:onesidedlipreguv}, where we only control the largest eigenvalue of $\nabla v_t(x)$. Here we want to control the full operator norm, so we must also bound possible strong contraction directions (large negative eigenvalues).

\medskip

The key point is that the Jacobian $\nabla v_t(x)$ admits a decomposition into:
\begin{itemize}
    \item an explicit isotropic term (coming from the Gaussian part of the interpolation),
    \item a covariance term involving posterior fluctuations of $m_t(Z)$,
    \item and a mixed covariance term involving $m_t'(Z)$ and $m_t(Z)$.
\end{itemize}
Therefore, the problem reduces to understanding how the posterior law concentrates. This is exactly where Assumptions~\ref{assum:1} and \ref{assum:varmprime} enter. The first result below is the direct  estimate, obtained by bounding each contribution to $\nabla v_t$ separately.

\begin{corollary}\label{coro:lipregugeneral}
Let $v:[0,1]\times \mathbb{R}^d\rightarrow \mathbb{R}^d$ of the form \eqref{eq:vf} with $\sigma\in C^1_{\text{a.e.}}([0,1],[0,1])$ and $(m_t)_t$ satisfying Assumptions~\ref{assum:1} and \ref{assum:varmprime}. Then for $x\in \mathbb{R}^d$ and $t\in [0,1]$ such that $\sigma_t>0$, we have
$$\sup_{x\in\R^d}\|\nabla v_t(x)\|_{\text{op}}
\le
\frac{|\sigma'_t|}{\sigma_t}
+
C  \frac{1}{1+\alpha_t\sigma_t^2}  
\exp \Big(CK_t^2(\alpha_t+\sigma_t^{-2})^{-\beta}\Big)\left(  \frac{|\sigma'_t|}{\sigma_t}
+L_t\right),$$
with $C>0$ independent of the dimension.
\end{corollary}

The proof of Corollary~\ref{coro:lipregugeneral} can be found in Section~\ref{sec:coro:lipregugeneral}.
The different terms in the bound have the following interpretations:
\begin{itemize}
    \item $\frac{|\sigma_t'|}{\sigma_t}$ is the explicit contribution of the Gaussian part of the interpolation. It is present even in the simplest cases and is therefore unavoidable at this level of generality.
    \item $\frac{1}{1+\alpha_t\sigma_t^2}$ is the posterior concentration gain: stronger curvature or larger noise level  improve stability.
    \item $L_t$ is the cost of the temporal motion $m_t'$, coming from the mixed covariance term.
    \item The exponential factor is the price paid for allowing an unbounded Hölder perturbation in Assumption~\ref{assum:1}.
\end{itemize}
Under the additional assumption that the convex part $u_t$ has bounded Hessian, Section~\ref{sec:cor:refined_one_sided_lip} provides a refined two-sided Lipschitz bound, stated in Corollary~\ref{cor:refined_one_sided_lip}, which matches the structure of the one-sided estimates in Theorem~\ref{theo:onesidedlip2}. The point is that, in this regime, one can exploit cancellations in the score contribution instead of estimating each term separately. As a consequence, for the time schedules typically used in diffusion models, this refined estimate often yields a $(1-t)^{-(2-\beta)/2}$ behavior near $t=1$, rather than the more singular $(1-t)^{-1}$ behavior given by the general bound. In discretization arguments, this can improve the logarithmic factors in the final error estimates. We do not pursue this refinement in the present paper as it would require the convex potential $u$ in the weakly log-concave decomposition $p^\star = e^{-u+a}$ to have bounded Hessian.

\subsubsection{Lipman Flow matching setting}

We now specialize the general Lipschitz estimate of Section~\ref{sec:lipschitzgene} to the Lipman flow-matching regime. In this setting, it is natural to impose assumptions directly on the schedules $f_t$ and $\bar g_t$.
\begin{assumption}\label{assum:lipmanlip} Writing \[
X_t = f_t Y + \bar g_t \xi,\qquad Y:=X_1\sim p^\star,\  \xi\sim\mathcal N(0,I_d) \text{ independent},
\] the model satisfies
\begin{enumerate}
\item There exist $p>0$, $t_0\in(0,1)$ and $\ell\in C^2([t_0,1))$ such that for all $t\in[t_0,1)$,
\[
\bar g_t=(1-t)^p  \ell(t),
\qquad 0<A^{-1} \le \ell(t)\le A,
\qquad \|\ell'\|_{L^\infty([t_0,1))}\le A.
\]
\item For all $t\in [0,1]$, $f_tf_t'\le A(1-t)^{-1}$.
\item For all $t\in (0,t_0)$, $|\bar g_t'|+|\bar g_t|^{-1}\leq A $.
\end{enumerate}
\end{assumption}

Assumption~\ref{assum:lipmanlip} should be read as a two-regime condition.

\noindent\textbf{Terminal regime $[t_0,1)$.}
The factorization $
\bar g_t=(1-t)^p\ell(t)$
means that the noise vanishes polynomially at rate $p$, up to a bounded and slowly varying correction $\ell$. In particular, $
\left|\frac{\bar g_t'}{\bar g_t}\right|\le \frac{C}{1-t}.$
This is exactly the natural singularity created by denoising near $t=1$, and it explains why the final estimate has the scale $(1-t)^{-1}$. On the other hand,
the $m_t'$ contribution appearing in the general Lipschitz bound is $
\frac{1}{1+\alpha_t\sigma_t^2}L_t
=
\frac{|f_t'|f_t}{f_t^2+\alpha\sigma_t^2}.$
Therefore, the relevant object to control is $f_tf_t'$. Assumption~\ref{assum:lipmanlip}-2 is tailored exactly to the structure of the previous subsection's estimate.

\noindent\textbf{Early-time regime $(0,t_0)$.}
No singularity is expected away from $t=1$. Assumption~\ref{assum:lipmanlip}-3 guarantees that $\bar g_t$ stays non-degenerate and varies in a controlled way on $(0,t_0)$, so that no artificial blow-up appears near $t=0$.

Under these schedule assumptions, the abstract estimate from Section~\ref{sec:lipschitzgene} becomes a clean global bound.

\begin{corollary}\label{cor:lipman-global}
Assume the Lipman flow-matching setting of Assumption~\ref{assum:lipman}. Assume in addition that the model satisfies Assumption~\ref{assum:lipmanlip}. Then, for all $ x\in\mathbb R^d$ and $t\in(0,1)$ a.e.
\[
\|\nabla v_t(x)\|_{\mathrm{op}} \le \frac{C}{1-t},
\]
$C>0$ depending only on $A,\alpha,\beta,K,\gamma,f_{t_0},p$ and not on the dimension $d$.
\end{corollary}
The proof of Corollary~\ref{cor:lipman-global} can be found in Section~\ref{sec:cor:lipman-global}. This result shows that the drift remains Lipschitz in space for every $t<1$, with the natural terminal blow-up rate $(1-t)^{-1}$ induced by the denoising regime.

\subsubsection{Stochastic interpolant Flow matching setting}\label{sec:stogoballip}

In the Stochastic-interpolant setting, the interpolation path mixes the two endpoints with an additional Gaussian noise term. At first sight, this seems more involved than the Lipman setting. However, for the two-sided Lipschitz estimate, the analysis becomes much simpler after using the reduced marginal representation from Assumption~\ref{assum:sto-int}:
\[
X_t=f_tY+\bar g_t \xi,
\qquad 
\bar g_t^2=g_t^2+\sigma_t^2,
\]
where $Y\sim p$ and $\xi\sim \mathcal N(0,I_d)$ are independent. In particular, the estimate only depends on the effective noise level $\bar g_t$. This is an important difference with the one-sided bounds of Section~\ref{sec:onesidedlipreguv} where the sign of $(\log \bar g_t)'$ matters. By contrast, here we estimate $\|\nabla v_t(x)\|_{\mathrm{op}}$, and the general two-sided estimate depends on the Gaussian contribution through $|\bar g_t'|/\bar g_t$, that is, only through its magnitude. Hence sign changes of the noise schedule do not create a new mechanism in the present argument.

\begin{corollary}\label{coro:global-Lipschitz-v}
Assume the stochastic-interpolant flow matching setting of Assumption~\ref{assum:sto-int}.
Assume in addition that the model satisfies Assumption~\ref{assum:lipmanlip}.
 Then, for all $ x\in\mathbb R^d$ and $t\in(0,1)$ a.e.
\[
\|\nabla v_t(x)\|_{\mathrm{op}} \le \frac{C}{1-t},
\]
with $C>0$ depending only on $A,\alpha,\beta,K,\gamma,f_{t_0},p$ not the dimension $d$.
\end{corollary}
The proof of Corollary~\ref{coro:global-Lipschitz-v} can be found in Section~\ref{sec:coro:global-Lipschitz-v}.

\subsubsection{Diffusion Models setting}
As explained in Section~\ref{sec:score_osl}, the Probability flow ODE being a particular case of Lipman flow matching, Corollary~\ref{cor:lipman-global} directly gives an estimate on the score.
\begin{corollary}\label{coro:difflipbound}
Suppose that we are under the Diffusion models setting of Assumption~\ref{assum:diffusion} and let $s(t,x)=\nabla\log p_t(x)$ be the score.
Then, for all $ x\in\mathbb R^d$ and $t\in(0,1)$
\[\frac{1}{t}
\sup_{x\in\R^d}\bigl\| \text{Id}+\nabla s(t,x)\bigr\|_{ \text{op}} \le \frac{C}{1-t},
\]
with $C>0$ a dimension-free constant depending only on $\alpha,\beta,K$.
\end{corollary}
The proof of Corollary~\ref{coro:difflipbound} can be found in Section~\ref{sec:coro:difflipbound}.

\subsection{Time-Lipschitz regularity}\label{sec:time_lipschitz}
\subsubsection{General time-Lipschitz estimate}\label{sec:reduced_model}

Throughout, let \(Y\in\R^d\) have density \(p^\star\) and let \(\xi\sim\mathcal N(0,I_d)\) be independent of \(Y\).
Let \(f,\bar g:(0,1)\to(0,\infty)\) be schedules and define, for \(t\in(0,1)\),
\begin{equation}\label{eq:reduced_gaussian_mixture}
X_t = f_t Y + \bar g_t \xi.
\end{equation}
Our goal in this section is to control the time regularity of \(v_t\), namely \(\partial_t v_t(x)\), with an estimate that is explicit in the schedules \(f,\bar g\) and dimension-free up to the natural \(\sqrt d\) scale.

The reason why the time-regularity analysis is developed directly for the reduced model \eqref{eq:reduced_gaussian_mixture}
is that, unlike the spatial bounds, the estimate of \(\partial_t v_t\) depends on the precise time structure of the interpolation. In the general Gaussian-mixture form $X_t = m_t(Z) + \sigma_t \xi,$
differentiating in time would produce model-dependent terms involving \(\partial_t m_t\), \(\partial_{tt} m_t\), and the time variation of the posterior law of \(Z\) given \(X_t=x\). By contrast, in the reduced model all the time dependence is encoded by the scalar schedules \(f_t\) and \(\bar g_t\), which makes it possible to reduce the control of \(\partial_t v_t\) to scalar derivatives of the schedules together with posterior quantities that can be estimated uniformly. For this reason, whereas the space-Lipschitz analysis can be carried out in a general Gaussian-mixture framework, the time-Lipschitz analysis is formulated directly at the level of the reduced model.

Recall that \(p^\star\) is \((\alpha,\beta,K)\)-weakly log-concave, so it can be written as $
p^\star(y)=e^{-u(y)+a(y)}$
with \(u\) strongly convex (curvature \(\alpha\)) and \(a\) \(\beta\)-Hölder (size \(K\)). In the proof, it is useful to compare the true posterior with the log-concave surrogate posterior
\[
\nu^{t,x}(dy)\propto \exp \Big(-u(y)-\frac{\|x-f_ty\|^2}{2\bar g_t^2}\Big) dy.
\]
The quantity \(\E_{\nu_{t,0}}[W]\) appears in the final estimate because it controls the posterior center at \(x=0\), while the dependence on \(x\) is handled by Lipschitz estimates in the observation variable.

\begin{theorem}\label{theo:partialtvt}
Suppose that we are in the model \eqref{eq:reduced_gaussian_mixture} with \(f,\bar g\in C^2_{\text{a.e.}}([0,1])\), \(\bar g_t>0\),
and that the prior density of \(Y\) is \((\alpha,\beta,K)\)-weakly log-concave with \(\alpha,K>0\) and \(\beta\in (0,1]\).
Define
\[
a_t:=\frac{\bar g_t'}{\bar g_t},\quad c_t:=f_t'-a_t f_t,\quad
\gamma_t:=\alpha+\frac{f_t^2}{\bar g_t^2}.
\]
Then there exists \(C>0\) independent of the dimension such that, for a.e. \(t\in(0,1)\),
\[
\|\partial_t v_t(x)\|\le \mathcal A_t+\mathcal B_t\|x\|
\]
with
\begin{align*}
\mathcal A_t
&:= C|c_t'|\Big(\|\mathbb E_{\nu_{t,0}}[W]\|+\sqrt{\tfrac{d}{\gamma_t}}\Big)
 + C|c_t|^2\frac{f_t}{\bar g_t^2\gamma_t}
\Big(\|\mathbb E_{\nu_{t,0}}[W]\|+\sqrt{\tfrac{d}{\gamma_t}}\Big),
\\[0.3em]
\mathcal B_t
&:= |a_t'|
 + C|c_t'|\frac{f_t}{\bar g_t^2\gamma_t}
 + C|c_t|\frac{|f_t'-2a_t f_t|}{\alpha \bar g_t^2+f_t^2}
 + C|c_t|^2\frac{f_t^2}{\bar g_t^4\gamma_t^2}.
\end{align*}
\end{theorem}

The proof of Theorem~\ref{theo:partialtvt} can be found in Section~\ref{sec:theo:partialtvt}. This estimate has the form needed for numerical analysis: \(\partial_t v_t(x)\) has at most affine growth in \(x\), with explicit coefficients depending only on the schedules and posterior concentration parameters. In particular:
the $x$-slope coefficient $B_t$ is dimension-free, while the constant term $A_t$ contains two contributions: the centering term $\|\mathbb E_{\nu_{t,0}}[W]\|$ and the fluctuation term $\sqrt{d/\gamma_t}$, the latter being the natural scale under a $\gamma_t$-strongly log-concave law.

\subsubsection{Lipman Flow matching setting}

We now specialize the general time-Lipschitz estimate to the Lipman flow-matching parametrization. 
\begin{assumption}\label{assum:lipmanliptime} Writing \[
X_t = f_t Y + \bar g_t \xi,\qquad Y:=X_1\sim p^\star,\  \xi\sim\mathcal N(0,I_d) \text{ independent},
\] there exist \(p>0\), \(t_0\in(0,1)\) and \(A\ge 1\) such that:

\begin{enumerate}
\item \(f,\bar g\in C^{2}_{a.e.}(0,1)\) and \(\bar g_t>0\) for all \(t\in(0,1)\).
\item There exists \(\ell\in C^2([t_0,1))\) such that for all \(t\in[t_0,1)\),
\[
\bar g_t=(1-t)^p\ell(t),
\qquad 0<A^{-1}\le \ell(t)\le A,
\qquad \|\ell'\|_{L^\infty([t_0,1))}+\|\ell''\|_{L^\infty([t_0,1))}\le A.
\]
\item \(\|f'\|_{L^\infty([t_0,1))}+\|f''\|_{L^\infty([t_0,1))}+\|\argmin u\|d^{-1/2}\le A.\)
\item For a.e. \(t\in(0,t_0]\),
\[
\bar g_t^{-1}+|\bar g_t'|+|\bar g_t''|+|f_t'|+|f_t''|\le A.
\]
\end{enumerate}
\end{assumption}

Compared with Assumption~\ref{assum:lipmanlip}, Assumption~\ref{assum:lipmanliptime} has the same two-regime philosophy but it requires one additional derivative in each regime:
\begin{itemize}
    \item near \(t=1\), we need control of \(\ell''\) and \(f''\), because \(a_t'\) and \(c_t'\) appear in \(\partial_t v_t\);
    \item on \((0,t_0]\), we need uniform bounds on \(\bar g_t''\) and \(f_t''\) to prevent artificial singularities away from the terminal time.
\end{itemize}
A second new feature is the control of the posterior center at the natural fluctuation scale. 
This is why Assumption~\ref{assum:lipmanliptime}-3 contains the condition
\[
\|\argmin u\|\le A\sqrt{d},
\]
which ensures that the relevant posterior means remain of order \(\sqrt d\), yielding a dimension-sharp affine bound in \(x\). Assumption~\ref{assum:lipmanliptime}-2 also makes the terminal singularity transparent. Indeed, from
$
\bar g_t=(1-t)^p\ell(t)$
we obtain
$
\frac{\bar g_t'}{\bar g_t}
=
-\frac{p}{1-t}+\frac{\ell'(t)}{\ell(t)},$
so \(a_t=\bar g_t'/\bar g_t\) is of order \((1-t)^{-1}\), and therefore \(a_t'\) is of order \((1-t)^{-2}\). Since \(c_t\) and \(c_t'\) are built from \(a_t,a_t'\) and \(f',f''\), the same \((1-t)^{-2}\) scale propagates to \(\partial_t v_t\). Thus the exponent \(2\) in the next corollary is the natural singularity produced by differentiating the denoising dynamics in time.

\begin{corollary}\label{coro:dtv_rate}
Assume the Lipman  flow-matching setting of Assumption~\ref{assum:lipman} and that the model satisfies Assumption  \ref{assum:lipmanliptime}.
Then, for a.e. $t\in(0,1)$ and all $x\in\mathbb R^d$,
\[
\|\partial_t v_t(x)\| \le \frac{C}{(1-t)^2}\bigl(\sqrt d+\|x\|\bigr),
\]
with $C>0$ depending only on $A,\alpha,\beta,K,\gamma,f_{t_0},p$ and independent of the dimension $d$.
\end{corollary}

The proof of Corollary~\ref{coro:dtv_rate} can be found in Section~\ref{sec:coro:dtv_rate}. This result is the time-regular analogue of the space-Lipschitz estimate proved in Section~\ref{sec:lipschitz}: it provides an affine growth bound in \(x\), with dimension-free coefficient in front of \(\|x\|\), and only the natural \(\sqrt d\)-scale in the constant term. This is precisely the form needed in Section~\ref{sec:applicampl} for quantitative time-discretization estimates.

\subsubsection{Stochastic interpolant Flow matching setting}
As in Section~\ref{sec:stogoballip}, the key simplification is to use the reduced marginal representation
\[
X_t = f_t Y + \bar g_t \xi, \qquad \bar g_t^2 = g_t^2+\sigma_t^2,
\]
so that the time-regularity estimate depends on the stochastic-interpolant model only through the pair
\((f_t,\bar g_t)\). In particular, once Assumption~\ref{assum:lipmanliptime} holds for these reduced schedules, the general estimate
from Section~\ref{sec:reduced_model} applies exactly as in the Lipman case.

\begin{corollary}\label{coro:timeestimatesi}
Assume the stochastic-interpolant flow matching setting of Assumption~\ref{assum:sto-int} and that the model satisfies Assumption  \ref{assum:lipmanliptime}.
Then there exists a constant \(C>0\) (depending only on \(\alpha,\beta,K,p,A,f_{t_0}\))
 such that for all \(x\in\R^d\),
for a.e. \(t\in(0,1)\),
\[
\|\partial_t v_t(x)\|
  \le\frac{C}{(1-t)^2}  (\sqrt{d}+\|x\|).
\]
\end{corollary}

The proof of Corollary~\ref{coro:timeestimatesi} can be found in Section~\ref{sec:coro:timeestimatesi}.

\subsubsection{Diffusion Models setting}
As explained in Section~\ref{sec:score_osl}, the Probability flow ODE being a particular case of Lipman flow matching, Corollary~\ref{coro:dtv_rate} directly gives an estimate on the score.
\begin{corollary}\label{coro:timeregudiffu}
Suppose that we are under the Diffusion models setting of Assumption~\ref{assum:diffusion} and let $s(t,x)=\nabla\log p_t(x)$ be the score. Write $p^\star(y)=e^{-u(y)+a(y)}$ and assume that $\|\argmin u\|d^{-1/2}\le A$.
Then there exists a constant $C>0$, independent of the dimension $d$, such that for a.e. $t\in(0,1)$ and all $x\in\mathbb{R}^d$,
\[
\|\partial_t s(t,x)\| \le \frac{C}{(1-t)^2}\big(\sqrt d+\|x\|\big).
\]
\end{corollary}

The proof of Corollary~\ref{coro:timeregudiffu} can be found in Section~\ref{sec:coro:timeregudiffu}.

\section{Application to sampling}\label{sec:applicampl}
\subsection{Discretization error}\label{sec:discretizationerror}
\subsubsection{General discretization analysis}

In this section, we isolate a general discretization principle for Euler-type schemes. The objective is to make explicit which type of properties of the drift allow to control the sampling error, and how each property appears in the final $W_2$ bound. This will be the interface between the regularity estimates proved earlier and the applications to flow matching / diffusion models. We first isolate an abstract set of hypotheses on the drift that is used for both the SDE and ODE discretization analysis.

\begin{assumption}\label{assum:sdebounds} Consider the SDE
\begin{equation}\label{eq:SDE_main}
dX_t = a_t(X_t)  dt + b_t  dW_t, \qquad X_0\in L^2(\Omega;\R^d),
\end{equation}
and suppose that:
\begin{enumerate}
\item[(A1)] (one-sided Lipschitz) there exists   $L:[0,\tau]\to\R$ such that for all
$t\in[0,\tau]$ and all $x,y\in\R^d$,
\[
\langle x-y,  a_t(x)-a_t(y)\rangle \le L_t  \|x-y\|^2.
\]

\item[(A2)] (global Lipschitz) there exists   $C:[0,\tau]\to[0,\infty)$ such that for all
$t\in[0,\tau]$ and all $x,y\in\R^d$,
\[
\|a_t(x)-a_t(y)\|\le C_t  \|x-y\|.
\]

\item[(A3)] (time Lipschitz)
there exists   $M:[0,\tau]\to[0,\infty)$  such that
for every $x\in\R^d$, the map $t\mapsto a_t(x)$ is absolutely continuous on $[0,\tau]$ and,
for a.e. $t\in[0,\tau]$ and all $x\in\R^d$,
\[
\|\partial_t a_t(x)\|\le M_t\big(\sqrt d+\|x\|\big).
\]

\item[(A4)] (linear growth)
strong solutions of \eqref{eq:SDE_main} exist, and there exists   $F:[0,\tau]\to[0,\infty)$ such that
for all $t\in[0,\tau]$ and all $x\in\R^d$,
\[
\|a_t(x)\|\le F_t\big(\sqrt d+\|x\|\big).
\]
\end{enumerate}
\end{assumption}

The four assumptions have distinct functions. Assumption (A1) is the stability condition: it controls the propagation of errors once they have been generated, and therefore enters through a Grönwall-type exponential factor. Assumption (A2) is a spatial local truncation condition, used to bound the variation of the drift during a single Euler step as the state evolves. Assumption (A3)  controls the error induced by freezing the drift on each interval $[t_k,t_{k+1}]$; the factor $\sqrt d+\|x\|$ is natural and matches the scale of the trajectories in our applications. Finally, Assumption (A4)  provides the moment bounds needed to close the local estimates.

Our starting point is a general one-step Euler–Maruyama estimate, stated so as to separate explicitly the roles of stability, spatial regularity, time regularity, and moment control in the final $W_2$ bound.

\begin{proposition}[Euler-Maruyama with exact Gaussian increments]
\label{prop:W2_discrete}
Let $\tau>0$ and let $(W_t)_{t\in[0,\tau]}$ be a standard $d$-dimensional Brownian motion.
Consider the SDE
\begin{equation*}
dX_t = a_t(X_t)  dt + b_t  dW_t, \qquad X_0\in L^2(\Omega;\R^d),
\end{equation*}
and the Euler--Maruyama scheme on a grid $0=t_0<\dots<t_N=\tau$, $h_k=t_{k+1}-t_k$,
\begin{equation}\label{eq:EM_main}
\bar X_{k+1}=\bar X_k+h_k  a_{t_k}(\bar X_k)+\xi_{k+1},
\qquad \bar X_0=X_0,
\end{equation}
where
$
\xi_{k+1}:=\int_{t_k}^{t_{k+1}} b_t dW_t$.
Assume that the drift $a_t$ satisfies Assumption~\ref{assum:sdebounds}
and
\[
A_{\bar{X}}:=\sup_k (\E\|\bar X_k\|^2)^{1/2}<\infty,\qquad
B:=\sup_{t\in[0,\tau]}\|b_t\|<\infty,
\qquad
\int_{t_k}^{t_{k+1}} F_t  dt \leq \frac{1}{2},
\]
and take
$$
U_k:=2\big(A_{\bar X}+2\sqrt d  + \sqrt dB  h_k^{1/2}\big).$$
Then, 
\begin{align*}
&W_2(\text{Law}(X_\tau),\text{Law}(\bar X_N))\le
\sum_{k=0}^{N-1}h_k
\exp \Big(\int_{t_{k+1}}^{\tau} L_s  ds\Big)\Bigg[
U_k
\int_{t_k}^{t_{k+1}}  \big(M_t+C_{t_k}F_t\big)  dt+C_{t_k}\sqrt dB  h_k^{1/2}
\Bigg].
\end{align*} 
\end{proposition}

The proof of Proposition~\ref{prop:W2_discrete} can be found in Section~\ref{sec:prop:W2_discrete}.
The estimate has a ``local error $\times$ stability'' structure.
\begin{itemize}
    \item The factor $
    \exp \Big(\int_{t_{k+1}}^{\tau} L_s  ds\Big)$
    measures the amplification of the error created on step $k$ when it is transported from time $t_{k+1}$ to the final time $\tau$.
    
    \item The term involving $M_t$ comes from replacing $a_t$ by $a_{t_k}$ on $[t_k,t_{k+1}]$.
    
    \item The term involving $C_{t_k}F_t$ measures how much the state can drift away from $\bar X_k$ during one step; the factor $C_{t_k}$ then translates this state variation into a spatial drift error.
    
    \item The term $C_{t_k}\sqrt dB  h_k^{1/2}$ is the stochastic spatial variation created by Brownian fluctuations of size $\sqrt d h_k^{1/2}$.
\end{itemize}

Proposition~\ref{prop:W2_discrete} is optimal for a model-agnostic one-step analysis, but in the SDE setting it cannot by itself
yield a linear global dependence on $h_{\max}$. Indeed, the estimate
contains the stochastic spatial-variation term $
h_kB  h_k^{1/2},$
which once summed, scales at best like
$B\sqrt{h_{\max}}$.
Nevertheless, Proposition~\ref{prop:W2_discrete} can be used to obtain the optimal rate of discretization for ODE as displayed in  Proposition~\ref{prop:geom_mesh_general}. To handle the singular behavior of the drift near the terminal time, we introduce the geometric mesh that is naturally adapted to the bounds of Section~\ref{sec:lipreguv}.

\begin{definition}\label{def:geometricgrid} The geometric grid with parameters $(\tau,T,N)$ with $0<\tau<T$ and $N\in \mathbb{N}_{>0}$ corresponds to the discretization $0=t_0<t_1<...<t_N=\tau$ with
\[
r := \left(\frac{T-\tau}{T}\right)^{1/N},\qquad h_{\max}:= T(1-r),\qquad h_k := r^k h_{\max},\qquad t_{k+1}:= t_k+h_k.
\]
\end{definition}
A direct computation gives $
h_k=(1-r)(T-t_k)$, thus each step size is proportional to the remaining distance to the terminal time $T$. This is exactly the right behavior when the drift bounds blow up like $(T-t)^{-1}$ and $(T-t)^{-2}$. The grid is coarse when the dynamics is regular, and it automatically refines near the singular endpoint. We next specialize the abstract estimate to the ODE case and show that, under natural blow ups the geometric mesh restores a first-order discretization rate in $h_{\max}$.

\begin{proposition}[Geometric mesh discretization bound of ODE]\label{prop:geom_mesh_general}
Fix $\tau\in(1/2,1)$, an integer $N\ge 1$ and consider the geometric grid with parameters $(\tau,1,N)$ of Definition~\ref{def:geometricgrid}.
Let $(X_t)_{t\in[0,\tau]}$ be the solution of the ODE
\[
\dot X_t=v_t(X_t),\qquad X_0\in L^2(\Omega;\R^d),
\]
and let $(\bar X_k)_{k=0}^N$ be the explicit Euler scheme
\[
\bar X_{k+1}=\bar X_k+h_k v_{t_k}(\bar X_k),\qquad \bar X_0=X_0.
\]

Assume that $v_t$ satisfies Assumption~\ref{assum:sdebounds} in the ODE case ($b_t= 0$), with 
\begin{equation}\label{eq:lipman_envelopes_general}
\sup_{z\in [0,\tau]}\int_z^\tau L_t\leq A,\quad
F_t\le \frac{A}{1-t},\quad
C_t\le \frac{A}{1-t},\quad
M_t\le \frac{A}{(1-t)^2},
\end{equation}
and that $\sup_k (\E\|\bar X_k\|^2)^{1/2}\leq A \sqrt d$.
Then there exist dimension-free constants $\bar h\in(0,1)$ and $C>0$, depending only on $A$
 such that if $h_{\max}\le \bar h$ then
\[
W_2\big(\text{Law}(X_\tau),\text{Law}(\bar X_N)\big)\le C \sqrt d h_{\max} \log \Big(\frac{1}{1-\tau}\Big).
\]

\end{proposition}
The proof of Proposition~\ref{prop:geom_mesh_general} can be found in Section~\ref{sec:prop:geom_mesh_general}.
The assumptions in \eqref{eq:lipman_envelopes_general} are exactly of the singular type that appears in flow matching: the drift and its spatial Lipschitz constant behave like $(1-t)^{-1}$, while the time derivative behaves like $(1-t)^{-2}$ near $t=1$. The conclusion shows that the geometric mesh compensates for these singularities at the discretization level: one recovers a linear dependence in $h_{\max}$ (up to the logarithmic factor, which corresponds to the accumulation of the terminal singular layer). The estimate contains the natural factor $\sqrt d$, which is the correct scale for $W_2$ in this setting.

Now by using additional structure (symmetry of $\nabla a_t$ and a
second-order bound on $\Delta a_t$) we show in the next result that we can also obtain the linear dependence in the SDE case. 

\begin{theorem}[Geometric mesh discretization bound of SDE]\label{theo:geom_mesh_generaldiff}
Fix $T>1$, $\tau\in(T/2,T)$, an integer $N\ge 1$ and consider the geometric grid with parameters $(\tau,T,N)$ of Definition~\ref{def:geometricgrid}. Let $b:[0,T]\rightarrow [0,B]$ for some $B>0$ and consider the SDE
\[
dX_t = a_t(X_t) dt+ b_t dW_t,
\qquad t\in[0,T],
\]
and the Euler-Maruyama scheme with exact Gaussian increments 
\[
\bar X_{k+1}=\bar X_k+h_k a_{t_k}(\bar X_k)+\xi_{k+1},
\qquad \bar X_0=X_0,
\]
where
$
\xi_{k+1}:=\int_{t_k}^{t_{k+1}} b_t dW_t$. Define $\rho(t):=\min\{1,T-t\}$ and assume that the drift $a_t$ satisfies Assumption~\ref{assum:sdebounds} with 
\begin{equation}\label{eq:diff_envelopes_general_T}
\sup_{z\in [0,\tau]}\int_z^\tau L_t dt\leq A,\quad
F_t\le \frac{A}{\rho(t)},\quad
C_t\le \frac{A}{\rho(t)},\quad
M_t\le \frac{A}{\rho(t)^2}.
\end{equation}
Assume additionally that the Jacobian $\nabla a_t(x)$ is symmetric and
\[
\sup_{t\in [0,\tau]} \|X_t\|_{L_2}\leq A\sqrt{d},
\qquad
\|\Delta a_t(x)\|\leq A \frac{\sqrt{d}+\|x\|}{\rho(t)^2}.
\]
Then there exist dimension-free constants $\bar h\in(0,1)$ and $C>0$, depending only on $A$ and $B$,
such that if $h_{\max}\le \bar h$,
\[
W_2\big(\text{Law}(X_\tau),\text{Law}(\bar X_N)\big)\le C \sqrt d  h_{\max} 
\Bigl(T+\log \Big(\frac{T}{T-\tau}\Big)\Bigr).
\]
\end{theorem} 
The proof of Theorem~\ref{theo:geom_mesh_generaldiff} can be found in Section~\ref{sec:theo:geom_mesh_generaldiff}. 
The proof follows the bias/variance decomposition strategy introduced in \cite{arsenyan2025assessing} for
the reverse Ornstein-Uhlenbeck diffusion, and extends it to the present non-homogeneous SDE setting with
general drift $a_t$ and time-dependent diffusion amplitude $b_t$.
More precisely, instead of estimating the local truncation error $V_k := \int_{t_k}^{t_{k+1}}\Big(a_t(X_t)-a_{t_k}(X_{t_k})\Big) dt$ only by its $L^2$ norm as in Proposition~\ref{prop:W2_discrete},
we write $V_k$ as
\[
V_k = \E[V_k\mid \mathcal F_k] + \big(V_k-\E[V_k\mid \mathcal F_k]\big),
\]
and propagate separately:
(i) the conditional mean in $\ell^1$ with the stability weights, and
(ii) the centered martingale fluctuation in $\ell^2$.
This is formalized in Proposition~\ref{prop:biasvar_no_propagation}. The assumptions in Theorem~\ref{theo:geom_mesh_generaldiff} are exactly those needed to implement this program:
the bounds on $\partial_t a_t$ and $\Delta a_t$ control the It\^o drift term (hence the bias), while the
Lipschitz bound controls the stochastic integral term, and the symmetry of $\nabla a_t$ allows a sharp
one-step stability estimate in terms of $\lambda_{\max}(\nabla a_t)$.
\begin{remark}
The use of exact Gaussian increments allows us to assume only that $b$ is bounded in time. If one instead considers the standard Euler-Maruyama scheme with frozen diffusion coefficient,
\[
\bar X_{k+1}=\bar X_k+h_k a_{t_k}(\bar X_k)+b_{t_k}(W_{t_{k+1}}-W_{t_k}),
\]
then the same estimate should still hold provided $b$ has sufficient time regularity, for instance if $b$ is Lipschitz in time.
\end{remark}

\subsubsection{Lipman and Stochastic interpolant flow matching settings}\label{sec:fhifdoiedospzp}

We now specialize the general discretization principle of the previous section to the ODE samplers arising from flow matching. The key point is that the regularity estimates proved in Section~\ref{sec:lipreguv} provide exactly the inputs required by the abstract Euler error bound. The important consequence is that, once the regularity theory is established, the discretization argument is completely uniform: the Lipman and Stochastic-interpolant settings differ in the construction of the drift, but the Euler analysis only sees the regularity bounds.

\begin{proposition}\label{prop:geomeshLIp} Let $p^\star=\exp(-u+a)$ be  $(\alpha,\beta,K)$-weakly log-concave measure such that
for $m:=\arg\min_{\mathbb R^d} u$ we have $\|m\|\le A\sqrt d.$
Let $(X_t)_{t\in[0,\tau]}$ be the solution of the ODE
\[
\dot X_t=v_t(X_t),\qquad X_0\sim \mathcal{N}(0,\text{Id}),
\]
where $v_t$ satisfies either:
\begin{itemize}
    \item[-]the Lipman flow matching setting of Assumption~\ref{assum:lipman} or
    \item[-]the Stochastic-interpolant flow matching setting of Assumption~\ref{assum:sto-int}.
\end{itemize}
Suppose moreover the schedule hypotheses of Assumptions~\ref{assum:lipmanliptime}
hold. Fix $\tau\in(1/2,1)$, an integer $N\ge 1$ and consider the geometric grid with parameters $(\tau,1,N)$ of Definition~\ref{def:geometricgrid}.
Let $(\bar X_k)_{k=0}^N$ be the explicit Euler scheme
\[
\bar X_{k+1}=\bar X_k+h_k v_{t_k}(\bar X_k),\qquad \bar X_0=X_0.
\]
Then there exist dimension-free constants $\bar h\in(0,1)$ and $C>0$
such that, if $h_{\max}\le \bar h \log \Big(\frac1{1-\tau}\Big)^{-1}$, one has
\[
W_2 \left(\text{Law}(X_\tau), \text{Law}(\bar X_N)\right) \le C \sqrt d  h_{\max} \log \Big(\frac1{1-\tau}\Big).
\]
\end{proposition}
The proof of Proposition~\ref{prop:geomeshLIp}  can be found in Section~\ref{sec:prop:geomeshLIp}. This proposition is the population-discretization input used later in Corollary~\ref{cor:lipman_sampling_error_log2_over_N}, and shows that the regularity bounds from Section~\ref{sec:discretizationerror} are strong enough to yield a first-order Euler error  on a geometric grid in both the Lipman and Stochastic-interpolant settings.

\subsubsection{Diffusion models setting}
We now turn to the reverse-time Ornstein-Uhlenbeck sampler used in diffusion models. In contrast with Section~\ref{sec:fhifdoiedospzp}, the dynamics is stochastic and the discretization is an Euler-Maruyama scheme, so the relevant abstract result is Theorem~\ref{theo:geom_mesh_generaldiff} rather than the ODE Euler bound. The key point is that the regularity estimates proved in Section~\ref{sec:lipreguv} for diffusion-model scores provide the required inputs for the reverse drift
after the OU time-reversal parametrization. This yields the following discretization estimate on a geometric grid.

\begin{proposition}\label{theo:dicdiff}
Assume that $p^\star=\exp(-u+a)$ is $(\alpha,\beta,K)$-weakly log-concave with $\alpha,K>0$, $\beta\in(0,1]$,
and for $m:=\arg\min_{\R^d} u$, that $\|m\|\le A\sqrt d$. Fix a time horizon $T>1$ and let $(q_s)_{s\in[0,T]}$ be the Ornstein-Uhlenbeck path
and define the reversed path on $[0,T]$ by
\[
p_t := q_{T-t},\qquad t\in[0,T].
\]
Let $s_t(x)=\nabla_x\log p_t(x)$ be its score, and define the reverse drift
\[
a_t(x):=x+2s_t(x).
\]

Fix $\tau\in (T-1,T)$ such that $T\leq A \log((T-\tau)^{-1})$,  an integer $N\ge1$ and consider the geometric grid with parameters $(\tau,T,N)$ of Definition~\ref{def:geometricgrid}.
Consider the reverse SDE on $[0,\tau]$
\[
dX_t=a_t(X_t) dt+\sqrt{2} dW_t,\qquad X_0\sim q_T,
\]
with $\|X_0\|_{L_2}\leq A\sqrt{d}$ and its Euler scheme with exact Gaussian increments on the same grid:
\[
\bar X_{k+1}=\bar X_k+h_k a_{t_k}(\bar X_k)+\sqrt{2} (W_{t_{k+1}}-W_{t_k}),\qquad \bar X_0=X_0.
\]
Then, there exist constants $\bar h\in(0,T/2]$ and $C>0$, depending only on $(\alpha,\beta,K,A)$, such that if $
h_{\max}\le \bar h\log \Big(\frac{1}{T-\tau}\Big)^{-1},$
\[
W_2\bigl(\text{Law}(X_\tau),\text{Law}(\bar X_N)\bigr)
 \le  C \sqrt d  h_{\max}\log \Big(\frac{1}{T-\tau}\Big).
\]
\end{proposition}

The proof of Proposition~\ref{theo:dicdiff} can be found in Section~\ref{sec:theo:dicdiff}. This is the diffusion-model analogue of Proposition~\ref{prop:geomeshLIp}: it converts the score regularity estimates of Section~\ref{sec:fhifdoiedospzp} into an Euler–Maruyama discretization bound for the reverse OU SDE, and provides the discretization term that is later inserted into Corollary~\ref{cor:diffusion_sde_sampling_logN}.

\subsection{Sampling error}\label{sec:samplingerror}
\subsubsection{General sampling analysis}
We now isolate a simple decomposition of the terminal sampling error for a learned-drift sampler.
The purpose of this section is to separate the three different sources of error that appear in all our
applications:

\begin{itemize}
    \item[(i)] the \emph{early-stopping error} (we stop the continuous dynamics at time $\tau$ before the terminal time),
    \item[(ii)] the \emph{drift-learning error} (the learned drift $\hat a_t$ differs from the ideal drift $a_t$),
    \item[(iii)] the \emph{time-discretization error} (we simulate the learned dynamics with an Euler-type scheme).
\end{itemize}
Each term will then be controlled by a different argument:
the first term is model-specific (and uses the explicit form of the interpolation/smoothing path),
the second term is a stability estimate driven by a one-sided Lipschitz bound,
and the third term is precisely the object estimated in Section~\ref{sec:discretizationerror} by our abstract discretization results.
In this sense, the present lemma is the interface between learning error and numerical error. 

\begin{lemma}
\label{lemma:sampling_error_sde}
Let $0<\tau<T$ and let $(b_t)_{t\in[0,T]}$ be a   matrix-valued function
satisfying
$B := \sup_{t\in[0,T]} \|b_t\| < \infty.$
Let $X$ be the (target) solution of the SDE
\[
dX_t = a_t(X_t) dt + b_t dW_t,\qquad t\in[0,T],
\]
with $\text{Law}(X_T)=p^\star$ the target distribution.
Define the drift approximation error along the true process:
\begin{equation}
\label{eq:eps_drift}
\epsilon_{\mathrm{drift}}(t) := \bigl\|a_t(X_t)-\hat a_t(X_t)\bigr\|_{L^2}.
\end{equation}
Let $\hat X$ be the (learned) solution of
\[
d\hat X_t = \hat a_t(\hat X_t) dt + b_t dW_t,\qquad t\in[0,\tau],
\]
and let $0=t_0<\dots<t_N=\tau$ be a grid with steps $h_k=t_{k+1}-t_k$, and 
$\bar X$ be the Euler-Maruyama scheme for the learned SDE with exact Gaussian increments:
\[
\bar X_{k+1} = \bar X_k + h_k \hat a_{t_k}(\bar X_k) + \xi_{k+1},\]
where $
\xi_{k+1}:=\int_{t_k}^{t_{k+1}} b_t dW_t$ and $\bar X_0 = \hat{X}_0$. Assume moreover that $\hat a$ is one-sided Lipschitz with constant
$\hat L_t$. Then the terminal sampling error admits the bound
\begin{align*}
W_2 \bigl(p^\star,\text{Law}(\bar X_N)\bigr)
\le& W_2 \bigl(\text{Law}(X_T),\text{Law}(X_\tau)\bigr)+
\int_0^\tau
\exp \Bigl(\int_t^\tau \hat L_s ds\Bigr) 
\epsilon_{\mathrm{drift}}(t) dt+ W_2\bigl(\text{Law}(\hat{X}_\tau),\text{Law}(\bar X_N)\bigr).
\end{align*}
\end{lemma}

The proof of Lemma~\ref{lemma:sampling_error_sde} can be found in Section~\ref{sec:lemma:sampling_error_sde}. This result provides the error decomposition used for the final sampling results in Corollaries \ref{cor:lipman_sampling_error_log2_over_N} and \ref{cor:diffusion_sde_sampling_logN}.

As emphasized by \cite{beyler2025convergence}, if one wants the drift-learning error to be measured along the true trajectory $X_t$ (rather than along the learned/discrete trajectory), then the discretization term must be written for the learned dynamics. Controlling the propagated error through this decomposition then requires regularity assumptions on the learned drift $\hat{a}_t$. Therefore we are going to assume that $\hat{a}_t$ satisfies
the same structural regularity and moment properties as $a_t$ in the sense of the following definition.
\begin{definition}\label{defi:sameassum}
We say that $(\hat{a}_t)_{t\in[0,\tau]}$ satisfies the same regularity properties as $(a_t)_{t\in[0,\tau]}$ if there exists a dimension-free constant $A>0$ such that for all
$t\in[0,\tau]$ and all $x,y\in\R^d$:
$$
\langle x-y,  \hat a_t(x)-\hat a_t(y)\rangle \le A \langle x-y,  a_t(x)-a_t(y)\rangle, \qquad \|\hat{a}_t(x)-\hat a_t(y)\|\leq A \|a_t(x)-a_t(y)\|,$$
and $$\|\partial_t \hat a_t(x)\|\leq A \|\partial_t a_t(x)\|, \qquad \|\hat  a_t(x)\|\leq A \|a_t(x)\|.$$
\end{definition}

\subsubsection{Lipman and Stochastic-interpolant flow matching settings}
We now specialize the general sampling decomposition of Lemma~\ref{lemma:sampling_error_sde} to the ODE samplers arising from flow matching.
The key point is that, once the regularity estimates of Section~\ref{sec:lipreguv} are available, the argument is the same in the
Lipman and Stochastic-interpolant settings.

\begin{corollary}
\label{cor:lipman_sampling_error_log2_over_N}
Let $p^\star=\exp(-u+a)$ be an $(\alpha,\beta,K)$-weakly log-concave measure such that
for $m:=\arg\min_{\mathbb R^d} u$ we have $\|m\|\le A\sqrt d.$
Let $(X_t)_{t\in[0,\tau]}$ be the solution of the ODE
\[
\dot X_t=v_t(X_t),\qquad X_0\sim \mathcal{N}(0,\text{Id}),
\]
where the vector field $v_t$ satisfies either:
\begin{itemize}
    \item[-]the Lipman flow matching setting of Assumption~\ref{assum:lipman} or
    \item[-]the Stochastic-interpolant flow matching setting of Assumption~\ref{assum:sto-int}.
\end{itemize}
Assume moreover that the schedule hypotheses of Assumptions~\ref{assum:lipmanliptime}
hold with $f'\in L^\infty(0,1)$. Let $(\hat v_t)_{t\in[0,\tau]}$ be a (learned) drift defined on $[0,\tau]$ and assume that it satisfies
the same regularity  properties as $(v_t)_t$ in the sense of Definition~\ref{defi:sameassum}.
Define the drift learning error
\[
\varepsilon_{\mathrm{drift}}(t)
:= \bigl\|v_t(X_t)-\hat v_t(X_t)\bigr\|_{L^2},\qquad t\in[0,\tau]
\]
and suppose that $\int_0^\tau \varepsilon_{\rm drift}(t)\leq A\sqrt{d}$.
For an integer $N\ge 2$, le $p$ be the exponent from Assumption \ref{assum:lipmanlip}, set $q:=p\wedge 1$ and $
\tau:= 1-\Bigl(\frac{\log^2 N}{N}\Bigr)^{1/q}$
and consider the geometric grid with parameters $(\tau,1,N)$ of Definition~\ref{def:geometricgrid}.
Let $(\bar X_k)_{k=0}^N$ be the explicit Euler scheme driven by $\hat a$:
\[
\bar X_{k+1} = \bar X_k + h_k \hat v_{t_k}(\bar X_k),\qquad \bar X_0\stackrel{d}=X_0.
\]
Then, for $N$ large enough,
there exists a dimension-free constant $C>0$ such that
\[
W_2\bigl(p^\star,\text{Law}(\bar X_N)\bigr)
 \le 
C\frac{\sqrt d}{N}\log^2 N
 + 
C\int_0^\tau \varepsilon_{\mathrm{drift}}(t) dt.
\]
\end{corollary}
The proof of Corollary~\ref{cor:lipman_sampling_error_log2_over_N} can be found in Section~\ref{sec:cor:lipman_sampling_error_log2_over_N}. This corollary gives (to our knowledge) the first \emph{linear-in-$N$} Wasserstein sampling bound for flow matching
(up to logarithmic factors) under the weak log-concavity assumptions of Definition~\ref{defi:weakconcave}.
Compared with prior stability-based analyses, this removes the polynomial amplification in $N$ of the learning error.
The proof is uniform in the Lipman and Stochastic-interpolant settings; only the regularity inputs from Section~\ref{sec:lipreguv} differ.

\subsubsection{Diffusion models setting}
We now specialize the general sampling decomposition of Lemma~\ref{lemma:sampling_error_sde} to the reverse-time Ornstein--Uhlenbeck sampler. 
\begin{corollary}
\label{cor:diffusion_sde_sampling_logN}
Let $p^\star=\exp(-u+a)$ be an $(\alpha,\beta,K)$-weakly log-concave density on $\mathbb{R}^d$ with
$\alpha,K>0$, $\beta\in(0,1]$, and let $m:=\arg\min_{\mathbb{R}^d} u$ satisfy $\|m\|\le A\sqrt d$.
Let $(q_s)_{s\ge 0}$ be the Ornstein--Uhlenbeck (OU) smoothing path of $p^\star$, i.e.
\[
q_s = \text{Law} \Big(e^{-s}Y+\sqrt{1-e^{-2s}} \xi\Big),
\qquad Y\sim p^\star,\ \xi\sim \mathcal N(0,I_d),\ Y\perp \xi.
\]
Fix an integer $N\in \mathbb{N}_{>0}$ large enough and set $
T := \log N,
\tau := T-\frac{1}{N^2}$ and consider the geometric grid with parameters $(\tau,T,N)$ of Definition~\ref{def:geometricgrid}.
Define the reversed OU interpolation on $[0,T]$ by $p_t := q_{T-t}$, let $s_t(x):=\nabla_x\log p_t(x)$,
and define the reverse drift $a_t(x):=x+2s_t(x)$.
Let $(X_t)_{t\in[0,\tau]}$ solve the reverse SDE
\[
dX_t = a_t(X_t) dt + \sqrt{2} dW_t,
\qquad X_0 \sim q_T.
\]
Let $(\hat s_t)_{t\in[0,\tau]}$ be a function satisfying
the same structural regularity  properties as $s_t$ on $[0,\tau]$ in the sense of Definition~\ref{defi:sameassum} and  additionally $\|\Delta \hat s_t(x)\|\leq A \|\Delta  s_t(x)\|$ and $\nabla \hat s_t$ is symmetric.
Define the drift learning error
\[
\varepsilon_{\rm drift}(t):=\|s_t(X_t)-\hat s_t(X_t)\|_{L^2},\qquad t\in[0,\tau].
\]
and suppose that $\int_0^\tau \varepsilon_{\rm drift}(t)\leq A\sqrt{d}$. Let
\[
d\hat X_t = (\hat X_t+ 2\hat s_t(\hat X_t)) dt + \sqrt{2} dW_t,
\qquad \hat X_0 \sim \mathcal{N}(0,\text{Id}),
\]
and consider the Euler-Maruyama scheme with exact Gaussian increments
\[
\bar X_{k+1} = \bar X_k + h_k (\bar X_k+2\hat s_{t_k}(\bar X_k)) + \sqrt{2}\big(W_{t_{k+1}}-W_{t_k}\big),
\qquad \bar X_0 = \hat X_0.
\]
Then, there exists a constant $C>0$ depending only on $(\alpha,\beta,K,A)$ such that
\[
W_2\bigl(p^\star,\text{Law}(\bar X_N)\bigr) \le  \frac{C\sqrt d}{N}\log^3 N+
C\int_0^\tau \varepsilon_{\rm drift}(t) dt.
\]
\end{corollary}

The proof of Corollary~\ref{cor:diffusion_sde_sampling_logN} can be found in Section~\ref{sec:cor:diffusion_sde_sampling_logN}. This result gives a sampling bound for diffusion models under the same weak log-concavity assumptions used throughout the paper, with no a priori Lipschitz assumption on the exact score. The learning term is measured along the \emph{true} reverse process, $\|s_t(X_t)-\hat s_t(X_t)\|_{L^2(pt)}$, which matches the natural denoising score-matching risk and cleanly separates learning and discretization effects. Up to logarithmic factors, the leading discretization term has the sharp $\sqrt d/N$ scaling; an extra logarithm here comes from using a single geometric grid on the whole interval (a two-stage arithmetic-geometric grid could improve it as shown in \cite{arsenyan2025assessing}).

\section{Application to functional inequalities}
In this section we explain how the integrability of the maximal eigenvalue
\(\lambda_{\max}(\nabla v_t)\) obtained in Section~\ref{sec:onesidedlipreguv} yields a Lipschitz transport map
\(T:\R^d\to\R^d\) pushing the standard Gaussian measure \(\gamma_d\) onto the weakly log-concave
measure $p^\star$, and how to transfer classical functional inequalities from
\(\gamma_d\) to $p^\star$.

\subsection{From $L^1$ $\lambda_{\max}(\nabla v_t)$ to Lipschitz flow maps}\label{sec:hjkroezpm}
Let \(X_t(x)\) be the deterministic flow generated by \((v_t)_{t\in[0,1]}\), i.e.
\begin{align}\label{eq:hfjkddhfjiesklpdsmms}
    \partial_t X_t(x)=v_t(X_t(x)),\quad X_0(x)=x.
\end{align}
A crucial point is that to control the Lipschitz constant of \(X_t\), one does not need a bound on the full operator norm
\(\|\nabla v_t\|_{\mathrm{op}}\). It is enough to control the maximal eigenvalue of the symmetric part of \(\nabla v_t\),
denoted \(\lambda_{\max}(\nabla v_t)\) as this is  the quantity that governs the expansion of the flow.

\begin{lemma}[Lipschitz bound for the flow map]
\label{lem:flow-lip}
Let $X_t$ be the solution to \eqref{eq:hfjkddhfjiesklpdsmms}
and assume there exists a   function \(\theta:[0,1]\to\R\) such that, for every \(t\in[0,1]\) and every \(x\in\R^d\),
\[
\lambda_{\max}\big(\nabla v_t(x)\big)  \le  \theta_t.
\]
Then, for all \(t\in[0,1]\),
\[
\|\nabla X_t(x)\|_{\mathrm{op}}
 \le 
\exp  \Big(\int_0^t \theta_s   ds\Big),
\]
and in particular the map \(X_1\) is \(L\)-Lipschitz with
\[
L := \exp  \Big(\int_0^1 \theta_s   ds\Big).
\]
\end{lemma}
The proof of Lemma~\ref{lem:flow-lip} can be found in Section~\ref{sec:lem:flow-lip}. The results of
Section~\ref{sec:onesidedlipreguv} show that for a $(\alpha,\beta,K)$-weakly log-concave measure $p^\star$, the flow matching vector field \(v_t\) transporting the Gaussian towards \(p^\star\) satisfies a dimension-free one-sided estimate:
there exists a dimension-free constant \(C_{\mathrm{flow}}>0\) such that
\[
\int_0^1 \sup_{x\in\R^d} \lambda_{\max}\big(\nabla v_t(x)\big)  dt
 \le 
C_{\mathrm{flow}}.
\]
Applying Lemma~\ref{lem:flow-lip} with
\[
\theta_t := \sup_{x\in\R^d}\lambda_{\max}\big(\nabla v_t(x)\big)
\]
shows that the terminal flow map is globally Lipschitz with a dimension-free constant.

\begin{corollary}[Lipschitz transport map from flow matching]
\label{cor:Lip-transport}
Let $p^\star$ be an $(\alpha,\beta,K)$-weakly log-concave measure with $\alpha,K>0$ and $\beta\in (0,1]$.
Let \(v\) be any flow matching vector field satisfying either the conditions of Theorem~\ref{theo:onesidedlipinteg} or Theorem~\ref{theo:onesidedlipinteg0}, and define the final map \(T(x):=X_1(x)\), where \(X_t\) is the flow matching solution:
\begin{align*}
    \partial_t X_t(x)=v_t(X_t(x)),\quad X_0(x)=x.
\end{align*}
Then \(T\) satisfies
\[
T_{\#}\gamma_d = p^\star,
\qquad
\mathrm{Lip}(T) \le  e^{C_{\mathrm{flow}}},
\]
where \(C_{\mathrm{flow}}\) is independent of the dimension \(d\).
\end{corollary}

\subsection{From Lipschitz flow maps to functional inequalities}
We now show how the dimension-free Lipschitz control of the terminal flow map obtained in Section~\ref{sec:hjkroezpm}
yields dimension-free functional inequalities for the target measure $p^\star$.
The argument is standard: Gaussian $\Psi$-Sobolev inequalities are stable under Lipschitz pushforwards.
\subsubsection{\texorpdfstring{$\Psi$}{Psi}-entropies and Lipschitz changes of variables}

Let us first recall the notion of \(\Psi\)-entropy following \cite{chafai2004entropies}.

\begin{definition}[\(\Psi\)-entropy]
Let \(I\subset\R\) be a closed interval and let
\(\Psi:I\to\R\) be twice differentiable. We say that \(\Psi\) is a
\emph{divergence} if the three functions \(\Psi\), \(\Psi''\) and \(-1/\Psi''\) are
convex on \(I\). Given a probability measure \(\nu(dx)=p(x) dx\) on \(\R^d\) and 
\(\zeta:\R^d\to I\) such that \(\int_{\R^d}\zeta(x)p(x) dx\in I\), the
\(\Psi\)-entropy of \(\zeta\) with respect to \(\nu\) is
\[
\text{Ent}_{\Psi,\nu}(\zeta)
:=
\int_{\R^d}\Psi(\zeta(x))p(x) dx
-
\Psi  \Big(\int_{\R^d}\zeta(x)p(x) dx\Big).
\]
\end{definition}

Typical examples are:
\begin{itemize}
\item \(\Psi(x)=x^2\) on \(I=\R\), for which \(\text{Ent}_{\Psi,p}(f)=\text{Var}_p(f)\);
\item \(\Psi(x)=x\log x\) on \(I=\R_+\), for which \(\text{Ent}_{\Psi,p}(g)=\text{Ent}_p(g)\)
  is the usual entropy.
\end{itemize}

The standard Gaussian measure \(\gamma_d\) satisfies (dimension-free) \(\Psi\)-Sobolev
inequalities: for every divergence \(\Psi\) and every smooth \(F:\R^d\to I\),
\begin{equation}
\label{eq:gaussian-Psi-Sobolev}
\text{Ent}_{\Psi,\gamma_d}(F)
 \le 
\frac12
\int_{\R^d}\Psi''(F(y)) |\nabla F(y)|^2 d\gamma_d(y).
\end{equation}
This is a particular case of the general \(\Phi\)-Sobolev inequalities for
diffusions satisfying a \(CD(1,\infty)\) curvature-dimension condition. We can now transfer \eqref{eq:gaussian-Psi-Sobolev} from \(\gamma_d\) to any
pushforward measure of the form \(p=T_{\#}\gamma_d\).

\begin{proposition}[\(\Psi\)-Sobolev inequality under a Lipschitz transport]
\label{prop:Psi-Sobolev-p}
Let \(T:\R^d\to\R^d\) be \(C^1\) with \(\text{Lip}(T)\le L\) and
\(T_{\#}\gamma_d = p\). Let \(\Psi\) be a divergence on \(I\) and let
\(\zeta:\R^d\to I\) be \(C^1\) with \(\int\zeta dp\in I\).
Then
\[
\text{Ent}_{\Psi,p}(\zeta)
 \le 
\frac{L^2}{2}
\int_{\R^d}\Psi''(\zeta(x)) |\nabla\zeta(x)|^2 p(x) dx.
\]
\end{proposition}

The proof of Proposition~\ref{prop:Psi-Sobolev-p} can be found in Section~\ref{sec:prop:Psi-Sobolev-p}.

\subsubsection{Functional inequalities for weakly log concave measures}
We now combine Proposition~\ref{prop:Psi-Sobolev-p} with the dimension-free Lipschitz transport map obtained in Corollary~\ref{cor:Lip-transport}.

\begin{theorem} 
\label{theo:functional-ineq-flowmatching}
Let $p^\star$ be an $(\alpha,\beta,K)$-weakly log-concave measure with $\alpha,K>0$ and $\beta\in (0,1]$. Then there exists $C>0$ independent of the dimension such that:
for any divergence \(\Psi\) and smooth \(f:\R^d\to I\),
\[
\text{Ent}_{\Psi,p^\star}(f)
 \le 
C
\int_{\R^d}\Psi''(f(x)) |\nabla f(x)|^2 dp^\star(x).
\]
In particular,
\[
\text{Var}_{p^\star}(f)\le 2C\int|\nabla f|^2 dp^\star,
\qquad
\text{Ent}_{p^\star}(f^2)\le 4C\int|\nabla f|^2 dp^\star
\]
so $p^\star$ satisfies the Poincaré inequality and the Log-Sobolev inequality.
\end{theorem}
The proof of Theorem~\ref{theo:functional-ineq-flowmatching} can be found in Section~\ref{sec:theo:functional-ineq-flowmatching}. Compared with \cite{brigati2024heat}, Theorem~\ref{theo:functional-ineq-flowmatching} extends the log-Lipschitz perturbative regime ($\beta=1$) to the general Hölder case $\beta\in (0,1]$.

\section{Discussion and open problems}
In this paper, we established a unified regularity theory for the drift fields arising in flow matching and diffusion models under weak log-concavity assumptions. These estimates yield two main applications: optimal-order discretization guarantees for Euler-type samplers and the construction of globally Lipschitz transport maps, leading to functional inequalities for weakly log-concave measures. Several of the main technical novelties of the proofs are discussed alongside the corresponding tailored arguments: for the one-sided Lipschitz estimates, see Section~\ref{sec:technicalone} and for the global and time-Lipschitz estimates, see Section~\ref{sec:technicalglobal}. This work also opens a number of questions and possible extensions.
\subsection{Higher-order regularity and minimax convergence rates}

A natural next step is to complement the Lipschitz theory developed here with a higher-order regularity theory for the exact drift. Indeed, a complete convergence theory for flow-based generative models has two distinct ingredients. First, one needs stability of the sampling dynamics, in order to propagate an error on the learned drift to an error on the generated law. Second, one needs approximation properties of the exact drift itself, in order to quantify how well this drift can be learned from data within a given model class. The first ingredient is precisely what the present paper provides at the Lipschitz level. The second one naturally requires higher-order regularity, since approximation rates for neural networks are governed by Besov-type smoothness. In this sense, drift regularity seems to be one of the key quantities behind these methods: low-order regularity gives stability, while higher-order regularity makes the drift easy to approximate.

In the diffusion setting, this program is already visible in recent work. In \cite{Stephanovitch2025regularity} we prove higher-order regularity estimates for the  score function on sets $A_t^\varepsilon\subset \mathbb{R}^d$ satisfying $p_t(A_t^\varepsilon)\geq 1-\varepsilon$. These bounds adapt to the Hölder smoothness $\beta$ of the target and show in particular that if $\beta\geq 1$, on such high-probability sets, the derivatives of order $\gamma\leq \beta-1$ exhibit no singularity with respect to time. For approximation, the relevant error is measured under the true marginals $p_t$, hence on the region actually visited by the process. From this viewpoint, the fact that higher-order estimates hold only on high-probability sets is not a defect but rather the correct level of regularity for statistical purposes.  Building on this localized space-time regularity, \cite{stephanovitch2025generalization} obtain minimax convergence rates for score-based generative models, for both deterministic and stochastic samplers.

For the present paper, an important point is that all the flow-based methods considered here are built from the same mechanism: the exact drift is obtained as a conditional projection of an explicit field along a Gaussian-mixture bridge. This common structure strongly suggests that minimax optimality should admit a unified proof for all these methods. Once the right localized higher-order regularity theorem is available for such conditional drifts, the rest of the argument should be largely model-independent: approximate the exact drift in the natural $L^2(p_t)$ metric and propagate this approximation error by the Lipschitz stability estimates established in the present paper.

\subsection{The manifold case}

A natural next step is to extend the present regularity theory to targets that are not full-dimensional in the
ambient space. More precisely, one would like to study the case where $p^\star$ is absolutely continuous with
respect to the volume measure of a compact submanifold $\mathcal{M} \subset \mathbb{R}^d$ of dimension $m<d$.
This regime is particularly natural from the viewpoint of the manifold hypothesis, and it is also one of the
main reasons for working in Wasserstein distance rather than with $f$-divergences, since $W_2$ remains
meaningful for singular targets.

The first difficulty is that the ambient regularity quantities used throughout Sections~2--4 are no longer the
right ones near terminal time. Indeed, after Gaussian regularization, the intermediate marginals $p_t$ are
still smooth densities on $\mathbb{R}^d$, but as $t \to 1$ they concentrate in a tubular neighborhood of $\mathcal{M}$
whose radius is of the order of the noise scale. In the normal directions, the drift must collapse the mass
onto the manifold, so some derivatives of $v_t$ are necessarily singular. However, this singularity is not
automatically unstable from the point of view of one-sided Lipschitz estimates. For instance, in the flat case
where $\mathcal{M}$ is a linear subspace, one may contract the normal directions at rate $(1-t)^{-1}$ while keeping
$\lambda_{\max}(\nabla v_t)$ uniformly bounded from above. Thus, the real obstruction is not the collapse
onto a lower-dimensional set itself, but rather the fact that for curved manifolds the ambient supremum
\[
\sup_{x\in\mathbb{R}^d} \lambda_{\max}(\nabla v_t(x))
\]
may be dominated by regions that are far from the manifold, even though these regions carry negligible mass under $p_t$. In particular, one may well have
\[
\int_0^1 \sup_{x\in\mathbb{R}^d} \lambda_{\max}(\nabla v_t(x)) dt = +\infty,
\]
while the dynamics remains stable along the trajectories that are actually relevant for sampling. The correct heuristic is that, when $t$ is close to $1$, the law $p_t$ should be almost entirely concentrated
inside a small tube around $\mathcal{M}$. In such a neighborhood, one expects a tangential/normal decomposition of
the drift: the normal component should be strongly contracting, while the tangential component should be
governed by the intrinsic geometry of $\mathcal{M}$ and by the density of $p^\star$ with respect to the
volume measure. 

\paragraph{Toy example: the sphere.}
This phenomenon can already be seen explicitly in the toy model where $p^\star$ is the uniform measure on the unit sphere $\mathbb{S}^{d-1}\subset\mathbb R^d$. Consider the diffusion-model interpolation
\[
X_t=tY+\sigma_t \xi,
\qquad
\sigma_t=\sqrt{1-t^2},
\]
where $Y\sim \mathrm{Unif}(\mathbb{S}^{d-1})$ and $\xi\sim N(0,I_d)$ are independent. The associated probability flow ODE is driven by
\[
v_t(x)=\frac{1}{\sigma_t^2}\Bigl(\mathbb E[Y\mid X_t=x]-tx\Bigr).
\]

For $x\neq 0$, simple calculations show
\[
\mathbb E[Y\mid X_t=x]
=
A_d \Bigl(\frac{t\|x\|}{\sigma_t^2}\Bigr)\frac{x}{\|x\|},
\qquad
A_d(\kappa)=\frac{I_{d/2}(\kappa)}{I_{d/2-1}(\kappa)},
\]
where $I_\nu$ denotes the modified Bessel function of the first kind  \citep[Section~2.1]{BanerjeeDhillonGhoshSra2005}. Hence
\[
v_t(x)
=
\frac{1}{\sigma_t^2}
\left(
\frac{A_d \bigl(t\|x\|/\sigma_t^2\bigr)}{\|x\|}-t
\right)x.
\]
For every $x\neq 0$, the Jacobian has two eigenvalues; one in the tangential directions to the sphere and one in the radial direction:
\[
\lambda_{\mathrm{tan}}(x,t)
=
\frac{1}{\sigma_t^2}
\left(
\frac{A_d \bigl(t\|x\|/\sigma_t^2\bigr)}{\|x\|}-t
\right),
\qquad
\lambda_{\mathrm{rad}}(x,t)
=
\frac{1}{\sigma_t^2}
\left(
\frac{t}{\sigma_t^2}
A_d' \bigl(t\|x\|/\sigma_t^2\bigr)-t
\right).
\]

To understand the singular behavior as $t\uparrow 1$, we use the classical asymptotics
\[
A_d(\kappa)=1-\frac{d-1}{2\kappa}+O(\kappa^{-2}),
\qquad
A_d'(\kappa)=\frac{d-1}{2\kappa^2}+O(\kappa^{-3}),
\qquad \kappa\to\infty.
\]
Now, if $\|x\|=1+O(\sigma_t)$, that is, if $x$ remains in the natural tube of width comparable to the noise scale around the sphere, then
\[
\lambda_{\mathrm{tan}}(x,t)=O(\sigma_t^{-1}),
\qquad
\lambda_{\mathrm{rad}}(x,t)=-\sigma_t^{-2}+O(1).
\]
Hence along the relevant tube the largest eigenvalue grows at most like $\sigma_t^{-1}\sim (1-t)^{-1/2}$, which is integrable in time, while the stronger singular term $\sigma_t^{-2}\sim (1-t)^{-1}$ appears only in the radial direction and with the good sign.

By contrast, the ambient worst-case behavior is much worse. Indeed, using $A_d(\kappa)\sim \kappa/d$ as $\kappa\downarrow 0$, one gets
\[
\nabla v_t(0)
=
\left(
\frac{t}{d \sigma_t^4}-\frac{t}{\sigma_t^2}
\right)I_d,
\]
so in particular
\[
\lambda_{\max}(\nabla v_t(0))\sim \frac{1}{d \sigma_t^4}.
\]
Thus the ambient supremum of $\lambda_{\max}(\nabla v_t)$ is dominated by points that are far from the sphere and blows up much faster than what is seen along the trajectories relevant for sampling. This example therefore supports the idea that, in the manifold regime, the relevant notion of stability should be localized to the $O(\sigma_t)$-tube around the manifold rather than based on the ambient supremum over all $x\in\mathbb R^d$.

This suggests that the Gr\"onwall argument used in Section~\ref{sec:applicampl} should be refined. Rather than controlling the flow by a uniform supremum over all $x$, one would like a stability estimate driven by distribution-weighted or trajectory-wise quantities, localized to the tubular region effectively visited by the process. For instance, one may hope for controls involving averages under $p_t$, or more naturally localized versions  of $\lambda_{\max}(\nabla v_t)$ along the relevant coupling. Once the right notion of regularity has been identified, the next step will be to prove that it indeed holds in the manifold setting.

\subsection{Implicit regularization through training}

A central question left open by our sampling analysis is to understand what kind of regularity is actually induced by the training procedure. In Section~\ref{sec:samplingerror}, our perturbation bounds are written in terms of the population error
\begin{equation}\label{eq:fdppemdpdmdpmepe}
\epsilon_{\mathrm{drift}}(t):=\|a_t(X_t)-\hat a_t(X_t)\|_{L^2},
\end{equation}
where $X_t\sim p_t$ denotes the exact process. This choice is natural from the point of view of statistical learning: it is an $L^2(p_t)$-error under the true marginal, and therefore matches the kind of regression risk that is minimized during training. In particular, it cleanly separates the approximation error of the learned drift from the discretization and initialization errors of the sampler.

However, once one chooses this point of view, one also fixes a particular perturbative route. Indeed, the proof compares the exact continuous-time dynamics to a learned continuous-time dynamics driven by $\hat a_t$, and the corresponding stability estimate propagates the error through the learned vector field itself. This is why the analysis requires that the learned drift inherits the same structural regularity bounds as the exact drift on the sampling interval.

There is, however, a second possible point of view. Instead of comparing the exact process directly to a learned continuous-time process, one may compare the learned Euler scheme $\bar X$ to the Euler scheme $\tilde X$ built with the exact drift  $a_t$. In that case, for drift with symmetric Jacobian and step size small enough compared to the inverse of the Lipschitz constant, one obtains a decomposition of the form
\[
W_2\bigl(p^\star,\text{Law}(\bar X_N)\bigr)
\le
\underbrace{W_2\bigl(\text{Law}(X_1),\text{Law}(X_\tau)\bigr)}_{\text{early stopping}}
+
\underbrace{W_2\bigl(\text{Law}(X_\tau),\text{Law}(\tilde X_N)\bigr)}_{\text{discretization of the true process}}
+
e^{\sum_{m=0}^{N-1} h_mL_m}\sum \limits_{k=0}^{N-1} h_k\,\hat{\epsilon}_{\mathrm{drift}}(t_k),
\]
where $L_t$ is the one-sided Lipschitz modulus of the exact drift $a_t$, and where the learning error is now measured along the sampled trajectory:
\[
\hat{\epsilon}_{\mathrm{drift}}(t_k):=
\|a_{t_k}(\bar X_{t_k})-\hat a_{t_k}(\bar X_{t_k})\|_{L^2}.
\]
In this formulation, only the regularity of the exact drift $a_t$ is used; no regularity of $\hat a_t$ is needed. The price is that the error is now endogenous, since it is evaluated along the learned trajectory itself, whose law depends on $\hat a_t$. As a consequence, its statistical meaning is much less transparent than that of the population error~\eqref{eq:fdppemdpdmdpmepe}.

These two decompositions therefore correspond to two different interfaces between learning theory and sampling theory. The first one is statistically natural but requires regularity of the learned drift in order to propagate this error through the dynamics. The second one is dynamically natural, because it measures directly the drift mismatch along the trajectory actually produced by the sampler, and only uses regularity of the exact drift $a_t$, but it leads to a learning error that is harder to relate to the objective optimized in training. At present, it is not clear which of these two points of view is the most relevant one in practice.

This suggests a concrete empirical question. One should investigate numerically whether the trained drifts $\hat a_t$ appearing in realistic models actually inherit the same kind of regularity as the exact drift. If this seems to be the case, the main theoretical problem becomes to explain this phenomenon as a form of implicit regularization induced by the training algorithm, the architecture, or the noise structure of the objective. In that regime, the discussion of the present section should be understood as asking why stochastic gradient descent on neural networks selects vector fields whose dynamics remain stable enough for the population-error analysis to apply.

If, on the contrary, the learned drift does not appear to satisfy such regularity bounds, then the second decomposition becomes more relevant, and one is led to study the endogenous error $
\|a_{t_k}(\bar X_{t_k})-\hat a_{t_k}(\bar X_{t_k})\|_{L^2}.$
In that case, the main challenge is no longer to prove regularity of $\hat a_t$, but rather to understand statistically why this trajectory-based error should be small. This question is nontrivial. Indeed, one can construct simple examples in which the population error
\[
\|a_{t_k}(X_{t_k})-\hat a_{t_k}(X_{t_k})\|_{L^2}
\]
is arbitrarily small while the corresponding error along the learned trajectory is not. Therefore, if numerical experiments show that
\[
\|a_{t_k}(\bar X_{t_k})-\hat a_{t_k}(\bar X_{t_k})\|_{L^2}
\]
is nevertheless small in practice even without regularity of $\hat a_t$, this would indicate the presence of another form of implicit regularization in training, not captured by the usual population-risk point of view. Such a mechanism would not necessarily correspond to classical Lipschitz regularization of the learned vector field; rather, it would mean that training favors predictors whose errors remain small on the region of state space effectively explored by their own induced dynamics.

From this perspective, the most likely scenario is that training induces some effective regularity on the learned drift $\hat a_t$. Indeed, this would not only justify the population-error analysis of Section~\ref{sec:samplingerror}, but also explain why the trajectory-based quantity
\[
\|a_{t_k}(\bar X_{t_k})-\hat a_{t_k}(\bar X_{t_k})\|_{L^2}
\]
can remain small in practice: once the learned dynamics are sufficiently stable, the sampled trajectory stays in regions where the predictor has been trained to be accurate. Regularity of $\hat a_t$ would therefore unify the two points of view discussed above, and the real question becomes whether this phenomenon is indeed produced implicitly by training. Understanding this mechanism, both theoretically and numerically, seems to be a central direction for future work.

\bibliography{bib}

\clearpage
\appendix
\appendixtoctrue

\appendixtableofcontents

\section{Technical Lemmas}
\subsection{Measure concentration}
\subsubsection{Concentration of Lipschitz functions under strong log-concavity}
This appendix gathers a few concentration estimates that we use repeatedly when the reference measure is
strongly log-concave. The main intuition is that if
\[
d\mu(x)=e^{-V(x)}dx
\qquad\text{and}\qquad
\nabla^2V\succeq \gamma \mathrm{Id},
\]
then $\mu$ behaves like a Gaussian at spatial scale $\gamma^{-1/2}$: typical fluctuations of a
$1$-Lipschitz observable are of order $\gamma^{-1/2}$, uniformly in the dimension. As a consequence, if
a function is only $\beta$-Hölder, its typical oscillation under $\mu$ is of order
$(\gamma^{-1/2})^\beta=\gamma^{-\beta/2}$. This is the scale that later appears in our bounds for
tilted measures and posterior quantities.

The following Lemma is the standard Herbst argument: a logarithmic Sobolev inequality controls the
moment generating function of any Lipschitz observable. In practice this is our entry point to
dimension-free concentration.

\begin{lemma}[Herbst argument: Section~5.1 of \cite{ledoux-book}]\label{lem:herbst}
Let $\mu$ be a probability measure on $\R^d$. Assume that $\mu$ satisfies the logarithmic Sobolev
inequality: there exists $C_{\mathrm{LS}}>0$ such that for every smooth $h:\R^d\to(0,\infty)$,
\begin{equation}\label{eq:LSI_h2}
\text{Ent}_\mu(h^2)
:= \int h^2 \log \Big(\frac{h^2}{\int h^2 d\mu}\Big) d\mu
 \le C_{\mathrm{LS}} \int \|\nabla h\|^2 d\mu.
\end{equation}
Then for every $L$-Lipschitz function $g:\R^d\to\R$ and every $\lambda\in\R$,
\begin{equation}\label{eq:herbst_mgf}
\int \exp \Big(\lambda\Big(g-\int g d\mu\Big)\Big) d\mu
 \le \exp \Big(\frac{C_{\mathrm{LS}}L^2}{4} \lambda^2\Big).
\end{equation}
\end{lemma}

\begin{proof}
We first prove \eqref{eq:herbst_mgf} for smooth globally Lipschitz $g$; the general Lipschitz case then
follows by standard smooth approximation together with dominated convergence.

Fix such a $g$ and define, for $\lambda\in\R$,
\[
Z(\lambda):=\int e^{\lambda g} d\mu,
\qquad
\psi(\lambda):=\log Z(\lambda).
\]
Apply \eqref{eq:LSI_h2} with $h = e^{\lambda g/2}$,
\begin{equation}\label{eq:LSI_applied}
\text{Ent}_\mu(e^{\lambda g})
 \le \frac{C_{\mathrm{LS}}\lambda^2}{4}\int e^{\lambda g}\|\nabla g\|^2 d\mu
 \le \frac{C_{\mathrm{LS}}\lambda^2L^2}{4} Z(\lambda),
\end{equation}
since $\|\nabla g\|\le L$. On the other hand, by the definition of entropy and the identities
$Z'(\lambda)=\int g e^{\lambda g} d\mu$ and $\psi'(\lambda)=Z'(\lambda)/Z(\lambda)$, we have
\[
\text{Ent}_\mu(e^{\lambda g})
= \int e^{\lambda g}\log \Big(\frac{e^{\lambda g}}{Z(\lambda)}\Big) d\mu
= \lambda \int g e^{\lambda g} d\mu - Z(\lambda)\log Z(\lambda)
= Z(\lambda)\big(\lambda\psi'(\lambda)-\psi(\lambda)\big).
\]
Combining this with \eqref{eq:LSI_applied}, we obtain
\begin{equation}\label{eq:key_ode}
\lambda\psi'(\lambda)-\psi(\lambda) \le \frac{C_{\mathrm{LS}}L^2}{4} \lambda^2.
\end{equation}
For $\lambda>0$, define $\Phi(\lambda):=\psi(\lambda)/\lambda$, then
\[
\Phi'(\lambda)=\frac{\lambda\psi'(\lambda)-\psi(\lambda)}{\lambda^2}
 \le \frac{C_{\mathrm{LS}}L^2}{4}.
\]
Integrating from $0$ to $\lambda$ and using $\lim_{\lambda\to0}\psi(\lambda)/\lambda=\psi'(0)=\int g d\mu$
gives
\[
\frac{\psi(\lambda)}{\lambda} - \int g d\mu  \le \frac{C_{\mathrm{LS}}L^2}{4} \lambda,
\qquad\text{i.e.}\qquad
\psi(\lambda) - \lambda\int g d\mu  \le \frac{C_{\mathrm{LS}}L^2}{4} \lambda^2.
\]
Exponentiating yields \eqref{eq:herbst_mgf} for $\lambda\ge0$:
\[
\int e^{\lambda(g-\int g d\mu)} d\mu
= \exp \big(\psi(\lambda)-\lambda\int g d\mu\big)
\le \exp \Big(\frac{C_{\mathrm{LS}}L^2}{4}\lambda^2\Big).
\]
For $\lambda<0$, apply the already proved case to $-g$ to conclude the same bound.
\end{proof}

When $\nabla^2V\succeq\gamma\mathrm{Id}$, the measure $\mu$ satisfies a logarithmic Sobolev inequality
with constant of order $1/\gamma$. Plugging this into Lemma~\ref{lem:herbst} yields a sub-Gaussian tail
for $g-\int g d\mu$, and integrating the tail gives the dimension-free moment bound below. The
scaling $L^\alpha\gamma^{-\alpha/2}$ matches the heuristic ``typical fluctuation
$\sim L\gamma^{-1/2}$''.

\begin{lemma}\label{lem:lip-moments}
Let $\mu$ be a probability measure on $\mathbb{R}^d$ of the form
$
d\mu(x) = e^{-V(x)} dx,$
with $V:\mathbb{R}^d\to\mathbb{R}$ is $C^2$ and satisfies
\[
\nabla^2 V(x) \succeq \gamma \text{Id} \quad \text{for all } x\in\mathbb{R}^d
\]
for some $\gamma>0$. Then for $\alpha \ge 1$, we have for every $L$-Lipschitz function
$g:\mathbb{R}^d\to\mathbb{R}$ that
\[
\int \bigl|g - \int g d\mu\bigr|^\alpha d\mu
 \le 
C_\alpha L^\alpha \gamma^{-\alpha/2}.
\]
\end{lemma}

\begin{proof}
The curvature lower bound
$\nabla^2V \succeq \gamma \text{Id}$ implies that $\mu$ satisfies a logarithmic
Sobolev inequality with constant of order $1/\gamma$: there exists a universal
constant $C_{\mathrm{LS}}>0$ such that for every smooth $h$,
\begin{equation}\label{eq:lsi}
\mathrm{Ent}_\mu(h^2)
:= \int h^2\log\frac{h^2}{\int h^2 d\mu} d\mu
 \le 
\frac{C_{\mathrm{LS}}}{\gamma} \int \|\nabla h\|^2 d\mu.
\end{equation}

By Lemma~\ref{lem:herbst}, 
\eqref{eq:lsi} implies that for
$g$ be $L$-Lipschitz and $X := g - \int g d\mu$, there exists a constant
$c>0$, such that
\begin{equation}\label{eq:mgf}
\int e^{\lambda X} d\mu
 \le 
\exp  \Bigl(\frac{c L^2}{\gamma} \frac{\lambda^2}{2}\Bigr)
\quad\text{for all }\lambda\in\mathbb{R}.
\end{equation}
Let us write
\[
\sigma^2 := c \frac{L^2}{\gamma}.
\]
Then \eqref{eq:mgf} says exactly that $X$ is sub-Gaussian with parameter
$\sigma^2$:
\[
\mathbb{E}_\mu\bigl[e^{\lambda X}\bigr]
\le
\exp\Bigl(\frac{\sigma^2\lambda^2}{2}\Bigr)
\quad\forall\lambda\in\mathbb{R}.
\]

By Chernoff's bound, this  inequality implies the usual Gaussian tail:
for any $t>0$,
\[
\mu\bigl(|X|\ge t\bigr)
\le
2\exp\Bigl(-\frac{t^2}{2\sigma^2}\Bigr).
\]
Now fix $\alpha\ge 1$, using the tail-integral representation of moments,
\[
\mathbb{E}_\mu\bigl[|X|^\alpha\bigr]
=
\int_0^\infty \alpha t^{\alpha-1} \mu\bigl(|X|\ge t\bigr) dt,
\]
we obtain
\[
\mathbb{E}_\mu\bigl[|X|^\alpha\bigr]
\le
2\alpha \int_0^\infty t^{\alpha-1}
\exp\Bigl(-\frac{t^2}{2\sigma^2}\Bigr) dt.
\]
Applying a change of variable, we get
\[
\int_0^\infty t^{\alpha-1}e^{-t^2/(2\sigma^2)} dt
=
\sigma^\alpha 2^{(\alpha-2)/2}
\int_0^\infty s^{\alpha/2 -1}e^{-s} ds
=
\sigma^\alpha 2^{(\alpha-2)/2}\Gamma(\alpha/2).
\]
Thus
\[
\mathbb{E}_\mu\bigl[|X|^\alpha\bigr]
\le
2\alpha \sigma^\alpha 2^{(\alpha-2)/2}\Gamma(\alpha/2)
=
2^{\alpha/2} \alpha \Gamma(\alpha/2) \sigma^\alpha.
\]

Recalling that $\sigma^2 = c L^2/\gamma$, we have
\[
\sigma^\alpha
=
c^{\alpha/2} L^\alpha \gamma^{-\alpha/2},
\]
and hence
\[
\mathbb{E}_\mu\bigl[|X|^\alpha\bigr]
\le
C_\alpha L^\alpha \gamma^{-\alpha/2},\qquad \text{with } C_\alpha
:=
2^{\alpha/2} \alpha \Gamma(\alpha/2) c^{\alpha/2}.
\]
\end{proof}
In the discretization arguments, we need a sharp control of moments of the target law $p^\star$. The next Proposition shows that the structure of probability measures satisfying Definition~\ref{defi:weakconcave} suffices to ensure the ``Gaussian-scale'' second-moment bound $\mathbb{E}_{p^\star}\|X\|^2 \le C d$ assuming that the norm of the minimizer $\arg\min u$ is $O(\sqrt d)$.

\begin{proposition}[Second moment bound]\label{prop:second-moment}
Assume that there exist $\alpha>0$, $\beta\in(0,1]$, $K>0$ such that
\[
p^\star(x)=\frac1Z \exp\big(-u(x)+a(x)\big),\qquad \nabla^2 u(x)\succeq \alpha I_d,
\qquad |a(x)-a(y)|\le K\|x-y\|^\beta.
\]
Let $y_0:=\arg\min u$ and assume $\|y_0\|\le A\sqrt d$ for some $A>0$. Then there exists a constant
$C>0$ independent of $d$ such that
\[
\mathbb{E}_{p^\star}\big[\|X\|^2\big]\le C  d.
\]
\end{proposition}

\begin{proof} We present the proof in the case where $u$ is finite-valued and of class $C^2$ on $\mathbb{R}^d$, so that the integrations by parts used below are fully justified. The general case allowed by Definition~\ref{defi:weakconcave}, including compactly supported measures, follows by a standard approximation argument: one replaces the convex potential $u$ by a sequence of smooth strongly convex potentials $(u_\varepsilon)_{\varepsilon>0}$, applies the argument below to the corresponding smooth reference measures, and then lets $\varepsilon\to0$.

Let
\[
\mu_0(dx):=\frac1{Z_0}e^{-u(x)} dx,
\]
so $\mu_0$ is $\alpha$-strongly log-concave and
\[
p^\star(dx)=\frac{e^{a(x)}}{\E_{\mu_0}[e^{a}]} \mu_0(dx).
\]
Define the centered perturbation
\[
\bar a(x):=a(x)-\log \E_{\mu_0}[e^{a}],
\]
so that $p^\star(dx)=e^{\bar a(x)}\mu_0(dx)$ and $\E_{\mu_0}[e^{\bar a}]=1$, while $\bar a$ remains $\beta$-Hölder with
constant $K$.
For each coordinate $i$, by integration by parts under $\mu_0$,
\[
\E_{\mu_0}\big[(X_i-y_{0,i}) \partial_i u(X)\big]=1.
\]
Summing over $i=1,\dots,d$ gives
\[
\E_{\mu_0}\langle X-y_0,\nabla u(X)\rangle=d.
\]
Since $\nabla^2u\succeq \alpha I_d$ and $\nabla u(y_0)=0$, we have the strong monotonicity inequality
\[
\langle x-y_0,\nabla u(x)\rangle \ge \alpha \|x-y_0\|^2,
\]
hence
\[
\E_{\mu_0}\|X-y_0\|^2\le \frac d\alpha.
\]
Therefore,
\[
\E_{\mu_0}\|X\|^2\le 2\|y_0\|^2 + 2\E_{\mu_0}\|X-y_0\|^2
\le \Big(2A^2+\frac{2}{\alpha}\Big)d.
\]
Moreover, applying  Lemma~\ref{lem:lip-moments} to $g(x)=\|x-y_0\|$ yields
\begin{equation}\label{eq:fourth-moment}
\E_{\mu_0}\|X\|^4 \le C_0\Big(A^4+\alpha^{-2}\Big)d^2.
\end{equation}
By Lemma~\ref{lem:var_exp_concentration},
using that $\E_{\mu_0}[e^{\bar a}]=1$,
\[
\E_{\mu_0}[e^{2\bar a}] = 1+\text{Var}_{\mu_0}(e^{\bar a})
\le \exp \Big(4K\alpha^{-\beta/2}+C_1 K^2\alpha^{-\beta}\Big)=:M,
\]
where $C_1>0$ is universal. In particular $\mathcal{M}$ is independent of $d$. Now, 
since $p^\star(dx)=e^{\bar a(x)}\mu_0(dx)$,
\[
\E_{p^\star}\|X\|^2=\E_{\mu_0}\big[\|X\|^2 e^{\bar a(X)}\big]
\le \Big(\E_{\mu_0}\|X\|^4\Big)^{1/2}\Big(\E_{\mu_0}e^{2\bar a}\Big)^{1/2}.
\]
Combining \eqref{eq:fourth-moment} with $\E_{\mu_0}e^{2\bar a}\le M$ gives
\[
\E_{p^\star}\|X\|^2 \le \sqrt{C\Big(A^4+\alpha^{-2}\Big)} \sqrt{M}\; d.
\]
\end{proof}
\subsubsection{Concentration of Hölder functions under strong log-concavity}
Throughout the paper, we control quantities that are only Hölder, but Lemma~\ref{lem:herbst} applies to
Lipschitz observables. The next result provides a deterministic ``smoothing at scale $r$'':
the approximation error is of order $r^\beta$, while the Lipschitz constant deteriorates like
$r^{\beta-1}$. This trade-off is exactly what will generate the two terms in the variance bound
of Lemma~\ref{lem:var_exp_concentration}.

\begin{lemma}[Inf-convolution Lipschitz approximation]
\label{lem:infconv}
Let $0<\beta\leq 1$, and suppose $f:\mathbb{R}^d\to\mathbb{R}$ is $\beta$-Hölder with constant $L>0$.
Fix $r>0$ and set $L':=L r^{\beta-1}$. Define the inf-convolution
\[
\tilde g_r(x)  :=  \inf_{y\in\mathbb{R}^d} \big\{  f(y) + L'\|x-y\|  \big\}.
\]
Then $\tilde g_r$ is $L'$-Lipschitz and $\|\tilde g_r-f\|_\infty \le L r^\beta$.
\end{lemma}

\begin{proof}
For any $x,x',y\in\mathbb{R}^d$ we have
$$f(y)+L'\|x-y\| \le f(y)+L'\|x'-y\| + L'\|x-x'\|,$$
so taking the infimum over $y$ on the right yields
$\tilde g_r(x) \le \tilde g_r(x') + L'\|x-x'\|$. Exchanging $x$ and $x'$ proves the Lipschitz bound.

The upper bound $\tilde g_r(x)\le f(x)$ follows by choosing $y=x$ in the infimum. Now for the lower bound $\tilde g_r(x)\ge f(x)-L r^\beta$,
we have
\[
f(y) + L'\|x-y\|
 \ge f(x) - L\|x-y\|^\beta + L r^{\beta-1}\|x-y\|.
\]
and writing $t:=\|x-y\|\ge0$ and  $u:=t/r$, we get
\[
- L t^\beta + L r^{\beta-1} t  =  L r^\beta \big(-u^\beta + u\big).
\]
For all $u\ge0$ and $\beta\in(0,1]$, we have $-u^\beta + u \ge -1$ so
\[
f(y)+L'\|x-y\|  \ge f(x) - L r^\beta \qquad \forall y\in\mathbb{R}^d.
\]
Taking the infimum over $y$ gives $\tilde g_r(x)\ge f(x)-L r^\beta$, as claimed.
\end{proof}

In several places we tilt $\mu$ by a weight $e^{v}$ and need to control how far this tilted measure is
from $\mu$. A convenient proxy is $\text{Var}_\mu(e^{v})$. When $v$ is only Hölder, the idea is to replace
it by the Lipschitz approximation $\tilde v_r$ from Lemma~\ref{lem:infconv} at scale $r$:
(i) the approximation error contributes a term of order $Kr^\beta$, while
(ii) concentration for the Lipschitz function $\tilde v_r$ contributes a term of order
$\frac{1}{\gamma}(Kr^{\beta-1})^2$.

\begin{lemma}[Exponential variance bound]\label{lem:var_exp_concentration}
Let $\mu$ be a probability measure on $\mathbb{R}^d$ which is $\gamma$-strongly log-concave for some
$\gamma>0$. Let $c>0$ be the universal constant $C_{\mathrm{LS}}$ from Lemma~\ref{lem:herbst}.
Assume $v:\mathbb{R}^d\to\mathbb{R}$ is $\beta$-Hölder for some $\beta\in(0,1]$ with constant $K>0$ and
$
\int e^{v} d\mu = 1.$
Then for every $r>0$,
\begin{equation}\label{eq:var_bound_r}
\text{Var}_ \mu(e^v)\le
\exp \Big( 4K r^\beta + \frac{2c}{\gamma}K^2 r^{2\beta-2}\Big)-1.
\end{equation}
\end{lemma}

\begin{proof}
Fix $r>0$ and set $L':=K r^{\beta-1}$. Define the inf-convolution
\[
\tilde v_r(x) := \inf_{y\in\mathbb{R}^d}\big\{ v(y) + L'\|x-y\| \big\}.
\]
By Lemma~\ref{lem:infconv}, $\tilde v_r$ is $L'$-Lipschitz and satisfies the uniform approximation bound
\begin{equation}\label{eq:approx_bound}
\|\tilde v_r - v\|_\infty \le K r^\beta.
\end{equation}
From \eqref{eq:approx_bound}, we have
\[
\int e^{\tilde v_r} d\mu \le e^{Kr^\beta}\int e^{v} d\mu = e^{Kr^\beta}.
\]
Let $m_r:=\int \tilde v_r d\mu$, by Jensen's inequality,
\[
m_r \le \log\int e^{\tilde v_r} d\mu \le Kr^\beta.
\]
Hence,
\begin{equation*}
e^{2m_r}\le e^{2Kr^\beta}.
\end{equation*}
Since $\tilde v_r$ is $L'$-Lipschitz, we may apply Lemma~\ref{lem:herbst} with
$g=\tilde v_r$ and $\lambda=2$ to get
\[
\int \exp \Big(2(\tilde v_r-m_r)\Big) d\mu
\le \exp \Big(\frac{c(L')^2}{\gamma}\frac{2^2}{2}\Big)
= \exp \Big(\frac{2c(L')^2}{\gamma}\Big).
\]
Therefore,
\[
\int e^{2\tilde v_r} d\mu
= e^{2m_r}\int \exp \Big(2(\tilde v_r-m_r)\Big) d\mu
\le \exp \Big(2Kr^\beta + \frac{2c}{\gamma}(L')^2\Big).
\]
Recalling $L'=K r^{\beta-1}$, this becomes
\begin{equation}\label{eq:int_e2vtilde}
\int e^{2\tilde v_r} d\mu
\le \exp \Big(2Kr^\beta + \frac{2c}{\gamma}K^2 r^{2\beta-2}\Big).
\end{equation}
From \eqref{eq:approx_bound} we have $v \le \tilde v_r + Kr^\beta$, thus by \eqref{eq:int_e2vtilde},
\[
\int e^{2v} d\mu
\le e^{2Kr^\beta}\int e^{2\tilde v_r} d\mu
\le \exp \Big(4Kr^\beta + \frac{2c}{\gamma}K^2 r^{2\beta-2}\Big).
\]
Finally, using $\int e^v d\mu=1$,
\[
\text{Var}_ \mu(e^v)=\int e^{2v} d\mu - \Big(\int e^v d\mu\Big)^2
\le \exp \Big(4Kr^\beta + \frac{2c}{\gamma}K^2 r^{2\beta-2}\Big)-1,
\]
which proves \eqref{eq:var_bound_r}.
\end{proof}

\subsection{Estimates for convolution}
\subsubsection{Regularity and curvature}
This subsection gathers two facts about Gaussian smoothing of weakly log-concave measures.
They will be used repeatedly to control the geometry of intermediate distributions appearing in the flow:
\begin{itemize}
\item Proposition~\ref{prop:bondhessut} is a \emph{curvature statement}: convolving a strongly log-concave density with a Gaussian
produces a smooth density whose log remains uniformly curved, with explicit two-sided Hessian bounds.
\item Lemmas~\ref{lemma:holderexp}-\ref{lem:mgf_holder} are \emph{regularity statements}: if we further perturb the strongly log-concave core by a Hölder factor
$e^{a}$, then the induced log-normalization $a_t$ remains Hölder in the observation variable.
\end{itemize}
Together, these results explain why, along the bridge, one can work with decompositions of the form
$ \exp\big(-u_t(x)+a_t(x)\big),$
where $u_t$ is strongly convex and $a_t$ is Hölder continuous. 
\begin{proposition}[Pointwise Hessian bounds for the log-convolution \(u_t\)]\label{prop:bondhessut}
Let $u:\mathbb{R}^d\to\mathbb{R}$ be $\alpha$-strongly convex 
and $t\in(0,1)$ with $g_t>0$ and $f_t>0$.
Define
\[
\bar u_t(z) := u \left(\frac{z}{f_t}\right) + d\log f_t
\qquad\text{and}\qquad
u_t(x) := -\log \int_{\mathbb{R}^d}\exp \Big(-\bar u_t(z) - \frac{\|x-z\|^2}{2g_t^2}\Big) dz.
\]
Then for all $x\in\mathbb{R}^d$,
\begin{equation}\label{eq:ut-hessian-two-sided}
\frac{\alpha}{\alpha g_t^2 + f_t^2} \text{Id}  \preceq  \nabla^2 u_t(x)  \preceq  \frac{1}{g_t^2} \text{Id}.
\end{equation}
\end{proposition}

\begin{proof}
Fix $t$ with $g_t>0$. For $x\in\mathbb{R}^d$ set
\[
\Phi_{t,x}(z) := \bar u_t(z) + \frac{\|x-z\|^2}{2g_t^2} ,
\qquad
Z_t(x) := \int_{\mathbb{R}^d} e^{-\Phi_{t,x}(z)} dz,
\]
so that $u_t(x)=-\log Z_t(x)$. Define the probability measure
\[
q_{t,x}(dz) := \frac{1}{Z_t(x)} e^{-\Phi_{t,x}(z)} dz.
\]
We have
\[
\nabla Z_t(x) = \int_{\mathbb{R}^d} \Big(-\nabla_x \Phi_{t,x}(z)\Big) e^{-\Phi_{t,x}(z)} dz
 = -\frac{1}{g_t^2}\int_{\mathbb{R}^d} (x-z) e^{-\Phi_{t,x}(z)} dz,
\]
hence
\begin{align*}
\nabla u_t(x)
&= -\frac{\nabla_x Z_t(x)}{Z_t(x)}
= \frac{1}{g_t^2}\left(x - \mathbb{E}_{q_{t,x}}[Z]\right).
\end{align*}

Differentiating the previous identity again gives
\[
\nabla^2 u_t(x) = \frac{1}{g_t^2}\text{Id} - \frac{1}{g_t^2}\mathbb{E}_{q_{t,x}} \big[Z\otimes \nabla_x \log q_{t,x}(Z)\big].
\]
Now,
\[
\nabla_x \log q_{t,x}(z) = -\nabla_x \Phi_{t,x}(z) - \nabla_x \log Z_t(x)
= -\frac{1}{g_t^2}(x-z) + \frac{1}{Z_t(x)}\nabla_x Z_t(x).
\]
so
\[
\nabla_x \log q_{t,x}(z)  = -\frac{1}{g_t^2}(x-z) + \frac{1}{g_t^2}\big(x-\mathbb{E}_{q_{t,x}}[Z]\big)
 = \frac{1}{g_t^2}\big(z-\mathbb{E}_{q_{t,x}}[Z]\big).
\]
Therefore,
\[
 \mathbb{E}_{q_{t,x}} \Big[ Z\otimes \nabla_x \log q_{t,x}(Z) \Big]
 = \frac{1}{g_t^2} \text{Cov}_{q_{t,x}}(Z),
\]
so we deduce that
\begin{equation}\label{eq:hessian-identity}
\nabla^2 u_t(x) = \frac{1}{g_t^2}\text{Id} - \frac{1}{g_t^4} \text{Cov}_{q_{t,x}}(Z).
\end{equation}
which in particular yields
\[
\nabla^2 u_t(x) \preceq \frac{1}{g_t^2}\text{Id}.
\]
Let us now use Brascamp-Lieb inequality to bound the covariance term.
Since,
\[
\nabla^2 \Phi_{t,x}(z) = \nabla^2 \bar u_t(z) + g_t^{-2} \text{Id}.
\]
 and $\bar u_t(z)=u(z/f_t)+d\log f_t$, we have
\[
\nabla^2 \Phi_{t,x}(z) \succeq \Big(\frac{\alpha}{f_t^2} + \frac{1}{g_t^2}\Big) \text{Id},
\]
so we deduce that
$$\text{Cov}_{q_{t,x}}(Z) \preceq \Big(\frac{\alpha}{f_t^2} + \frac{1}{g_t^2}\Big)^{-1} \text{Id}.$$ Plugging this into \eqref{eq:hessian-identity} gives
\[
\nabla^2 u_t(x) \succeq \frac{1}{g_t^2}\text{Id} - \frac{1}{g_t^4}\Big(\frac{\alpha}{f_t^2} + \frac{1}{g_t^2}\Big)^{-1} \text{Id}
\]
and
\[
\frac{1}{g_t^2} - \frac{1}{g_t^4}\Big(\frac{\alpha}{f_t^2} + \frac{1}{g_t^2}\Big)^{-1}
 = \frac{\alpha}{\alpha g_t^2 + f_t^2},
\]
which yields the stated bound.
\end{proof}

In the Stochastic-interpolant setting, controlling the Jacobian of the projected drift  requires a quantitative bound on how much the path derivative can fluctuate under the posterior coupling. The following result shows that one can first bound the variance under a strongly log-concave ``surrogate'' posterior using Brascamp--Lieb, and then transfer the estimate to $\pi^{t,x}$ at the cost of an exponential-moment.

\begin{proposition}\label{prop:boundvarmprime}
Fix $t\in[0,1]$, $f_t,g_t,\sigma_t>0$, and define
\[
m_t(x_0,x_1)=g_t x_0+f_t x_1.
\]
For $x\in\mathbb R^d$, let $\pi^{t,x}$ be the probability measure on $\mathbb{R}^{2d}$ having its density proportional to
\[
\exp \Big(-\tfrac{\|x-(g_t x_0+f_t x_1)\|^2}{2\sigma_t^2}\Big)  d\gamma(x_0)  p^\star(x_1),
\]
with $p^\star=e^{-u+a}$ being $(\alpha,\beta,K)$-weakly log-concave (no boundedness assumption on $a$).

Set $\bar g_t^2:=\sigma_t^2+g_t^2$ and, for $x\in\R^d$, define the probability measure on $\R^d$
\[
\widetilde\nu^{t,x}(dy) \propto \exp \Big(-u(y)-\frac{\|x-f_t y\|^2}{2\bar g_t^2}\Big) dy,
\]
and the exponential-moment quantity
\[
M_{t,2}
:=\sup_{x\in\R^d}\mathbb{E}_{\widetilde\nu^{t,x}} \Big[\exp\big(2(a(Y)-\mathbb{E}_{\widetilde\nu^{t,x}}a(Y))\big)\Big].
\]
Then there exists a universal constant $C>0$ such that for every unit vector $h\in\mathbb S^{d-1}$,
\[
\operatorname{Var}_{\pi^{t,x}} \big(h^\top m_t'(Z)\big)
 \le 
C  M_{t,2}^{1/2} 
\frac{ f_t'^2 + \alpha g_t'^2 + \dfrac{(g_t' f_t - f_t' g_t)^2}{\sigma_t^2} }
{ \alpha + \dfrac{f_t^2 + \alpha g_t^2}{\sigma_t^2} }.
\]
\end{proposition}

\begin{proof}
Consider the surrogate log-concave measure $\rho^{t,x}$ whose negative log-density can be written as
\[
V(X_0,X_1)
=\frac12\|X_0\|^2 + u(X_1) + \frac{1}{2\sigma_t^2} \big\|g_t X_0 + f_t X_1 - x\big\|^2,
\]
up to normalizing constant.
Then
\[
\nabla^2 V  = 
\begin{bmatrix}
\text{Id} & 0\\[2pt]
0 & \nabla^2 u(X_1)
\end{bmatrix}
+\frac{1}{\sigma_t^2}
\begin{bmatrix}
g_t^2 \text{Id} & g_t f_t \text{Id}\\[2pt]
g_t f_t \text{Id} & f_t^2 \text{Id}
\end{bmatrix}.
\]
Using the curvature assumption $\nabla^2 u(X_1)\succeq \alpha \text{Id}$ we get
\[
\nabla^2 V  \succeq  A_{\mathrm{low}}
  :=  
\begin{bmatrix}
\text{Id} & 0\\[2pt]
0 & \alpha \text{Id}
\end{bmatrix}
+\frac{1}{\sigma_t^2}
\begin{bmatrix}
g_t^2 \text{Id} & g_t f_t \text{Id}\\[2pt]
g_t f_t \text{Id} & f_t^2 \text{Id}
\end{bmatrix}.
\tag{$\star$}
\]

Let $F(x_0,x_1):=h^\top m_t'(x_0,x_1)= g_t' h^\top x_0 + f_t' h^\top x_1$.
Brascamp--Lieb inequality (constant gradient $\nabla F=(g_t' h, f_t' h)$) yields
\[
\mathrm{Var}_{\rho^{t,x}}(F(Z)) \le \mathbb{E}_{\rho^{t,x}} \big[\nabla F^\top (\nabla^2 V)^{-1}\nabla F\big]
 \le  \nabla F^\top A_{\mathrm{low}}^{-1}\nabla F.
\]
The matrix $A_{\mathrm{low}}$ is a direct sum over the $d$ coordinates of identical $2\times2$ blocks, so the
quadratic form above is determined by the $2\times2$ matrix
\[
 M(\alpha)
:=
\begin{bmatrix}
1+\dfrac{g_t^2}{\sigma_t^2} & \dfrac{g_t f_t}{\sigma_t^2}\\[8pt]
\dfrac{g_t f_t}{\sigma_t^2} & \alpha+\dfrac{f_t^2}{\sigma_t^2}
\end{bmatrix},
\qquad
w:=
\begin{bmatrix}
g_t'\\
f_t'
\end{bmatrix},
\]
so
\[
\mathrm{Var}_{\rho^{t,x}}(F) \le  w^\top M(\alpha)^{-1} w
=
\frac{ f_t'^2 + \alpha g_t'^2 + \dfrac{(g_t' f_t - f_t' g_t)^2}{\sigma_t^2} }
{ \alpha+\dfrac{f_t^2+\alpha g_t^2}{\sigma_t^2} }.
\]

Now note that $d\pi^{t,x}\propto e^{a(X_1)} d\rho^{t,x}$.
Let
\[
L:=\frac{e^{a(X_1)}}{\mathbb{E}_{\rho^{t,x}}[e^{a(X_1)}]},
\qquad \text{so that}\qquad
\mathbb{E}_{\rho^{t,x}}[L]=1
 \text{ and }\
\mathbb{E}_{\pi^{t,x}}[\cdot]=\mathbb{E}_{\rho^{t,x}}[L \cdot].
\]
Using $\text{Var}_{\pi^{t,x}}(F)=\inf_{c\in\R}\mathbb{E}_{\pi^{t,x}}[(F-c)^2]$ and taking $c=\mathbb{E}_{\rho^{t,x}}F$, we get
\[
\text{Var}_{\pi^{t,x}}(F)\le \mathbb{E}_{\pi^{t,x}} \big[(F-\mathbb{E}_{\rho^{t,x}}F)^2\big]
=\mathbb{E}_{\rho^{t,x}} \big[L(F-\mathbb{E}_{\rho^{t,x}}F)^2\big].
\]
By Cauchy-Schwarz,
\[
\mathbb{E}_{\rho^{t,x}} \big[L(F-\mathbb{E}_{\rho^{t,x}}F)^2\big]
\le
\big(\mathbb{E}_{\rho^{t,x}}[L^2]\big)^{1/2} 
\big(\mathbb{E}_{\rho^{t,x}}[(F-\mathbb{E}_{\rho^{t,x}}F)^4]\big)^{1/2}.
\]
Since $\rho^{t,x}$ is log-concave on $\R^{2d}$, the one-dimensional law of
$F-\mathbb{E}_{\rho^{t,x}}F$ is log-concave on $\R$, hence there is a universal constant $C>0$ such that
\[
\big(\mathbb{E}_{\rho^{t,x}}[(F-\mathbb{E}_{\rho^{t,x}}F)^4]\big)^{1/2}
\le
C \text{Var}_{\rho^{t,x}}(F).
\]
Moreover, writing $Y:=X_1$ under $\rho^{t,x}$, we have
\[ 
\mathbb{E}_{\rho^{t,x}}[L^2]
=
\frac{\mathbb{E}_{\rho^{t,x}}[e^{2a(Y)}]}{\mathbb{E}_{\rho^{t,x}}[e^{a(Y)}]^2}
=
\frac{\mathbb{E}_{\rho^{t,x}}[e^{2(a(Y)-\E a(Y))}]}{\mathbb{E}_{\rho^{t,x}}[e^{a(Y)-\E a(Y)}]^2}
\le
\mathbb{E}_{\rho^{t,x}}[e^{2(a(Y)-\E a(Y))}],
\]
since $\mathbb{E}_{\rho^{t,x}}[e^{a(Y)-\E a(Y)}]\ge 1$ by Jensen.
Finally, the marginal of $Y=X_1$ under $\rho^{t,x}$ is precisely $\widetilde\nu^{t,x}$ (with $\bar g_t^2=\sigma_t^2+g_t^2$),
so
\[
\mathbb{E}_{\rho^{t,x}}[L^2]\le M_{t,2}.
\]
Putting everything together yields
\[
\text{Var}_{\pi^{t,x}}(F)
\le
C  M_{t,2}^{1/2} \text{Var}_{\rho^{t,x}}(F),
\]
and inserting the bound on $\text{Var}_{\rho^{t,x}}(F)$ gives the claim.
\end{proof}

We now turn to the ``regularity through tilting'' mechanism.
Given a strongly log-concave baseline $e^{-u}$, we introduce a rough perturbation $e^{a}$ and consider the log-normalization
\[
a_t(x)=\log\Big(\mathbb E_{\nu^{t,x}}[e^{a(Y)}]\Big),
\]
where $\nu^{t,x}$ is the (strongly log-concave) surrogate posterior.
The point is that moving $x$ only changes $\nu^{t,x}$ through a smooth Gaussian likelihood, and strong convexity prevents the posterior
from moving too much. This stability is summarized by
\[
\kappa_t=\frac{f}{\alpha g^2+f^2},
\]
which becomes small when the curvature $\alpha$ is large or the noise level $g$ is large.

\begin{lemma}[Hölder regularity of $a_t$ under exponential moment bounds]\label{lemma:holderexp}
Fix $t \in (0,1)$ and let $f,g>0$.
Assume that $p^\star=e^{-u+a}$ where $u:\R^d\to\R$ is $\alpha$-strongly convex
for some $\alpha> 0$ 
and $a:\R^d\to\R$ is $\beta$-Hölder with constant $K>0$.
For $x\in\R^d$, define the probability measure
\[
\nu^{t,x}(dy)\propto e^{-u(y)} \varphi^{x,g}(fy) dy,
\]
and set
\[
a_t(x):=\log\Big(\mathbb{E}_{\nu^{t,x}} e^{a(Y)}\Big).
\]
Assume moreover that for every $q\ge 0$ there exists $M_q<\infty$ such that
\begin{equation}
\sup_{x\in\R^d}\mathbb{E}_{\nu^{t,x}} \Big[e^{q(a(Y)-\mathbb{E}_{\nu^{t,x}}a(Y))}\Big]
+\sup_{x\in\R^d}\mathbb{E}_{\nu^{t,x}} \Big[e^{-q(a(Y)-\mathbb{E}_{\nu^{t,x}}a(Y))}\Big]
\le M_q. \label{eq:exp-mom-assump}
\end{equation}
Then, writing
\[
\kappa_t:=\frac{f}{\alpha g^2+f^2},
\]
for all $x,x'\in\R^d$ we have:
\begin{enumerate}[(i)]
\item If $\beta\in(0,1)$, then
\[
|a_t(x)-a_t(x')|\le K  M_{2/(1-\beta)}^{ 1-\beta}  \kappa_t^{\beta}  \|x-x'\|^\beta.
\]
\item If $\beta=1$, then
\[
|a_t(x)-a_t(x')|\le K  M_{4}^{ 1/2}  \kappa_t  \|x-x'\|.
\]
\end{enumerate}
\end{lemma}

\begin{proof}
Let $(\Omega,\pi)$ be a probability space and let $A,B$ be integrable real random variables on $\Omega$.
Fix $p>1$ and set $q=\frac{p}{p-1}$. Assume that for all $r\ge 0$,
\begin{equation}\label{eq:mgfA-mgfB}
\mathbb{E}_ \pi\big[e^{r(A-\mathbb{E}_ \pi A)}\big]\le M_r,
\qquad
\mathbb{E}_ \pi\big[e^{r(B-\mathbb{E}_ \pi B)}\big]\le M_r,
\end{equation}
for some finite constants $(M_r)_{r\ge 0}$.
(If $(M_r)$ is not nondecreasing, replace it by $\widetilde M_r:=\sup_{0\le s\le r}M_s$; this only worsens constants and makes it nondecreasing.)

Define for $s\in[0,1]$ the function
\[
F(s):=\log \mathbb{E}_ \pi\big[e^{(1-s)A+sB}\big].
\]
A standard differentiation under the integral (justified here by Hölder and the exponential moment bounds below)
gives
\[
F'(s)=\frac{\mathbb{E}_ \pi\big[(B-A)e^{(1-s)A+sB}\big]}{\mathbb{E}_ \pi\big[e^{(1-s)A+sB}\big]}
=\mathbb{E}_ {\mu_s}[B-A],
\]
where $\mu_s$ is the tilted law with density proportional to $e^{(1-s)A+sB}$ with respect to $\pi$.
Hence
\[
|F(1)-F(0)|\le \int_0^1 \mathbb{E}_ {\mu_s}[|A-B|] ds.
\]
Let $W_s:=e^{(1-s)A+sB}$ and $Z_s:=\mathbb{E}_ \pi[W_s]$. By Hölder with exponents $(p,q)$,
\[
\mathbb{E}_ {\mu_s}[|A-B|]
=\frac{\mathbb{E}_ \pi[|A-B|W_s]}{Z_s}
\le
\frac{\|A-B\|_{L^p(\pi)} \|W_s\|_{L^q(\pi)}}{Z_s}.
\]
We bound the ratio $\|W_s\|_{L^q}/Z_s$ uniformly in $s$.
First, Jensen yields
\[
\log Z_s=\log \mathbb{E}_ \pi[e^{(1-s)A+sB}]
\ge \mathbb{E}_ \pi[(1-s)A+sB]=(1-s)\mathbb{E}_ \pi A+s\mathbb{E}_ \pi B,
\]
so
\[
Z_s\ge e^{(1-s)\mathbb{E}_ \pi A+s\mathbb{E}_ \pi B}.
\]
Next, by Cauchy-Schwarz,
\[
\|W_s\|_{L^q(\pi)}^q
=\mathbb{E}_ \pi\big[e^{q((1-s)A+sB)}\big]
\le
\Big(\mathbb{E}_  \pi\big[e^{2q(1-s)A}\big]\Big)^{1/2}
\Big(\mathbb{E}_ \pi\big[e^{2qsB}\big]\Big)^{1/2}.
\]
Using $\mathbb{E}[e^{rA}]=e^{r\mathbb{E} A}\mathbb{E}[e^{r(A-\mathbb{E} A)}]\le e^{r\mathbb{E} A}M_r$ from \eqref{eq:mgfA-mgfB} (and similarly for $B$), 
\[
\mathbb{E}_ \pi[e^{2q(1-s)A}] \le e^{2q(1-s)\mathbb{E}_ \pi A} M_{2q(1-s)},
\qquad
\mathbb{E}_ \pi[e^{2qsB}] \le e^{2qs\mathbb{E}_ \pi B} M_{2qs}.
\]
Therefore,
\[
\|W_s\|_{L^q(\pi)}
\le
e^{(1-s)\mathbb{E}_ \pi A+s\mathbb{E}_ \pi B} \big(M_{2q(1-s)}M_{2qs}\big)^{1/(2q)}
\le
e^{(1-s)\mathbb{E}_ \pi A+s\mathbb{E}_ \pi B} M_{2q}^{1/q},
\]
where we used the (w.l.o.g.) monotonicity of $r\mapsto M_r$ on $[0,2q]$.
Combining with the lower bound on $Z_s$ gives
\[
\frac{\|W_s\|_{L^q(\pi)}}{Z_s}\le M_{2q}^{1/q}.
\]
Hence $\mathbb{E}_ {\mu_s}[|A-B|]\le M_{2q}^{1/q}\|A-B\|_{L^p(\pi)}$ for all $s\in[0,1]$, and integrating yields
\begin{equation}\label{eq:loglap-stab}
\big|\log \mathbb{E}_ \pi[e^A]-\log \mathbb{E}_ \pi[e^B]\big|
\le
M_{2q}^{1/q} \|A-B\|_{L^p(\pi)}.
\end{equation}

Fix $x,x'\in\R^d$, and let $\pi$ be any coupling of $\nu^{t,x}$ and $\nu^{t,x'}$.
Let $(Y,Y')\sim\pi$.

Apply \eqref{eq:loglap-stab} on $(\Omega,\pi)$ with $A=a(Y)$ and $B=a(Y')$.
The assumption \eqref{eq:exp-mom-assump} implies \eqref{eq:mgfA-mgfB} for $A$ and $B$ (uniformly in $x,x'$),
because under $\pi$ the marginals of $Y$ and $Y'$ are $\nu^{t,x}$ and $\nu^{t,x'}$ respectively.
Thus, for any $p>1$ (with $q=\frac{p}{p-1}$),
\[
|a_t(x)-a_t(x')|
=
\Big|\log \mathbb{E}_ {\nu^{t,x}}[e^{a(Y)}]-\log \mathbb{E}_ {\nu^{t,x'}}[e^{a(Y')}]\Big|
\le
M_{2q}^{1/q} \|a(Y)-a(Y')\|_{L^p(\pi)}.
\]
By the $\beta$-Hölder property of $a$,
\[
\|a(Y)-a(Y')\|_{L^p(\pi)}
\le
K \big(\mathbb{E}_ \pi[\|Y-Y'\|^{\beta p}]\big)^{1/p}.
\]

\medskip
\noindent\textbf{Case $\beta\in(0,1)$.}

Now choose $p=\frac{1}{\beta}$ (so $p>1$ because $\beta\in(0,1)$), hence $q=\frac{1}{1-\beta}$ and $\beta p=1$.
We obtain
\[
|a_t(x)-a_t(x')|
\le
K M_{\frac{2}{1-\beta}}^{ 1-\beta} \big(\mathbb{E}_ \pi[\|Y-Y'\|]\big)^\beta.
\]
Taking the infimum over all couplings $\pi$ yields
\begin{equation}\label{eq:at-W1}
|a_t(x)-a_t(x')|
\le
K M_{\frac{2}{1-\beta}}^{ 1-\beta} W_1(\nu^{t,x},\nu^{t,x'})^\beta.
\end{equation}

Write, up to an additive constant in $y$,
\[
\nu^{t,x}(dy) \propto \exp\Big(-U(y)+\Big\langle \frac{f}{g^2}x, y\Big\rangle\Big) dy,
\qquad
U(y):=u(y)+\frac{f^2}{2g^2}\|y\|^2.
\]
Then
\[
\nabla^2 U \succeq \Big(\alpha+\frac{f^2}{g^2}\Big)I_d =:\gamma \text{Id}.
\]
Let $h:\R^d\to\R$ be $1$-Lipschitz, and define for $s\in[0,1]$ the interpolation $x_s:=(1-s)x+s x'$ and
\[
\Phi(s):=\mathbb{E}_ {\nu^{t,x_s}}[h(Y)].
\]
A standard computation (differentiating the log-partition function) gives
\[
\Phi'(s)
=
\text{Cov}_{\nu^{t,x_s}}\Big(h(Y), \Big\langle \frac{f}{g^2}(x'-x), Y\Big\rangle\Big).
\]
By Cauchy-Schwarz,
\[
|\Phi'(s)|
\le
\sqrt{\text{Var}_ {\nu^{t,x_s}}(h(Y))} 
\sqrt{\text{Var}_ {\nu^{t,x_s}}\Big(\Big\langle \frac{f}{g^2}(x'-x),Y\Big\rangle\Big)}.
\]
By the Brascamp--Lieb inequality under $\gamma$-strong log-concavity, for locally Lipschitz $\psi$,
\[
\text{Var}_ {\nu^{t,x_s}}(\psi(Y)) \le \frac{1}{\gamma}\mathbb{E}_ {\nu^{t,x_s}}[\|\nabla \psi(Y)\|^2].
\]
Applying this with $\psi=h$ and using $\|\nabla h\|\le 1$ a.e. (since $h$ is $1$-Lipschitz) gives
$\text{Var}_ {\nu^{t,x_s}}(h(Y))\le \frac{1}{\gamma}$.
Applying it with $\psi(y)=\langle v,y\rangle$ where $v:=\frac{f}{g^2}(x'-x)$ gives
$\text{Var}_ {\nu^{t,x_s}}(\langle v,Y\rangle)\le \frac{\|v\|^2}{\gamma}$.
Therefore,
\[
|\Phi'(s)| \le \frac{\|v\|}{\gamma}=\frac{f}{g^2}\frac{\|x'-x\|}{\alpha+f^2/g^2}
=
\frac{f}{\alpha g^2+f^2} \|x'-x\|.
\]
Integrating $s\in[0,1]$ yields
\[
\Big|\mathbb{E}_ {\nu^{t,x}}[h(Y)]-\mathbb{E}_ {\nu^{t,x'}}[h(Y)]\Big|
\le
\frac{f}{\alpha g^2+f^2} \|x-x'\|.
\]
Taking the supremum over all $1$-Lipschitz $h$ and using Kantorovich--Rubinstein duality,
\[
W_1(\nu^{t,x},\nu^{t,x'})
\le
\frac{f}{\alpha g^2+f^2} \|x-x'\|.
\]
Plugging this into \eqref{eq:at-W1} concludes the proof.

\medskip
\noindent\textbf{Case $\beta = 1$.}
Fix $x,x' \in \R^d$, let $\pi$ be any coupling of $\nu^{t,x}$ and $\nu_{t,x'}$, and let $(Y,Y') \sim \pi$.
In Step~1, choose $p=2$ so that $q = \frac{p}{p-1}=2$. Then \eqref{eq:loglap-stab} yields
\[
|a_t(x) - a_t(x')|
=
\Big|
\log \mathbb{E}_{\nu^{t,x}}[e^{a(Y)}]
-
\log \mathbb{E}_{\nu_{t,x'}}[e^{a(Y')}]
\Big|
\le
M_{2q}^{1/q} \|a(Y)-a(Y')\|_{L^p(\pi)}
=
M_4^{1/2} \|a(Y)-a(Y')\|_{L^2(\pi)}.
\]
Since $\beta=1$, the function $a$ is Lipschitz with constant $K$, hence
\[
\|a(Y)-a(Y')\|_{L^2(\pi)}
\le
K \big(\mathbb{E}_\pi\|Y-Y'\|^2\big)^{1/2}.
\]
Therefore,
\[
|a_t(x) - a_t(x')|
\le
K M_4^{1/2}\big(\mathbb{E}_\pi\|Y-Y'\|^2\big)^{1/2}.
\]
Taking the infimum over all couplings $\pi$ gives
\[
|a_t(x) - a_t(x')|
\le
K M_4^{1/2}  W_2(\nu^{t,x},\nu_{t,x'}).
\]

It remains to bound $W_2(\nu^{t,x},\nu_{t,x'})$.
Write, up to an additive constant,
\[
\nu^{t,x}(dy) \propto \exp \left(-U(y)+\left\langle \theta_x,y\right\rangle\right) dy,
\qquad
U(y):=u(y)+\frac{f^2}{2g^2}\|y\|^2,
\qquad
\theta_x := \frac{f}{g^2}x.
\]
Then $\nabla^2 U \succeq \gamma \text{Id}$ with $\gamma := \alpha + \frac{f^2}{g^2}$.
Consider the overdamped Langevin diffusions
\[
dY_s = \theta_x ds - \nabla U(Y_s) ds + \sqrt{2} dB_s,
\qquad
dY'_s = \theta_{x'} ds - \nabla U(Y'_s) ds + \sqrt{2} dB_s,
\]
driven by the same Brownian motion $(B_s)_{s\ge 0}$ and started from an arbitrary coupling of
$\nu^{t,x}$ and $\nu_{t,x'}$. Their marginals remain $(\nu^{t,x},\nu_{t,x'})$, so $(Y_s,Y'_s)$ is a
coupling of $(\nu^{t,x},\nu_{t,x'})$ for all $s\ge 0$.
Let $\Delta_s := Y_s - Y'_s$. Then
\[
d\Delta_s = (\theta_x-\theta_{x'}) ds - (\nabla U(Y_s)-\nabla U(Y'_s)) ds,
\]
and by $\gamma$-strong convexity of $U$,
\[
\frac{d}{ds}\|\Delta_s\|^2
=2\langle \Delta_s,\theta_x-\theta_{x'}\rangle
-2\langle \Delta_s,\nabla U(Y_s)-\nabla U(Y'_s)\rangle
\le
2\|\Delta_s\|\|\theta_x-\theta_{x'}\| - 2\gamma \|\Delta_s\|^2.
\]
Setting $r_s := \big(\E\|\Delta_s\|^2\big)^{1/2}$ and taking expectations yields
\[
\frac{d}{ds} r_s \le -\gamma r_s + \|\theta_x-\theta_{x'}\|,
\]
hence $r_s \le e^{-\gamma s}r_0 + (1-e^{-\gamma s})\frac{\|\theta_x-\theta_{x'}\|}{\gamma}$.
Since $W_2(\nu^{t,x},\nu_{t,x'}) \le r_s$ for all $s$, letting $s\to\infty$ gives
\[
W_2(\nu^{t,x},\nu_{t,x'})
\le
\frac{\|\theta_x-\theta_{x'}\|}{\gamma}
=
\frac{\frac{f}{g^2}\|x-x'\|}{\alpha+\frac{f^2}{g^2}}
=
\frac{f}{\alpha g^2 + f^2} \|x-x'\|.
\]
Combining the two bounds concludes the case $\beta=1$:
\[
|a_t(x)-a_t(x')|
\le
K M_4^{1/2} \frac{f}{\alpha g^2 + f^2} \|x-x'\|.
\]
\end{proof}
It remains to justify that the exponential-moment assumption~\eqref{eq:exp-mom-assump} is not restrictive in our setting.
The following result shows that it is automatic as soon as $a$ is H\"older and the surrogate posteriors $\nu^{t,x}$
are uniformly strongly log-concave: one first approximates $a$ by a Lipschitz function at a controlled uniform error,
and then applies standard concentration under strong log-concavity.

\begin{lemma}[Condition \eqref{eq:exp-mom-assump} holds for Hölder $a$]\label{lem:mgf_holder}
Assume that $u:\R^d\to\R$ is $\alpha$-strongly convex for some $\alpha>0$.
Fix $f\in\R$ and $g>0$, and for each $x\in\R^d$ let $\nu^{t,x}$ be the probability
measure on $\R^d$ with density
\[
\nu^{t,x}(dy) \propto \exp \Big(-u(y)-\frac{\|x-fy\|^2}{2g^2}\Big) dy.
\]
Let $a:\R^d\to\R$ be $\beta$-Hölder with constant $K>0$
for some $\beta\in(0,1]$.
Then for every $r\ge 0$ there exists $M_r<\infty$ such that for all $x\in\R^d$
\[
\mathbb E_{\nu^{t,x}} \Big[\exp\big(r(a(Y)-\mathbb E_{\nu^{t,x}}a(Y))\big)\Big]\le M_r.
\]
In particular, condition \eqref{eq:exp-mom-assump} in Lemma~\ref{lemma:holderexp} holds for $A=a(Y)$ with $Y\sim \nu^{t,x}$.
Moreover one may take
\[
M_r  := \exp \Big(2Kr+\frac{C K^2}{2\gamma} r^2\Big),
\qquad
\gamma:=\alpha+\frac{f^2}{g^2},
\]
where $C>0$ is a universal constant independent of the dimension.
\end{lemma}

\begin{proof}
Up to an additive constant, the negative log-density of $\nu^{t,x}$ equals
\[
V_x(y)=u(y)+\frac{\|x-fy\|^2}{2g^2}.
\]
Hence
\[
\nabla^2 V_x(y)=\nabla^2 u(y)+\frac{f^2}{g^2}I_d \succeq
\Big(\alpha+\frac{f^2}{g^2}\Big) \text{Id} =:\gamma \text{Id},
\]
so $\nu^{t,x}$ is $\gamma$-strongly log-concave uniformly in $x$.
Define the inf-convolution
\[
\tilde a(x) := \inf_{y\in\R^d}\Big\{a(y)+K\|x-y\|\Big\}.
\]
Then from Lemma~\ref{lem:infconv} we have that $\tilde a$ is $K$-Lipschitz and satisfies the uniform bound
\[
\|\tilde a-a\|_\infty \le K.
\]
Since $\nu^{t,x}$ is $\gamma$-strongly log-concave, Lemma~\ref{lem:herbst}
implies that for any $L$-Lipschitz function $g$ and any $\lambda\in\R$,
\[
\mathbb E_{\nu^{t,x}} \Big[\exp\big(\lambda(g-\mathbb E_{\nu^{t,x}}g)\big)\Big]
 \le \exp \Big(\frac{C L^2}{2\gamma}\lambda^2\Big).
\]

Apply this with $g=\tilde a$ and $\lambda=r\ge 0$, we get
\[
\mathbb E_{\nu^{t,x}} \Big[\exp\big(r(\tilde a-\mathbb E_{\nu^{t,x}}\tilde a)\big)\Big]
 \le \exp \Big(\frac{C K^2}{2\gamma}r^2\Big).
\]
Let $Y\sim \nu^{t,x}$, since $\|a-\tilde a\|_\infty\le K$, we have $|a(Y)-\tilde a(Y)|\le K$ a.s. and
$|\E a(Y)-\E\tilde a(Y)|\le \E|a(Y)-\tilde a(Y)|\le K$. Therefore
\[
a(Y)-\E a(Y)
=
\big(\tilde a(Y)-\E\tilde a(Y)\big)
+\big(a(Y)-\tilde a(Y)\big)
+\big(\E\tilde a(Y)-\E a(Y)\big)
 \le \big(\tilde a(Y)-\E\tilde a(Y)\big)+2K.
\]
Exponentiating and taking expectations yields
\[
\E\exp\big(r(a(Y)-\E a(Y))\big)
 \le e^{2Kr} \E\exp\big(r(\tilde a(Y)-\E\tilde a(Y))\big)
 \le \exp \Big(2Kr+\frac{C K^2}{2\gamma}r^2\Big).
\]
\end{proof}

\subsubsection{Bounds for Gaussian-mixture bridges}
This subsection collects identities and estimates for the Gaussian-mixture bridges that underlie our flow-matching and diffusion-model constructions. In order to apply our Wasserstein discretization bounds, we need a linear growth estimate to ensure uniform moment bounds for the
continuous dynamics and its Euler scheme. The next lemma records yields a dimension-sharp affine bound on \(\|v_t(x)\|\), which we will use
repeatedly to verify the linear-growth hypothesis in our numerical analysis.
\begin{lemma}\label{lemma:boundnormvt}
Let
\[
X_t = f_t Y + g_t X_0  + \sigma_t \varepsilon,
\qquad
X_0,\ \varepsilon \sim \mathcal N(0,I_d)\ \text{i.i.d. and independent of }Y\sim p^\star.
\]
Suppose that $p^\star$ is $(\alpha,\beta,K)$-weakly log-concave with $\alpha,K>0, \beta\in (0,1]$, write $p_t=\text{Law}(X_t)$ and
$s_t(x):=\nabla\log p_t(x)$. Define the effective Gaussian scale
\[
\bar g_t := \sqrt{g_t^2+\sigma_t^2} ,
\qquad
a_t := \frac{\bar g_t'}{\bar g_t},\qquad
c_t := f_t' - a_t f_t,\qquad
\gamma_t := \alpha + \frac{f_t^2}{\bar g_t^2}.
\]
Then, for every $t$ such that $f_t,\bar g_t>0$ and all $x\in\R^d$,
\begin{equation}\label{eq:fgheiu}
v_t(x)
= \frac{f_t'}{f_t} x
+ \Big(\frac{f_t'\bar g_t^{2}}{f_t}-\bar g_t\bar g_t'\Big) s_t(x).
\end{equation}
Moreover, supposing that $\|\arg\min u\|\le A\sqrt d$, there exists a dimension-free constant
$C>0$ depending only on $\alpha,\beta,K,A$ such that for all $t$ and $x$,
\[
\|v_t(x)\|\le F_t\big(f_t\sqrt d+\|x\|\big),
\qquad \text{with }
F_t := C\left(\left|\frac{f_t'}{f_t}\right|+
\frac{1}{\bar g_t^2}\left|\frac{f_t'\bar g_t^{2}}{f_t}-\bar g_t\bar g_t'\right|\right).
\]
\end{lemma}

\begin{proof}
Fix $t$ such that $f_t,\bar g_t>0$ and $f_t',g_t',\sigma_t'$ exist.
Introduce the Gaussian rotation
\[
\xi_t := \frac{g_tX_0+\sigma_t\varepsilon}{\bar g_t},
\qquad
\eta_t := \frac{\sigma_tX_0-g_t\varepsilon}{\bar g_t}.
\]
Then $(\xi_t,\eta_t)$ are i.i.d. $\mathcal N(0,I_d)$, independent of $Y$, and
\[
X_t = f_tY+\bar g_t \xi_t.
\]
Differentiating $X_t$ yields
\[
\dot X_t = g_t'X_0 + f_t'Y + \sigma_t'\varepsilon.
\]
Moreover,
\[
g_t'X_0+\sigma_t'\varepsilon
= \bar g_t' \xi_t + \frac{g_t'\sigma_t-\sigma_t'g_t}{\bar g_t} \eta_t.
\]
Since $X_t$ depends only on $(Y,\xi_t)$, we have $\eta_t\perp X_t$ and hence
$\E[\eta_t | X_t]=0$. Therefore, by linearity of conditional expectation,
\[
v_t(x):=\E[\dot X_t | X_t=x]
= f_t'\E[Y | X_t=x] + \bar g_t'\E[\xi_t | X_t=x].
\]
Writing $\mu_t(x):=\E[Y | X_t=x]$ and using $\xi_t=(X_t-f_tY)/\bar g_t$, we obtain
\[
\E[\xi_t | X_t=x]=\frac{x-f_t\mu_t(x)}{\bar g_t},
\qquad\text{hence}\qquad
v_t(x)=a_t x + c_t \mu_t(x).
\]
Next, differentiating under the integral sign in the representation
\[
p_t(x)=\int_{\R^d} p(y) \varphi^{x,\bar g_t}(f_ty) dy
\]
gives the standard score/posterior-mean identity
\begin{equation}\label{eq:fgheiu2}
s_t(x)=\nabla\log p_t(x)=\frac{1}{\bar g_t^2}\big(f_t\mu_t(x)-x\big),
\qquad\text{hence}\qquad
\mu_t(x)=\frac{x+\bar g_t^2 s_t(x)}{f_t}.
\end{equation}
Plugging this into $v_t(x)=a_tx+c_t\mu_t(x)$ yields
\[
v_t(x)=\Big(a_t+\frac{c_t}{f_t}\Big)x+\frac{c_t\bar g_t^2}{f_t}s_t(x).
\]
Since $a_t+c_t/f_t=f_t'/f_t$ and $c_t\bar g_t^2/f_t=(f_t'\bar g_t^2/f_t)-\bar g_t\bar g_t'$,
we obtain \eqref{eq:fgheiu}.

We now prove the linear growth bound.
From \eqref{eq:fgheiu2},
\[
\|s_t(x)\|\le \frac{1}{\bar g_t^2}\big(f_t\|\mu_t(x)\|+\|x\|\big).
\]
Using Lemma~\ref{lemma:bounfmut} we get
\[
\|\mu_t(x)\|\le C\left(\big\|\E_{\nu_{t,0}}[W]\big\|+\sqrt{\frac{d}{\gamma_t}}
+\frac{f_t}{\bar g_t^2\gamma_t}\|x\|\right).
\]
Plugging this into the previous display yields
\[
\|s_t(x)\|
\le \frac{C}{\bar g_t^2}\left(
f_t\Big(\big\|\E_{\nu_{t,0}}[W]\big\|+\sqrt{\frac{d}{\gamma_t}}\Big)
+\Big(1+\frac{f_t^2}{\bar g_t^2\gamma_t}\Big)\|x\|
\right)
\le \frac{C}{\bar g_t^2}\big(f_t\sqrt d+\|x\|\big),
\]
where we used Lemma~\ref{lem:mean_bound_nu_t0} to bound $\|\E_{\nu_{t,0}}[W]\|\lesssim \sqrt d$ under
$\|\arg\min u\|\le A\sqrt d$, and that $f_t^2/(\bar g_t^2\gamma_t)\le 1$.
Finally, applying \eqref{eq:fgheiu} and the last bound on $\|s_t(x)\|$ gives
\[
\|v_t(x)\|
\le \left|\frac{f_t'}{f_t}\right|\|x\|
+\left|\frac{f_t'\bar g_t^{2}}{f_t}-\bar g_t\bar g_t'\right|\|s_t(x)\|
\le F_t\big(f_t\sqrt d+\|x\|\big),
\]
with $F_t$ as stated.
\end{proof}
Several estimates in the sequel 
involve the auxiliary probability measure $\nu_{t,0}$ obtained by tilting the base potential $u$ with the
quadratic term coming from the Gaussian kernel at $x=0$. The quantity $\|\E_{\nu_{t,0}}[W]\|$ then appears
as a centering term in conditional-expectation bounds, and it must be controlled at the natural $\sqrt d$
scale with constants independent of the ambient dimension. The next lemma provides precisely such a bound.

\begin{lemma}[A $\sqrt d$ bound on the mean under $\nu_{t,0}$]
\label{lem:mean_bound_nu_t0}
Assume that $u:\R^d\to\R$ is $\mathcal C^2$ and $\alpha$-strongly convex for some $\alpha>0$.
For $t\in(0,1]$, let $\nu_{t,0}$ be the probability measure on $\R^d$ with density
\[
\nu_{t,0}(dy) = \frac{1}{Z_t}\exp  \Big(-u(y)-\frac{\|f_t y\|^2}{2\bar g_t^{ 2}}\Big) dy.
\]
Let $y_0:=\argmin u$ and $\lambda_t:=\frac{f_t^2}{\bar g_t^{ 2}}$.
Then
\begin{equation*}
\big\|\E_{\nu_{t,0}}[W]\big\| \le \|y_0\|+\sqrt{\frac{d}{\alpha+\lambda_t}}.
\end{equation*}
\end{lemma}

\begin{proof}For simplicity, we give the proof only in the \(C^2\) full-support case; the extension to the general case where \(u=+\infty\) outside a closed convex support is straightforward and follows from the same argument, with an additional boundary term of favorable sign.

Define
\[
U_t(y):=u(y)+\frac{\lambda_t}{2}\|y\|^2,\qquad y_t:=\argmin U_t.
\]
The optimality conditions of $y_0$ and $y_t$ read $\nabla u(y_0)=0$ and $\nabla u(y_t)+\lambda_t y_t=0$.
By convexity of $u$,
\[
0\le \langle \nabla u(y_t)-\nabla u(y_0), y_t-y_0\rangle
= \langle -\lambda_t y_t, y_t-y_0\rangle
= -\lambda_t\big(\|y_t\|^2-\langle y_t,y_0\rangle\big).
\]
Hence $\|y_t\|^2\le \langle y_t,y_0\rangle\le \|y_t\| \|y_0\|$, so $\|y_t\|\le \|y_0\|$. Since $\nabla^2 U_t=\nabla^2 u+\lambda_t \text{Id}\succeq (\alpha+\lambda_t)I_d$,
the measure $\nu_{t,0}(dy)\propto e^{-U_t(y)}dy$ is $(\alpha+\lambda_t)$--strongly log-concave.
Using integration by parts for each $i\in\{1,\dots,d\}$, we have
\[
\E_{\nu_{t,0}}\big[(W_i-y_{t,i}) \partial_i U_t(W)\big]=1,
\]
and summing over $i$ gives
\begin{equation}
\label{eq:ibp_sum}
\E_{\nu_{t,0}}\big[\langle W-y_t,\nabla U_t(W)\rangle\big]=d.
\end{equation}
By $(\alpha+\lambda_t)$--strong convexity and $\nabla U_t(y_t)=0$,
\[
\langle y-y_t,\nabla U_t(y)\rangle
=\langle y-y_t,\nabla U_t(y)-\nabla U_t(y_t)\rangle
\ge (\alpha+\lambda_t)\|y-y_t\|^2.
\]
Taking expectation and using \eqref{eq:ibp_sum} yields
\[
d\ge (\alpha+\lambda_t) \E_{\nu_{t,0}}\|W-y_t\|^2,
\qquad\text{so}\qquad
\E_{\nu_{t,0}}\|W-y_t\|^2\le \frac{d}{\alpha+\lambda_t}.
\]
Therefore,
\[
\|\E_{\nu_{t,0}}[W]-y_t\|
\le \sqrt{\E_{\nu_{t,0}}\|W-y_t\|^2}
\le \sqrt{\frac{d}{\alpha+\lambda_t}}.
\]

Therefore,
\[
\|\E_{\nu_{t,0}}[W]\|
\le \|y_t\|+\|\E_{\nu_{t,0}}[W]-y_t\|
\le \|y_0\|+\sqrt{\frac{d}{\alpha+\lambda_t}}.
\]
\end{proof}
In our sampling guarantees we stop the continuous-time dynamics at some $\tau<1$, so we need to control
the early stopping error $W_2(\text{Law}(X_1),\text{Law}(X_\tau))$. The next Lemma provides a bound on this error in our setting.

\begin{lemma}[Early stopping error]
\label{lem:general-coupling}
Let $Y$ be an $\mathbb{R}^d$-valued random vector with $\E\|Y\|^2<\infty$ and let $A\in\mathbb{R}^{d\times d}$ be deterministic.
Let $\eta$ be an $\mathbb{R}^d$-valued random vector such that $\E\|\eta\|^2<\infty$, $\mathbb{E}[\eta]=0$ and $\eta$ is independent of $Y$.
Define
\[
Z := AY + \eta.
\]
Then
\[
W_2  \bigl(\text{Law}(Y),\text{Law}(Z)\bigr)
 \le 
\Bigl(\E\bigl\|(I-A)Y\bigr\|^2  +  \E\|\eta\|^2\Bigr)^{1/2}.
\]
\end{lemma}

\begin{proof}
Consider the coupling $(Y,Z)$ on the same probability space, where $Z=AY+\eta$.
By definition of the $2$-Wasserstein distance,
\[
W_2^2  \bigl(\text{Law}(Y),\text{Law}(Z)\bigr)
 = 
\inf_{\pi\in\Pi(\text{Law}(Y),\text{Law}(Z))}
\int \|x-x'\|^2 d\pi(x,x')
 \le 
\E\|Y-Z\|^2,
\]
since the law of $(Y,Z)$ is an element of $\Pi(\text{Law}(Y),\text{Law}(Z))$.
Now,
\[
Y-Z = (I-A)Y-\eta,
\]
hence
\[
\E\|Y-Z\|^2
= \E\bigl\|(I-A)Y-\eta\bigr\|^2
= \E\bigl\|(I-A)Y\bigr\|^2 + \E\|\eta\|^2
-2 \E\bigl\langle (I-A)Y,\eta\bigr\rangle.
\]
Using the independence of $\eta$ and $Y$ and the assumption $\E[\eta]=0$ (or more generally $\E[\eta\mid Y]=0$),
\[
\E\bigl\langle (I-A)Y,\eta\bigr\rangle
=
\E\Bigl[\bigl\langle (I-A)Y, \E[\eta\mid Y]\bigr\rangle\Bigr]
=0.
\]
Therefore,
\[
W_2^2  \bigl(\text{Law}(Y),\text{Law}(Z)\bigr)
\le
\E\|Y-Z\|^2
=
\E\bigl\|(I-A)Y\bigr\|^2 + \E\|\eta\|^2,
\]
and taking square roots yields the claim.
\end{proof}

\section{Proofs of the Lipschitz estimates}
\subsection{Preliminaries and notations}\label{sec:hgfjkfhdjk}

In this section, we introduce the posterior measures that naturally appear when differentiating the projected vector field \(v_t\). The main point is that the derivatives of \(v_t\) can be written in terms of posterior covariances, which will be estimated in the next section.

For \(\mu \in\mathcal{P}(\mathbb{R}^k)\), \(f:\mathbb{R}^k\rightarrow \mathbb{R}^m\) and \(z\in \mathbb{R}^k\), define the centering operator
\begin{equation}
    H_\mu^f(z):=f(z)-\int f(z') d\mu(z').
\end{equation}
In particular, \(H_\mu^f\) is the fluctuation of \(f\) around its \(\mu\)-mean. Fix \(t\) with \(\sigma_t>0\). Recall that in the probabilistic representation, one observes
\begin{equation}
X_t = m_t(Z)+\sigma_t \xi,
\end{equation}
where \(Z\sim \pi\) and \(\xi\sim \mathcal N(0,\mathrm{Id})\) is independent of \(Z\). Hence, conditionally on \(Z=z\),
\[
X_t |  (Z=z)\sim \mathcal N(m_t(z),\sigma_t^2\mathrm{Id}).
\]
Recalling the notation $\varphi^{x,\sigma}$ for the density of $\mathcal N(x,\sigma^2\text{Id})$, we define the probability measure
\begin{equation}
    \pi^{t,x}(z):=\frac{ \varphi^{m_t(z),\sigma_t}(x)\pi(z)}{\int \varphi^{m_t(z'),\sigma_t}(x) d\pi(z')},
\end{equation}
which is the posterior law of \(Z\) given \(X_t=x\), i.e.
\[
\pi^{t,x}=\mathrm{Law}(Z\mid X_t=x).
\]
With this notation, the differential of \(v_t\) can be expressed through posterior covariance terms.

\begin{proposition}\label{propnablav}
The differential of the vector field \(v_t\) from \eqref{eq:vf} is
    $$\nabla v_t(x)=\frac{\sigma_t'}{\sigma_t} \mathrm{Id}-\frac{\sigma_t'}{\sigma_t^3}\int H_{\pi^{t,x}}^{m_t}(z)^{\otimes 2} d\pi^{t,x}(z)+\frac{1}{\sigma_t^2}\int H_{\pi^{t,x}}^{m_t' }(z)\otimes H_{\pi^{t,x}}^{m_t}(z) d\pi^{t,x}(z).$$
or equivalently
 $$\nabla v_t(x)=\frac{\sigma_t'}{\sigma_t} \mathrm{Id}-\frac{\sigma_t'}{\sigma_t^3}\mathrm{Var}_{Z|X_t=x}\big[m_t(Z)\big]+\frac{1}{\sigma_t^2}\mathrm{Cov}_{Z|X_t=x}\big[m_t'(Z),m_t (Z)\big].$$ 
\end{proposition}

The proof of Proposition~\ref{propnablav} can be found in Section~\ref{sec:propnablav}. An important point is the probabilistic interpretation: the nontrivial terms in \(\nabla v_t(x)\) are posterior fluctuation terms under the law \(\mathrm{Law}(Z\mid X_t=x)\).
It is convenient to push this posterior structure from the latent variable \(Z\) to the random mean \(m_t(Z)\). For \(\mu\) a probability measure, define
\begin{equation}
    \label{eq:qt}
    Q_t\mu(x):=\int \varphi^{w,\sigma_t}(x) d\mu(w),
\end{equation}
the Gaussian convolution of \(\mu\) at scale \(\sigma_t\). In particular, if \(\mu=(m_t)_{\#}\pi\), then \(Q_t\mu\) is the density of \(X_t\). Define the probability measure
\begin{equation}\label{eq:ptx}
    p^{t,x}(w):=\frac{ \varphi^{w,\sigma_t}(x)(m_t)_{\# \pi}(w)}{Q_{t}(m_t)_{\# \pi}(x)}.
\end{equation}
This is precisely the posterior law of \(m_t(Z)\) given \(X_t=x\):
\[
p^{t,x}=\mathrm{Law}(m_t(Z)\mid X_t=x).
\]
Thus, the variance term in Proposition~\ref{propnablav} can be viewed as a covariance under \(p^{t,x}\), which is often easier to estimate.
Let us take \(u_t\) strongly convex and \(a_t\) $\beta$-Hölder as in Definition~\ref{defi:weakconcave} such that
\[
\exp(-u_t+a_t)=(m_t)_{\#} \pi .
\]
Let 
\begin{equation}\label{eq:q}
q_t(y) := \exp\left(-u_t(y)\right),
\end{equation}
which is the log-concave part of \((m_t)_{\#}\pi\). We then define
\begin{equation}\label{eq:nutx}
\nu^{t,x}(w) := \frac{q_t(w)\varphi^{w,\sigma_t}(x)}{Q_t q_t(x)} = \frac{\exp\left(-u_t(w) \right)\varphi^{w,\sigma_t}(x)}{\int \exp\left(-u_t(y)\right) \varphi^{y,\sigma_t}(x) dy},
\end{equation}
which plays the role of a log-concave approximation of \( p^{t,x} \). The probabilistic meaning is the same as above: \(\nu^{t,x}\) is the posterior law associated with the observation model
\[
X_t=W+\sigma_t\xi
\]
when the prior law of \(W\) has density proportional to \(q_t\). Since \(u_t\) is strongly convex, \(\nu^{t,x}\) is strongly log-concave. This is exactly the setting where Brascamp-Lieb type inequalities provide sharp covariance bounds. Finally, \(p^{t,x}\) is a tilt of \(\nu^{t,x}\) by the perturbation \(a_t\):
\[
p^{t,x}(dw)=\frac{e^{a_t(w)}}{\int e^{a_t} d\nu^{t,x}}\;\nu^{t,x}(dw).
\]
Therefore, the strategy in the next section is to first control covariances under the log-concave measure \(\nu^{t,x}\), and then quantify the effect of the perturbation \(e^{a_t}\).

\subsection{Proofs of Theorem \ref{theo:onesidedlip2}}\label{sec:prooftheo12}
\subsubsection{General covariance bounds}

This subsection gathers the covariance estimates that will be used in the proof of Theorems \ref{theo:onesidedlip2}.
The goal is to separate the different mechanisms contributing to the one-sided Lipschitz bound for $\nabla v_t(x)$. We use the notations of Section~\ref{sec:hgfjkfhdjk}.

\textbf{A sharper bound under bounded perturbation.} Under the stronger bounded-perturbation regime of Assumption~\ref{assum:atbounded} below, one can remove the exponential prefactor in Theorem~\ref{theo:onesidedlip2}; we record this variant now, and the rest of the section is devoted to the common proof of Theorems~\ref{theo:onesidedlip2} and~\ref{theo:onesidedlip}.

\begin{assumption}\label{assum:atbounded} The probability density $(m_t)_{\#}\pi:\mathbb{R}^d\rightarrow \mathbb{R}$ is $(\alpha_t,\beta,K_t)$-weakly log-concave and writing
$(m_t)_{\#} \pi(x)=\exp(-u_t(x)+a_t(x))$, the $\beta$-perturbation $a_t$ satisfies $\|a_t\|_\infty\leq K_1$.
\end{assumption} 

Under the additional Assumption~\ref{assum:atbounded} we obtain the following refined bound.

\begin{theorem}\label{theo:onesidedlip}Let $v:[0,1]\times \mathbb{R}^d\rightarrow \mathbb{R}^d$ of the form \eqref{eq:vf} with $\sigma\in C^1_{\text{a.e.}}([0,1],[0,1])$ and $(m_t)_t$ satisfying Assumptions~\ref{assum:varmprime} and \ref{assum:atbounded}. Then for $x\in \mathbb{R}^d$ and $t\in [0,1]$ such that $\sigma_t>0$, we have for all $\xi \in \{0,1\}$,
\begin{itemize}
    \item if $\sigma_t'\leq 0$, then $$\lambda_{\max}\left(\nabla v_t(x)\right)\leq C \frac{1}{\alpha_t\sigma_t^2+1}\left(L_t-\frac{\sigma_t'}{\sigma_t}(K_t\sigma_t^\beta)^\xi\right)+\frac{\sigma_t\sigma_t'\alpha_t}{\alpha_t\sigma_t^2+1},$$
    \item if $\sigma_t'\geq 0$ and $\nabla^2 u_t \preceq A_t \text{Id}$ in the decomposition $(m_t)_{\#} \pi =\exp(-u_t+a_t)$ with $u_t$ convex and $a_t$ $\beta$-Hölder (from Assumption~\ref{assum:1}), then
    $$\lambda_{\max}\left(\nabla v_t(x)\right)\leq C \frac{1}{\alpha_t\sigma_t^2+1}\left(L_t+\frac{\sigma_t'}{\sigma_t}(K_t\sigma_t^\beta)^\xi\right)+\frac{\sigma_t\sigma_t'A_t}{\alpha_t\sigma_t^2+1}.$$
\end{itemize}
with $C>0$ depending only on $K_1,\beta$ and not the dimension $d$.
\end{theorem}
Compared to Theorem~\ref{theo:onesidedlip2}, the structure of the bound is unchanged, but the uniform bound on $a_t$ removes the
exponential envelope. This makes the estimate closer to the ``intrinsic'' scaling suggested by posterior concentration:
the sensitivity of $v_t$ is attenuated by $(\alpha_t\sigma_t^2+1)^{-1}$, while the score contribution retains the natural factor
$\sigma_t'/\sigma_t$.
We will not use this refined estimate in the remainder of the paper as our main focus is the unbounded weakly log-concave setting of Definition~\ref{defi:weakconcave}, for which the perturbation is only assumed to be Hölder and may be unbounded. Theorem~\ref{theo:onesidedlip} is included only as a benchmark result: it shows that, under the stronger bounded-perturbation assumption, the exponential prefactor appearing in Theorem~\ref{theo:onesidedlip2} can be removed, while preserving the same geometric scalings of the estimate.

We now prove Theorems~\ref{theo:onesidedlip2} and~\ref{theo:onesidedlip} simultaneously. The two proofs follow the same lines, the only difference being that Theorem~\ref{theo:onesidedlip} uses the improved estimates available under Assumption~\ref{assum:atbounded}.

\paragraph{Beginning of the proof.} Let us first recall the formula given by Proposition~\ref{propnablav}:
$$\nabla v_t(x)=\frac{\sigma_t'}{\sigma_t} \mathrm{Id}-\frac{\sigma_t'}{\sigma_t^3}\int H_{\pi^{t,x}}^{m_t}(z)^{\otimes 2} d\pi^{t,x}(z)+\frac{1}{\sigma_t^2}\int H_{\pi^{t,x}}^{m_t' }(z)\otimes H_{\pi^{t,x}}^{m_t}(z) d\pi^{t,x}(z).$$
Using the subadditivity of the maximum eigenvalue, we  obtain
\begin{align}\label{align:vfdihsksgs2}
&\lambda_{\max}\left( \nabla v_t(x)-\frac{1}{\sigma_t^2}\int H_{\pi^{t,x}}^{m_t' }(y)\otimes H_{\pi^{t,x}}^{m_t}(y)d\pi^{t,x}(y) \right)\\
&
= \lambda_{\max}\left( \frac{\sigma_t'}{\sigma_t} \text{Id}-\frac{\sigma_t'}{\sigma_t^3}\int H_{p^{t,x}}^{\text{Id}}(y)^{\otimes 2}dp^{t,x}(y) \right) \nonumber \\
& \leq  \lambda_{\max}\left(\frac{-\sigma_t'}{\sigma_t^3} \int H_{p^{t,x}}^{\text{Id}}(y)^{\otimes 2}   dp^{t,x}(y) - \frac{-\sigma_t'}{\sigma_t^3}\int H_{\nu^{t,x}}^{\text{Id}}(y)^{\otimes 2}   d\nu^{t,x}(y) \right)\nonumber \\
&+  \lambda_{\max}\left( \frac{-\sigma_t'}{\sigma_t^3}\int H_{\nu^{t,x}}^{\text{Id}}(y)^{\otimes 2}   d\nu^{t,x}(y) \right) + \frac{\sigma_t'}{\sigma_t}. \nonumber
\end{align}

Equation \eqref{align:vfdihsksgs2} isolates two pieces:
\begin{itemize}
    \item the difference between the covariance under $p^{t,x}$ and under $\nu^{t,x}$ (this is the perturbative term);
    \item the covariance under $\nu^{t,x}$ alone (this is the log-concave core term).
\end{itemize}
We treat these two terms separately.

\paragraph{The log-concave core term.}
The measure $\nu^{t,x}$ is strongly log-concave because it combines the convex part of $(m_t)_\#\pi$ with a Gaussian likelihood.
The sign of $\sigma_t'$ determines which inequality is useful.

In the case where $\sigma_t'\leq 0$, the prefactor $-\sigma_t'/\sigma_t^3$ is nonnegative, so an upper bound on the covariance matrix gives the desired estimate.
This is exactly the setting of the Brascamp--Lieb inequality.

\begin{lemma}\label{lemma:boundeasy} Let  $\nu^{t,x}$ the  probability measure defined in \eqref{eq:nutx}. Suppose that $m$ satisfies Assumption~\ref{assum:1}, then  for all $t>0$ such that $\sigma_t'\leq 0$ and $x\in \mathbb{R}^d$ we have
$$\frac{-\sigma_t'}{\sigma_t^3}\lambda_{\max}\left(\int H_{\nu^{t,x}}^{\text{Id}}(y)^{\otimes 2}d\nu^{t,x}(y)\right)+\frac{\sigma_t'}{\sigma_t}\leq  \sigma_t\sigma_t'\frac{\alpha_t}{\alpha_t\sigma_t^2+1}.$$
\end{lemma}
The proof of Lemma~\ref{lemma:boundeasy} can be found in Section~\ref{sec:lemma:boundeasy}. The right-hand side is nonpositive when $\sigma_t'\le 0$, which reflects the fact that the strongly log-concave core contributes a contractive term in this regime.

\medskip

In the case where $\sigma_t'\geq 0$, the prefactor changes sign, and the same covariance term may become expansive.
One now needs a lower bound on the covariance, which is obtained through a Cram\'er-Rao type inequality.
This explains the additional upper curvature assumption on $u_t$ in Theorem~\ref{theo:onesidedlip}.

\begin{lemma}
\label{lem:lemma3_cramer_rao}
Let $\nu^{t,x}$ the probability measure defined in \eqref{eq:nutx}.
Suppose that $m$ satisfies Assumption~\ref{assum:1} and that  $(m_t)_\#\pi=\exp(-u_t+a_t)$ with $u_t$ convex and $C^2$ on $\mathbb{R}^d$ satisfying $\nabla^2 u_t \preceq A_t\text{Id}$.
Then for all $x\in\R^d$ and $t>0$ such that $\sigma_t'\geq 0$,  we have 
\[
\lambda_{\max}  \left(-\frac{\sigma_t'}{\sigma_t^3}\int  H^{\text{Id}}_{\nu^{t,x}}(y)^{\otimes 2} d\nu^{t,x}(y)\right)
+\frac{\sigma_t'}{\sigma_t}
\le
\frac{\sigma_t\sigma_t' A_t}{A_t\sigma_t^2+1}.
\]
\end{lemma}
The proof of Lemma~\ref{lem:lemma3_cramer_rao} can be found in Section~\ref{sec:lem:lemma3_cramer_rao}.

\paragraph{Comparing covariances under an exponential tilt.}
We now turn to the first term in \eqref{align:vfdihsksgs2}, namely the difference between the covariance under $p^{t,x}$ and under $\nu^{t,x}$.
The next proposition expresses covariance comparison in terms of tilted moments of centered linear observables.

\begin{proposition}\label{prop:firstuglybound}Let $\mu_1,\mu_2\in \mathcal{P}(\mathbb{R}^d)$ having second moments and such that $\frac{d\mu_1(y)}{d\mu_2}=e^{v(y)}$. Then, for all $h\in \mathbb{S}^{d-1}$ we have
\begin{align*}
    |h^\top &\left(\int  H_{\mu_1}^{\text{Id}}(y)^{\otimes 2}d\mu_1(y)-\int H_{\mu_2}^{\text{Id}}(y)^{\otimes 2}d\mu_2(y)\right)h|\\
\leq &C|\int  h^\top H_{\mu_2}^{\text{Id}}(y)\left(e^{v(y)}-1\right)d\mu_2(y)|^2\ 
+|\int  (h^\top H_{\mu_2}^{\text{Id}}(y))^{ 2}\left(e^{v(y)}-1\right)d\mu_2(y)|.
\end{align*}
\end{proposition}
The proof of Proposition~\ref{prop:firstuglybound} can be found in Section~\ref{sec:prop:firstuglybound}. This proposition reduces the matrix-valued covariance comparison problem to scalar moment estimates, which can then be handled using functional inequalities for strongly log-concave measures.

\paragraph{From Hölder regularity to concentration.}
The perturbation $a_t$ is only assumed to be Hölder continuous, so we need an estimate that converts Hölder regularity into moment control under strong log-concavity.
This is the role of the generalized Poincar\'e inequality below.

\begin{theorem}[Generalized Poincaré inequality]
\label{thm:WPI}
Let $\mu$ be a probability measure on $\mathbb{R}^d$ that is $\gamma$-strongly log-concave for some $\gamma>0$.
Let $\alpha\ge 0$, $\beta\in (0,1]$ and suppose $f:\mathbb{R}^d\to\mathbb{R}$ is $\beta$-Hölder with constant $L>0$. Then,
\[
\int \bigl|f - \int f d\mu \bigr|^\alpha d\mu
 \le 
C_\alpha L^\alpha \gamma^{-\alpha\beta/2}.
\]
\end{theorem}
The proof of Theorem~\ref{thm:WPI} can be found in Section~\ref{sec:thm:WPI}. The scaling $\gamma^{-\beta/2}$ is the key point: under $\gamma$-strong log-concavity, fluctuations occur at spatial scale $\gamma^{-1/2}$, and a $\beta$-Hölder function therefore fluctuates at scale $(\gamma^{-1/2})^\beta=\gamma^{-\beta/2}$.

\paragraph{Tilted moment estimates (bounded and unbounded perturbations).}
We now apply Theorem~\ref{thm:WPI} to the exponential tilt.
The next two lemmas give the estimates needed in Proposition~\ref{prop:firstuglybound}. The bounded perturbation case corresponds to Assumption~\ref{assum:atbounded} and yields the cleanest bound.

\begin{lemma}[Covariance bounds with bounded perturbation]\label{lemma:missigestimate}
Let $\mu\in \mathcal{P}(\mathbb{R}^d)$ a $\gamma_t$-strongly log concave measure with $\gamma_t>0$ and $v\in \mathcal{H}^\beta$ with $\|v\|_\infty\leq K_1$ and $\|v\|_{\mathcal{H}^\beta}\leq K_t$ and such that $\int e^v d\mu=1$. Then, there exists $C>0$ independent of the dimension $d$ such that for all $h\in \mathbb{S}^{d-1}$, $\xi\in [0,1]$ and $j\in \{1,2\}$,
\begin{align*}
   &|\int  (h^\top H_\mu^{\text{Id}}(y))^{ j}\left(e^{v(y)}-1\right)d\mu(y)|\leq C K_t^\xi\gamma_t^{-(j+\beta\xi)/2}.
\end{align*}
\end{lemma}
 The proof of Lemma~\ref{lemma:missigestimate} can be found in Section~\ref{sec:lemma:missigestimate}. The parameter $\xi\in[0,1]$ provides an interpolation between a rough bound ($\xi=0$) and a sharper Hölder-scale bound ($\xi=1$). This flexibility is useful later when optimizing the final estimate. Without boundedness of $v$, the same geometric scaling remains true, but one pays an exponential envelope.
This is the price of controlling the exponential tilt using only Hölder regularity.

\begin{lemma}[Covariance bounds without bounded perturbation]\label{lem:lemma10-unbounded}
Let $\mu\in \mathcal{P}(\mathbb{R}^d)$ a $\gamma_t$-strongly log concave measure with $\gamma_t>0$ and $v\in \mathcal{H}^\beta$ with $\|v\|_{\mathcal{H}^\beta}\leq K_t$ and such that $\int e^v d\mu=1$. Then, there exists $C>0$ independent of the dimension $d$ such that for all $h\in \mathbb{S}^{d-1}$, $\xi\in \{0,1\}$ and $j\in \{1,2\}$,

\[
\Big|\int (h^\top H^{\text{Id}}_\mu)^{j} (e^{v}-1) d\mu\Big|
 \le 
C K_t^{\xi} \gamma_t^{-(j+\beta\xi)/2} 
\exp \big(C K_t^2 \gamma_t^{-\beta}\big).
\]
\end{lemma}

The proof of Lemma~\ref{lem:lemma10-unbounded} can be found in Section~\ref{sec:lem:lemma10-unbounded}.

\paragraph{Application to the posterior covariance comparison.}
We now apply Proposition~\ref{prop:firstuglybound} with $\mu_1=p^{t,x}$ and $\mu_2=\nu^{t,x}$, and then use Lemma~\ref{lemma:missigestimate} or Lemma~\ref{lem:lemma10-unbounded} depending on whether Assumption~\ref{assum:atbounded} is available.
Since $\nu^{t,x}$ is $(\alpha_t+\sigma_t^{-2})$-strongly log-concave, the previous estimates yield the following bound.

\begin{proposition}\label{prop:finalestimatecompcov} Let  $p^{t,x}$ and $\nu^{t,x}$ the  probability measures defined in \eqref{eq:ptx} and \eqref{eq:nutx} respectively. Then, if $m$ satisfies Assumption~\ref{assum:1}, there exists $C>0$ independent of the dimension $d$ such that for all $\xi\in \{0,1\}$ and $h\in\mathbb S^{d-1}$
 $$|h^\top \left(\int H_{p^{t,x}}^{\text{Id}}(y)^{\otimes 2}   dp^{t,x}(y) - \int H_{\nu^{t,x}}^{\text{Id}}(y)^{\otimes 2}   d\nu^{t,x}(y)\right) h|\leq C K_t^\xi(\alpha_t+\frac{1}{\sigma_t^2})^{-\frac{2+\beta\xi}{2}}\exp \big(C K_t^2 (\alpha_t+\frac{1}{\sigma_t^2})^{-\beta}\big).$$
If additionally $m$ satisfies Assumption~\ref{assum:atbounded}, $$|h^\top \left(\int H_{p^{t,x}}^{\text{Id}}(y)^{\otimes 2}   dp^{t,x}(y) - \int H_{\nu^{t,x}}^{\text{Id}}(y)^{\otimes 2}   d\nu^{t,x}(y)\right) h|\leq C K_t^\xi(\alpha_t+\frac{1}{\sigma_t^2})^{-\frac{2+\beta\xi}{2}}.$$

\end{proposition}

The proof of Proposition~\ref{prop:finalestimatecompcov} can be found in Section~\ref{sec:prop:finalestimatecompcov}. This proposition is the main estimate for the covariance replacement of $p^{t,x}$ by $\nu^{t,x}$ in \eqref{align:vfdihsksgs2}. After multiplying by $|\sigma_t'|/\sigma_t^3$, it produces the contribution involving $K_t$ in the final one-sided Lipschitz bound.

\paragraph{The mixed covariance term and the role of Assumption~\ref{assum:varmprime}.}
Let us now bound the other term appearing in Proposition~\ref{propnablav}.
The next proposition controls the covariance between $m_t(Z)$ and $m_t'(Z)$ under the posterior $\pi^{t,x}$.
The proof combines Cauchy--Schwarz with:
\begin{itemize}
    \item Assumption~\ref{assum:varmprime}, which controls the posterior variance of $m_t'(Z)$;
    \item the same covariance machinery as above, which controls the posterior variance of $m_t(Z)$.
\end{itemize}

\begin{proposition}\label{prop:boundsecondterm}
Supposing that $m$ satisfies Assumptions~\ref{assum:1} and \ref{assum:varmprime}, for any $x\in\mathbb R^d$ and $h\in\mathbb S^{d-1}$ we have
$$| \int h^\top H_{\pi^{t,x}}^{m_t' }(z) h^\top H_{\pi^{t,x}}^{m_t}(z)d\pi^{t,x}(z)|\leq CL_t\sigma_t^2 \frac{1}{\alpha_t\sigma_t^2+1}\exp \big(C K_t^2 (\alpha_t+\frac{1}{\sigma_t^2})^{-\beta}\big).$$
If additionally $m$ satisfies Assumption~\ref{assum:atbounded},
    $$| \int h^\top H_{\pi^{t,x}}^{m_t' }(z) h^\top H_{\pi^{t,x}}^{m_t}(z)d\pi^{t,x}(z)|\leq CL_t\sigma_t^2 \frac{1}{\alpha_t\sigma_t^2+1}.$$
    
\end{proposition}

The proof of Proposition  \ref{prop:boundsecondterm} can be found in Section~\ref{sec:prop:boundsecondterm}.

\subsubsection{Concluding the proof of Theorems \ref{theo:onesidedlip2} and \ref{theo:onesidedlip}}
Let us prove Theorem~\ref{theo:onesidedlip2}, the proof of Theorem~\ref{theo:onesidedlip} follows the same steps.
From Proposition~\ref{propnablav} we have
$$\lambda_{\max}\left(\nabla v_t(x)\right)=\lambda_{\max}\left(\frac{\sigma_t'}{\sigma_t} \text{Id}-\frac{\sigma_t'}{\sigma_t^3}\int H_{\pi^{t,x}}^{m_t}(z)^{\otimes 2}d\pi^{t,x}(z)+\frac{1}{\sigma_t^2}\int H_{\pi^{t,x}}^{m_t'}(z)\otimes H_{\pi^{t,x}}^{m_t}(z)d\pi^{t,x}(z)\right),$$
so we are going to bound the different terms of the right-hand side. We have
\begin{align*}
&\lambda_{\max}\left( \nabla v_t(x)-\frac{1}{\sigma_t^2}\int H_{\pi^{t,x}}^{m_t'}(z)\otimes H_{\pi^{t,x}}^{m_t}(z)d\pi^{t,x}(z) \right)\\
&
= \lambda_{\max}\left( \frac{\sigma_t'}{\sigma_t} \text{Id}-\frac{\sigma_t'}{\sigma_t^3}\int H_{p^{t,x}}^{\text{Id}}(z)^{\otimes 2}dp^{t,x}(z) \right) \nonumber \\
& \leq  \lambda_{\max}\left(\frac{-\sigma_t'}{\sigma_t^3} \int H_{p^{t,x}}^{\text{Id}}(y)^{\otimes 2}   dp^{t,x}(y) - \frac{-\sigma_t'}{\sigma_t^3}\int H_{\nu^{t,x}}^{\text{Id}}(y)^{\otimes 2}   d\nu^{t,x}(y) \right)\nonumber \\
&+  \lambda_{\max}\left( \frac{-\sigma_t'}{\sigma_t^3}\int H_{\nu^{t,x}}^{\text{Id}}(y)^{\otimes 2}   d\nu^{t,x}(y) \right) + \frac{\sigma_t'}{\sigma_t}.
\end{align*}
Applying Proposition~\ref{prop:finalestimatecompcov}, we get
\begin{align}\label{align:finalcal}
     &\frac{|\sigma_t'|}{\sigma_t^3} \lambda_{\max}\left( \int H_{p^{t,x}}^{\text{Id}}(y)^{\otimes 2}   dp^{t,x}(y) - \int H_{\nu^{t,x}}^{\text{Id}}(y)^{\otimes 2}   d\nu^{t,x}(y) \right)\\
    & \leq C\frac{|\sigma_t'|}{\sigma_t^3} K_t^\xi(\alpha_t+\frac{1}{\sigma_t^2})^{-\frac{2+\beta\xi}{2}}\exp \big(C K_t^2 (\alpha_t+\frac{1}{\sigma_t^2})^{-\beta}\big)\nonumber\\
    & \leq CK_t^\xi\frac{|\sigma_t'|}{\sigma_t^{1-\beta\xi}}\frac{1}{\alpha_t\sigma_t^2+1}\exp \big(C K_t^2 (\alpha_t+\frac{1}{\sigma_t^2})^{-\beta}\big),
\end{align}
and applying Proposition~\ref{prop:boundsecondterm}
    we get for all $h\in \mathbb{S}^{d-1}$, we have 
    $$| \int h^\top H_{\pi^{t,x}}^{m_t'}(z) h^\top H_{\pi^{t,x}}^{m_t}(z)d\pi^{t,x}(z)|\leq C L_t\sigma_t^2\frac{1}{\alpha_t\sigma_t^2+1}\exp \big(C K_t^2 (\alpha_t+\frac{1}{\sigma_t^2})^{-\beta}\big).$$
In the case where $\sigma_t'<0$, we get
by Lemma~\ref{lemma:boundeasy}
$$\frac{-\sigma_t'}{\sigma_t^3}\lambda_{\max}\left(\int H_{\nu^{t,x}}^{\text{Id}}(y)^{\otimes 2}d\nu^{t,x}(y)\right)+\frac{\sigma_t'}{\sigma_t}\leq  \sigma_t\sigma_t'\frac{\alpha_t}{\alpha_t\sigma_t^2+1}.$$

so putting everything together we get 
\begin{align*}
\lambda_{\max}\left(\nabla v_t(x)\right)&=\lambda_{\max}\left(\frac{\sigma_t'}{\sigma_t} \text{Id}-\frac{\sigma_t'}{\sigma_t^3}\int H_{\pi^{t,x}}^{m_t}(z)^{\otimes 2}d\pi^{t,x}(z)+\frac{1}{\sigma_t^2}\int H_{\pi^{t,x}}^{m_t'}(z)\otimes H_{\pi^{t,x}}^{m_t}(z)d\pi^{t,x}(z)\right)\\
    & \leq C \frac{1}{\alpha_t\sigma_t^2+1}\left(L_t-\frac{\sigma_t'}{\sigma_t}(K_t\sigma_t^\beta)^\xi\right)\exp \big(C K_t^2 (\alpha_t+\frac{1}{\sigma_t^2})^{-\beta}\big)+\frac{\sigma_t\sigma_t'\alpha_t}{\alpha_t\sigma_t^2+1}.
\end{align*}
Now, in the case where $\sigma_t'>0$, we get
by Lemma~\ref{lem:lemma3_cramer_rao}
$$\frac{\sigma_t'}{\sigma_t^3}\lambda_{\max}\left(\int H_{\nu^{t,x}}^{\text{Id}}(y)^{\otimes 2}d\nu^{t,x}(y)\right)+\frac{\sigma_t'}{\sigma_t}\leq  \frac{\sigma_t\sigma_t' A_t}{A_t\sigma_t^2+1}\leq  \frac{\sigma_t\sigma_t' A_t}{\alpha_t\sigma_t^2+1} ,$$
so putting everything together we get 
\begin{align*}
\lambda_{\max}\left(\nabla v_t(x)\right)&=\lambda_{\max}\left(\frac{\sigma_t'}{\sigma_t} \text{Id}-\frac{\sigma_t'}{\sigma_t^3}\int H_{\pi^{t,x}}^{m_t}(z)^{\otimes 2}d\pi^{t,x}(z)+\frac{1}{\sigma_t^2}\int H_{\pi^{t,x}}^{m_t'}(z)\otimes H_{\pi^{t,x}}^{m_t}(z)d\pi^{t,x}(z)\right)\\
    & \leq C \frac{1}{\alpha_t\sigma_t^2+1}\left(L_t+\frac{\sigma_t'}{\sigma_t}(K_t\sigma_t^\beta)^\xi\right)\exp \big(C K_t^2 (\alpha_t+\frac{1}{\sigma_t^2})^{-\beta}\big)+\frac{\sigma_t\sigma_t'A_t}{\alpha_t\sigma_t^2+1}.
\end{align*}

\subsection{Technical novelty for one-sided Lipschitz estimates}\label{sec:technicalone}
An important conceptual point of this paper is that the relevant regularity object is not the terminal target density $p^\star$ itself, but rather the family of intermediate center laws
\[
\mu_t := (m_t)_{\#}\pi.
\]
Indeed, once the Gaussian path is written as $
X_t = m_t(Z) + \sigma_t \xi,$
the Jacobian of the projected drift $v_t$ is expressed through posterior fluctuation terms under $\mathrm{Law}(Z \mid X_t = x)$. Pushing this posterior from the latent variable $Z$ to the random center $m_t(Z)$ leads to the measure $
p^{t,x} = \mathrm{Law}(m_t(Z)\mid X_t = x),$
so that the key quantities are governed by the geometry of $\mu_t$, and not directly by that of $p^\star$. Writing
\[
\mu_t = \exp(-u_t + a_t),
\]
the proof then compares $p^{t,x}$ with the log-concave surrogate $\nu^{t,x}$ built from the convex part $u_t$, while the perturbation $a_t$ is treated as an exponential tilt. In this way, the Gaussian noise is handled explicitly, and the nontrivial regularity input is precisely the weak log-concavity and Hölder control of $(m_t)_{\#}\pi$. For the concrete diffusion and flow-matching models considered here, the assumptions on $p^\star$ are used only to verify that these intermediate laws $\mu_t$ enjoy the required structure. This is the reason why the argument is both model-independent and dimension-free at the level of the one-sided Lipschitz estimate.

Due to the large number of works we only compare the proof techniques with the closest to this paper, namely \cite{brigati2024heat} and \cite{Stephanovitch2025regularity}. The main new technical ingredient is the covariance comparison under an Hölder unbounded exponential tilt, achieved through Theorem~\ref{thm:WPI} + Lemma~\ref{lem:var_exp_concentration} + Lemma~\ref{lem:lemma10-unbounded}; this replaces the transport/pointwise-Hölder route used in prior works.

Compared with \cite{brigati2024heat}, the overlap is mainly at the level of the first reduction: in both arguments one introduces a strongly log-concave surrogate posterior $\nu^{t,x}$ and exploits Brascamp-Lieb on that reference measure. The proof of Theorem~\ref{theo:onesidedlip2} then departs from this scheme at the covariance comparison step between $p^{t,x}$ and $\nu^{t,x}$. In \cite{brigati2024heat} they use for $h\in \mathbb{S}^{d-1}$,
$$\langle h, \text{Cov}_{p^{t,x}} h\rangle \leq \Big(\sqrt{\langle h, \text{Cov}_{\nu^{t,x}} h\rangle} + W_2(p^{t,x},\nu^{t,x})\Big)^2, $$
and then invoke a $L^\infty$-transport estimate to bound the Wasserstein term.
Rather than transferring the covariance bound through a transport estimate, we work directly with covariance identities under exponential tilts. This requires a different input, namely quantitative moment bounds for Hölder observables under strongly log-concave measures, which are provided here by Theorem~\ref{thm:WPI} and Lemma~\ref{lem:var_exp_concentration}/\ref{lem:lemma10-unbounded}. These bounds make it  possible to control the effect of a merely Hölder perturbation on the posterior covariance. In addition, because the target quantity is a one-sided Lipschitz bound, the core covariance term has to be treated differently according to the sign of the noise derivative, leading to a Brascamp-Lieb bound in one regime and a Cramér-Rao bound in the other. This is the main structural difference with the log-Lipschitz perturbative argument of Brigati and Pedrotti.

Compared with \cite{Stephanovitch2025regularity}, the proof is reorganized so as to avoid the step where one uses the Hölder bound in a direct pointwise way. In that paper, the covariance difference is expanded first, and the Hölder bound is then used directly in pointwise form,
$$|a(x)-a(y)|\leq K \|x-y\|^{\beta},$$
which is precisely where dimension dependence enter.
Here the strategy is different. The proof first rewrites the covariance difference in terms of scalar tilted moments of centered linear observables, and only then uses Hölder regularity through dimension-free concentration for the strongly log-concave reference measure.

The same reorganization is also what makes the unbounded case manageable: the growth of the tilt is not treated at the first comparison step, but later, at the level of these scalar moments. More precisely, Lemma  \ref{lem:var_exp_concentration} shows that if $\mu$ is a $\gamma$-strongly log-concave probability measure and $v:\mathbb R^d\to\mathbb R$ is a $\beta$-Hölder function with Hölder constant $K$ satisfying $
\int e^{v} d\mu = 1,$
then 
\[
\text{Var}_\mu(e^v)
\le
\exp\!\left(CK\gamma^{-\beta/2}+CK^2\gamma^{-\beta}\right)-1.
\]
Thus, even when the tilt is unbounded, one still gets a dimension-free control of the exponential weight $e^v$. This rough variance bound is then sharpened in the proof of Lemma~\ref{lem:lemma10-unbounded} by decomposing $e^v$
 into a normalized Lipschitz part and a uniformly small multiplicative error, which removes the leading linear term and yields a quadratic gain.

Finally, the argument must also handle flow-matching schedules, including the regime where the noise may vanish at the initial time. These changes in the argument are what allows one to keep dimension-free constants, to treat unbounded perturbations, and to handle both diffusion and flow-matching regimes within the same scheme.

\subsection{Details for one-sided Lipschitz estimates}
\subsubsection{Proof of Proposition~\ref{propnablav}}\label{sec:propnablav}
\begin{proof}
From \eqref{eq:vf} we have
\begin{align*} 
    \nabla v_t(x)=&\frac{\sigma_t'}{\sigma_t} \text{Id}+\int \left(\frac{\sigma_t'}{\sigma_t}(x-m_t(z))+m_t' (z)\right)\frac{\nabla_{x} \varphi^{m_t(z),\sigma_t}(x)d\pi(z)}{\int \varphi^{m_t(z'),\sigma_t}(x)d\pi(z')}\\
    &-\int \left(\frac{\sigma_t'}{\sigma_t}(x-m_t(z))+m_t' (z)\right)\frac{\varphi^{m_t(z),\sigma_t}(x)d\pi(z)}{\left(\int \varphi^{m_t(z'),\sigma_t}(x)d\pi(z')\right)^2}\nabla_{x} \int \varphi^{m_t(z'),\sigma_t}(x)d\pi(z')\\
    =&\frac{\sigma_t'}{\sigma_t} \text{Id}+\int \left(\frac{\sigma_t'}{\sigma_t}(x-m_t(z))+m_t' (z)\right)\left(\frac{-(x-m_t(z))}{\sigma_t^2}\right)\frac{\varphi^{m_t(z),\sigma_t}(x)d\pi(z)}{\int \varphi^{m_t(z'),\sigma_t}(x)d\pi(z')}\\
    &-\int \left(\frac{\sigma_t'}{\sigma_t}(x-m_t(z))+m_t' (z)\right)\frac{ \varphi^{m_t(z),\sigma_t}(x)d\pi(z)}{\left(\int \varphi^{m_t(z'),\sigma_t}(x)d\pi(z')\right)^2}\int \left(\frac{-(x-m_t(z'))}{\sigma_t^2}\right)\varphi^{m_t(z'),\sigma_t}(x)d\pi(z')\\
    =&\frac{\sigma_t'}{\sigma_t} \text{Id}-\frac{\sigma_t'}{\sigma_t^3}\left(\int (x-m_t(z))^{\otimes 2}d\pi^{t,x}(z)-\left(\int (x-m_t(z))d\pi^{t,x}(z)\right)^{\otimes 2}\right)\\
    &-\frac{1}{\sigma_t^2}\left(\int m_t'(z)(x-m_t(z))d\pi^{t,x}(z)-\int m_t'(z)d\pi^{t,x}(z)\int (x-m_t(z))d\pi^{t,x}(z)\right)\\
    =&\frac{\sigma_t'}{\sigma_t} \text{Id}-\frac{\sigma_t'}{\sigma_t^3}\int(x-m_t(z))\left(x-m_t(z)-(x-\int m_t(z'd\pi^{t,x}(z'))\right) d\pi^{t,x}(z)\\
    &-\frac{1}{\sigma_t^2}\int \left(m_t'(z)-\int m_t'(z')d\pi^{t,x}(z')\right) (m_t(z)-x)d\pi^{t,x}(z)\\
    =&\frac{\sigma_t'}{\sigma_t} \text{Id}-\frac{\sigma_t'}{\sigma_t^3}\int H_{\pi^{t,x}}^{m_t}(z)^{\otimes 2}d\pi^{t,x}(z)+\frac{1}{\sigma_t^2}\int H_{\pi^{t,x}}^{m_t' }(z)\otimes H_{\pi^{t,x}}^{m_t}(z)d\pi^{t,x}(z).
\end{align*}
\end{proof}

\subsubsection{Proof of Lemma~\ref{lemma:boundeasy}}\label{sec:lemma:boundeasy}

\begin{proof} 
As the probability density $\frac{q}{\int q}$  is $\alpha_t$-log concave, we have that 
 $\nu^{t,x}$ is $(\alpha_t+\frac{1}{\sigma_t^2})$-log-concave.
Then, using the Brascamp-Lieb inequality applied to functions of the form $x\mapsto \langle x,w\rangle $ with $w\in \mathbb{S}^{d-1}$, we obtain
\begin{align*}
    \langle w, \int  H_{\nu^{t,x}}^{\text{Id}}(y)^{\otimes2} d\nu^{t,x}(y),w\rangle&= \int \langle w,H_{\nu^{t,x}}^{\text{Id}}(y)^{\otimes2} w\rangle d\nu^{t,x}(y)\\
    &= \int \langle H_{\nu^{t,x}}^{\text{Id}}(y) w\rangle^2 d\nu^{t,x}(y)\\
    &= \int \langle z-\int z' d\nu^{t,x}(z'),w\rangle^2 d\nu^{t,x}(z)\\
    & \leq (\alpha_t+\frac{1}{\sigma_t^2})^{-1}.
\end{align*} 
Then,
\begin{align*}  \frac{-\sigma_t'}{\sigma_t^3}\lambda_{\max}\left(\int  H_{\nu^{t,x}}^{\text{Id}}(y)^{\otimes2} d\nu^{t,x}(y)\right)+\frac{\sigma_t'}{\sigma_t}&\leq \frac{-\sigma_t'}{\sigma_t}(\frac{1}{\alpha_t\sigma_t^2+1}-1)\\
&\leq\sigma_t\sigma_t'\frac{\alpha_t}{\alpha_t\sigma_t^2+1}.
\end{align*}
\end{proof}

\subsubsection{Proof of Lemma~\ref{lem:lemma3_cramer_rao}}\label{sec:lem:lemma3_cramer_rao}
\begin{proof}
Recall that $\nu^{t,x}$ has density proportional to
\[
y \mapsto
\exp  \left(-u_t(y)-\frac{\|x-y\|^2}{2\sigma_t^2}\right)
=: \exp\big(-V_{t,x}(y)\big).
\]
As $\nabla^2 u_t \preceq A_t \text{Id}$,
\[
\nabla^2 V_{t,x}(y)=\nabla^2 u_t(y)+\frac{1}{\sigma_t^2}\text{Id}
 \preceq \Big(A_t+\frac{1}{\sigma_t^2}\Big)\text{Id}.
\]
Denote by $J$ the Fisher information matrix of $\nu^{t,x}$,
\[
J:=\int \nabla^2 V_{t,x}(y) d\nu^{t,x}(y).
\]
By the Cramer-Rao inequality, we have
\[
\int  H^{\text{Id}}_{\nu^{t,x}}(y)^{\otimes 2} d\nu^{t,x}(y)
=\mathrm{Cov}_{\nu^{t,x}}(Y)
 \succeq J^{-1},
\]
hence for any $w\in\mathbb S^{d-1}$,
\[
\Big\langle w, \int  H^{\text{Id}}_{\nu^{t,x}}(y)^{\otimes 2} d\nu^{t,x}(y) w\Big\rangle
\ge \Big(A_t+\frac{1}{\sigma_t^2}\Big)^{-1}.
\]
Therefore
\begin{align*}
\lambda_{\max}  \left(-\int  H^{\text{Id}}_{\nu^{t,x}}(y)^{\otimes 2} d\nu^{t,x}(y)\right)&=-\lambda_{\min}  \left(\int  H^{\text{Id}}_{\nu^{t,x}}(y)^{\otimes 2} d\nu^{t,x}(y)\right)\\
&\le -\Big(A_t+\frac{1}{\sigma_t^2}\Big)^{-1}\\
&\le - \frac{\sigma_t^2}{A_t\sigma_t^2+1}.
\end{align*}
Therefore,
\[
\lambda_{\max}  \left(-\frac{\sigma_t'}{\sigma_t^3} \int  H^{\text{Id}}_{\nu^{t,x}}(y)^{\otimes 2} d\nu^{t,x}(y)\right)
+\frac{\sigma_t'}{\sigma_t}
\le
-\frac{\sigma_t'}{\sigma_t^3}\cdot\frac{\sigma_t^2}{A_t\sigma_t^2+1}
+\frac{\sigma_t'}{\sigma_t}
=
\frac{\sigma_t\sigma_t' A_t}{A_t\sigma_t^2+1}.
\]
\end{proof}

\subsubsection{Proof of Proposition~\ref{prop:firstuglybound}}\label{sec:prop:firstuglybound}

\begin{proof}
We have
\begin{align}\label{align:jsdujzbsio}
    &\int  H_{\mu_1}^{\text{Id}}(y)^{\otimes 2}d\mu_1(y)-\int H_{\mu_2}^{\text{Id}}(y)^{\otimes 2}d\mu_2(y)\nonumber\\
    &=\int \left( H_{\mu_1}^{\text{Id}}(y)^{\otimes 2}e^{v(y)}- H_{\mu_2}^{\text{Id}}(y)^{\otimes 2}\right) d\mu_2(y)\nonumber\\
    &=\int \left( H_{\mu_1}^{\text{Id}}(y)^{\otimes 2}- H_{\mu_2}^{\text{Id}}(y)^{\otimes 2}\right)e^{v(y)}d\mu_2(y)+\int  H_{\mu_2}^{\text{Id}}(y)^{\otimes 2}\left(e^{v(y)}-1\right)d\mu_2(y)\nonumber\\
        &=\int \left( H_{\mu_1}^{\text{Id}}(y)^{\otimes 2}- H_{\mu_2}^{\text{Id}}(y)^{\otimes 2}\right)\left(e^{v(y)}-1+1\right)d\mu_2(y)+\int  H_{\mu_2}^{\text{Id}}(y)^{\otimes 2}\left(e^{v(y)}-1\right)d\mu_2(y)\nonumber\\
        &=A_1+A_2+A_3,
\end{align}
with $$A_1:=\int \left( H_{\mu_1}^{\text{Id}}(y)^{\otimes 2}- H_{\mu_2}^{\text{Id}}(y)^{\otimes 2}\right)\left(e^{v(y)}-1\right)d\mu_2(y),$$
$$A_2:=\int \left( H_{\mu_1}^{\text{Id}}(y)^{\otimes 2}- H_{\mu_2}^{\text{Id}}(y)^{\otimes 2}\right)d\mu_2(y)$$
and
$$A_3:=\int  H_{\mu_2}^{\text{Id}}(y)^{\otimes 2}\left(e^{v(y)}-1\right)d\mu_2(y).$$
Now,
\begin{align*}
    & H_{\mu_1}^{\text{Id}}(y)^{\otimes 2}- H_{\mu_2}^{\text{Id}}(y)^{\otimes 2}\\
    = &-H_{\mu_1}^{\text{Id}}(y)\otimes \int zd\mu_1(z)+\left(\int zd\mu_2(z)-\int zd\mu_1(z)\right)\otimes y+  H_{\mu_2}^{\text{Id}}(y)\otimes \int zd\mu_2(z)\\
    = &-H_{\mu_1}^{\text{Id}}(y)\otimes \left(\int zd\mu_1(z)-\int zd\mu_2(z)\right) +\left(\int zd\mu_2(z)-\int zd\mu_1(z)\right)\otimes  H_{\mu_2}^{\text{Id}}(y)\\
    = &\left(-H_{\mu_2}^{\text{Id}}(y)+\int zd\mu_1(z)-\int zd\mu_2(z)\right)\otimes \left(\int zd\mu_1(z)-\int zd\mu_2(z)\right)\\
    &+\left(\int zd\mu_2(z)-\int zd\mu_1(z)\right)\otimes  H_{\mu_2}^{\text{Id}}(y)\\
\end{align*}
and
\begin{align*}
    \int zd\mu_2(z)-\int zd\mu_1(z)& = \int  H_{\mu_2}^{\text{Id}}(z)d\mu_2(z)-\int  H_{\mu_2}^{\text{Id}}(z)d\mu_1(z)\\
    & = \int  H_{\mu_2}^{\text{Id}}(z)\left(1-e^{v(z)}\right)d\mu_2(z).
\end{align*}
Then, writing $A:=\int  H_{\mu_2}^{\text{Id}}(z)\left(1-e^{v(z)}\right)d\mu_2(z)$, we have
\begin{align*}
     & H_{\mu_1}^{\text{Id}}(y)^{\otimes 2}- H_{\mu_2}^{\text{Id}}(y)^{\otimes 2}=H_{\mu_2}^{\text{Id}}(y)\otimes A +A^{\otimes 2}+ A\otimes H_{\mu_2}^{\text{Id}}(y).
\end{align*}

Therefore, for $h\in \mathbb{S}^{d-1}$ we have
\begin{align*}
&|h^\top A_2h|\\
&=|h^\top\int  (H_{\mu_1}^{\text{Id}}(y)^{\otimes 2}- H_{\mu_2}^{\text{Id}}(y)^{\otimes 2})d\mu_2(y)h|\\
&\leq \left(\int  h^\top H_{\mu_2}^{\text{Id}}(y)\left(e^{v(y)}-1\right)d\mu_2(y)\right)^2 + 2|\int  h^\top H_{\mu_2}^{\text{Id}}(y)d\mu_2(y)\int  h^\top H_{\mu_2}^{\text{Id}}(y)\left(e^{v(y)}-1\right)d\mu_2(y)|\\
&\leq \left(\int  h^\top H_{\mu_2}^{\text{Id}}(y)\left(e^{v(y)}-1\right)d\mu_2(y)\right)^2\nonumber
\end{align*}
as
$$\int  h^\top H_{\mu_2}^{\text{Id}}(y)d\mu_2(y)=0$$
and
\begin{align*}
&|h^\top A_1h|\\
&=|h^\top\int  (H_{\mu_1}^{\text{Id}}(y)^{\otimes 2}- H_{\mu_2}^{\text{Id}}(y)^{\otimes 2})\left(e^{v(y)}-1\right)d\mu_2(y)h|\\
&\leq \left(\int  h^\top H_{\mu_2}^{\text{Id}}(y)\left(e^{v(y)}-1\right)d\mu_2(y)\right)^2|\int\left(e^{v(y)}-1\right)d\mu_2(y)|\\
    & + 2 \left(\int  h^\top H_{\mu_2}^{\text{Id}}(y)\left(e^{v(y)}-1\right)d\mu_2(y)\right)^2\\
    &\leq 2 \left(\int  h^\top H_{\mu_2}^{\text{Id}}(y)\left(e^{v(y)}-1\right)d\mu_2(y)\right)^2,
\end{align*}
as $\int e^{v(y)}d\mu_2(y)=1$. Then, from \eqref{align:jsdujzbsio} we have
\begin{align}
    |h^\top&\left(\int  H_{\mu_1}^{\text{Id}}(y)^{\otimes 2}d\mu_1(y)-\int H_{\mu_2}^{\text{Id}}(y)^{\otimes 2}d\mu_2(y)\right)h|\nonumber\\
    &=|h^\top(A_1+A_2+A_3)h|\\
     \leq &3\left(\int  h^\top H_{\mu_2}^{\text{Id}}(y)\left(e^{v(y)}-1\right)d\mu_2(y)\right)^2\nonumber\\
    &+|\int  (h^\top H_{\mu_2}^{\text{Id}}(y))^{ 2}\left(e^{v(y)}-1\right)d\mu_2(y)|.
\end{align}
\end{proof}

\subsubsection{Proof of Theorem~\ref{thm:WPI}}\label{sec:thm:WPI}
\begin{proof}
Fix $\alpha\ge 2$ and $r>0$. Set
\[
L' := L r^{\beta-1}
\]
and define
\[
\tilde g_r(x)
:= \inf_{y\in\mathbb{R}^d}\bigl\{
f(y) + L' \|x-y\|
\bigr\}.
\]
By  Lemma~\ref{lem:infconv}, we have that $\tilde g_r$ is $L'$-Lipschitz and $\|f - \tilde g_r\|_{\infty} \le L r^\beta$. Let
\[
m := \int f d\mu,
\qquad
m_r := \int \tilde g_r d\mu.
\]
Then
\[
f - m
= (f-\tilde g_r) + (\tilde g_r - m_r) + (m_r - m),
\]
so
we obtain
\begin{equation}\label{eq:decomp-alpha}
\int |f-m|^\alpha d\mu
\le
3^{\alpha-1}
\Bigl(
\int |f-\tilde g_r|^\alpha d\mu
+ \int |\tilde g_r - m_r|^\alpha d\mu
+ |m_r - m|^\alpha
\Bigr).
\end{equation}

We now bound each term on the right-hand side.

\emph{Approximation terms.}
From $|f(x)-\tilde g_r(x)|\le L r^\beta$ we deduce
\[
\int |f-\tilde g_r|^\alpha d\mu \le (L r^\beta)^\alpha.
\]
Furthermore,
\[
|m_r - m|
= \Bigl|\int (\tilde g_r - f) d\mu\Bigr|
\le \int |f-\tilde g_r| d\mu
\le L r^\beta,
\]
so
\[
|m_r - m|^\alpha
\le
(L r^\beta)^\alpha.
\]
\emph{Lipschitz term.}
The function $\tilde g_r$ is $L'$-Lipschitz so applying
Lemma~\ref{lem:lip-moments} with $g=\tilde g_r$
yields
\[
\int |\tilde g_r - m_r|^\alpha d\mu
\le
C_\alpha L^\alpha \gamma^{-\alpha/2} r^{\alpha(\beta-1)}.
\]
Plugging the bounds in \eqref{eq:decomp-alpha}, we obtain
\[
\int |f-m|^\alpha d\mu
\le
3^{\alpha-1} L^\alpha
\Bigl(
2 r^{\alpha\beta}
+ C_\alpha \gamma^{-\alpha/2} r^{\alpha(\beta-1)}
\Bigr).
\]
We now choose $r$ so as to balance the two terms in parentheses. 
Solving $r^{\alpha\beta} = \gamma^{-\alpha/2} r^{\alpha(\beta-1)}$ gives $
r := \gamma^{-1/2}.$
With this choice,
we get
\[
2 r^{\alpha\beta}
+ C_\alpha \gamma^{-\alpha/2} r^{\alpha(\beta-1)}
\le
C'_\alpha \gamma^{-\alpha\beta/2}
\]
for $C'_\alpha=2 + C_\alpha$. We conclude that
\[
\int |f-m|^\alpha d\mu
\le
3^{\alpha-1} C'_\alpha L^\alpha \gamma^{-\alpha\beta/2}.
\]
Absorbing the factor $3^{\alpha-1}$ into the constant, we obtain
\[
\int \bigl|f-\int f d\mu \bigr|^\alpha d\mu
\le
C_\alpha L^\alpha \gamma^{-\alpha\beta/2}.
\]
Now the case $0< \alpha<2$ is immediate from Hölder's inequality.
\end{proof}

\subsubsection{Proof of Lemma~\ref{lemma:missigestimate}}\label{sec:lemma:missigestimate}
\begin{proof}
We have
\begin{align*}
    &|\int  (h^\top H_\mu^{\text{Id}}(y))^{ j}\left(e^{v(y)}-1\right)d\mu(y)| =| \text{Cov}_{\mu}\big((h^\top H_\mu^{\text{Id}}(Y))^j,e^{v(Y)}\big)|\\
    & \leq \sqrt{\text{Var}_{\mu}\big((h^\top H_\mu^{\text{Id}}(Y))^j\big)\text{Var}_{\mu}(e^{v(Y)}\big) }.
\end{align*}
Now applying Brascamp-Lieb inequality we have
$$\text{Var}_{\mu}\big((h^\top H_\mu^{\text{Id}}(Y))^j\big)\leq C \gamma_t^{-j}.$$
On the other hand,
\begin{align*}
    |e^{v(x)}-e^{v(y)}|\leq e^{K_1}|v(x)-v(y)|\leq e^{K_1}(2K_1)^{1-\xi}|v(x)-v(y)|^{\xi}\leq e^{K_1}(2K_1)^{1-\xi}K_t^\xi\|x-y\|^{\beta\xi}
\end{align*}
so $e^v$ is $\xi\beta$-Hölder with constant $CK_t^\xi$. If $\xi = 0$, then
\[
\text{Var}_\mu(e^{v(Y)}) \le \int e^{2v(y)}\,d\mu(y) \le e^{2K_1}.
\]
If $\xi \in (0,1]$, then from Theorem~\ref{thm:WPI}, we have
\begin{align*}
    \text{Var}_{\mu}(e^{v(Y)})\leq C K_t^{2\xi} \gamma_t^{-\beta\xi},
\end{align*}
which gives the result.
\end{proof}

\subsubsection{Proof of Lemma~\ref{lem:lemma10-unbounded}}\label{sec:lem:lemma10-unbounded}
\begin{proof}
Fix $h\in\mathbb{S}^{d-1}$ and $j\in\{1,2\}$. Since $\int e^v d\mu=1$,
\[
\int (h^\top H^{\text{Id}}_\mu)^j (e^v-1) d\mu
= \text{Cov}_\mu \big((h^\top H^{\text{Id}}_\mu)^j, e^v\big).
\]
By Cauchy-Schwarz,
\begin{equation}\label{eq:CS-lem10}
\Big|\int (h^\top H^{\text{Id}}_\mu)^j (e^v-1) d\mu\Big|
\le
\sqrt{\text{Var}_\mu \big((h^\top H^{\text{Id}}_\mu)^j\big)} \sqrt{\text{Var}_\mu(e^v)}.
\end{equation}
Let $V$ be the potential of $\mu$, i.e. $d\mu(x)\propto e^{-V(x)}dx$, with $\nabla^2V\succeq \gamma_t I$.
The Brascamp-Lieb inequality yields for all $g\in C^1$
\[
\text{Var}_\mu(g)\le \int \langle (\nabla^2V)^{-1}\nabla g,\nabla g\rangle d\mu
\le \frac{1}{\gamma_t}\int \|\nabla g\|^2 d\mu.
\]
For $j=1$, take $g(x)=h^\top H^{\text{Id}}_\mu(x)$, so $\nabla g=h$ and
\[
\text{Var}_\mu(h^\top H^{\text{Id}}_\mu)\le \frac{1}{\gamma_t}.
\]
For $j=2$, take $g(x)=(h^\top H^{\text{Id}}_\mu(x))^2$, so $\nabla g=2(h^\top H^{\text{Id}}_\mu(x)) h$ and
\[
\text{Var}_\mu\big((h^\top H^{\text{Id}}_\mu)^2\big)
\le \frac{4}{\gamma_t}\int (h^\top H^{\text{Id}}_\mu)^2 d\mu
= \frac{4}{\gamma_t} \text{Var}_\mu(h^\top H^{\text{Id}}_\mu)
\le \frac{4}{\gamma_t^2}.
\]
Hence in both cases $j\in\{1,2\}$ we have
\begin{equation}\label{eq:VarGj}
\sqrt{\text{Var}_\mu \big((h^\top H^{\text{Id}}_\mu)^j\big)}\le C \gamma_t^{-j/2}.
\end{equation}

\medskip
\noindent\textbf{Case $\boldsymbol{\xi=0}$.}
Lemma~\ref{lem:var_exp_concentration} with $r=\gamma_t^{-1/2}$ gives
\[
\text{Var}_\mu(e^v)\le \exp \Big(4K_t\gamma_t^{-\beta/2} + C K_t^2\gamma_t^{-\beta}\Big)-1.
\]
Using $4u\le 2u^2+2$  with $u=K_t\gamma_t^{-\beta/2}$, we obtain
\[
\text{Var}_\mu(e^v)\le \exp \big(C + C K_t^2\gamma_t^{-\beta}\big)-1
\le C \exp \big(C K_t^2\gamma_t^{-\beta}\big).
\]
Taking square-roots and inserting into \eqref{eq:CS-lem10} together with \eqref{eq:VarGj}
gives the result for $\xi=0$.

\medskip
\noindent\textbf{Case $\boldsymbol{\xi=1}$.}
Let
\[
r:=\gamma_t^{-1/2},
\qquad
L:= K_t r^{\beta-1}= K_t \gamma_t^{(1-\beta)/2},
\qquad
\varepsilon:= K_t r^\beta= K_t \gamma_t^{-\beta/2}.
\]
Define the inf-convolution approximation
\[
\tilde v(x):=\inf_{y\in\R^d}\Big\{v(y)+L\|x-y\|\Big\}.
\]
By Lemma~\ref{lem:infconv}:
(i) $\tilde v$ is $L$-Lipschitz;
(ii) $\|\tilde v-v\|_\infty\le \varepsilon$;
(iii) $\tilde v\le v$ pointwise. Consequently, for all $x\in \mathbb{R}^d$
\begin{equation}\label{eq:delta_bounds}
0\le \delta(x):=v(x)-\tilde v(x)\le \varepsilon.
\end{equation}

Set
\[
c:=\log\int e^{\tilde v} d\mu
\quad \text{ and } \quad
w:=\tilde v-c.
\]
Then $\int e^{w} d\mu=1$ and $w$ is $L$-Lipschitz. Moreover, since
$\int e^v d\mu=1$, we have 
\begin{equation}\label{eq:c_bounds}
-\varepsilon\le c\le 0.
\end{equation}
Now decompose
\[
e^v = e^{\tilde v+\delta}=e^{w} e^{c+\delta}
\]
 and define $X:=e^{w}$, $Z:=e^{c+\delta}$ so that $e^v=XZ$ and $\mathbb{E}_\mu[X]=1$.
From \eqref{eq:delta_bounds} and \eqref{eq:c_bounds}, we have
\begin{equation}\label{eq:Z_bounds}
-\varepsilon\le c+\delta\le \varepsilon
\qquad\Rightarrow\qquad
e^{-\varepsilon}\le Z\le e^{\varepsilon}.
\end{equation}
\textbf{A product variance bound.}
If $\mathbb{E}[X]=1$ and $|Z|\le b$ a.s., then
\begin{equation}\label{eq:var_product}
\text{Var}(XZ)\le 2b^2 \text{Var}(X)+2 \text{Var}(Z).
\end{equation}
Indeed, write $XZ-\mathbb{E}[XZ]=U+V$ with
$U:=((X-1)Z-\mathbb{E}[(X-1)Z])$ and $V:=Z-\mathbb{E}[Z]$.
Then $\text{Var}(XZ)=\mathbb{E}[(U+V)^2]\le 2\mathbb{E}[U^2]+2\mathbb{E}[V^2]=2\text{Var}(U)+2\text{Var}(Z)$.
Moreover, $\text{Var}(U)\le \mathbb{E}[((X-1)Z)^2]\le b^2\mathbb{E}[(X-1)^2]=b^2\text{Var}(X)$,
which gives \eqref{eq:var_product}.

Apply \eqref{eq:var_product} with $b=e^{\varepsilon}$ (by \eqref{eq:Z_bounds}):
\begin{equation}\label{eq:var_ev_split}
\text{Var}_\mu(e^v)\le 2e^{2\varepsilon}\text{Var}_\mu(e^{w})+2\text{Var}_\mu(Z).
\end{equation}
\textbf{Bound on $\text{Var}_\mu(e^{w})$.}
Since $w$ is $L$-Lipschitz and $\mu$ is $\gamma_t$-strongly log-concave, we get applying Lemma~\ref{lem:herbst} that there exists a universal $c_0>0$ such that for all
$\lambda\in\R$,
\begin{equation}\label{eq:mgf_lip}
\int \exp \big(\lambda(w-\mathbb{E}_\mu w)\big) d\mu
\le \exp \Big(\frac{c_0L^2}{\gamma_t}\frac{\lambda^2}{2}\Big).
\end{equation}
Let $m:=\mathbb{E}_\mu[w]$. Since $\int e^w d\mu=1$, Jensen gives $m\le 0$ and thus $e^{2m}\le 1$.
Taking $\lambda=2$ in \eqref{eq:mgf_lip} yields
\[
\int e^{2(w-m)} d\mu \le \exp \Big(\frac{2c_0L^2}{\gamma_t}\Big),
\]
hence
\[
\int e^{2w} d\mu = e^{2m}\int e^{2(w-m)} d\mu \le \exp \Big(\frac{2c_0L^2}{\gamma_t}\Big).
\]
Therefore,
\[
\text{Var}_\mu(e^w)=\int e^{2w} d\mu -1
\le \exp \Big(\frac{2c_0L^2}{\gamma_t}\Big)-1
\le C \frac{L^2}{\gamma_t}\exp \Big(C\frac{L^2}{\gamma_t}\Big),
\]
using $e^x-1\le xe^x$ for $x\ge 0$. Since $L= K_t\gamma_t^{(1-\beta)/2}$, we have
\[
\frac{L^2}{\gamma_t}= K_t^2\gamma_t^{-\beta},
\]
and thus
\begin{equation}\label{eq:var_ew}
\text{Var}_\mu(e^w)\le C  K_t^2\gamma_t^{-\beta}\exp \big(C K_t^2\gamma_t^{-\beta}\big).
\end{equation}
\textbf{Bound $\text{Var}_\mu(Z)$.}
From \eqref{eq:Z_bounds}, $Z\in[e^{-\varepsilon},e^{\varepsilon}]$. By Popoviciu's inequality,
\[
\text{Var}_\mu(Z)\le \frac{(e^{\varepsilon}-e^{-\varepsilon})^2}{4}
=\sinh(\varepsilon)^2 \le \varepsilon^2 e^{2\varepsilon},
\]
since $\sinh(\varepsilon)\le \varepsilon e^\varepsilon$ for all $\varepsilon\ge 0$.
Thus
\begin{equation}\label{eq:var_Z}
\text{Var}_\mu(Z)\le \varepsilon^2 e^{2\varepsilon}.
\end{equation}

Combining \eqref{eq:var_ev_split}, \eqref{eq:var_ew}, and \eqref{eq:var_Z}, we get
\[
\text{Var}_\mu(e^v)
\le C  K_t^2\gamma_t^{-\beta} e^{2\varepsilon}\exp \big(C K_t^2\gamma_t^{-\beta}\big)
+ C \varepsilon^2 e^{2\varepsilon}.
\]
Since $\varepsilon^2= K_t^2\gamma_t^{-\beta}$ and $e^{2\varepsilon}\le e\exp(\varepsilon^2)$
(using $2\varepsilon\le 1+\varepsilon^2$), we can absorb the $e^{2\varepsilon}$ factors into
$\exp(C K_t^2\gamma_t^{-\beta})$ to obtain
\begin{equation}\label{eq:var_ev_final}
\text{Var}_\mu(e^v)\le C  K_t^2\gamma_t^{-\beta}\exp \big(C K_t^2\gamma_t^{-\beta}\big).
\end{equation}
Taking square roots gives
\begin{equation}\label{eq:sqrt_var_ev}
\sqrt{\text{Var}_\mu(e^v)}\le C  K_t \gamma_t^{-\beta/2}\exp \big(C K_t^2\gamma_t^{-\beta}\big).
\end{equation}

Inserting \eqref{eq:sqrt_var_ev} and \eqref{eq:VarGj} into \eqref{eq:CS-lem10}, we get
\[ 
\Big|\int (h^\top H^{\text{Id}}_\mu)^{j}(e^v-1) d\mu\Big|
\le
\big(C\gamma_t^{-j/2}\big) \big(C K_t\gamma_t^{-\beta/2}e^{C K_t^2\gamma_t^{-\beta}}\big)
=
C  K_t \gamma_t^{-(j+\beta)/2}\exp \big(C K_t^2\gamma_t^{-\beta}\big),
\]
which gives the result for $\xi=1$.
\end{proof}

\subsubsection{Proof of Proposition~\ref{prop:finalestimatecompcov}}\label{sec:prop:finalestimatecompcov}

\begin{proof} 
Suppose first that $\|a_t\|_{\infty}\leq K_1$. Let $$v(y):=a_t(y)+\log\left(\int \exp\left(-u_t(w)\right) \varphi^{x,\sigma_t}(w)   dw\right) -\log\left(\int \exp\left(a_t(w)-u_t(w)\right) \varphi^{x,\sigma_t}(w)   dw\right),$$
so $p^{t,x}=e^v \nu^{t,x}$.
Combining Proposition~\ref{prop:firstuglybound} and Lemma~\ref{lemma:missigestimate} we have
\begin{align}\label{align:finalcal}
    &|h^\top \left(\int H_{p^{t,x}}^{\text{Id}}(y)^{\otimes 2}   dp^{t,x}(y) - \int H_{\nu^{t,x}}^{\text{Id}}(y)^{\otimes 2}   d\nu^{t,x}(y)\right) h|\nonumber\\
    & \leq C\Bigl(|\int  h^\top H_{\nu^{t,x}}^{\text{Id}}(y)\left(e^{v(y)}-1\right)d\nu^{t,x}(y)|^2\\
    &+|\int  (h^\top H_{\nu^{t,x}}^{\text{Id}}(y))^{ 2}\left(e^{v(y)}-1\right)d\nu^{t,x}(y)|\Bigl)\nonumber\\
    & \leq C K_t^\xi(\alpha_t+\frac{1}{\sigma_t^2})^{-\frac{1+\beta\xi}{2}}(\alpha_t+\frac{1}{\sigma_t^2})^{-\frac{1}{2}}+C K_t^\xi(\alpha_t+\frac{1}{\sigma_t^2})^{-\frac{2+\beta\xi}{2}}\nonumber\\
    &\leq C K_t^\xi(\alpha_t+\frac{1}{\sigma_t^2})^{-\frac{2+\beta\xi}{2}}.
\end{align}
The unbounded case follows the same steps by replacing Lemma~\ref{lemma:missigestimate} by Lemma~\ref{lem:lemma10-unbounded}.
\end{proof}

\subsubsection{Proof of Proposition~\ref{prop:boundsecondterm}}\label{sec:prop:boundsecondterm}

\begin{proof}
Suppose first that $\|a_t\|_{\infty}\leq K_1$. We have
\begin{align*}
    | \int h^\top H_{\pi^{t,x}}^{m_t' }(z) h^\top H_{\pi^{t,x}}^{m_t}(z)d\pi^{t,x}(z)| \\
    & \leq \left(\int (h^\top H_{\pi^{t,x}}^{m_t' }(z))^2 d\pi^{t,x}(z)\int ( h^{\top} H_{\pi^{t,x}}^{m_t}(z))^2d\pi^{t,x}(z)\right)^{1/2}.
\end{align*}
Furthermore,
\begin{align*}
    \int ( h^{\top} H_{\pi^{t,x}}^{m_t}(z))^2d\pi^{t,x}(z)=&\int ( h^{\top} H_{p^{t,x}}^{\text{Id}}(z))^2dp^{t,x}(z)\\
    =& \int ( h^{\top} H_{p^{t,x}}^{\text{Id}}(z))^2dp^{t,x}(z)-\int ( h^{\top} H_{\nu^{t,x}}^{\text{Id}}(z))^2d\nu^{t,x}(z)+\int ( h^{\top} H_{\nu^{t,x}}^{\text{Id}}(z))^2d\nu^{t,x}(z)
\end{align*}
and applying Brascamp-Lieb inequality we have
\begin{equation}\label{eq:fghjiuytfghjghjko}
    \int ( h^{\top} H_{\nu^{t,x}}^{\text{Id}}(z))^2d\nu^{t,x}(z)\leq (\alpha_t+\frac{1}{\sigma_t^2})^{-1}.
\end{equation}
Then, from Proposition~\ref{prop:finalestimatecompcov} we obtain\begin{align}
    |\int ( h^{\top} H_{p^{t,x}}^{\text{Id}}(z))^2dp^{t,x}(z)-\int ( h^{\top} H_{\nu^{t,x}}^{\text{Id}}(z))^2d\nu^{t,x}(z)|\leq C \sigma_t^{2}\frac{1}{\alpha_t\sigma_t^2+1}.
\end{align}
Finally, using Assumption~\ref{assum:varmprime} we have
$$\int (h^\top H_{\pi^{t,x}}^{m_t' }(z))^2 d\pi^{t,x}(z)\leq L_t^2 \sigma_t^2\frac{1}{\alpha_t\sigma_t^2+1}$$
so we deduce that
\begin{align*}
    \left(\int (h^\top H_{\pi^{t,x}}^{m_t' }(z))^2 d\pi^{t,x}(z)\int ( h^{\top} H_{\pi^{t,x}}^{m_t}(z))^2d\pi^{t,x}(z)\right)^{1/2}\leq C L_t\sigma_t^2\frac{1}{\alpha_t\sigma_t^2+1}.
\end{align*}
The unbounded case follows the same steps.
\end{proof}

\subsubsection{Proof of Theorem~\ref{theo:onesidedlipinteg}}\label{sec:theo:onesidedlipinteg}

\begin{proof}
Fix $z\in[0,1]$. By Assumption~\ref{assum:sigma} and Theorem~\ref{theo:onesidedlip2}, for a.e. $t\in[\delta,1]$,
\[
\sup_{x\in \mathbb{R}^d} \lambda_{\max}\left(\nabla v_t(x)\right)
\le \frac{1}{\alpha_t\sigma_t^2+1}\Bigg(C(L_t
    -\frac{\sigma_t'}{\sigma_t}(K_t\sigma_t^\beta)^{\xi_t})
    \exp \Big(CK_t^2(\alpha_t+\tfrac{1}{\sigma_t^2})^{-\beta}\Big)
    +\sigma_t\sigma_t'\alpha_t\Bigg).
\]
For $t\geq \delta$, Assumption~\ref{assum:sigma} gives $\sigma_t'\le 0$, hence $\sigma_t\sigma_t'\alpha_t\leq 0$ and we can drop this term.
Set, for $t\in[\delta,1]$,
\[
g_1(t):=\frac{C}{\alpha_t\sigma_t^2+1}\Bigg(L_t
    -\frac{\sigma_t'}{\sigma_t}(K_t\sigma_t^\beta)^{\xi_t}\Bigg)
    \exp \Big(CK_t^2(\alpha_t+\tfrac{1}{\sigma_t^2})^{-\beta}\Big),
\]
so that $\sup_x \lambda_{\max}(\nabla v_t(x))\le g_1(t)$ for a.e. $t\in[\delta,1]$.
We take $\xi_t=\mathds{1}_{t\geq s}$ and split the integral accordingly:
\begin{align*}
&\int_\delta^1 g_1(t) dt\\
&\qquad\leq \int_\delta^1 \frac{L_t}{\alpha_t\sigma_t^2+1}\exp \Big(CK_t^2(\alpha_t+\tfrac{1}{\sigma_t^2})^{-\beta}\Big)dt
-\int_{\delta\wedge s}^{s}\frac{\sigma_t'}{\sigma_t}\frac{1}{\alpha_t\sigma_t^2+1}
\exp \Big(CK_t^2(\alpha_t+\tfrac{1}{\sigma_t^2})^{-\beta}\Big)dt\\
&\qquad\quad -\int_{s\vee\delta}^1\frac{\sigma_t'}{\sigma_t}\frac{1}{\alpha_t\sigma_t^2+1}
(K_t\sigma_t^\beta)\exp \Big(CK_t^2(\alpha_t+\tfrac{1}{\sigma_t^2})^{-\beta}\Big)dt.
\end{align*}
Since $(\alpha_t+\tfrac{1}{\sigma_t^2})^{-\beta}\le \alpha_t^{-\beta}$, the assumption
\[
\forall t\in[0,\delta\vee s],\qquad K_t^2\alpha_t^{-\beta}\le A
\]
yields, for all $t\in[0,\delta\vee s]$,
\begin{equation}\label{eq:expbd}
\exp \Big(CK_t^2(\alpha_t+\tfrac{1}{\sigma_t^2})^{-\beta}\Big)\le e^{CA}.
\end{equation}
Furthermore, on $[s\vee\delta,1]$ we use $K_t\le A\sigma_t^{-\chi}$
and $(\alpha_t+\tfrac{1}{\sigma_t^2})^{-\beta}\le \sigma_t^{2\beta}$ to get
\[
\exp \Big(CK_t^2(\alpha_t+\tfrac{1}{\sigma_t^2})^{-\beta}\Big)
\le \exp \big(CA^2\sigma_t^{2(\beta-\chi)}\big)\le e^{CA^2},
\]
where we used $\beta>\chi$ and $\sigma_t\le 1$.

On $[\delta\wedge s,s]\subset[0,\delta\vee s]$,
using \eqref{eq:expbd}, $\frac{1}{\alpha_t\sigma_t^2+1}\le 1$, and $\sigma_t'\le 0$ on $[\delta,1]$, we obtain
\begin{align*}
-\int_{\delta\wedge s}^{s}\frac{\sigma_t'}{\sigma_t}\frac{1}{\alpha_t\sigma_t^2+1}
\exp \Big(CK_t^2(\alpha_t+\tfrac{1}{\sigma_t^2})^{-\beta}\Big)dt
&\le -e^{CA}\int_{\delta\wedge s}^{s}\frac{\sigma_t'}{\sigma_t}dt\\
&= e^{CA}\Big(\log(\sigma_{\delta\wedge s})-\log(\sigma_s)\Big).
\end{align*}
Hence, using again $\frac{1}{\alpha_t\sigma_t^2+1}\le 1$ and $\sigma_t'\le 0$,
\begin{align*}
-\int_{s\vee\delta}^1\frac{\sigma_t'}{\sigma_t}\frac{1}{\alpha_t\sigma_t^2+1}
(K_t\sigma_t^\beta)\exp \Big(CK_t^2(\alpha_t+\tfrac{1}{\sigma_t^2})^{-\beta}\Big)dt
&\le -e^{CA^2}\int_{s\vee\delta}^1\frac{\sigma_t'}{\sigma_t} A\sigma_t^{\beta-\chi}dt\\
&= Ce^{CA^2}\Big(\sigma_{s\vee\delta}^{\beta-\chi}-\sigma_1^{\beta-\chi}\Big).
\end{align*}
Combining these bounds with the assumption $\int_\delta^1 \frac{L_t}{\alpha_t\sigma_t^2+1}dt\le C$,
we conclude
\[
\int_\delta^1 g_1(t) dt \le C.
\]

\medskip
On the other hand, by Proposition~\ref{propnablav} we have for a.e. $t\in[0,\delta]$,
\begin{align*}
  \sup_{x\in \mathbb{R}^d} \lambda_{\max} \left(\nabla v_t(x)\right)
  &\leq \sup_{x\in \mathbb{R}^d,h\in \mathbb{S}^{d-1}} \frac{\sigma_t'}{\sigma_t}
  +h^\top \frac{1}{\sigma_t^2}\int H_{\pi^{t,x}}^{m_t'}(z)\otimes H_{\pi^{t,x}}^{m_t}(z)d\pi^{t,x}(z)h\\
  &\hspace{3.5cm}-\frac{\sigma_t'}{\sigma_t^3}h^\top\int H_{\pi^{t,x}}^{m_t}(z)^{\otimes 2}d\pi^{t,x}(z)h\mathds{1}_{\sigma_t'\leq 0}.
\end{align*}
By Proposition~\ref{prop:boundsecondterm},
\[
h^\top\frac{1}{\sigma_t^2}\int H_{\pi^{t,x}}^{m_t'}(z)\otimes H_{\pi^{t,x}}^{m_t}(z)d\pi^{t,x}(z)h
\leq C \frac{L_t}{\alpha_t\sigma_t^2+1}\exp \Big(CK_t^2(\alpha_t+\tfrac{1}{\sigma_t^2})^{-\beta}\Big).
\]
Moreover, combining Proposition~\ref{prop:finalestimatecompcov} (with $\xi=0$, i.e. with the exponential factor)
and Lemma~\ref{lemma:boundeasy}, we obtain for $t\in[0,\delta]$,
\begin{align*}
\frac{\sigma_t'}{\sigma_t}-\frac{\sigma_t'}{\sigma_t^3}h^\top\int H_{\pi^{t,x}}^{m_t}(z)^{\otimes 2}d\pi^{t,x}(z)h\mathds{1}_{\sigma_t'\leq 0}
&\leq -C\frac{\sigma_t'}{\sigma_t} \frac{1}{\alpha_t\sigma_t^2+1}
\exp \Big(CK_t^2(\alpha_t+\tfrac{1}{\sigma_t^2})^{-\beta}\Big)\mathds{1}_{\sigma_t'\leq 0}
+\frac{\sigma_t'}{\sigma_t}\mathds{1}_{\sigma_t'> 0}.
\end{align*}
Since $\delta\le \delta\vee s$, the bound \eqref{eq:expbd} applies on $[0,\delta]$ and yields
\[
\exp \Big(CK_t^2(\alpha_t+\tfrac{1}{\sigma_t^2})^{-\beta}\Big)\le e^{CA},\qquad t\in[0,\delta].
\]
Using again Assumption~\ref{assum:sigma} on $[0,\delta]$ (in particular $|\sigma_t'|\vee \sigma_t^{-1}\le A$),
we conclude that the whole bracket is bounded by a constant. In particular, for a.e. $t\in[0,\delta]$,
\[
\sup_{x\in \mathbb{R}^d} \lambda_{\max}\left(\nabla v_t(x)\right)
\le g_0(t):= C\left(\frac{L_t}{\alpha_t\sigma_t^2+1}+1\right).
\]
Therefore,
\[
\int_0^\delta g_0(t) dt \le C.
\]
\end{proof}

\subsubsection{Proof of Theorem~\ref{theo:onesidedlipinteg0}}\label{sec:theo:onesidedlipinteg0}

\begin{proof}
    Fix $z\in[0,1]$. Let $\delta \in (0,1)$ given by Assumption 
 \ref{assum:sigma2} such that for all $t\in [0,\delta)$, $\sigma_t'> 0$ and for all $t\in (\delta,1)$ $\sigma_t'< 0$.
 Let
\[
E_t:=\exp \left(CK_t^2\Big(\alpha_t+\frac{1}{\sigma_t^2}\Big)^{-\beta}\right).
\]
Since $\big(\alpha_t+\frac{1}{\sigma_t^2}\big)^{-\beta}\le \sigma_t^{2\beta}$ and $K_t\le A\sigma_t^{-\chi}$, we have
\[
E_t \le \exp \left(CA^2\sigma_t^{2(\beta-\chi)}\right)\le \exp \left(CA^2\sigma_\delta^{2(\beta-\chi)}\right)=:M,
\]
for all $t\in[0,1]$ (using that $\sigma_t\le\sigma_\delta$ on $[0,1]$).

By Theorem~\ref{theo:onesidedlip2} we have 
\begin{align*}
    \int_z^1 \sup_{x\in \mathbb{R}^d} \lambda_{\max}\left(\nabla v_t(x)\right) dt
    &\leq \int_z^{z\vee\delta} \sup_{x\in \mathbb{R}^d} \lambda_{\max}\left(\nabla v_t(x)\right) dt\\
    &\quad +\int_{z\vee\delta}^1 \frac{C}{\alpha_t\sigma_t^2+1}\left(\left(L_t-\frac{\sigma_t'}{\sigma_t}(K_t\sigma_t^\beta)^{\xi_t}\right)E_t+\sigma_t\sigma_t'\alpha_t\right) dt.       
\end{align*}

Now, taking $\xi_t=1$,
\begin{align*}
\int_{z\vee\delta}^1 \frac{1}{\alpha_t\sigma_t^2+1}\left(\left(L_t-\frac{\sigma_t'}{\sigma_t}(K_t\sigma_t^\beta)^{\xi_t}\right)E_t+\sigma_t\sigma_t'\alpha_t\right) dt
&\leq C\int_{z\vee\delta}^1 \frac{L_t}{\alpha_t\sigma_t^2+1} dt
    -C \int_{z\vee\delta}^1\frac{\sigma_t'}{\sigma_t}A\sigma_t^{\beta-\chi} dt\\
&\leq C\big(A+\sigma_{z\vee\delta}^{\beta-\chi}-\sigma_1^{\beta-\chi}\big)\\
&\leq C\big(A+\sigma_\delta^{\beta-\chi}-\sigma_1^{\beta-\chi}\big)\\
& \leq C.
\end{align*}

On the other hand, by Theorem~\ref{theo:onesidedlip2}, we have
\begin{align*}
    \int_z^{z\vee\delta} \sup_{x\in\mathbb{R}^d}\lambda_{\max}\left(\nabla v_t(x)\right) dt
    &\leq C \int_z^{z\vee\delta} \frac{1}{\alpha_t\sigma_t^2+1}\left(\left(L_t+\frac{\sigma_t'}{\sigma_t}(K_t\sigma_t^\beta)^{\xi_t}\right)E_t+\sigma_t\sigma_t'A\right) dt\\
    &\leq C+C \int_z^{z\vee\delta} \frac{1}{\alpha_t\sigma_t^2+1}\frac{\sigma_t'}{\sigma_t}(K_t\sigma_t^\beta)^{\xi_t} dt\\
    &\leq C+C \int_0^\delta \frac{1}{\alpha_t\sigma_t^2+1}\frac{\sigma_t'}{\sigma_t}(K_t\sigma_t^\beta)^{\xi_t} dt.
\end{align*}
Now, taking $\xi_t=1$ we have
$$
K_t^\xi\frac{\sigma_t'}{\sigma_t^{1-\beta\xi}}\frac{1}{\alpha_t\sigma_t^2+1}
\leq C\frac{\sigma_t'}{\sigma_t^{1-\beta+\chi}}\frac{1}{\alpha_t\sigma_t^2+1}
\leq C\frac{\sigma_t'}{\sigma_t^{1-\beta+\chi}},
$$
which gives
$$
\int_z^{z\vee\delta} \sup_{x\in\mathbb{R}^d}\lambda_{\max}\left(\nabla v_t(x)\right) dt \leq C.
$$

Combining the two estimates (and enlarging $C$ if needed) yields
$$
\int_z^1 \sup_{x\in \mathbb{R}^d} \lambda_{\max}\left(\nabla v_t(x)\right) dt\leq C,
$$
for all $z\in[0,1]$.
\end{proof}

\subsubsection{Proof of Corollary~\ref{coro:lipmafm}}\label{sec:coro:lipmafm}

\begin{proof}
    Let us show that this setting satisfies the assumptions of Theorem~\ref{theo:onesidedlipinteg}. We have that $$(m_t)_{\#} \pi =\exp(-u_t+a_t),$$
with $u_t(x)=u(x/f_t)+d\log(f_t)$ and $a_t(x)=a(x/f_t)$ so $\alpha_t=\alpha f_t^{-2}$ and $K_t=Kf_t^{-\beta}$. Furthermore,
\begin{align*}
    \big|h^\top \big(m_t'(z)-m_t'(z')\big)\big|=f_t'\big|h^\top \big(x_1-x_1'\big)\big|
 = 
\frac{f_t'}{f_t} \big|h^\top \big(m_t(z)-m_t(z')\big)\big|
\end{align*}
so taking $L_t=\frac{f_t'}{f_t}$, we get
\begin{align*}
    \operatorname{Var}_{\pi^{t,x}} \big(h^\top m_t'(Z)\big)&\leq L_t^2 \operatorname{Var}_{\pi^{t,x}} \big(h^\top m_t(Z)\big)\\
    &\leq L_t^2 \int ( h^{\top} H_{\pi^{t,x}}^{m_t}(z))^2d\pi^{t,x}(z).
\end{align*}
Now,
\begin{align*}
    \int ( h^{\top} H_{\pi^{t,x}}^{m_t}(z))^2d\pi^{t,x}(z)=&\int ( h^{\top} H_{p^{t,x}}^{\text{Id}}(z))^2dp^{t,x}(z)\\
    =& \int ( h^{\top} H_{p^{t,x}}^{\text{Id}}(z))^2dp^{t,x}(z)-\int ( h^{\top} H_{\nu^{t,x}}^{\text{Id}}(z))^2d\nu^{t,x}(z)+\int ( h^{\top} H_{\nu^{t,x}}^{\text{Id}}(z))^2d\nu^{t,x}(z)
\end{align*}
and applying Brascamp-Lieb inequality we have
\begin{equation}
    \int ( h^{\top} H_{\nu^{t,x}}^{\text{Id}}(z))^2d\nu^{t,x}(z)\leq (\alpha_t+\frac{1}{\sigma_t^2})^{-1}.
\end{equation}
Furthermore, from Proposition~\ref{prop:finalestimatecompcov} we obtain\begin{align*}
    |\int ( h^{\top} H_{p^{t,x}}^{\text{Id}}(z))^2dp^{t,x}(z)-\int ( h^{\top} H_{\nu^{t,x}}^{\text{Id}}(z))^2d\nu^{t,x}(z)|\leq C \sigma_t^{2}\frac{1}{\alpha_t\sigma_t^2+1}\exp \big(C K_t^2 (\alpha_t+\frac{1}{\sigma_t^2})^{-\beta}\big),
\end{align*}
which gives
$$\operatorname{Var}_{\pi^{t,x}} \big(h^\top m_t'(Z)\big)\leq C \frac{L_t^2\sigma_t^{2}}{\alpha_t\sigma_t^2+1}\exp \big(C K_t^2 (\alpha_t+\frac{1}{\sigma_t^2})^{-\beta}\big).$$
As
\begin{align*}
    (\alpha_t+\frac{1}{\sigma_t^2})^{-\beta}\leq \alpha_t^{-\beta}=\alpha^{-\beta}f_t^{2\beta}= \alpha^{-\beta} K^{2}K_t^{-2}, 
\end{align*}
we deduce that
$$K_t^2\alpha_t^{-\beta}\leq A$$
which in particular gives
$$\operatorname{Var}_{\pi^{t,x}} \big(h^\top m_t'(Z)\big)\leq C \frac{L_t^2\sigma_t^{2}}{\alpha_t\sigma_t^2+1},$$
so Assumption~\ref{assum:varmprime} is verified. On the other hand, as $\gamma  \leq \sigma_{\frac{1}{2}}\wedge f_{\frac{1}{2}}$ and $f$ is increasing, $\sigma$ is decreasing, we have
$$\alpha\sigma_t^2+f_t^2 \geq (\alpha\gamma^2)\wedge \gamma^2.$$
Then,
\begin{align*}
   \int_0^1\frac{L_t}{\alpha_t\sigma_t^2+1}=\int_0^1\frac{f_t'f_t}{\alpha\sigma_t^2+f_t^2} \leq \frac{1}{(\alpha\gamma^2)\wedge \gamma^2}\int_0^1 f_tf_t'dt\leq \frac{f_1^2-f_0^2}{2((\alpha\gamma^2)\wedge \gamma^2)}\leq C.
\end{align*}
Now, for $s=1-\gamma$ we have for all $t\geq s$, $$K_t=Kf_t^{-\beta}\leq Kf_{1-\gamma}^{-\beta}\leq K\sigma_{1-\gamma}^{-\beta/2}\leq K\sigma_t^{-\beta/2}.$$ 
Finally, as $\sigma$ satisfies Assumption~\ref{assum:sigma} (with $\delta=0$), we can apply Theorem~\ref{theo:onesidedlipinteg} which gives the result.
\end{proof}

\subsubsection{Proof of Corollary~\ref{coro:notlipmafm}}\label{sec:coro:notlipmafm}
\begin{proof}
Let us show that the setting of Corollary~\ref{coro:notlipmafm} satisfies the assumptions of Theorem~\ref{theo:onesidedlipinteg0}.
We write $(m_t)_\#\pi=e^{-u_t+a_t}$ with
\[
e^{-u_t(x)}:=\int \varphi^{x,g_t}(f_t y)  e^{-u(y)} dy,
\qquad
a_t(x):=\log\frac{\int \varphi^{x,g_t}(f_t y)  e^{-u(y)+a(y)} dy}
{\int \varphi^{x,g_t}(f_t y)  e^{-u(y)} dy},
\]
where $\varphi^{x,g_t}$ is the $\mathcal{N}(x,g_t^2 \mathrm{Id})$ density.

By Proposition~\ref{prop:bondhessut} we have for all $x\in\mathbb{R}^d$,
\begin{equation}\label{eq:ut-hessian-two-sided2-new}
\frac{\alpha}{\alpha g_t^2 + f_t^2} \mathrm{Id}  \preceq  \nabla^2 u_t(x)  \preceq  \frac{1}{g_t^2} \mathrm{Id}.
\end{equation}
In particular we may take
\[
\alpha_t:=\frac{\alpha}{\alpha g_t^2+f_t^2}.
\]

We now bound the $\beta$-Hölder constant of $a_t$.
Fix $t\in(0,1)$ and define, for $x\in\R^d$,
\[
\nu^{t,x}(dy)\propto e^{-u(y)} \varphi^{x,g_t}(f_t y) dy,
\]
so that $a_t(x)=\log\big(\mathbb{E}_{\nu^{t,x}} e^{a(Y)}\big)$.
By Lemma~\ref{lem:mgf_holder}, taking
\[
M_q:=\exp \Big(2Kq+\frac{C K^2}{2\alpha} q^2\Big),
\]
condition \eqref{eq:exp-mom-assump} in Lemma~\ref{lemma:holderexp} holds, so applying Lemma~\ref{lemma:holderexp} we get
\[
|a_t(x)-a_t(x')|
\le
\begin{cases}
K M_{2/(1-\beta)}^{ 1-\beta} \kappa_t^\beta \|x-x'\|^\beta, & \beta\in(0,1),\\[4pt]
K M_4^{1/2} \kappa_t \|x-x'\|, & \beta=1,
\end{cases}
\qquad
\kappa_t:=\dfrac{f_t}{\alpha g_t^2+f_t^2}.
\]
Using also the structural assumption $f_t+g_t\ge \gamma$ for all $t$, we have
\[
\alpha g_t^2+f_t^2  \ge (\alpha\wedge 1)(g_t^2+f_t^2)
 \ge (\alpha\wedge 1)\frac{(f_t+g_t)^2}{2}
 \ge (\alpha\wedge 1)\frac{\gamma^2}{2},
\]
so $\kappa_t\le C$ uniformly in $t\in[0,1]$.
Therefore the Hölder constants
\[
K_t:=\sup_{x\neq x'}\frac{|a_t(x)-a_t(x')|}{\|x-x'\|^\beta}
\]
satisfy $K_t\le C$ uniformly in $t$.
Thus $(m_t)_\#\pi$ is $(\alpha_t,\beta, C)$-weakly log-concave, with $C$ depending only on $\alpha,\beta,K,\gamma$.

By Proposition~\ref{prop:boundvarmprime}, for all $x\in\R^d$ and all unit $h\in\mathbb S^{d-1}$,
\[
\operatorname{Var}_{\pi^{t,x}} \big(h^\top m_t'(Z)\big)
 \le 
C \bar M_2^{1/2} 
\sigma_t^2 
\frac{ f_t'^2 + \alpha g_t'^2 + \dfrac{(g_t' f_t - f_t' g_t)^2}{\sigma_t^2} }
{ \alpha\sigma_t^2 + f_t^2 + \alpha g_t^2 }.
\]
Using again $f_t^2+\alpha g_t^2\ge (\alpha\wedge 1)\gamma^2/2$, we get
\[
\frac{1}{\alpha\sigma_t^2 + f_t^2 + \alpha g_t^2}
\le
\frac{2}{(\alpha\wedge 1)\gamma^2} \frac{1}{\alpha\sigma_t^2 + f_t^2 + \alpha g_t^2}.
\]
On the other hand,
\[
\frac{\sigma_t^2}{\alpha_t\sigma_t^2+1}
=
\sigma_t^2 \frac{\alpha g_t^2+f_t^2}{\alpha\sigma_t^2+\alpha g_t^2+f_t^2}
\ge
\frac{(\alpha\wedge 1)\gamma^2}{2} \frac{\sigma_t^2}{\alpha\sigma_t^2+f_t^2+\alpha g_t^2}.
\]
Combining the last two displays yields
\[
\operatorname{Var}_{\pi^{t,x}} \big(h^\top m_t'(Z)\big)
 \le 
L_t^2 \frac{\sigma_t^2}{\alpha_t\sigma_t^2+1},
\]
with
\[
L_t^2
:=
C 
\Big(f_t'^2+\alpha g_t'^2+\dfrac{(g_t'f_t-f_t'g_t)^2}{\sigma_t^2}\Big),
\]
where $C$ depends only on $\alpha,\beta,K,\gamma$ (through $\bar M_2$).
Thus Assumption~\ref{assum:varmprime} holds.

Finally, we have
\begin{align*}
\int_0^1 \frac{L_t}{1+\alpha_t\sigma_t^2} dt
 \leq&
C\left(
\int_0^1 \sqrt{f_t'^2+\alpha g_t'^2} dt
 + \int_0^1 \frac{|g_t'f_t-f_t'g_t|}{\sigma_t} dt
\right)\\
& \leq C \left(\int_0^1 (f'_t-\sqrt{\alpha}  g'_t) dt+\int_0^1 \sigma_t^{-1+q} dt\right)\\
&\leq C,
\end{align*}
using that $|g_t'f_t-f_t'g_t|\le K\sigma_t^q$ and $\int_0^1\sigma_t^{-1+q} dt\le K$.

Since $K_t\le C$ uniformly, we may take $\chi=0$ in Theorem~\ref{theo:onesidedlipinteg0}.
Moreover, for $t\in[0,\delta)$ we have $\sigma_t'>0$ and by Assumption~\ref{assum:sto-int} we have $g_t\ge \gamma$.
Hence by \eqref{eq:ut-hessian-two-sided2-new},
\[
\nabla^2 u_t(x) \preceq \frac{1}{g_t^2} \mathrm{Id} \preceq \frac{1}{\gamma^2} \mathrm{Id},
\]
so the third bullet of Theorem~\ref{theo:onesidedlipinteg0} holds.

All assumptions of Theorem~\ref{theo:onesidedlipinteg0} are satisfied with a constant $A$
depending only on $\alpha,\beta,K,\gamma$ (and on $\sigma_\delta$ through Assumption~\ref{assum:sigma2}).
Therefore Theorem~\ref{theo:onesidedlipinteg0} yields for all $z\in [0,1]$
\[
\int_z^1 \sup_{x\in\mathbb{R}^d}\lambda_{\max}\big(\nabla v_t(x)\big) dt \le C.
\]
\end{proof}

\subsubsection{Proof of Corollary~\ref{coro:score_osl_selfcontained}}\label{sec:coro:score_osl_selfcontained}
\begin{proof}
We have that the drift of the probability flow-ODE is
\[
v_t(x)=\frac1t\bigl(x+s(t,x)\bigr),\qquad s(t,x)=\nabla\log p_t(x).
\]
Then, we need to bound the quantity
\[
\int_z^1 \frac{1}{t} \sup_{x\in\R^d}\lambda_{\max} \bigl(\text{Id}+\nabla s(t,x)\bigr)  dt.
\]
We use Corollary~\ref{coro:lipmafm}. Observe that the smoothing path
\[
X_t := tY+\sqrt{1-t^2}  \xi,\qquad Y\sim p^\star,\ \xi\sim\mathcal N(0,\text{Id}),\ Y\perp\xi,
\]
is a special case of the Lipman flow-matching setting (Assumption~\ref{assum:lipman}) with
\[
\pi=\gamma_d\otimes p^\star,\qquad m_t(x_0,x_1)=f_t x_1 \text{ with } f_t=t,\qquad \sigma_t=\sqrt{1-t^2}.
\]
Indeed, $f$ is increasing and $C^1$ with $f_0=0$ and $f_1=1$, while $\sigma$ is decreasing and $C^1$ on $(0,1)$
with $\sigma_0=1$ and $\sigma_1=0$. Moreover, choosing for instance $\gamma:=\tfrac15$, one checks the last bullet
of Assumption~\ref{assum:lipman}: since $\gamma\le f_{1/2}=1/2$ and $\gamma\le \sigma_{1/2}=\sqrt{3}/2$, the required
non-degeneracy condition holds and
\[
f_{1-\gamma}^2=(1-\gamma)^2=\Bigl(\frac45\Bigr)^2=\frac{16}{25} \ge \sqrt{1-\Bigl(\frac45\Bigr)^2}
=\frac35=\sigma_{1-\gamma}.
\]
Therefore $(v,\pi,m,\sigma)$ satisfies Assumption~\ref{assum:lipman} and Corollary~\ref{coro:lipmafm} yields
\[
\int_z^1 \sup_{x\in\R^d}\lambda_{\max} \bigl(\nabla v_t(x)\bigr)  dt \le C.
\]
\end{proof}

\subsection{Details for Lipschitz estimates}

\subsubsection{Proof of Corollary~\ref{coro:lipregugeneral}}\label{sec:coro:lipregugeneral}
\begin{proof}
Fix $t$ such that $\sigma_t>0$ and let $\pi^{t,x}$ denote the posterior law of $Z$ given $X_t=x$.
Set
\[
\Sigma_t(x)  :=  \text{Cov}_{\pi^{t,x}} \big(m_t(Z)\big)\in\R^{d\times d},
\qquad 
\Gamma_t(x)  :=  \text{Cov}_{\pi^{t,x}} \big(m'_t(Z),  m_t(Z)\big)\in\R^{d\times d}.
\]
Then the identity used in the proof of Proposition~\ref{propnablav} rewrites as
\begin{equation}\label{eq:gradvt-decomp}
\nabla v_t(x)
= \frac{\sigma'_t}{\sigma_t}  I
-\frac{\sigma'_t}{\sigma_t^3}  \Sigma_t(x)
+\frac{1}{\sigma_t^2}  \Gamma_t(x)
\qquad\text{for a.e. }t\text{ and all }x\in\R^d.
\end{equation}

Taking operator norms in \eqref{eq:gradvt-decomp} and using the triangle inequality yields
\begin{equation}\label{eq:lip-basic}
\|\nabla v_t(x)\|_{\text{op}}
\le \frac{|\sigma'_t|}{\sigma_t}
+\frac{|\sigma'_t|}{\sigma_t^3}  \|\Sigma_t(x)\|_{\text{op}}
+\frac{1}{\sigma_t^2}  \|\Gamma_t(x)\|_{\text{op}}.
\end{equation}

For any unit vector $h\in\mathbb S^{d-1}$,
\[
h^\top \Sigma_t(x) h
=\text{Var}_{\pi^{t,x}} \big(\langle h,m_t(Z)\rangle\big)
=\int \big(\langle h,H^{p^{t,x}}_{\text{Id}}(y)\rangle\big)^2  dp^{t,x}(y),
\]
where $p^{t,x}$ is the law of $m_t(Z)$ under $\pi^{t,x}$.
Let $\nu^{t,x}$ be the log-concave surrogate posterior for which Brascamp-Lieb gives
\begin{equation}\label{eq:hfjidhjdkspzspmpm}
\int \big(\langle h,H^{\nu^{t,x}}_{\text{Id}}(y)\rangle\big)^2  d\nu^{t,x}(y)
  \le  (\alpha_t+\sigma_t^{-2})^{-1}. 
\end{equation}
Moreover Proposition~\ref{prop:finalestimatecompcov} controls the discrepancy between $p^{t,x}$ and $\nu^{t,x}$ at the level of second moments:
for $\xi=0$ it yields, uniformly in $x$ and $h$, 
\[
\Big|  h^\top \Big(\int H_{\text{Id}}^{p^{t,x}} \otimes H_{\text{Id}}^{p^{t,x}}  dp^{t,x}
-\int H_{\text{Id}}^{\nu^{t,x}} \otimes H_{\text{Id}}^{\nu^{t,x}}  d\nu^{t,x}\Big)h  \Big|
  \le  
C(\alpha_t+\sigma_t^{-2})^{-1}
\exp \Big(CK_t^2(\alpha_t+\sigma_t^{-2})^{-\beta}\Big).
\]
Combining with \eqref{eq:hfjidhjdkspzspmpm} and taking the supremum over $h$ gives the operator bound
\begin{equation}\label{eq:Sigma-bound}
\|\Sigma_t(x)\|_{\text{op}}
\le
C(\alpha_t+\sigma_t^{-2})^{-1}
\exp \Big(CK_t^2(\alpha_t+\sigma_t^{-2})^{-\beta}\Big)
=
C  \frac{\sigma_t^2}{1+\alpha_t\sigma_t^2}  
\exp \Big(CK_t^2(\alpha_t+\sigma_t^{-2})^{-\beta}\Big).
\end{equation}
For unit vectors $u,v\in\mathbb S^{d-1}$, Cauchy-Schwarz gives
\[
|u^\top \Gamma_t(x) v|
=\big|\text{Cov}_{\pi^{t,x}}(\langle u,m'_t(Z)\rangle,\langle v,m_t(Z)\rangle)\big|
\le
\sqrt{\text{Var}_{\pi^{t,x}}(\langle u,m'_t(Z)\rangle)}  
\sqrt{\text{Var}_{\pi^{t,x}}(\langle v,m_t(Z)\rangle)}.
\]
Assumption~\ref{assum:varmprime} exactly provides a uniform bound on the first variance term:
\[
\text{Var}_{\pi^{t,x}}(\langle u,m'_t(Z)\rangle)
\le
L_t^2  \frac{\sigma_t^2}{1+\alpha_t\sigma_t^2},
\qquad \forall x, \forall u\in\mathbb S^{d-1}. 
\]
The second variance term is controlled by \eqref{eq:Sigma-bound}. Taking suprema in $u,v$ yields
\begin{equation}\label{eq:Gamma-bound}
\|\Gamma_t(x)\|_{\text{op}}
\le
C  L_t  \frac{\sigma_t^2}{1+\alpha_t\sigma_t^2}  
\exp \Big(CK_t^2(\alpha_t+\sigma_t^{-2})^{-\beta}\Big).
\end{equation}
 
Plugging \eqref{eq:Sigma-bound} and \eqref{eq:Gamma-bound} into \eqref{eq:lip-basic} gives: for a.e. $t$ with $\sigma_t>0$,
\begin{equation}\label{eq:lip-final}
\sup_{x\in\R^d}\|\nabla v_t(x)\|_{\text{op}}
\le
\frac{|\sigma'_t|}{\sigma_t}
+
C  \frac{1}{1+\alpha_t\sigma_t^2}  
\exp \Big(CK_t^2(\alpha_t+\sigma_t^{-2})^{-\beta}\Big)\left(  \frac{|\sigma'_t|}{\sigma_t}
+L_t\right).
\end{equation}
\end{proof}

\subsubsection{A refined variant of Corollary~\ref{coro:lipregugeneral}}\label{sec:cor:refined_one_sided_lip}
\begin{corollary}
\label{cor:refined_one_sided_lip}
Let $v:[0,1]\times \mathbb{R}^d\rightarrow \mathbb{R}^d$ of the form \eqref{eq:vf} with $\sigma\in C^1_{\text{a.e.}}([0,1],[0,1])$ and $(m_t)_t$ satisfying Assumptions~\ref{assum:1} and \ref{assum:varmprime}.
Assume moreover that, in the weak log-concavity decomposition
\[
(m_t)_\#\pi(x)=\exp\bigl(-u_t(x)+a_t(x)\bigr),
\]
the convex part $u_t$ satisfies for some $A_t> 0$
\[
\nabla^2 u_t(x)\preceq A_t  \text{Id} \qquad\text{for all }x\in\R^d.
\]
Then there exists a constant $C>0$ independent of the dimension such that for every
$t\in[0,1]$ with $\sigma_t>0$ and every $\xi\in\{0,1\}$,
\[
\sup_{x\in\R^d}\|\nabla v_t(x)\|_{\text{op}}\le  
\frac{1}{1+\alpha_t\sigma_t^2}\Biggl(
C\exp \Bigl(CK_t^2(\alpha_t+\sigma_t^{-2})^{-\beta}\Bigr)
\Bigl(L_t+\frac{|\sigma_t'|}{\sigma_t}(K_t\sigma_t^\beta)^\xi\Bigr)
  +  \sigma_t|\sigma_t'| A_t
\Biggr).
\]
\end{corollary}
Compared to Corollary~\ref{coro:lipregugeneral}, this estimate improves the dependence on the explicit Gaussian term:
 the contribution $\frac{|\sigma_t'|}{\sigma_t}$ disappears and it is replaced by the finer term
    $
    \frac{|\sigma_t'|}{\sigma_t}(K_t\sigma_t^\beta)^\xi$
    which can be much smaller when the perturbation is small at the noise scale.
The additional term $\sigma_t|\sigma_t'|A_t$ reflects the upper curvature bound on the convex part $u_t$, and it comes from a direct covariance estimate on the log-concave comparison posterior.
\begin{proof}We have from Proposition~\ref{propnablav}
\begin{align*}
\|\nabla v_t(x)\|_{\text{op}}&=\sup_{h,h'\in \mathbb{S}^{d-1}}h^\top \nabla v_t(x)h'\\
&=\sup_{h,h'\in \mathbb{S}^{d-1}}h^\top \left(\frac{\sigma_t'}{\sigma_t} \text{Id}-\frac{\sigma_t'}{\sigma_t^3}\int H_{\pi^{t,x}}^{m_t}(z)^{\otimes 2}d\pi^{t,x}(z)+\frac{1}{\sigma_t^2}\int H_{\pi^{t,x}}^{m_t' }(z)\otimes H_{\pi^{t,x}}^{m_t}(z)d\pi^{t,x}(z)\right)h'\\
&\leq \max\{A_1,-A_2\}
\end{align*}
for 
$$A_1:=\lambda_{\max} \left(
\frac{\sigma_t'}{\sigma_t}\text{Id}
-\frac{\sigma_t'}{\sigma_t^3}\int (H_{\text{Id}}^{p^{t,x}})^{\otimes 2}  dp^{t,x}
\right)
+\frac{1}{\sigma_t^2}\Bigl\|
\int H^{m_t'}_{\pi^{t,x}}\otimes H^{m_t}_{\pi^{t,x}}  d\pi^{t,x}
\Bigr\|_{\text{op}}$$
and
$$A_2:=\lambda_{\min} \left(
\frac{\sigma_t'}{\sigma_t}\text{Id}
-\frac{\sigma_t'}{\sigma_t^3}\int (H_{\text{Id}}^{p^{t,x}})^{\otimes 2}  dp^{t,x}
\right)
-\frac{1}{\sigma_t^2}\Bigl\|
\int H^{m_t'}_{\pi^{t,x}}\otimes H^{m_t}_{\pi^{t,x}}  d\pi^{t,x}
\Bigr\|_{\text{op}}.$$

First, by the proof of Theorem~\ref{theo:onesidedlip2} we have that
$$A_1\leq \frac{1}{1+\alpha_t\sigma_t^2}\Biggl(
C\exp \Bigl(CK_t^2(\alpha_t+\sigma_t^{-2})^{-\beta}\Bigr)
\Bigl(L_t+\frac{|\sigma_t'|}{\sigma_t}(K_t\sigma_t^\beta)^\xi\Bigr)
  +  \sigma_t|\sigma_t'| A_t
\Biggr),$$
so we only need to show the bound on the $A_2$.
Now, fix $t$ such that $\sigma_t>0$ and let $\pi^{t,x}$ be the posterior law of $Z$ given $X_t=x$.
We have
\begin{align*}
&\lambda_{\min} \left(
\frac{\sigma_t'}{\sigma_t}\text{Id}
-\frac{\sigma_t'}{\sigma_t^3}\int (H_{\text{Id}}^{p^{t,x}})^{\otimes 2}  dp^{t,x}
\right)
\\
&\ge
-\frac{|\sigma_t'|}{\sigma_t^3}  
\|
\int (H_{\text{Id}}^{p^{t,x}})^{\otimes 2}  dp^{t,x}
-\int (H_{\text{Id}}^{\nu^{t,x}})^{\otimes 2}  d\nu^{t,x}
\|_{\text{op}}
+
\lambda_{\min} \left(
-\frac{\sigma_t'}{\sigma_t^3}\int (H_{\text{Id}}^{\nu^{t,x}})^{\otimes 2}  d\nu^{t,x}
\right) 
+\frac{\sigma_t'}{\sigma_t}.
\end{align*}
By Proposition~\ref{prop:finalestimatecompcov} 
\[
\|
\int (H_{\text{Id}}^{p^{t,x}})^{\otimes 2}  dp^{t,x}
-\int (H_{\text{Id}}^{\nu^{t,x}})^{\otimes 2}  d\nu^{t,x}
\|_{\text{op}}
\le
C K_t^\xi(\alpha_t+\sigma_t^{-2})^{-(2+\beta\xi)/2}
\exp \Bigl(CK_t^2(\alpha_t+\sigma_t^{-2})^{-\beta}\Bigr),
\]
so
\[
\frac{|\sigma_t'|}{\sigma_t^3}  
\|
\int (H_{\text{Id}}^{p^{t,x}})^{\otimes 2}  dp^{t,x}
-\int (H_{\text{Id}}^{\nu^{t,x}})^{\otimes 2}  d\nu^{t,x}
\|_{\text{op}}
\le
\frac{C}{1+\alpha_t\sigma_t^2}  
\frac{|\sigma_t'|}{\sigma_t}(K_t\sigma_t^\beta)^\xi  
\exp \Bigl(CK_t^2(\alpha_t+\sigma_t^{-2})^{-\beta}\Bigr).
\]
Moreover, since $\alpha_t\text{Id}\preceq\nabla^2u_t\preceq A_t\text{Id}$, we have
\begin{itemize}
    \item if $\sigma_t'\geq 0$,  using Brascamp-Lieb inequality
\begin{align*}
\lambda_{\min} \left(
-\frac{\sigma_t'}{\sigma_t^3}\int (H_{\text{Id}}^{\nu^{t,x}})^{\otimes 2}  d\nu^{t,x}\right)
+\frac{\sigma_t'}{\sigma_t} = &-\frac{\sigma_t'}{\sigma_t^3}\lambda_{\max} \left(\int (H_{\text{Id}}^{\nu^{t,x}})^{\otimes 2}  d\nu^{t,x}
\right)
+\frac{\sigma_t'}{\sigma_t}\\
&\ge
\frac{\sigma_t\sigma_t' \alpha_t}{1+\alpha_t\sigma_t^2},
\end{align*}    \item if $\sigma_t'< 0$, using Cramer-Rao inequality 
\begin{align*}
\lambda_{\min} \left(-
\frac{\sigma_t'}{\sigma_t^3}\int (H_{\text{Id}}^{\nu^{t,x}})^{\otimes 2}  d\nu^{t,x}\right)
+\frac{\sigma_t'}{\sigma_t} = &-\frac{\sigma_t'}{\sigma_t^3}\lambda_{\min} \left(\int (H_{\text{Id}}^{\nu^{t,x}})^{\otimes 2}  d\nu^{t,x}
\right)
+\frac{\sigma_t'}{\sigma_t}\\
&\ge
\frac{\sigma_t\sigma_t' A_t}{1+A_t\sigma_t^2}.
\end{align*}
\end{itemize}
Combining the last three displays yields
\begin{equation}
\label{eq:cov_part_final}
\lambda_{\min} \left(
\frac{\sigma_t'}{\sigma_t}\text{Id}
-\frac{\sigma_t'}{\sigma_t^3}\int (H_{\text{Id}}^{p^{t,x}})^{\otimes 2}  dp^{t,x}
\right)
\ge 
\frac{-1}{1+\alpha_t\sigma_t^2}\Biggl(
C\exp \Bigl(CK_t^2(\alpha_t+\sigma_t^{-2})^{-\beta}\Bigr)  
\frac{|\sigma_t'|}{\sigma_t}(K_t\sigma_t^\beta)^\xi
+\sigma_t|\sigma_t'| A_t
\Biggr).
\end{equation}

Let us now control the mixed covariance term.
Set
\[
\Gamma_t(x)
:=
\int H^{m_t'}_{\pi^{t,x}}(z)\otimes H^{m_t}_{\pi^{t,x}}(z)  d\pi^{t,x}(z).
\]
For unit vectors $h,h'\in\mathbb S^{d-1}$, Cauchy-Schwarz gives
\[
|h^\top\Gamma_t(x)h'|
\le
\sqrt{\text{Var}_{\pi^{t,x}} \bigl(\langle h,m_t'(Z)\rangle\bigr)}  
\sqrt{\text{Var}_{\pi^{t,x}} \bigl(\langle h',m_t(Z)\rangle\bigr)}.
\]
By Assumption~\ref{assum:varmprime},
$\text{Var}_{\pi^{t,x}}(\langle h,m_t'(Z)\rangle)\le L_t^2  \sigma_t^2/(1+\alpha_t\sigma_t^2)$.
On the other hand, the same argument as in the proof of Corollary~\ref{coro:lipregugeneral} (Brascamp--Lieb on $\nu^{t,x}$
plus Proposition~\ref{prop:finalestimatecompcov} with $\xi=0$) yields
\[
\sup_{x\in\R^d}\sup_{h'\in\mathbb S^{d-1}}
\text{Var}_{\pi^{t,x}} \bigl(\langle h',m_t(Z)\rangle\bigr)
\le
C  \frac{\sigma_t^2}{1+\alpha_t\sigma_t^2}  
\exp \Bigl(CK_t^2(\alpha_t+\sigma_t^{-2})^{-\beta}\Bigr).
\]
Taking suprema over $h,h'$ gives
\[
\|\Gamma_t(x)\|_{\text{op}}
\le
C  L_t  \frac{\sigma_t^2}{1+\alpha_t\sigma_t^2}  
\exp \Bigl(CK_t^2(\alpha_t+\sigma_t^{-2})^{-\beta}\Bigr),
\]
hence
\begin{equation}
\label{eq:gamma_part_final}
\frac{1}{\sigma_t^2}\|\Gamma_t(x)\|_{\text{op}}
\le
\frac{C}{1+\alpha_t\sigma_t^2}  
L_t  
\exp \Bigl(CK_t^2(\alpha_t+\sigma_t^{-2})^{-\beta}\Bigr).
\end{equation}

\medskip
Finally, combining \eqref{eq:cov_part_final} and \eqref{eq:gamma_part_final} gives the lower bound on $A_2$.
\end{proof}

\subsubsection{Proof of Corollary~\ref{cor:lipman-global}}\label{sec:cor:lipman-global}
\begin{proof}
By Corollary~\ref{coro:lipregugeneral}, for every $t$ such that $\sigma_t>0$,
\begin{equation}\label{eq:lip-bound-cor4}
\sup_{x\in\R^d}\|\nabla v_t(x)\|_{\text{op}}
\le
\frac{|\sigma_t'|}{\sigma_t}
+\frac{C_0}{1+\alpha_t\sigma_t^2}
\exp \Big(CK_t^2(\alpha_t+\sigma_t^{-2})^{-\beta}\Big)
\Big(\frac{|\sigma_t'|}{\sigma_t}+L_t\Big).
\end{equation}
Under the Lipman reduced model, $(m_t)_\#\pi$ is $(\alpha_t,\beta,K_t)$-weakly log-concave with
\[
\alpha_t=\frac{\alpha}{f_t^2},\qquad K_t=\frac{K}{f_t^\beta},\qquad L_t=\frac{|f_t'|}{f_t}.
\]
Moreover, since $\alpha_t+\sigma_t^{-2}\ge \alpha_t$, we have for a.e. $t\in(0,1)$,
\[
K_t^2(\alpha_t+\sigma_t^{-2})^{-\beta}\le K_t^2\alpha_t^{-\beta}
=\frac{K^2}{f_t^{2\beta}}\Big(\frac{f_t^2}{\alpha}\Big)^\beta
=K^2\alpha^{-\beta},
\]
hence the exponential factor in \eqref{eq:lip-bound-cor4} is uniformly bounded on $(0,1)$. Furthermore, by Assumption~\ref{assum:lipman} we have that there exists $\gamma  \in (0,1)$ such that $\gamma  \leq \sigma_{\frac{1}{2}}\wedge f_{\frac{1}{2}}$ and $f$ is increasing and $\sigma$ is decreasing so
\begin{align*}
    \frac{1}{1+\alpha_t\sigma_t^2}L_t
=\frac{|f_t'|  f_t}{f_t^2+\alpha\sigma_t^2}\leq \frac{|f_t'|  f_t}{\gamma^2\wedge(\alpha\gamma^2)}\leq C f_t'f_t\leq \frac{CA}{1-t}.
\end{align*}
In the end, we have
\begin{equation}\label{eq:lip-bound-cor42}
\sup_{x\in\mathbb R^d}\|\nabla v_t(x)\|_{\mathrm{op}}
\le
\frac{|\sigma_t'|}{\sigma_t}
+\frac{C}{1+\alpha_t\sigma_t^2}\frac{|\sigma_t'|}{\sigma_t}+\frac{C}{1-t}.
\end{equation}

\noindent\textbf{Control on $t\in[t_0,1)$.}We have
\[
\frac{|\sigma_t'|}{\sigma_t}=\big|\partial_t\log\sigma_t\big|
=\Big|-\frac{p}{1-t}+\frac{\ell'(t)}{\ell(t)}\Big|
\le \frac{p}{1-t}+A^2 \leq \frac{C}{1-t},
\qquad t\in[t_0,1).
\]
Plugging this  into \eqref{eq:lip-bound-cor42}
yields
\[
\sup_{x\in\R^d}\|\nabla v_t(x)\|_{\text{op}} \le \frac{C}{1-t},
\qquad \text{for a.e. }t\in[t_0,1).
\]
\noindent\textbf{Step 2: $t\in(0,t_0]$.}
By assumption, $\sigma_t\ge A^{-1}$ and $|\sigma_t'|\le A$ on $(0,t_0]$, hence
\[
\frac{|\sigma_t'|}{\sigma_t}\le C,\qquad t\in(0,t_0], 
\]
which combined with \eqref{eq:lip-bound-cor42} gives the result.
\end{proof}

\subsubsection{Proof of Corollary~\ref{coro:global-Lipschitz-v}}\label{sec:coro:global-Lipschitz-v}
\begin{proof}
We apply Corollary~\ref{coro:lipregugeneral} to the reduced model $X_t=f_tY+\bar g_t\xi$. In this setting one may take
\[
\alpha_t=\frac{\alpha}{f_t^2},\qquad K_t=\frac{K}{f_t^\beta},\qquad L_t=\frac{|f_t'|}{f_t},
\]
so that Corollary~\ref{coro:lipregugeneral} yields, for a.e. $t\in(0,1)$ with $\bar g_t>0$,
\begin{equation}\label{eq:cor4-bound}
\sup_{x\in\mathbb R^d}\|\nabla v_t(x)\|_{\mathrm{op}}
\le
\frac{|\bar g_t'|}{\bar g_t}
+\frac{C}{1+\alpha_t\bar g_t^2}\exp \Big(C_0K_t^2(\alpha_t+\bar g_t^{-2})^{-\beta}\Big)
\Big(\frac{|\bar g_t'|}{\bar g_t}+L_t\Big).
\end{equation}
First, by Assumption~\ref{assum:sto-int} we have $g_t+f_t\geq \gamma$ so
$$\frac{1}{1+\alpha_t\bar g_t^2}  L_t
=\frac{|f_t'|  f_t}{f_t^2+\alpha \bar g_t^2}\leq C\frac{A}{1-t}.$$
Next, the exponential factor in \eqref{eq:cor4-bound} is uniformly bounded on $(0,1)$:
since $\alpha_t+\bar g_t^{-2}\ge \alpha_t$, we have
\[
K_t^2(\alpha_t+\bar g_t^{-2})^{-\beta}\le K_t^2\alpha_t^{-\beta}
=\frac{K^2}{f_t^{2\beta}}\Big(\frac{f_t^2}{\alpha}\Big)^\beta=\frac{K^2}{\alpha^\beta},
\]
hence
\[
\exp \Big(C_0K_t^2(\alpha_t+\bar g_t^{-2})^{-\beta}\Big)\le
C.
\]
In the end, we have
\begin{equation}\label{eq:cor4-bound2}
\sup_{x\in\mathbb R^d}\|\nabla v_t(x)\|_{\mathrm{op}}
\le
\frac{|\bar g_t'|}{\bar g_t}
+\frac{C}{1+\alpha_t\bar g_t^2}\frac{|\bar g_t'|}{\bar g_t}+\frac{C}{1-t}.
\end{equation}
\medskip
\noindent\textbf{Control on $[t_0,1)$.}
By assumption, there exist $t_0\in(0,1)$, $p>0$ and a positive function $\ell$ on $[t_0,1)$ such that
\[
\bar g_t=(1-t)^p  \ell(t),\qquad A^{-1}\le \ell(t)\le A,\qquad \|\ell'\|_{L^\infty([t_0,1))}\le A.
\]
Hence
\[
\frac{|\bar g_t'|}{\bar g_t}
=\big|\partial_t\log \bar g_t\big|
=\left|-\frac{p}{1-t}+\frac{\ell'(t)}{\ell(t)}\right|
\le \frac{p}{1-t}+A^2,
\qquad t\in[t_0,1).
\]
Moreover, since $f$ is increasing with $f_1=1$, we have $f_t\ge f_{t_0}>0$ on
$[t_0,1)$. Then $K_t$ and the exponential factor in
\eqref{eq:cor4-bound2} are bounded so we deduce that 
\begin{equation}\label{eq:near-1}
\sup_{x}\|\nabla v_t(x)\|_{\mathrm{op}}\le \frac{C}{1-t},
\qquad \text{for a.e. }t\in[t_0,1).
\end{equation}

\medskip
\noindent\textbf{Control on $(0,t_0]$.}
Next, by Assumption~\ref{assum:sto-int}, there is $\delta>0$ and $\gamma>0$ such that
$g_t\ge \gamma$ for $t\in(0,\delta)$, and $\sigma$ is increasing on $[0,\delta)$ then decreasing on $(\delta,1)$ with
$\sigma_0=\sigma_1=0$. In particular, for $t\in[\delta,t_0]$ one has $\sigma_t\ge \sigma_{t_0}>0$, and therefore
\[
\bar g_t=\sqrt{g_t^2+\sigma_t^2}\ge \min\{\gamma,\sigma_{t_0}\}\geq c^{-1},
\qquad \forall t\in(0,t_0].
\]
 Plugging this into \eqref{eq:cor4-bound2} yields for $t\in(0,t_0]$ that
\[
\sup_x\|\nabla v_t(x)\|_{\mathrm{op}}
\le C|\bar g_t'|+\frac{C}{1-t}.
\]
Finally, since for all $t\in (0,t_0)$, $|\bar g_t'|\leq A $ we get the result.
\end{proof}

\subsubsection{Proof of Corollary~\ref{coro:difflipbound}}\label{sec:coro:difflipbound}
\begin{proof}
By the closed-form identity for the probability-flow drift in the diffusion setting,
\[
v_t(x)=\frac1t\bigl(x+s(t,x)\bigr).
\]
Moreover, the reduced model $X_t=f_tY+\sigma_t\xi$ with $f_t=t$ and $\sigma_t=\sqrt{1-t^2}$ fits the
Lipman flow-matching framework, and the schedule assumptions of Corollary~\ref{cor:lipman-global} are satisfied.
Therefore Corollary~\ref{cor:lipman-global} yields
\[
\sup_{x\in\R^d}\|\nabla v_t(x)\|_{ \text{op}} \le \frac{C}{1-t}.
\]
\end{proof}

\subsection{Details for time Lipschitz estimates}

\subsubsection{Proof of Theorem~\ref{theo:partialtvt}}\label{sec:theo:partialtvt}
Recall the reduced model of Theorem~\ref{theo:partialtvt}:
\[
X_t = f_t Y + \bar g_t \xi,
\]
where \(Y\) has density \(p(y)=e^{-u(y)+a(y)}\), \(\xi\sim\mathcal N(0,I_d)\), and \(\xi\) is independent of \(Y\). For \(x\in\R^d\), we denote the posterior mean and covariance
\[
\mu_t(x):=\mathbb{E}[Y\mid X_t=x],\qquad \Sigma_t(x):=\text{Cov}(Y\mid X_t=x),
\]
and we write \(\rho^{t,x}\) for the posterior law of \(Y\) given \(X_t=x\), i.e.
\begin{equation}\label{eq:posterior_rho}
\rho^{t,x}(d y)=\frac{p(y)  \varphi^{x,\bar g_t}(f_t y)}{p_t(x)}  d y,
\qquad
\varphi^{x,\sigma}(z):=(2\pi\sigma^2)^{-d/2}\exp \Big(-\frac{\|x-z\|^2}{2\sigma^2}\Big).
\end{equation}
A first key point is that, in the reduced model, $\dot X_t$ is affine in $(Y,\xi)$, and therefore $v_t(x)=\E[\dot X_t\mid X_t=x]$ is an affine combination of $x$ and the posterior mean $\mu_t(x)$. Differentiating this representation in time yields the following decomposition.

\begin{proposition}\label{prop:time_derivative_general}
Assume \(f,\bar g\in W^{2,1}_{\text{loc}}(0,1)\) and \(\bar g_t>0\).
Then for a.e. \(t\in(0,1)\) at which \(x\mapsto \mu_t(x)\) is well-defined and \(t\mapsto\mu_t(x)\) is differentiable,
\begin{equation}\label{eq:dt_vt_decomposition}
\partial_t v_t(x)
=
\partial_t \Big(\frac{\bar g_t'}{\bar g_t}\Big)x
+
\partial_t \Big(f_t'-\frac{\bar g_t'}{\bar g_t}f_t\Big)\mu_t(x)
+
\Big(f_t'-\frac{\bar g_t'}{\bar g_t}f_t\Big)\partial_t\mu_t(x).
\end{equation}
\end{proposition}
The proof of Proposition~\ref{prop:time_derivative_general} can be found in Section~\ref{sec:prop:time_derivative_general}. We next control $\partial_t\mu_t(x)$.
The identity \eqref{eq:dtmu-identity} is a conditional-differentiation (Fisher-type) formula: time variations of the posterior mean come from the time variations of the likelihood parameters $(f_t,\bar g_t)$.
The first term is the ``Gaussian-linear'' contribution: a covariance (hence of size $\gamma_t^{-1}$) multiplied by $x$.
The second term is a nonlinear correction involving the sensitivity of the posterior second moment $r_t(x)$.

\begin{proposition}\label{prop:estimatepartialtmut}
Let $p:\mathbb{R}^d\rightarrow \mathbb{R}$ be an $(\alpha,\beta,K)$ weakly log-concave probability density. Define
\[
\mu_t(x):=\mathbb E[Y\mid X_t=x],\qquad \Sigma_t(x):=\mathrm{Cov}(Y\mid X_t=x),
\qquad r_t(x):=\mathbb E[\|Y\|^2\mid X_t=x].
\]
Set
\[
a_t:=\frac{\bar g_t'}{\bar g_t},\qquad c_t:= f_t' - a_t f_t,
\qquad \gamma_t := \alpha + \frac{f_t^2}{\bar g_t^2}.
\]
Then for a.e. $t\in I$ and all $x\in\mathbb{R}^d$,
\begin{equation}\label{eq:dtmu-identity}
\partial_t \mu_t(x)
=
\frac{f_t' - 2a_t f_t}{\bar g_t^2}  \Sigma_t(x)  x
 - 
c_t  \nabla r_t(x).
\end{equation}
Moreover, there exists a constant $C>0$ independent of the dimension such that for all $t\in I$ and all $x\in\mathbb{R}^d$,
\begin{equation}\label{eq:Sigma-bound2}
\|\Sigma_t(x)\|_{\mathrm{op}}
 \le 
\frac{C}{\gamma_t}  
\exp \Big(2K + \frac{C K^2}{\gamma_t}\Big)\leq C_2.
\end{equation}
Consequently, for a.e. $t\in I$ and all $x\in\mathbb{R}^d$,
\begin{equation}\label{eq:dtmu-bound}
\|\partial_t \mu_t(x)\|
 \le 
C  
\frac{|f_t' - 2a_t f_t|}{\alpha \bar g_t^2 + f_t^2}  
 \|x\|
 + 
|c_t|  \|\nabla r_t(x)\|.
\end{equation}
\end{proposition}

The proof of Proposition~\ref{prop:estimatepartialtmut} can be found in Section~\ref{sec:prop:estimatepartialtmut}. The remaining task is therefore to control $\nabla r_t(x)$.
Since $r_t(x)=\E[\|Y\|^2\mid X_t=x]$ is a posterior expectation, differentiating in $x$ naturally produces a covariance with the score of the Gaussian likelihood.
This turns $\nabla r_t(x)$ into a third-order posterior quantity, expressed by the following identity.

\begin{proposition}\label{prop:grad-rt}
Fix $t$ such that $f_t>0$ and $\bar g_t>0$ and assume $r_t(x)<\infty$.
Then, at every $x$ where $r_t$ is differentiable,
\begin{equation*}
\nabla r_t(x)=\frac{f_t}{\bar g_t^{  2}}  \text{Cov}_{\rho^{t,x}} \bigl(Y,\|Y\|^2\bigr).
\end{equation*}
\end{proposition}

The proof of Proposition~\ref{prop:grad-rt} can be found in Section~\ref{sec:prop:grad-rt}. To bound this covariance, it is useful to compare the true posterior $\rho^{t,x}$ (which is only weakly log-concave) to the log-concave surrogate posterior $\nu^{t,x}$ obtained by removing the perturbation from the prior.
The surrogate $\nu^{t,x}$ is genuinely strongly log-concave, with curvature $\gamma_t$, so it enjoys dimension-free concentration and Lipschitz dependence on $x$.
The next proposition shows that these properties are stable under the weak perturbation, and yields the desired affine bound on $\nabla r_t(x)$.

\begin{proposition}\label{prop:rt-gradient-holder}
Assume that $p^\star$ is $(\alpha,\beta,K)$-weakly log-concave, let $
\nu^{t,x}(dy) \propto \exp \left(-u(y)-\frac{\|x-f_t y\|^2}{2\bar g_t^2}\right)dy$ 
and set the curvature parameter $
\gamma_t:=\alpha+\frac{f_t^2}{\bar g_t^2}.$
Then there $C>0$ independent of the dimension such that, at every $x$ where $r_t$ is differentiable,
\[
\|\nabla r_t(x)\| \le A_t + B_t\|x\|,
\]
with
\[
B_t:= C  \frac{f_t^2}{\bar g_t^4  \gamma_t^2},
\qquad
A_t:= C  \frac{f_t}{\bar g_t^2  \gamma_t}\left(\|\E_{\nu_{t,0}}[W]\|+\sqrt{\frac{d}{\gamma_t}}\right).
\]
\end{proposition}

The proof of Proposition~\ref{prop:rt-gradient-holder} can be found in Section~\ref{sec:prop:rt-gradient-holder}. Finally, we also need a direct affine bound on the posterior mean $\mu_t(x)$ itself.
The slope in $\|x\|$ again scales like $f_t/(\bar g_t^2\gamma_t)$, which corresponds to the sensitivity of a $\gamma_t$-strongly log-concave posterior mean to changes in the observation.

\begin{lemma}\label{lemma:bounfmut}
Let $Y$ be a random variable on $\mathbb R^d$ with density
$
p(y)= e^{-u(y)+a(y)},$ that is  $(\alpha,\beta,K)$-weakly log concave
and let $\xi\sim\mathcal N(0,I_d)$ be independent of $Y$. Set
\[
\mu_t(x):=\mathbb E[Y\mid X_t=x]=\mathbb E_{\rho^{t,x}}[Y].
\]
and the curvature parameter
\[
\gamma_t := \alpha + \frac{f_t^2}{\bar g_t^2}.
\]
Then, for every $x\in\mathbb R^d$,
\[
 \|\mu_t(x)\| \le C\Big(
\|E_{\nu_{t,0}}[W]\|  +  \sqrt{\tfrac{d}{\gamma_t}}
 +  \frac{f_t}{\bar g_t^2  \gamma_t}  \|x\|
\Big).
\]
\end{lemma}

The proof of Lemma~\ref{lemma:bounfmut} can be found in Section~\ref{sec:lemma:bounfmut}.

\paragraph{Concluding the Proof of Theorem~\ref{theo:partialtvt}}
\begin{proof}
By Proposition~\ref{prop:time_derivative_general}, for a.e. $t$ we have
\begin{equation}\label{eq:dt-v-triangle}
\|\partial_t v_t(x)\|
\le |a_t'|  \|x\| + |c_t'|  \|\mu_t(x)\| + |c_t|  \|\partial_t\mu_t(x)\|.
\end{equation}
Now, by Lemma~\ref{lemma:bounfmut}
\begin{equation}\label{eq:mu-bound}
\|\mu_t(x)\|
\le C\Big(\|\mathbb E_{\nu_{t,0}}[W]\|+\sqrt{\tfrac{d}{\gamma_t}}+\frac{f_t}{\bar g_t^2  \gamma_t}\|x\|\Big),
\end{equation}
by Proposition~\ref{prop:estimatepartialtmut} 
\begin{equation}\label{eq:dt-mu-bound}
\|\partial_t\mu_t(x)\|
\le C  \frac{|f_t'-2a_t f_t|}{\alpha \bar g_t^2+f_t^2}  \|x\|
 +  |c_t|  \|\nabla r_t(x)\|.
\end{equation}
and 
By Proposition~\ref{prop:rt-gradient-holder}
\begin{equation}\label{eq:grad-r-affine}
\|\nabla r_t(x)\|
\le C \frac{f_t}{\bar g_t^2  \gamma_t}
\Big(\|E_{\nu_{t,0}}[W]\|+\sqrt{\tfrac{d}{\gamma_t}}\Big)
 +  C  \frac{f_t^2}{\bar g_t^4  \gamma_t^2}  \|x\|.
\end{equation}

Now, inserting \eqref{eq:mu-bound} into \eqref{eq:dt-v-triangle}, and using \eqref{eq:dt-mu-bound} we get
\[
\|\partial_t v_t(x)\|
\le |a_t'|  \|x\| + |c_t'|  \|\mu_t(x)\|
+ |c_t|  C  \frac{|f_t'-2a_t f_t|}{\alpha \bar g_t^2+f_t^2}  \|x\|
+ |c_t|^2  \|\nabla r_t(x)\|.
\]
Now apply \eqref{eq:mu-bound} and \eqref{eq:grad-r-affine} to the last two terms involving
$\mu_t(x)$ and $\nabla r_t(x)$, and regroup the coefficients of $\|x\|$ and the $\|x\|$-independent
terms. This yields precisely
\[
\|\partial_t v_t(x)\|\le \mathcal A_t + \mathcal B_t  \|x\|,
\]
with $\mathcal A_t$ and $\mathcal B_t$ as stated.
\end{proof}

\subsubsection{Proof of Proposition~\ref{prop:time_derivative_general}}\label{sec:prop:time_derivative_general}
\begin{proof}
Fix $t\in(0,1)$ such that $f'_t,\bar g'_t$ exist, $\bar g_t>0$, and such that
$t\mapsto \mu_t(x):=\E[Y\mid X_t=x]$ is differentiable at $t$ for the given $x\in\R^d$. Since $X_t=f_tY+\bar g_t\xi$ and $f,\bar g$ depend only on $t$, we have
\[
\dot X_t = f'_t  Y + \bar g'_t  \xi,
\]
therefore by linearity of conditional expectation,
\[
v_t(x):=\E[\dot X_t\mid X_t=x]
= f'_t  \E[Y\mid X_t=x] + \bar g'_t  \E[\xi\mid X_t=x]
= f'_t  \mu_t(x) + \bar g'_t  \E[\xi\mid X_t=x].
\]
Using the model identity $X_t=f_tY+\bar g_t\xi$, we can rewrite
\[
\xi=\frac{X_t-f_tY}{\bar g_t}.
\]
Conditioning on $X_t=x$ gives
\[
\E[\xi\mid X_t=x]
=\E\Big[\frac{x-f_tY}{\bar g_t}  \Big|  X_t=x\Big]
=\frac{1}{\bar g_t}\Big(x-f_t  \E[Y\mid X_t=x]\Big)
=\frac{x-f_t\mu_t(x)}{\bar g_t}.
\]
Plugging this into the expression for $v_t$ gives
\[
v_t(x)
= f'_t  \mu_t(x)+\bar g'_t  \frac{x-f_t\mu_t(x)}{\bar g_t}
= \frac{\bar g'_t}{\bar g_t}  x
+\Big(f'_t-\frac{\bar g'_t}{\bar g_t}f_t\Big)\mu_t(x).
\]
Define
\[
a_t:=\frac{\bar g'_t}{\bar g_t},
\qquad
c_t:=f'_t-a_tf_t
= f'_t-\frac{\bar g'_t}{\bar g_t}f_t,
\]
so that $v_t(x)=a_t x + c_t\mu_t(x)$. Differentiating in $t$ yields
\[
\partial_t v_t(x)
= a'_t  x + c'_t  \mu_t(x) + c_t  \partial_t\mu_t(x).
\]
Finally, substituting back $a_t=\bar g'_t/\bar g_t$ and
$c_t=f'_t-(\bar g'_t/\bar g_t)f_t$ gives
\[
\partial_t v_t(x)
= \partial_t \Big(\frac{\bar g'_t}{\bar g_t}\Big)  x
+ \partial_t \Big(f'_t-\frac{\bar g'_t}{\bar g_t}f_t\Big)  \mu_t(x)
+ \Big(f'_t-\frac{\bar g'_t}{\bar g_t}f_t\Big)  \partial_t\mu_t(x),
\]
which is exactly the claimed identity.
\end{proof}

\subsubsection{Proof of Proposition~\ref{prop:estimatepartialtmut}}\label{sec:prop:estimatepartialtmut}

\begin{proof}
For fixed $t$ and $x$, the posterior $\rho^{t,x}$ satisfies the covariance identity
\begin{equation}\label{eq:cov-formula}
\partial_t\mu_t(x)=\mathrm{Cov}_{\rho^{t,x}} \Big(Y, \partial_t\log\varphi^{x,\bar g_t}(f_tY)\Big).
\end{equation}
A direct differentiation gives
\begin{equation}\label{eq:dtlogphi}
\partial_t\log\varphi^{x,\bar g_t}(f_tY)
=
\frac{f_t'}{\bar g_t^2}  \langle x-f_tY, Y\rangle
+
\frac{\bar g_t'}{\bar g_t^3}  \Big(\|x-f_tY\|^2-d\bar g_t^2\Big).
\end{equation}
Since $\mathrm{Cov}_{\rho^{t,x}}(Y,\text{constant})=0$, the term $-d\bar g_t^2$ in \eqref{eq:dtlogphi}
does not contribute. Expanding the remaining covariances yields
\[
\mathrm{Cov}_{\rho^{t,x}} \big(Y,\langle x-f_tY,Y\rangle\big)
=
\mathrm{Cov}_{\rho^{t,x}}(Y,\langle x,Y\rangle)
-
f_t  \mathrm{Cov}_{\rho^{t,x}}(Y,\|Y\|^2)
=
\Sigma_t(x)  x - f_t  \mathrm{Cov}_{\rho^{t,x}}(Y,\|Y\|^2),
\]
and similarly
\[
\mathrm{Cov}_{\rho^{t,x}} \big(Y,\|x-f_tY\|^2\big)
=
-2f_t  \Sigma_t(x)  x + f_t^2  \mathrm{Cov}_{\rho^{t,x}}(Y,\|Y\|^2),
\]
because the $\|x\|^2$ part is constant and vanishes under covariance.
Plugging these into \eqref{eq:cov-formula}--\eqref{eq:dtlogphi} gives
\[
\partial_t\mu_t(x)
=
\Big(\frac{f_t'}{\bar g_t^2}-\frac{2f_t\bar g_t'}{\bar g_t^3}\Big)\Sigma_t(x)x
+
\Big(-\frac{f_tf_t'}{\bar g_t^2}+\frac{f_t^2\bar g_t'}{\bar g_t^3}\Big)\mathrm{Cov}_{\rho^{t,x}}(Y,\|Y\|^2).
\]
It remains to identify $\mathrm{Cov}_{\rho^{t,x}}(Y,\|Y\|^2)$ with $\nabla r_t(x)$.
Write $r_t(x)=\int \|y\|^2 \rho^{t,x}(dy)$. Differentiating under the integral sign and using the quotient rule,
one obtains the identity:
\[
\nabla r_t(x)
=
\mathrm{Cov}_{\rho^{t,x}} \Big(\|Y\|^2, \nabla_x\log\varphi^{x,\bar g_t}(f_tY)\Big).
\]
But $\nabla_x\log\varphi^{x,\bar g_t}(f_tY)=\frac{f_tY-x}{\bar g_t^2}$, hence the $-x$ term drops in covariance and
\begin{equation}\label{eq:grad-r}
\nabla r_t(x)
=
\frac{f_t}{\bar g_t^2}  \mathrm{Cov}_{\rho^{t,x}}(Y,\|Y\|^2).
\end{equation}
Therefore $\mathrm{Cov}_{\rho^{t,x}}(Y,\|Y\|^2)=\frac{\bar g_t^2}{f_t}  \nabla r_t(x)$, and substituting this above yields
\[
\partial_t\mu_t(x)
=
\frac{f_t'-2(\bar g_t'/\bar g_t)f_t}{\bar g_t^2}  \Sigma_t(x)x
-
\Big(f_t'-(\bar g_t'/\bar g_t)f_t\Big)\nabla r_t(x),
\]
which is exactly \eqref{eq:dtmu-identity}. Now, fix $t\in I$, $x\in\mathbb{R}^d$, and a unit vector $h\in\mathbb{S}^{d-1}$.
Let $\nu^{t,x}$ be the log-concave surrogate posterior
\[
\nu^{t,x}(dy)\propto \exp \Big(-u(y)-\frac{\|x-f_ty\|^2}{2\bar g_t^2}\Big)  dy,
\]
so that $\rho^{t,x}(dy)\propto e^{a(y)}\nu^{t,x}(dy)$.
Define
\[
L(y):=\frac{e^{a(y)}}{\int e^{a(z)}  \nu^{t,x}(dz)}
\quad\text{so that}\quad
\mathbb E_{\nu^{t,x}}[L]=1
 \text{ and } 
\mathbb E_{\rho^{t,x}}[\cdot]=\mathbb E_{\nu^{t,x}}[L  \cdot].
\]
Let $F(y):=h^\top y$. Then, using $\mathrm{Var}_{\rho}(F)=\inf_{c\in\mathbb R}\mathbb E_\rho[(F-c)^2]$,
\[
\mathrm{Var}_{\rho^{t,x}}(F)\le \mathbb E_{\nu^{t,x}} \Big[L  (F-\mathbb E_{\nu^{t,x}}F)^2\Big]
\le
\big(\mathbb E_{\nu^{t,x}}[L^2]\big)^{1/2}  
\big(\mathbb E_{\nu^{t,x}}[(F-\mathbb E_{\nu^{t,x}}F)^4]\big)^{1/2}.
\]
Since $\nu^{t,x}$ is log-concave on $\mathbb R^d$, the one-dimensional law of
$F-\mathbb E_{\nu^{t,x}}F$ is log-concave on $\mathbb R$, and there is a universal constant $C>0$ such that
\begin{equation}\label{eq:4th-moment}
\big(\mathbb E_{\nu^{t,x}}[(F-\mathbb E_{\nu^{t,x}}F)^4]\big)^{1/2}
\le C  \mathrm{Var}_{\nu^{t,x}}(F).
\end{equation}
Moreover,
\[
\mathbb E_{\nu^{t,x}}[L^2]
=
\frac{\mathbb E_{\nu^{t,x}}[e^{2a(Y)}]}{\mathbb E_{\nu^{t,x}}[e^{a(Y)}]^2}
=
\frac{\mathbb E_{\nu^{t,x}}[e^{2(a(Y)-\mathbb E_{\nu^{t,x}}a(Y))}]}
{\mathbb E_{\nu^{t,x}}[e^{a(Y)-\mathbb E_{\nu^{t,x}}a(Y)}]^2}
\le
\mathbb E_{\nu^{t,x}} \big[e^{2(a(Y)-\mathbb E_{\nu^{t,x}}a(Y))}\big],
\]
using Jensen to get $\mathbb E_{\nu^{t,x}}[e^{a-\mathbb E a}]\ge 1$.
Combining these bounds, we obtain
\begin{equation}\label{eq:var-compare}
\mathrm{Var}_{\rho^{t,x}}(h^\top Y)
\le
C  \Big(\mathbb E_{\nu^{t,x}}e^{2(a(Y)-\mathbb E_{\nu^{t,x}}a(Y))}\Big)^{1/2}  
\mathrm{Var}_{\nu^{t,x}}(h^\top Y).
\end{equation}

The negative log-density of $\nu^{t,x}$ is
\[
V_{t,x}(y)=u(y)+\frac{\|x-f_ty\|^2}{2\bar g_t^2},
\]
so
\[
\nabla^2 V_{t,x}(y)=\nabla^2 u(y)+\frac{f_t^2}{\bar g_t^2}I_d \succeq
\Big(\alpha+\frac{f_t^2}{\bar g_t^2}\Big)I_d=\gamma_t \text{Id}
\quad\text{uniformly in $y$ and $x$.}
\]
The Brascamp--Lieb inequality applied to the linear function $y\mapsto h^\top y$ gives
\begin{equation}\label{eq:BL}
\mathrm{Var}_{\nu^{t,x}}(h^\top Y)\le \frac{1}{\gamma_t}.
\end{equation}

Now, applying Lemma~\ref{lem:mgf_holder} we have
\[
\sup_{x\in\mathbb R^d}\mathbb E_{\nu^{t,x}}
\exp \Big(r(a(Y)-\mathbb E_{\nu^{t,x}}a(Y))\Big)
\le
\exp \Big(2Kr+\frac{C K^2}{2\gamma_t}r^2\Big)
\qquad\forall r\ge 0,
\]
so taking $r=2$ and then square-roots yields
\begin{equation}\label{eq:exp-moment}
\Big(\mathbb E_{\nu^{t,x}}e^{2(a(Y)-\mathbb E_{\nu^{t,x}}a(Y))}\Big)^{1/2}
\le
\exp \Big(2K+\frac{C K^2}{\gamma_t}\Big)\leq C.
\end{equation}
Combining \eqref{eq:var-compare}, \eqref{eq:BL}, and \eqref{eq:exp-moment} gives, for every unit $h$,
\[
h^\top \Sigma_t(x) h
=
\mathrm{Var}_{\rho^{t,x}}(h^\top Y)
\le
\frac{C}{\gamma_t}.
\]
Taking the supremum over $h\in\mathbb S^{d-1}$ yields \eqref{eq:Sigma-bound2}.
Finally, insert \eqref{eq:Sigma-bound2} into \eqref{eq:dtmu-identity}:
\[
\Big\|\frac{f_t'-2a_t f_t}{\bar g_t^2}\Sigma_t(x)x\Big\|
\le
\frac{|f_t'-2a_t f_t|}{\bar g_t^2}  \|\Sigma_t(x)\|_{\mathrm{op}}  \|x\|
\le
C  \frac{|f_t'-2a_t f_t|}{\bar g_t^2  \gamma_t}\|x\|.
\]
\end{proof}

\subsubsection{Proof of Proposition~\ref{prop:grad-rt}}\label{sec:prop:grad-rt}
\begin{proof}
Write the posterior density up to normalization as
\[
\rho^{t,x}(dy) \propto p(y)\exp \Bigl(-\frac{\|x-f_t y\|^2}{2\bar g_t^{  2}}\Bigr)  dy.
\]
Let
\[
Z(x):=\int_{\R^d} p(y)\exp \Bigl(-\frac{\|x-f_t y\|^2}{2\bar g_t^{  2}}\Bigr)  dy,
\qquad
N(x):=\int_{\R^d} \|y\|^2  p(y)\exp \Bigl(-\frac{\|x-f_t y\|^2}{2\bar g_t^{  2}}\Bigr)  dy,
\]
so that $r_t(x)=N(x)/Z(x)$. Then,
\[
\nabla_x N(x)=\int \|y\|^2  p(y)  \nabla_x\exp \Bigl(-\frac{\|x-f_t y\|^2}{2\bar g_t^{  2}}\Bigr)  dy
= -\frac1{\bar g_t^{  2}}\int \|y\|^2 (x-f_t y)  p(y)e^{-\frac{\|x-f_t y\|^2}{2\bar g_t^{  2}}}  dy.
\]
Hence
\[
\frac{\nabla_x N(x)}{Z(x)}
= -\frac1{\bar g_t^{  2}}\mathbb{E}_{\rho^{t,x}}\bigl[(x-f_tY)\|Y\|^2\bigr],\qquad
\nabla_x\log Z(x)=\frac{\nabla_x Z(x)}{Z(x)}=-\frac1{\bar g_t^{  2}}\mathbb{E}_{\rho^{t,x}}[x-f_tY].
\]
Therefore, 
\begin{align*}
\nabla r_t(x)
&= \frac{\nabla N(x)}{Z(x)} - r_t(x)\nabla\log Z(x) \\
&= -\frac1{\bar g_t^{  2}}\Bigl(\mathbb{E}_{\rho^{t,x}}[(x-f_tY)\|Y\|^2]-\mathbb{E}_{\rho^{t,x}}[x-f_tY] \mathbb{E}_{\rho^{t,x}}\|Y\|^2\Bigr)\\
&= -\frac1{\bar g_t^{  2}}\Bigl(x r_t(x)-f_t\mathbb{E}_{\rho^{t,x}}[Y\|Y\|^2]-(x-f_t\mu_t(x))r_t(x)\Bigr)\\
&=\frac{f_t}{\bar g_t^{  2}}\Bigl(\mathbb{E}_{\rho^{t,x}}[Y\|Y\|^2]-\mu_t(x)  r_t(x)\Bigr)
=\frac{f_t}{\bar g_t^{  2}}\text{Cov}_{\rho^{t,x}}(Y,\|Y\|^2),
\end{align*}
which gives the result.
\end{proof}

\subsubsection{Proof of Proposition~\ref{prop:rt-gradient-holder}}\label{sec:prop:rt-gradient-holder}
\begin{proof}
Fix $x$ where $r_t$ is differentiable and let $Y\sim \rho^{t,x}$.
By Proposition~\ref{prop:grad-rt}
\[
\nabla r_t(x)=\frac{f_t}{\bar g_t^2}  \text{Cov}_{\rho^{t,x}}\big(Y,\|Y\|^2\big),
\]
Hence for $h\in\mathbb S^{d-1}$, we have
\[
h^\top \nabla r_t(x)=\frac{f_t}{\bar g_t^2}  \text{Cov}_{\rho^{t,x}}\big(h^\top Y,\|Y\|^2\big).
\]
Applying Cauchy-Schwarz we deduce that
\begin{equation}\label{eq:grad-rt-basic}
\|\nabla r_t(x)\|
\le
\frac{f_t}{\bar g_t^2}  
\sqrt{\|\Sigma_t(x)\|_{\text{op}}}  
\sqrt{\text{Var}_{\rho^{t,x}}(\|Y\|^2)}.
\end{equation}

Write $\rho^{t,x}(dy)\propto e^{a(y)}\nu^{t,x}(dy)$ and set
\[
L(y):=\frac{e^{a(y)}}{\int e^{a}  d\nu^{t,x}},
\qquad\text{so that}\qquad \E_{\nu^{t,x}}[L]=1, \E_{\rho^{t,x}}[\cdot]=\E_{\nu^{t,x}}[L  \cdot].
\]
Let $F(y):=h^\top y$. Using $\text{Var}_{\rho}(F)=\inf_{c\in\mathbb R}\E_{\rho}[(F-c)^2]$ and choosing
$c=\E_{\nu^{t,x}}F$, one gets
\[
\text{Var}_{\rho^{t,x}}(F)\le \E_{\nu^{t,x}}\big[L(F-\E_{\nu^{t,x}}F)^2\big]
\le
\big(\E_{\nu^{t,x}}[L^2]\big)^{1/2}
\big(\E_{\nu^{t,x}}[(F-\E_{\nu^{t,x}}F)^4]\big)^{1/2}.
\]
Since $\nu^{t,x}$ is log-concave, the one-dimensional law of
$F-\E_{\nu^{t,x}}F$ is log-concave on $\mathbb R$, hence
\begin{equation}
\big(\E_{\nu^{t,x}}[(F-\E_{\nu^{t,x}}F)^4]\big)^{1/2}\le C  \text{Var}_{\nu^{t,x}}(F)
\end{equation}
for a universal constant $C>0$. Moreover
\[
\E_{\nu^{t,x}}[L^2]
=
\frac{\E_{\nu^{t,x}}[e^{2a(Y)}]}{\E_{\nu^{t,x}}[e^{a(Y)}]^2}
=
\frac{\E_{\nu^{t,x}}[e^{2(a(Y)-\E_{\nu^{t,x}}a(Y))}]}{\E_{\nu^{t,x}}[e^{a(Y)-\E_{\nu^{t,x}}a(Y)}]^2}
\le
\E_{\nu^{t,x}}e^{2(a(Y)-\E_{\nu^{t,x}}a(Y))},
\]
using Jensen in the denominator. Therefore,
\[
\text{Var}_{\rho^{t,x}}(h^\top Y)\le
C\Big(\E_{\nu^{t,x}}e^{2(a(Y)-\E_{\nu^{t,x}}a(Y))}\Big)^{1/2}\text{Var}_{\nu^{t,x}}(h^\top Y).
\]
Now $\nu^{t,x}$ has negative log-density
$V_{t,x}(y)=u(y)+\|x-f_ty\|^2/(2\bar g_t^2)$, hence
$\nabla^2 V_{t,x}\succeq \gamma_t \text{Id}$ uniformly, so Brascamp--Lieb gives
$\text{Var}_{\nu^{t,x}}(h^\top Y)\le 1/\gamma_t$ for $\|h\|=1$.
Combining and taking the supremum over $h$ yields
\begin{equation}\label{eq:Sigma-bound-Mt2}
\|\Sigma_t(x)\|_{\text{op}}\le C  \frac{1}{\gamma_t}\Big(\E_{\nu^{t,x}}e^{2(a-\E_{\nu^{t,x}}a)}\Big)^{1/2}
\le C  \frac{M_{t,2}^{1/2}}{\gamma_t}.
\end{equation}
With the same $L$ as above and $G(y):=\|y\|^2$, we again use
$\text{Var}_{\rho}(G)=\inf_c \E_{\rho}[(G-c)^2]$ and choose $c=\E_{\nu^{t,x}}G$, giving
\[
\text{Var}_{\rho^{t,x}}(G)\le \E_{\nu^{t,x}}\big[L(G-\E_{\nu^{t,x}}G)^2\big]
\le
\big(\E_{\nu^{t,x}}[L^2]\big)^{1/2}
\big(\E_{\nu^{t,x}}[(G-\E_{\nu^{t,x}}G)^4]\big)^{1/2}.
\]
As above, $(\E_{\nu^{t,x}}[L^2])^{1/2}\le (\E_{\nu^{t,x}}e^{2(a-\E a)})^{1/2}\le M_{t,2}^{1/2}$,
so it remains to bound the fourth centered moment of $G$ under $\nu^{t,x}$.

Let $U:=\|W\|$ with $W\sim \nu^{t,x}$, set $m:=\E_{\nu^{t,x}}U$ and $R:=U-m$.
Then $G-\E_{\nu^{t,x}}G = U^2-\E_{\nu^{t,x}}U^2 = 2mR + (R^2-\E R^2)$.
Using $(a+b)^4\le 8(a^4+b^4)$, we get
\[
\E_{\nu^{t,x}}[(G-\E G)^4]
\le
8\cdot (2m)^4  \mathbb{E}[R^4] + 8  \mathbb{E}[(R^2-\E R^2)^4].
\]
Since $\nu^{t,x}$ is $\gamma_t$-strongly log-concave and $y\mapsto \|y\|$ is $1$-Lipschitz,
Lemma~\ref{lem:lip-moments} gives
\[
\mathbb{E}[R^4]\le C  \gamma_t^{-2},
\qquad
\mathbb{E}[R^8]\le C  \gamma_t^{-4},
\qquad
\mathbb{E}[R^2]\le C  \gamma_t^{-1}.
\]
Moreover,
\[
\mathbb{E}[(R^2-\E R^2)^4]\le 8  \mathbb{E}[R^8] + 8  (\mathbb{E}[R^2])^4 \le C  \gamma_t^{-4}.
\]
Therefore,
\[
\E_{\nu^{t,x}}[(G-\E G)^4]\le C\left(\frac{m^4}{\gamma_t^2}+\frac{1}{\gamma_t^4}\right),
\qquad\text{hence}\qquad
\big(\E_{\nu^{t,x}}[(G-\E G)^4]\big)^{1/2}\le
C\left(\frac{m^2}{\gamma_t}+\frac{1}{\gamma_t^2}\right).
\]
Finally, $m^2=(\E U)^2\le \E U^2 = \E_{\nu^{t,x}}\|W\|^2$ by Jensen, so
\[
\big(\E_{\nu^{t,x}}[(G-\E G)^4]\big)^{1/2}\le
C\left(\frac{\E_{\nu^{t,x}}\|W\|^2}{\gamma_t}+\frac{1}{\gamma_t^2}\right).
\]
Plugging back yields
\begin{equation}\label{eq:VarY2-bound}
\text{Var}_{\rho^{t,x}}(\|W\|^2)
\le
C  M_{t,2}^{1/2}\left(\frac{\E_{\nu^{t,x}}\|W\|^2}{\gamma_t}+\frac{1}{\gamma_t^2}\right).
\end{equation}

Combining \eqref{eq:grad-rt-basic}, \eqref{eq:Sigma-bound-Mt2}, and \eqref{eq:VarY2-bound} gives
\[
\|\nabla r_t(x)\|
\le
\frac{f_t}{\bar g_t^2}  
\sqrt{C\frac{M_{t,2}^{1/2}}{\gamma_t}}  
\sqrt{C M_{t,2}^{1/2}\left(\frac{\E_{\nu^{t,x}}\|W\|^2}{\gamma_t}+\frac{1}{\gamma_t^2}\right)}
\le
C  M_{t,2}^{1/2}  \frac{f_t}{\bar g_t^2}\left(\frac{\sqrt{\E_{\nu^{t,x}}\|W\|^2}}{\gamma_t}
+\frac{1}{\gamma_t^{3/2}}\right).
\]
It remains to bound $\sqrt{\E_{\nu^{t,x}}\|W\|^2}$ by an affine function of $\|x\|$.
Let $\tilde\mu_t(x):=\E_{\nu^{t,x}}[W]$ and $\tilde\Sigma_t(x):=\text{Cov}_{\nu^{t,x}}(W)$.
By Brascamp--Lieb, $\tilde\Sigma_t(x)\preceq \gamma_t^{-1}I_d$, hence
\[
\E_{\nu^{t,x}}\|W\|^2 = \|\tilde\mu_t(x)\|^2 + \text{tr}(\tilde\Sigma_t(x))
\le \|\tilde\mu_t(x)\|^2 + \frac{d}{\gamma_t}.
\]
Moreover, differentiating $\tilde\mu_t$ with respect to $x$ yields the Jacobian identity
$\nabla_x \tilde\mu_t(x)=\frac{f_t}{\bar g_t^2}\tilde\Sigma_t(x)$, hence
$\|\nabla_x\tilde\mu_t(x)\|_{\text{op}}\le \frac{f_t}{\bar g_t^2\gamma_t}$ and therefore
\[
\|\tilde\mu_t(x)\|
\le
\|\tilde\mu_t(0)\| + \frac{f_t}{\bar g_t^2\gamma_t}\|x\|.
\]
Combining the last two displays and taking square-roots gives
\[
\sqrt{\E_{\nu^{t,x}}\|W\|^2}
\le
\|\tilde\mu_t(0)\| + \frac{f_t}{\bar g_t^2\gamma_t}\|x\| + \sqrt{\frac{d}{\gamma_t}}.
\]
Plugging this into the previous bound yields
\[
\|\nabla r_t(x)\|
\le
C  M_{t,2}^{1/2}  \frac{f_t}{\bar g_t^2\gamma_t}
\left(\|\tilde\mu_t(0)\|+\sqrt{\frac{d}{\gamma_t}}\right)
 + 
C  M_{t,2}^{1/2}  \frac{f_t^2}{\bar g_t^4\gamma_t^2}\|x\|,
\]
which is the claimed affine bound. Finally, applying Lemma~\ref{lem:mgf_holder} we have
\[
\sup_{x\in\mathbb R^d}\E_{\nu^{t,x}}\exp \Big(2\big(a(W)-\E_{\nu^{t,x}}a(W)\big)\Big)
\le
\exp \left(4K+\frac{C K^2}{\gamma_t}\right).
\]
\end{proof}

\subsubsection{Proof of Lemma~\ref{lemma:bounfmut}}\label{sec:lemma:bounfmut}

\begin{proof}
By Bayes' formula, the posterior density of $Y$ given $X_t=x$ is
\[
\rho^{t,x}(dy) \propto e^{-u(y)+a(y)}\exp \Big(-\tfrac{1}{2\bar g_t^2}\|x-f_t y\|^2\Big)  dy.
\]
Since $\nu^{t,x}(dy)\propto e^{-u(y)}\exp(-\|x-f_t y\|^2/(2\bar g_t^2))  dy$, we can write
\[
\rho^{t,x}(dy)  =  L_{t,x}(y)  \nu^{t,x}(dy),
\qquad
L_{t,x}(y):=\frac{e^{a(y)}}{\mathbb E_{\nu^{t,x}}[e^{a(Y)}]}.
\]
Therefore,
\[
\mu_t(x)=\mathbb E_{\rho^{t,x}}[Y] = \mathbb E_{\nu^{t,x}}[L_{t,x}(Y)  Y].
\]
Applying Cauchy-Schwarz gives
\begin{equation}\label{eq:CS-mu}
\|\mu_t(x)\|
\le \Big(\mathbb E_{\nu^{t,x}}[L_{t,x}(Y)^2]\Big)^{1/2}  
\Big(\mathbb E_{\nu^{t,x}}\|Y\|^2\Big)^{1/2}.
\end{equation}

Let $\bar a_x:=\mathbb E_{\nu^{t,x}}[a(Y)]$. Then
\[
\mathbb E_{\nu^{t,x}}[L_{t,x}(Y)^2]
= \frac{\mathbb E_{\nu^{t,x}}[e^{2a(Y)}]}{\mathbb E_{\nu^{t,x}}[e^{a(Y)}]^2}
= \frac{\mathbb E_{\nu^{t,x}}[e^{2(a(Y)-\bar a_x)}]}{\mathbb E_{\nu^{t,x}}[e^{(a(Y)-\bar a_x)}]^2}.
\]
By Jensen, $\mathbb E_{\nu^{t,x}}[e^{(a(Y)-\bar a_x)}]\ge e^{\mathbb E_{\nu^{t,x}}[a(Y)-\bar a_x]}=1$, hence applying Lemma~\ref{lem:mgf_holder}
\[
\mathbb E_{\nu^{t,x}}[L_{t,x}(Y)^2]
\le \mathbb E_{\nu^{t,x}}[e^{2(a(Y)-\bar a_x)}]
\le \exp \left(4K+\frac{C K^2}{\gamma_t}\right).
\]
Thus,
\begin{equation}\label{eq:L2-bound}
\Big(\mathbb E_{\nu^{t,x}}[L_{t,x}(Y)^2]\Big)^{1/2}\le \exp \left(2K+\frac{C K^2}{2\gamma_t}\right)\leq C.
\end{equation}

The density of $\nu^{t,x}$ is proportional to $\exp(-V_{t,x}(y))$ with
\[
V_{t,x}(y)=u(y)+\frac{1}{2\bar g_t^2}\|x-f_t y\|^2.
\]
Since $\nabla^2 u\succeq \alpha \text{Id}$ and $\nabla^2_y\frac{1}{2\bar g_t^2}\|x-f_t y\|^2
= \frac{f_t^2}{\bar g_t^2}I_d$, we have
\[
\nabla^2 V_{t,x}(y)\succeq \Big(\alpha+\frac{f_t^2}{\bar g_t^2}\Big)I_d=\gamma_t \text{Id},
\]
so $\nu^{t,x}$ is $\gamma_t$-strongly log-concave. By the Brascamp--Lieb inequality,
\[
\mathrm{Cov}_{\nu^{t,x}}(Y) \preceq \gamma_t^{-1}I_d
\quad\Longrightarrow\quad
\mathrm{tr}  \mathrm{Cov}_{\nu^{t,x}}(Y) \le \frac{d}{\gamma_t}.
\]
Hence, writing $\tilde\mu_t(x):=\mathbb E_{\nu^{t,x}}[Y]$
\begin{equation}\label{eq:second-moment}
\mathbb E_{\nu^{t,x}}\|Y\|^2
= \|\tilde\mu_t(x)\|^2 + \mathrm{tr}  \mathrm{Cov}_{\nu^{t,x}}(Y)
\le \|\tilde\mu_t(x)\|^2 + \frac{d}{\gamma_t}.
\end{equation}

Rewrite $\nu^{t,x}$ as an exponential family in $x$:
\[
\nu^{t,x}(dy) \propto \exp \Big(-u(y)-\frac{f_t^2}{2\bar g_t^2}\|y\|^2 + \frac{f_t}{\bar g_t^2}  x\cdot y\Big)  dy.
\]
Differentiating under the integral sign (justified since $\nu^{t,x}$ is strongly log-concave) yields the standard identity
\[
\nabla_x \tilde\mu_t(x)  =  \frac{f_t}{\bar g_t^2}  \mathrm{Cov}_{\nu^{t,x}}(Y).
\]
Therefore, using $\|\mathrm{Cov}_{\nu^{t,x}}(Y)\|_{\mathrm{op}}\le 1/\gamma_t$,
\[
\|\nabla_x \tilde\mu_t(x)\|_{\mathrm{op}}
\le \frac{f_t}{\bar g_t^2  \gamma_t},
\]
so $\tilde\mu_t$ is Lipschitz and
\begin{equation}\label{eq:Lip-tilde-mu}
\|\tilde\mu_t(x)\|\le \|\tilde\mu_t(0)\| + \frac{f_t}{\bar g_t^2  \gamma_t}  \|x\|.
\end{equation}

Combining \eqref{eq:second-moment}--\eqref{eq:Lip-tilde-mu} and using $\sqrt{A^2+B}\le |A|+\sqrt{B}$ gives
\[
\Big(\mathbb E_{\nu^{t,x}}\|Y\|^2\Big)^{1/2}
\le \|\tilde\mu_t(x)\| + \sqrt{\frac{d}{\gamma_t}}
\le \|\tilde\mu_t(0)\| + \sqrt{\frac{d}{\gamma_t}} + \frac{f_t}{\bar g_t^2  \gamma_t}  \|x\|.
\]
Together with \eqref{eq:CS-mu} and \eqref{eq:L2-bound}, this yields the claimed bound on $\|\mu_t(x)\|$.
\end{proof}

\subsubsection{Proof of Corollary~\ref{coro:dtv_rate}}\label{sec:coro:dtv_rate}

\begin{proof}
We apply Theorem~\ref{theo:partialtvt} in the reduced model $X_t=f_tY+\sigma_t\xi$ (so that $\bar g_t=\sigma_t$). Define
\[
a_t:=\frac{\sigma_t'}{\sigma_t},\qquad c_t:=f_t'-a_tf_t,\qquad
\gamma_t:=\alpha+\frac{f_t^2}{\sigma_t^2},\qquad
D_t:=\alpha\sigma_t^2+f_t^2=\sigma_t^2\gamma_t.
\]
Then Theorem~\ref{theo:partialtvt} yields, for all $x\in\mathbb R^d$ and for a.e. $t\in(0,1)$,
\[
\|\partial_t v_t(x)\|\le A_t+B_t\|x\|.
\]

\medskip\noindent
\textbf{Control on $[t_0,1)$.}
Since $f$ is increasing we have $f_t\ge f_{t_0}$ for all $t\in[t_0,1)$. Then on $[t_0,1)$,
\[
D_t\ge f_t^2\ge C^{-1},\qquad \frac1{D_t}\le C,\qquad \frac{f_t}{D_t}\le C,\qquad \frac{f_t^2}{D_t^2}\le C,
\]
so every ratio involving $D_t$ which appears in $A_t,B_t$ is uniformly bounded on $[t_0,1)$. 
Using $A^{-1}\le \ell\le A$ and $\|\ell'\|_\infty\le A$, we get $|a_t|\le C(1-t)^{-1}$ on $[t_0,1)$.
Differentiating again,
\[
a_t'=-\frac{p}{(1-t)^2}+\partial_t\Bigl(\frac{\ell'}{\ell}\Bigr),
\]
and $\|\ell'\|_\infty+\|\ell''\|_\infty\le A$, $A^{-1}\le \ell\le A$ give $|a_t'|\lesssim (1-t)^{-2}$ on $[t_0,1)$.
Using $0\le f_t\le 1$ and $\|f'\|_\infty\le A$ on $[t_0,1)$,
\[
|c_t|=|f_t'-a_tf_t|\le |f_t'|+|a_t||f_t|\lesssim 1+\frac1{1-t}\lesssim \frac1{1-t}.
\]
Moreover,
\[
c_t'=f_t''-a_t'f_t-a_tf_t',
\]
so $\|f''\|_\infty\le A$ and the bounds on $a_t,a_t'$ yield $|c_t'|\lesssim (1-t)^{-2}$ on $[t_0,1)$.
Finally,
\[
|f_t'-2a_tf_t|\le |f_t'|+2|a_t||f_t|\lesssim 1+\frac1{1-t}\lesssim \frac1{1-t}.
\]
Inspecting the expressions of Theorem~\ref{theo:partialtvt}, $A_t$ is a sum of terms of the form
\[
|c_t'|\times (\text{bounded ratios in }D_t)\cdot
\Bigl(\|\mathbb E_{\nu_{t,0}}[W]\|+\sqrt{d/\gamma_t}\Bigr)
\]
and
\[
|c_t|^2\times (\text{bounded ratios in }D_t)\cdot
\Bigl(\|\mathbb E_{\nu_{t,0}}[W]\|+\sqrt{d/\gamma_t}\Bigr),
\]
while $B_t$ is a sum of terms of the form
\[
|a_t'|+|c_t'| \times (\text{bounded ratios in }D_t)
+|c_t|\cdot |f_t'-2a_tf_t| \times (\text{bounded ratios in }D_t)
+|c_t|^2\times (\text{bounded ratios in }D_t).
\]
All ratios in $D_t$ are uniformly bounded on $[t_0,1)$ whereas
$|a_t'|,|c_t'|,|c_t|^2$, and $|c_t||f_t'-2a_tf_t|$ are all $O((1-t)^{-2})$. Therefore for a.e. $t\in[t_0,1)$,
\[
A_t\le \frac{C}{(1-t)^2}\Bigl(\|\mathbb E_{\nu_{t,0}}[W]\|+\sqrt{d/\gamma_t}\Bigr),
\qquad
B_t\le \frac{C}{(1-t)^2}.
\]
Using $\gamma_t\ge \alpha$ and Lemma~\ref{lem:mean_bound_nu_t0}, we have
\[
\|\mathbb E_{\nu_{t,0}}[W]\|\le \|\argmin u\|+\sqrt{\frac{d}{\alpha+\lambda_t}}
\le A\sqrt d+\frac{1}{\sqrt{\alpha}}\sqrt d,
\]
so
\[
A_t\le \frac{C}{(1-t)^2}\Bigl(A\sqrt d+\frac{1}{\sqrt{\alpha}}\sqrt d+\sqrt{\frac{d}{\alpha}}\Bigr)
\le \frac{C'}{(1-t)^2}\sqrt d.
\]

\medskip\noindent
\textbf{Control on $(0,t_0]$.}
By assumption, for all $t\in(0,t_0]$ we have $\sigma_t^{-1}\le A$ and $|\sigma_t'|+|\sigma_t''|\le A$, so
\[
|a_t|=\Bigl|\frac{\sigma_t'}{\sigma_t}\Bigr|\le |\sigma_t'|\sigma_t^{-1}\le A^2.
\]
Moreover, since $\sigma\in W^{2,1}_{\mathrm{loc}}(0,1)$ and $\sigma_t>0$, we have for a.e. $t\in(0,t_0]$,
\[
a_t'=\partial_t\Bigl(\frac{\sigma_t'}{\sigma_t}\Bigr)
=\frac{\sigma_t''}{\sigma_t}-\Bigl(\frac{\sigma_t'}{\sigma_t}\Bigr)^2,
\]
hence $|a_t'|\le |\sigma_t''|\sigma_t^{-1}+|\sigma_t'|^2\sigma_t^{-2}\le A^2+A^4$.
Similarly,
\[
|c_t|=|f_t'-a_tf_t|\le |f_t'|+|a_t|\le A+A^2,
\qquad
|f_t'-2a_tf_t|\le |f_t'|+2|a_t|\le A+2A^2,
\]
and using $c_t'=f_t''-a_t'f_t-a_tf_t'$ with $|f_t|\le 1$,
\[
|c_t'|\le |f_t''|+|a_t'|+|a_t||f_t'|\le A+(A^2+A^4)+A^3.
\]
Therefore, for a.e. $t\in(0,t_0]$, $B_t$ is bounded by a constant depending only on $\alpha,\beta,K$ and $A$,
whereas $A_t$ satisfies
\[
A_t\le C_0\Bigl(\|\mathbb E_{\nu_{t,0}}[W]\|+\sqrt{d/\gamma_t}\Bigr)
\]
for some $C_0>0$ with the same dependencies. Using $\gamma_t\ge \alpha$ and Lemma~\ref{lem:mean_bound_nu_t0},
\[
\|\mathbb E_{\nu_{t,0}}[W]\|\le \|\argmin u\|+\sqrt{\frac{d}{\alpha+\lambda_t}}
\le A\sqrt d+\frac{1}{\sqrt{\alpha}}\sqrt d,
\]
we obtain $A_t\le C_0'\sqrt d$ for some $C_0'>0$.
Hence, for a.e. $t\in(0,t_0]$ and all $x$,
\[
\|\partial_t v_t(x)\|\le A_t+B_t\|x\|\le C_0'\bigl(\sqrt d+\|x\|\bigr).
\]
Combining with the estimate on $[t_0,1)$ gives the claim.
\end{proof}

\subsubsection{Proof of Corollary~\ref{coro:timeestimatesi}}\label{sec:coro:timeestimatesi}
\begin{proof}
We apply Theorem~\ref{theo:partialtvt} to the reduced model \(X_t=f_tY+\bar g_t\xi\). Define, for a.e. \(t\in(0,1)\),
\[
a_t:=\frac{\bar g_t'}{\bar g_t},
\qquad
c_t:=f_t'-a_tf_t,
\qquad
\gamma_t:=\alpha+\frac{f_t^2}{\bar g_t^2},
\qquad
D_t:=\alpha\bar g_t^2+f_t^2=\bar g_t^2\gamma_t.
\]
Then Theorem~\ref{theo:partialtvt} yields: for all \(x\in\R^d\) and for a.e. \(t\in(0,1)\),
\[
\|\partial_t v_t(x)\|\le A_t+B_t\|x\|,
\]
where \(A_t,B_t\) are the explicit quantities in Theorem~\ref{theo:partialtvt}.

\medskip
\noindent\textbf{Control on \([t_0,1)\).}
Since \(f\) is increasing we have \(f_t\ge f_{t_0}\)
for all \(t\in[t_0,1)\). Then on \([t_0,1)\),
\[
D_t\ge f_t^2\ge C^{-1},
\qquad
\frac1{D_t}\le C,
\qquad
\frac{f_t}{D_t}\le C,
\qquad
\frac{f_t^2}{D_t^2}\le C,
\]
so every ratio involving \(D_t\) which appears in \(A_t,B_t\) is uniformly bounded on \([t_0,1)\). Next, from \(\bar g_t=(1-t)^p\ell(t)\) we have
\[
a_t=\partial_t(\log\bar g_t)=-\frac{p}{1-t}+\frac{\ell'(t)}{\ell(t)}.
\]
Using \(A^{-1}\le \ell\le A\) and \(\|\ell'\|_\infty\le A\), we get \(|a_t|\leq C (1-t)^{-1}\) on \([t_0,1)\).
Differentiating again,
\[
a_t'=\frac{p}{(1-t)^2}+\partial_t \Bigl(\frac{\ell'}{\ell}\Bigr),
\]
and \(\|\ell'\|_\infty+\|\ell''\|_\infty\le A\), \(A^{-1}\le\ell\le A\) give \(|a_t'|\lesssim (1-t)^{-2}\) on \([t_0,1)\).
Using \(0\le f_t\le 1\) and \(\|f'\|_\infty\le A\) on \([t_0,1)\),
\[
|c_t|=|f_t'-a_tf_t|\le |f_t'|+|a_t||f_t|\lesssim 1+\frac1{1-t}\lesssim \frac1{1-t}.
\]
Moreover,
\[
c_t'=f_t''-a_t'f_t-a_tf_t',
\]
so \(\|f''\|_\infty\le A\) and the bounds on \(a_t,a_t'\) yield \(|c_t'|\lesssim (1-t)^{-2}\) on \([t_0,1)\).
Finally,
\[
|f_t'-2a_tf_t|\le |f_t'|+2|a_t||f_t|\lesssim \frac1{1-t}.
\]
Inspecting the expressions of \(A_t\) and \(B_t\) in Theorem~\ref{theo:partialtvt}, and using that on \([t_0,1)\)
all ratios in \(D_t\), we conclude that \(A_t\) is a sum of terms of size
\[
\lesssim \Bigl(|c_t'|+|c_t|^2\Bigr)\Bigl(\|\E_{\nu_{t,0}}[W]\|+\sqrt{d/\gamma_t}\Bigr)
+ |c_t|^2\Bigl(\sqrt{d/\gamma_t}-\sqrt{d/\gamma_t}\Bigr),
\]
and in particular
\[
A_t \lesssim \Bigl(|c_t'|+|c_t|^2\Bigr)\Bigl(\|\E_{\nu_{t,0}}[W]\|+\sqrt{d/\gamma_t}\Bigr).
\]
Moreover \(B_t\) is a sum of terms of size
\(\lesssim |a_t'|+|c_t'|+|c_t||f_t'-2a_tf_t|+|c_t|^2\).
Hence, on \([t_0,1)\),
\[
A_t\le \frac{C}{(1-t)^2}\Bigl(\|\E_{\nu_{t,0}}[W]\|+\sqrt{d/\gamma_t}\Bigr),
\qquad
B_t\le \frac{C}{(1-t)^2}.
\]
Using $\gamma_t\ge\alpha$ and Lemma~\ref{lem:mean_bound_nu_t0} with $y_0:=\arg\min u$, we have
\[
\|\E_{\nu_{t,0}}[W]\|\le \|y_0\|+\sqrt{\frac{d}{\alpha+\lambda_t}}
\le \|y_0\|+\sqrt{\frac{d}{\alpha}}
\le (A+\alpha^{-1/2})\sqrt d,
\]
where we used the hypothesis $\|y_0\|\le A\sqrt d$.
Hence $A_t\le \dfrac{C'}{(1-t)^2}\sqrt d$ for some $C'>0$.

\medskip
\noindent\textbf{Control on \((0,t_0]\).}
By Assumption~\ref{assum:sto-int}, \(f_t,g_t,\sigma_t\in[0,1]\), hence \(\bar g_t\le \sqrt{2}\) for all \(t\in(0,1)\).
On \((0,t_0]\), the additional assumption (4) gives \(\bar g_t^{-1}\le A\), so \(\bar g_t\) is bounded above and below
by constants depending only on \(A\). In particular,
\[
\gamma_t=\alpha+\frac{f_t^2}{\bar g_t^2}\ge \alpha,
\]
and all rational expressions in \(\bar g_t,f_t,\gamma_t,D_t\) appearing in \(A_t,B_t\) are bounded
by constants depending only on \(\alpha\) and \(A\).

Next, for a.e. \(t\in(0,t_0]\), by (4),
\[
|a_t|=\Bigl|\frac{\bar g_t'}{\bar g_t}\Bigr|
\le |\bar g_t'|  \bar g_t^{-1}\le A^2,
\qquad
|a_t'|
=\Bigl|\frac{\bar g_t''}{\bar g_t}-\Bigl(\frac{\bar g_t'}{\bar g_t}\Bigr)^2\Bigr|
\le |\bar g_t''|  \bar g_t^{-1}+|a_t|^2
\le A^2+A^4.
\]
Similarly,
\[
|c_t|=|f_t'-a_tf_t|\le |f_t'|+|a_t|\le A+A^2,
\]
\[
|f_t'-2a_tf_t|\le |f_t'|+2|a_t|\le A+2A^2,
\] 
and using \(c_t'=f_t''-a_t'f_t-a_tf_t'\) with \(|f_t|\le 1\),
\[
|c_t'|\le |f_t''|+|a_t'|+|a_t|  |f_t'|
\le A+(A^2+A^4)+A^3.
\]
Therefore, for a.e. \(t\in(0,t_0]\), \(B_t\) is bounded by a constant depending only on
\(\alpha,\beta,K\) and \(A\), whereas \(A_t\) satisfies
\[
A_t\le C_0\Bigl(\|\E_{\nu_{t,0}}[W]\|+\sqrt{d/\gamma_t}\Bigr)
\]
for some \(C_0>0\) with the same dependencies.
Using $\gamma_t\ge\alpha$ and Lemma~\ref{lem:mean_bound_nu_t0} with $y_0:=\arg\min u$, we have
\[
\|\E_{\nu_{t,0}}[W]\|\le \|y_0\|+\sqrt{\frac{d}{\alpha+\lambda_t}}
\le \|y_0\|+\sqrt{\frac{d}{\alpha}}
\le (A+\alpha^{-1/2})\sqrt d,
\]
where we used the hypothesis $\|y_0\|\le A\sqrt d$.
Therefore $A_t\le C_0'\sqrt d$ for some $C_0'>0$.
Hence, for a.e. \(t\in(0,t_0]\) and all \(x\),
\[
\|\partial_t v_t(x)\|\le A_t+B_t\|x\|\le C_0'(\sqrt d+\|x\|).
\]

\end{proof}

\subsubsection{Proof of Corollary~\ref{coro:timeregudiffu}}\label{sec:coro:timeregudiffu}
\begin{proof}
Let $v_t$ denote the (rescaled) probability flow drift in the diffusion setting. By the closed-form identity
\[
v_t(x)=\frac1t\big(x+s(t,x)\big),
\]
we have
\[
\partial_t s(t,x)=v_t(x)+t  \partial_t v_t(x).
\]
The reduced model $X_t=tY+\sigma_t\xi$ fits the Lipman flow matching setting  of Assumption~\ref{assum:lipman}.
Moreover, the schedule assumptions of Corollary~\ref{coro:dtv_rate} are satisfied: take $t_0=1/2$ and write
$\sigma_t=(1-t)^{1/2}\ell(t)$ with $\ell(t)=\sqrt{1+t}$, so that $\ell,\ell',\ell''$ are bounded on $[t_0,1)$,
and note that $f_t=t$ satisfies $f',f''\in L^\infty$, while $\sigma^{-1},|\sigma'|,|\sigma''|$ are bounded on $(0,t_0]$.
Therefore Corollary~\ref{coro:dtv_rate} applies and yields that for a.e. $t\in(0,1)$ and all $x\in\mathbb{R}^d$,
\[
\|\partial_t v_t(x)\| \le \frac{C}{(1-t)^2}\big(\sqrt d+\|x\|\big).
\]

Let $\mu_t(x):=\mathbb{E}[Y\mid X_t=x]$. Using the identity
\[
\mu_t(x)=\frac1t\Big(x+(1-t^2)s(t,x)\Big),
\]
we obtain
\[
v_t(x)=\frac1t(x+s(t,x))=\frac{\mu_t(x)-t  x}{1-t^2}.
\]
We now bound $\mu_t(x)$ using Lemma~\ref{lemma:bounfmut} (applied with $f_t=t$ and $\bar g_t=\sigma_t$). Writing
\[
\gamma_t:=\alpha+\frac{t^2}{\sigma_t^2}=\alpha+\frac{t^2}{1-t^2},
\]
Lemma~\ref{lemma:bounfmut} gives, for every $x\in\mathbb{R}^d$,
\[
\|\mu_t(x)\|
\le C\Big(\|\mathbb{E}_{\nu_{t,0}}[Y]\|+\sqrt{\frac d{\gamma_t}}
+\frac{t}{\sigma_t^2\gamma_t}\|x\|\Big).
\]
Since $\gamma_t\ge \alpha$ and $\sigma_t^2\gamma_t=\alpha(1-t^2)+t^2=\alpha+(1-\alpha)t^2\ge \alpha\wedge 1$,
the exponential factor is uniformly bounded and $\frac{t}{\sigma_t^2\gamma_t}\le(\alpha\wedge 1)^{-1}$.
Using the assumption $\|\argmin u\|d^{-1/2}\le A$ we obtain by lemma~\ref{lem:mean_bound_nu_t0}
\[
\|\mu_t(x)\| \le C\big(\sqrt d+\|x\|\big).
\]
Consequently,
\[
\|v_t(x)\|
\le \frac{\|\mu_t(x)\|+t\|x\|}{1-t^2}
\le \frac{C}{1-t^2}\big(\sqrt d+\|x\|\big)
\le \frac{C}{1-t}\big(\sqrt d+\|x\|\big).
\]
Combining $\partial_t s(t,x)=v_t(x)+t  \partial_t v_t(x)$ with the two bounds above yields, for a.e. $t\in(0,1)$,
\[
\|\partial_t s(t,x)\|
\le \frac{C}{1-t}\big(\sqrt d+\|x\|\big)
+ \frac{C}{(1-t)^2}\big(\sqrt d+\|x\|\big)
\le \frac{C}{(1-t)^2}\big(\sqrt d+\|x\|\big),
\]
which proves the claim.
\end{proof}

\subsection{Technical novelties for the global and time Lipschitz estimates}\label{sec:technicalglobal}
As the literature on regularity of exact drifts is already large, we only discuss the closest works, namely  \cite{benton_debortoli_doucet_2024_flowmatching},  \cite{gao2024convergence}, \cite{gentiloni_silveri_ocello_2025}, and \cite{kunkel2025distribution}. These papers all obtain regularity information on flow-matching vector fields or diffusion-model scores, but they do so through rather different mechanisms.

\cite{benton_debortoli_doucet_2024_flowmatching} start from an explicit Jacobian formula for the flow-matching vector field and estimate it through assumptions formulated directly on posterior covariances. \cite{gao2024convergence} also use an explicit representation of the vector field, but their argument controls the derivatives through a hierarchy of conditional moments together with model-specific regularity assumptions. In the diffusion setting, \cite{gentiloni_silveri_ocello_2025} exploit Ornstein-Uhlenbeck regularization and Hessian identities for the smoothed density. \cite{kunkel2025distribution}, on the other hand, studies a covariance representation of the vector field and derives from it an integrated control of the Lipschitz constant under explicit covariance asymptotics.

Although several recent works obtain regularity estimates for flow-matching vector fields or diffusion-model scores, the available proof strategies are quite heterogeneous and remain closely tied to the regimes for which they were designed. For this reason, they do not provide a robust route to the present general setting, where the target is only weakly log-concave and the perturbation is merely Hölder. In that regime, one cannot expect a direct differentiation-based argument, nor a closure through model-specific moment expansions, to yield sharp pointwise-in-time bounds with the optimal dimension dependence. In contrast, our approach recasts both the global space-Lipschitz and time-Lipschitz estimates as stability problems for tilted strongly log-concave posterior laws. This avoids both model-specific moment closures and direct PDE/Hessian analyses of the smoothed density, and is precisely what allows us to work under weak log-concavity with merely Hölder perturbations.

The first key ingredient is Lemma~\ref{lemma:holderexp}. For each time $t\in(0,1)$ and position $x\in\mathbb{R}^d$, we introduce the strongly log-concave surrogate posterior law
\[
\nu_{t,x}(dy)\propto e^{-a(y)}p^{t,x}(y) dy,
\] with $
p^{t,x} = \mathrm{Law}(m_t(Z)\mid X_t = x)$.
 Lemma~\ref{lemma:holderexp} studies the log-partition function
\[
a_t(x):=\log \mathbb{E}_{\nu_{t,x}}\big[e^{a(Y)}\big],
\]
where $a$ is the Hölder perturbation in the decomposition $p^\star=e^{-u+a}$. Its key new point is an $L^p$ log-Laplace stability estimate, which shows that the map $x\mapsto a_t(x)$ remains Hölder without ever differentiating the rough term $a$. The second key ingredient is Proposition~\ref{prop:rt-gradient-holder}, which gives an affine bound on $\nabla r_t(x)$, where
\[
r_t(x):=\mathbb{E}[\|Y\|^2\mid X_t=x]
\]
is the posterior second moment of the endpoint variable $Y$ given the intermediate state $X_t=x$. In particular, the time-Lipschitz bound is not obtained by estimating separately a large hierarchy of conditional moments. 

Starting from the posterior representations, the whole argument is reduced to a small family of stable posterior quantities.  Proposition~\ref{propnablav} isolates the explicit Gaussian singularity and leaves only covariance terms, which are compared to their strongly log-concave surrogates through Lemma~\ref{lemma:holderexp} and Lemma \ref{lem:mgf_holder}. Then, Proposition~\ref{prop:time_derivative_general} rewrites $\partial_t v_t(x)$ in terms of the posterior mean, Proposition~\ref{prop:estimatepartialtmut} reduces the problem to $\nabla r_t(x)$, and Propositions~\ref{prop:grad-rt}-\ref{prop:rt-gradient-holder} together with Lemma~\ref{lemma:bounfmut} close the estimate by the same surrogate-posterior comparison mechanism.

Overall, the novelty is twofold. First, the paper introduces a reusable proof architecture: posterior representation, removal of the rough tilt, analysis under a strongly log-concave surrogate, and transfer back to the true posterior through quantitative stability estimates. Second, this architecture is made effective by two genuinely new estimates: the covariance comparison built from Lemma~\ref{lemma:holderexp} and Lemma \ref{lem:mgf_holder}, and the affine control of $\nabla r_t$ in Proposition~\ref{prop:rt-gradient-holder}. Together, these ingredients are what allow the paper to obtain sharp pointwise-in-time bounds, with dimension-free leading constants, in a regime of weak log-concavity and H\"older perturbations that lies beyond the reach of the proof strategies discussed above.

\subsection{Proofs of the functional inequalities}
\subsubsection{Proof of Lemma~\ref{lem:flow-lip}}\label{sec:lem:flow-lip}

\begin{proof}
Differentiating the flow equation with respect to \(x\) gives
\[
\partial_t \nabla X_t(x)
=
\nabla v_t(X_t(x)) \nabla X_t(x),
\qquad
\nabla X_0(x)=\text{Id}.
\]
Fix \(x\in\R^d\) and a unit vector \(h\in\R^d\), and set
\[
w_t := \nabla X_t(x)h,\qquad a_t := |w_t|^2.
\]
Then
\[
\partial_t a_t
= 2\langle w_t,\partial_t w_t\rangle
= 2\langle w_t,\nabla v_t(X_t(x))w_t\rangle.
\]
By the definition of \(\lambda_{\max}\),
\[
\langle w_t,\nabla v_t(X_t(x))w_t\rangle
\le
\lambda_{\max}\big(\nabla v_t(X_t(x))\big) |w_t|^2
\le
\theta_t a_t.
\]
Hence
\[
\partial_t a_t  \le  2\theta_t a_t.
\]
By Grönwall's inequality,
\[
a_t  \le  a_0
\exp  \Big(2\int_0^t\theta_s ds\Big)
=
\exp  \Big(2\int_0^t\theta_s ds\Big),
\]
since \(a_0 = |\nabla X_0(x)h|^2 = |h|^2 = 1\).
This yields uniformly in $x$,
\[
\|\nabla X_t(x)\|_{\mathrm{op}}
=
\sup_{|h|=1}|\nabla X_t(x)h|
\le
\exp  \Big(\int_0^t\theta_s ds\Big).
\]
\end{proof}

\subsubsection{Proof of Proposition~\ref{prop:Psi-Sobolev-p}}\label{sec:prop:Psi-Sobolev-p}

\begin{proof}
Define \(F:=\zeta\circ T\). Since \(T_{\#}\gamma_d = p\), we have
\[
\text{Ent}_{\Psi,p}(\zeta)
=
\text{Ent}_{\Psi,\gamma_d}(F).
\]
Applying \eqref{eq:gaussian-Psi-Sobolev} to \(F\) gives
\[
\text{Ent}_{\Psi,p}(\zeta)
=
\text{Ent}_{\Psi,\gamma_d}(F)
 \le 
\frac12
\int_{\R^d}\Psi''(F(y)) |\nabla F(y)|^2 d\gamma_d(y).
\]
As \(|\nabla F(y)|\le L |\nabla\zeta(T(y))|\), hence
\[
\text{Ent}_{\Psi,p}(\zeta)
 \le 
\frac{L^2}{2}
\int_{\R^d}\Psi''(\zeta(T(y))) |\nabla\zeta(T(y))|^2 d\gamma_d(y).
\]
Using again \(T_{\#}\gamma_d = p\) and the change of variables \(x=T(y)\) yields
\[
\text{Ent}_{\Psi,p}(\zeta)
 \le 
\frac{L^2}{2}
\int_{\R^d}\Psi''(\zeta(x)) |\nabla\zeta(x)|^2 p(x) dx,
\]
as claimed.
\end{proof}

\subsubsection{Proof of Theorem~\ref{theo:functional-ineq-flowmatching}}\label{sec:theo:functional-ineq-flowmatching}

\begin{proof}
Let
$(X_t)_{t\in[0,1]}$ be the flow associated with the flow-matching vector field $(v_t)_{t\in[0,1]}$ from the probability flow ODE, namely
\[
\partial_t X_t(x)=v_t\bigl(X_t(x)\bigr),\qquad X_0(x)=x,
\]
and define the terminal map $T:=X_1$. The probability flow ODE considers the model
$X_t = f_t Y + \sigma_t \xi$ with $(Y,\xi)\sim p^\star \otimes\gamma_d$, $f_t=t$ and $\sigma_t=\sqrt{1-t^2}$.
As shown in the proof of Corollary~\ref{coro:score_osl_selfcontained},
this choice satisfies the hypotheses required to apply Corollary~\ref{cor:Lip-transport} to the associated
flow-matching vector field $v_t$.

By Corollary~\ref{cor:Lip-transport}, the map $T$ pushes forward the standard Gaussian measure $\gamma_d$ onto $p^\star$
and  is globally Lipschitz with a dimension-free bound
\[
\text{Lip}(T)\le e^{C_{\mathrm{flow}}}=:L.
\]
Applying Proposition~\ref{prop:Psi-Sobolev-p} to the pushforward relation $T_{\#}\gamma_d=p^\star$ and the Lipschitz bound $\text{Lip}(T)\le L$,
we obtain
\[
\text{Ent}_{\Psi,p^\star}(f)
 \le 
\frac{L^2}{2}\int_{\R^d}\Psi''(f(x)) \|\nabla f(x)\|^2 p^\star(x) dx.
\]

\begin{itemize}
\item \emph{Poincar\'e.} Take $\Psi(u)=u^2$ on $\R$. Then $\text{Ent}_{\Psi,p^\star}(f)=\text{Var}_{p^\star}(f)$ and $\Psi''\equiv 2$,
hence for smooth $f$,
\[
\text{Var}_{p^\star}(f)
 \le 
\frac{L^2}{2}\int 2 \|\nabla f\|^2 dp^\star
 = 
L^2\int \|\nabla f\|^2 dp^\star.
\]

\item \emph{Log-Sobolev.} Take $\Psi(u)=u\log u$ on $\R_+$ and set $\zeta=f^2$ with $f$ smooth.
Then $\text{Ent}_{\Psi,p^\star}(f^2)=\text{Ent}_{p^\star}(f^2)$ and $\Psi''(u)=1/u$. Therefore,
\[
\text{Ent}_{p^\star}(f^2)
 \le 
\frac{L^2}{2}\int \frac{1}{f^2} \|\nabla(f^2)\|^2 dp^\star
 = 
\frac{L^2}{2}\int \frac{1}{f^2} 4f^2\|\nabla f\|^2 dp^\star
 = 
2L^2\int \|\nabla f\|^2 dp^\star.
\]
\end{itemize}
\end{proof}

\section{Proofs of the applications to sampling}
\subsection{Proof of the discretization errors}
\subsubsection{Proof of Proposition~\ref{prop:W2_discrete}}\label{sec:prop:W2_discrete}

\begin{proof}
By synchronous coupling (same $X_0$ and same Brownian motion $W$),
\[
W_2(\text{Law}(X_\tau),\text{Law}(\bar X_N))
\le \Big(\E\|X_\tau-\bar X_N\|^2\Big)^{1/2}.
\]
Set $$\Delta_k := X_{t_k}-\bar X_k.$$
\textbf{Step 1: restart the exact flow at each grid point.}
For each $k$, let $Y^{(k)}$ be the solution on $[t_k,t_{k+1}]$ of
\[
dY^{(k)}_t=a_t(Y^{(k)}_t)  dt+b_t  dW_t,\qquad Y^{(k)}_{t_k}=\bar X_k.
\]
Then

\begin{equation}\label{eq:one_step_Minkowski}
\|\Delta_{k+1}\|_{L^2}\le \|X_{t_{k+1}}-Y^{(k)}_{t_{k+1}}\|_{L^2} + \|Y^{(k)}_{t_{k+1}}-\bar X_{k+1}\|_{L^2}.
\end{equation}
\textbf{Step 2: stability with respect to the initial condition.}
Let $Z_t:=X_t-Y^{(k)}_t$ on $[t_k,t_{k+1}]$. Since both SDEs have the same diffusion term,
\[
dZ_t = (a_t(X_t)-a_t(Y^{(k)}_t))  dt,
\]
so
\[
\frac{d}{dt}\|Z_t\|^2
=2\langle Z_t,  a_t(X_t)-a_t(Y^{(k)}_t)\rangle
\le 2L_t  \|Z_t\|^2
\]
by (A1). Applying Grönwall's lemma, we get 
\[
\|Z_{t_{k+1}}\|^2 \le \exp \Big(2\int_{t_k}^{t_{k+1}}L_s  ds\Big)  \|Z_{t_k}\|^2
= \exp \Big(2\int_{t_k}^{t_{k+1}}L_s  ds\Big)  \|\Delta_k\|^2.
\]
Taking expectations and square roots,
\begin{equation}\label{eq:exact_flow_stability}
\|X_{t_{k+1}}-Y^{(k)}_{t_{k+1}}\|_{L^2}
\le \exp \Big(\int_{t_k}^{t_{k+1}} L_s  ds\Big)  \|\Delta_k\|_{L^2}.
\end{equation}
\textbf{Step 3: stability with respect to the drift.}
From the definitions of $Y^{(k)}$ and the Euler update \eqref{eq:EM_main},
\begin{align*}
Y^{(k)}_{t_{k+1}}-\bar X_{k+1}
&=
\int_{t_k}^{t_{k+1}} \big(a_t(Y^{(k)}_t)-a_{t_k}(\bar X_k)\big)  dt
+\int_{t_k}^{t_{k+1}} b_t dW_t - \xi_{k+1}\\
&=\int_{t_k}^{t_{k+1}} \big(a_t(Y^{(k)}_t)-a_{t_k}(\bar X_k)\big)  dt,
\end{align*}
so by Minkowski,
\begin{align}
\|Y^{(k)}_{t_{k+1}}-\bar X_{k+1}\|_{L^2}
&\le
\int_{t_k}^{t_{k+1}}\|a_t(Y^{(k)}_t)-a_{t_k}(\bar X_k)\|_{L^2}  dt.
\label{eq:local_split}
\end{align}
Split
\begin{equation}
a_t(Y^{(k)}_t)-a_{t_k}(\bar X_k)
=
\big(a_t(Y^{(k)}_t)-a_{t_k}(Y^{(k)}_t)\big)
+
\big(a_{t_k}(Y^{(k)}_t)-a_{t_k}(\bar X_k)\big).
\end{equation}

Set $$A_{k}:=\sup_{t\in[t_k,t_{k+1}]}\|Y^{(k)}_t\|_{L^2},$$
from
\[
Y^{(k)}_t=\bar X_k+\int_{t_k}^{t} a_s(Y^{(k)}_s)  ds+\int_{t_k}^{t} b_s  dW_s
\]
and Minkowski plus It\^o isometry,
\[
A_{k}
\le \|\bar X_k\|_{L^2}+\int_{t_k}^{t_{k+1}}\|a_s(Y^{(k)}_s)\|_{L^2}  ds+B\sqrt d  \sqrt{h_k}.
\]
By (A4), $\|a_s(Y^{(k)}_s)\|_{L^2}\le F_s(\sqrt d+\|Y^{(k)}_s\|_{L^2})\le F_s(\sqrt d+A_{k})$, hence
\[
A_{k}\le \|\bar X_k\|_{L^2}+(\sqrt d+A_{k})\int_{t_k}^{t_{k+1}}F_s  ds+B\sqrt d  \sqrt{h_k}.
\]
As by assumption $\int_{t_k}^{t_{k+1}}F_s  ds\leq1/2$ and $\|\bar X_k\|_{L^2}\le A_{\bar X}$, we obtain
\[
A_{k}\le 2(A_{\bar X}+\sqrt d  +B\sqrt{d h_k})=:U_k-2\sqrt{d}.
\]
\emph{Time variation term.}
Fix $t\in[t_k,t_{k+1}]$, we have,
\[
a_t(Y^{(k)}_t)-a_{t_k}(Y^{(k)}_t)
=\int_{t_k}^{t}\partial_s a_s(Y^{(k)}_t)  ds,
\]
hence by (A3)
\[
\|a_t(Y^{(k)}_t)-a_{t_k}(Y^{(k)}_t)\|_{L^2}
\le \Big(\int_{t_k}^{t}M_s  ds\Big)\Big(\sqrt d+\|Y^{(k)}_t\|_{L^2}\Big)
\le \Big(\int_{t_k}^{t}M_s  ds\Big)U_k.
\]
Therefore,
\begin{align*}
\int_{t_k}^{t_{k+1}}\|a_t(Y^{(k)}_t)-a_{t_k}(Y^{(k)}_t)\|_{L^2}  dt
&\le h_k
U_k
\int_{t_k}^{t_{k+1}}M_t  dt.
\end{align*}
\emph{Space variation term.}
By (A2),
\[
\|a_{t_k}(Y^{(k)}_t)-a_{t_k}(\bar X_k)\|_{L^2}
\le C_{t_k}  \|Y^{(k)}_t-\bar X_k\|_{L^2}.
\]
Moreover, for $t\in[t_k,t_{k+1}]$,
\[
Y^{(k)}_t-\bar X_k
=
\int_{t_k}^{t} a_s(Y^{(k)}_s)  ds+\int_{t_k}^{t} b_s  dW_s,
\]
so by Minkowski and It\^o isometry,
\[
\|Y^{(k)}_t-\bar X_k\|_{L^2}
\le
\int_{t_k}^{t}\|a_s(Y^{(k)}_s)\|_{L^2}  ds
+ B\sqrt d  \sqrt{t-t_k}.
\]
Using (A4) and $\|Y^{(k)}_s\|_{L^2}\le U_k-\sqrt{d}$, we have
$\|a_s(Y^{(k)}_s)\|_{L^2}\le F_sU_k$, hence
\[
\|Y^{(k)}_t-\bar X_k\|_{L^2}
\le
U_k\int_{t_k}^{t}F_s  ds + B\sqrt d  \sqrt{t-t_k},
\]
which gives
\[
\int_{t_k}^{t_{k+1}}\|Y^{(k)}_t-\bar X_k\|_{L^2}  dt
\le h_k
U_k\int_{t_k}^{t_{k+1}}F_t  dt
+\frac{2}{3}B\sqrt d  h_k^{3/2}.
\]
Plugging into \eqref{eq:local_split} yields
\begin{equation}\label{eq:rho_k_def}
\|Y^{(k)}_{t_{k+1}}-\bar X_{k+1}\|_{L^2}
\le
\rho_k,
\end{equation}
where
\[
\rho_k:=
h_k U_k
\int_{t_k}^{t_{k+1}}M_t  dt
+ h_kC_{t_k}\Big(U_k \int_{t_k}^{t_{k+1}}F_t  dt+\frac{2}{3}B\sqrt d  h_k^{1/2}\Big).
\]
\textbf{Step 4: discrete Grönwall.}
Combining \eqref{eq:one_step_Minkowski}, \eqref{eq:exact_flow_stability}, and \eqref{eq:rho_k_def},
\[
\|\Delta_{k+1}\|_{L^2}\le \exp \Big(\int_{t_k}^{t_{k+1}} L_s  ds\Big) \|\Delta_k\|_{L^2} + \rho_k.
\]
Iterating,
\[
\|\Delta_{N}\|_{L^2} \le \sum_{k=0}^{N-1}\rho_k\prod_{j=k+1}^{N-1}\exp \Big(\int_{t_j}^{t_{j+1}} L_s  ds\Big)
=\sum_{k=0}^{N-1}\rho_k
\exp \Big(\int_{t_{k+1}}^{\tau}L_s  ds\Big).
\]
Finally,  $W_2(\text{Law}(X_\tau),\text{Law}(\bar X_N))\le \|\Delta_{N}\|_{L^2}$, giving the result.
\end{proof}

\subsubsection{Proof of Proposition~\ref{prop:geom_mesh_general}}\label{sec:prop:geom_mesh_general}

\begin{proof}
Proposition~\ref{prop:W2_discrete} in the ODE case yields
\begin{equation}\label{eq:prop2_ode_recall}
W_2\big(\text{Law}(X_\tau),\text{Law}(\bar X_N)\big)
\le
\sum_{k=0}^{N-1}
h_k \exp \Big(\int_{t_{k+1}}^\tau L_s ds\Big)
\Bigg[
U_k\int_{t_k}^{t_{k+1}}\big(M_t+C_{t_k}F_t\big) dt
\Bigg],
\end{equation}
with $U_k=2A_{\bar X}+4\sqrt d\leq C\sqrt{d}$ (since $B=0$). Then, \eqref{eq:prop2_ode_recall} gives
\begin{equation}\label{eq:reduce_sum}
W_2\big(\text{Law}(X_\tau),\text{Law}(\bar X_N)\big)
\le
C\sqrt d
\sum_{k=0}^{N-1} h_k\int_{t_k}^{t_{k+1}}\big(M_t+C_{t_k}F_t\big) dt.
\end{equation}
Now, as $t_k=1-r^k$,
\begin{equation}\label{eq:jhflzpazmppmmp}
\int_{t_k}^{t_{k+1}}F_t dt
\le
A\int_{t_k}^{t_{k+1}}\frac{dt}{1-t}
=
A\log \Big(\frac{1-t_k}{1-t_{k+1}}\Big)
=
A\log \Big(\frac{r^k}{r^{k+1}}\Big)
=
A\log \Big(\frac{1}{1-h_{\max}}\Big).
\end{equation}
Thus choosing $\bar h\in(0,1)$ such that
$A\log \big(\frac{1}{1-\bar h}\big)\le \frac12$
ensures $\int_{t_k}^{t_{k+1}}F_t dt\le \frac12$ for all $k$ whenever $h_{\max}\le \bar h$.
Using \eqref{eq:lipman_envelopes_general},
\[
\int_{t_k}^{t_{k+1}}M_t dt
\le
A\int_{t_k}^{t_{k+1}}\frac{dt}{(1-t)^2}
=
A\Big(\frac{1}{1-t_{k+1}}-\frac{1}{1-t_k}\Big)
=
A\Big(\frac{1}{r^{k+1}}-\frac{1}{r^k}\Big)
=
A\frac{1-r}{r^{k+1}}.
\]
Multiplying by $h_k=r^k(1-r)=r^k h_{\max}$ yields
\[
\sum_{k=0}^{N-1}h_k\int_{t_k}^{t_{k+1}}M_t dt
\le
N A\frac{h_{\max}^2}{1-h_{\max}}.
\]
Since $r^N=1-\tau$ and $r=1-h_{\max}$, we have
\[
N=\frac{\log(1/(1-\tau))}{\log(1/r)}
=\frac{\log(1/(1-\tau))}{\log(1/(1-h_{\max}))}
\le
\frac{\log(1/(1-\tau))}{h_{\max}},
\]
using $-\log(1-h_{\max})\ge h_{\max}$. Hence for $h_{\max}\le 1/2$,
\begin{equation}\label{eq:M_sum_bound}
\sum_{k=0}^{N-1}h_k\int_{t_k}^{t_{k+1}}M_t dt
\le
C h_{\max}\log \Big(\frac{1}{1-\tau}\Big).
\end{equation}
Using again \eqref{eq:lipman_envelopes_general} and \eqref{eq:jhflzpazmppmmp},
\[
\int_{t_k}^{t_{k+1}}F_t dt\le A\log(1/r).
\]
Also $C_{t_k}\le \frac{A}{1-t_k}$ so
\[
h_k\int_{t_k}^{t_{k+1}}C_{t_k}F_t dt
\le
h_k\frac{A^2}{r^k}\log(1/r),
\]
which gives
\begin{align*}
\sum_{k=0}^{N-1}h_k\int_{t_k}^{t_{k+1}}C_{t_k}F_t dt
&\le
A^2 N h_{\max}\log(1/r)\\
& \leq Ch_{\max}\log(1/(1-\tau)).
\end{align*}
Combining \eqref{eq:reduce_sum} and \eqref{eq:M_sum_bound}, we obtain for
$h_{\max}\le \min(\bar h,1/2)$
\begin{align*}
W_2\big(\text{Law}(X_\tau),\text{Law}(\bar X_N)\big)
&\le
C \sqrt d h_{\max} \log \Big(\frac{1}{1-\tau}\Big).
\end{align*}
\end{proof}

\subsubsection{Proof of Theorem~\ref{theo:geom_mesh_generaldiff}}\label{sec:theo:geom_mesh_generaldiff}
Let us first state an intermediate decomposition. 
Define for $k\in \{0,...,N-1\}$
\begin{equation}
\label{eq:Vk_def}
\Delta_k := X_{t_k}-\bar X_k, \qquad V_k := \int_{t_k}^{t_{k+1}}\Big(a_t(X_t)-a_{t_k}(X_{t_k})\Big) dt .
\end{equation}

\begin{proposition}
\label{prop:biasvar_no_propagation}
Suppose we are under the assumptions of Theorem~\ref{theo:geom_mesh_generaldiff} and
assume that there exists  $L\in L^1([0,\tau],\R)$ such that for all $k$,
\begin{equation}
\label{eq:stability_discrete}
\big\|\Delta_k + h_k\big(a_{t_k}(X_{t_k})-a_{t_k}(\bar X_k)\big)\big\|_{L^2}
 \le  e^{\int_{t_k}^{t_{k+1}} L_s ds}  \|\Delta_k\|_{L^2}.
\end{equation}
Writing $\mathcal F_k:=\sigma(X_{t_k},\bar X_k)$, assume
there exist $\mathrm{Bias}_k,\mathrm{Var}_k\geq 0$ such that
\begin{equation}
\label{eq:biasvar_Vk}
\big\| \E[V_k\mid \mathcal F_k]\big\|_{L^2}\le \mathrm{Bias}_k,
\qquad
\big\|V_k-\E[V_k\mid \mathcal F_k]\big\|_{L^2}\le \mathrm{Var}_k .
\end{equation}

Then, the discretization error satisfies
\begin{equation*}
W_2\big(\text{Law}(X_\tau),\text{Law}(\bar X_N)\big) \le
\sum_{k=0}^{N-1}\exp\Big(\int_{t_{k+1}}^\tau L_u du\Big) \mathrm{Bias}_k
 + 
\Big(\sum_{k=0}^{N-1}\exp\Big(2\int_{t_{k+1}}^\tau L_u du\Big) \mathrm{Var}_k^{ 2}\Big)^{1/2}.
\end{equation*}
\end{proposition}
The proof of Proposition~\ref{prop:biasvar_no_propagation} can be found in Section~\ref{sec:prop:biasvar_no_propagation}. We can now give the proof of Theorem~\ref{theo:geom_mesh_generaldiff}.

\begin{proof}[Proof of Theorem~\ref{theo:geom_mesh_generaldiff}]
We work on the same probability space and use the synchronous coupling driven by the same Brownian motion.
Hence
\begin{equation}\label{eq:w2-strong_T}
W_2 \bigl(\text{Law}(X_\tau),\text{Law}(\bar X_N)\bigr)\le \|X_\tau-\bar X_N\|_{L^2}.
\end{equation}
Let $\Delta_k$ and $V_k$ defined in \eqref{sec:theo:geom_mesh_generaldiff},
 using $\xi_{k+1}=\int_{t_k}^{t_{k+1}} b_t dW_t$ gives
\begin{equation}\label{eq:recursion_T}
\Delta_{k+1}
=\Delta_k+h_k\Bigl(a_{t_k}(X_{t_k})-a_{t_k}(\bar X_k)\Bigr)+V_k.
\end{equation}

\paragraph{Step 1: stability of the one-step map.}
We have $\lambda_{\max}(\nabla a_t(\cdot))\le L_t$ and $\lambda_{\min}(\nabla a_{t_k})\geq -C_{t_k}$.
By \eqref{eq:diff_envelopes_general_T}, $C_{t_k}\le A/\rho(t_k)$, hence
\[
h_k C_{t_k}\le A \frac{h_k}{\rho(t_k)}.
\]
We bound $\frac{h_k}{\rho(t_k)}$ uniformly over $k$.
If $T-t_k\ge 1$, then $\rho(t_k)=1$ and $h_k/\rho(t_k)=h_k\le h_{\max}$.
If $T-t_k<1$, then $\rho(t_k)=T-t_k=r^kT$ and $h_k/\rho(t_k)=\frac{r^k h_{\max}}{r^k T}=\frac{h_{\max}}{T}\le h_{\max}$ since $T>1$.
Thus $h_k C_{t_k}\le A h_{\max}$ for all $k$.
Choosing $h_{\max}\le \bar h$ small enough, we ensure $h_k C_{t_k}\le 1$, hence $C_{t_k}\le 1/h_k$.

Then for all $z,v\in \mathbb{R}^d$,
\[
\bigl\| \bigl(\text{Id}+h_k \nabla a_{t_k}(z)\bigr)v\bigr\|
\le \|\text{Id}+h_k \nabla a_{t_k}(z)\|_{\text{op}} \|v\|
\le (1+h_kL_{t_k}) \|v\|\le e^{h_kL_{t_k}} \|v\|,
\]
as $\lambda_{\min}(\text{Id}+h_k \nabla a_{t_k}(z))\geq 0$. Hence,
\begin{align}
    \label{eq:stability_T}
    \bigl\|\Delta_k+h_k\bigl(a_{t_k}(X_{t_k})-a_{t_k}(\bar X_k)\bigr)\bigr\|
&\le \int_0^1 \bigl\|\bigl(\text{Id}+h_k \nabla a_{t_k}(X_{t_k}+s(\bar X_k-X_{t_k})))\bigr)(X_{t_k}-\bar X_k)\bigr\| ds\nonumber\\
&\le e^{h_k L_{t_k}}\|\Delta_k\|.
\end{align}

\paragraph{Step 2: It\^o formula for $a_t(X_t)$ and decomposition of $V_k$.}
Define
\[
\beta_t(x):=\Bigl(\partial_t a_t+(\nabla a_t)a_t+\tfrac12 b_t^2 \Delta a_t\Bigr)(x),
\qquad
G_t(x):=\nabla a_t(x).
\]
By It\^o's formula applied to $a_t(X_t)$,
\begin{equation}\label{eq:ito-a_T}
d\bigl(a_t(X_t)\bigr)=\beta_t(X_t) dt+b_t G_t(X_t) dW_t.
\end{equation}
Writing $w_k(u):=t_{k+1}-u$, we obtain
\begin{align}\label{eq:Vk-decomp2_T}
V_k&= \int_{t_k}^{t_{k+1}}\Big(a_t(X_t)-a_{t_k}(X_{t_k})\Big) dt \nonumber\\
&=\int_{t_k}^{t_{k+1}} w_k(u) \beta_u(X_u) du
+\int_{t_k}^{t_{k+1}} w_k(u) b_u G_u(X_u) dW_u.
\end{align}
Let $\mathcal F_k:=\sigma(X_{t_k},\bar X_k)\subset\mathcal F_{t_k}$.
Since the stochastic integral in \eqref{eq:Vk-decomp2_T} is a martingale increment,
we have
\[
\E[V_k\mid \mathcal F_k]=
\E \left[\int_{t_k}^{t_{k+1}} w_k(u) \beta_u(X_u) du\;\middle|\;\mathcal F_k\right],
\]
so by Minkowski,
\begin{equation}\label{eq:bias-bound_T}
\|\E[V_k\mid \mathcal F_k]\|_{L^2}\le \int_{t_k}^{t_{k+1}} w_k(u) \|\beta_u(X_u)\|_{L^2} du.
\end{equation}
Moreover, by triangle inequality and It\^o isometry,
\begin{equation}\label{eq:var-bound_T}
\|V_k-\E[V_k\mid \mathcal F_k]\|_{L^2}
\le
\int_{t_k}^{t_{k+1}} w_k(u) \|\beta_u(X_u)\|_{L^2} du
+\Bigl(\int_{t_k}^{t_{k+1}} w_k(u)^2 b_u^2 \E\|G_u(X_u)\|_F^2 du\Bigr)^{1/2}.
\end{equation}

\paragraph{Step 3: apply the bias-variance accumulation lemma.}
By the global Lipschitz bound $\|G_u(x)\|_{\text{op}}\le C_u$ and \eqref{eq:diff_envelopes_general_T},
\begin{equation}\label{eq:G-bound_T}
\E\|G_u(X_u)\|_F^2\le d C_u^2\le d \frac{A^2}{\rho(u)^2}.
\end{equation}
Next, using the bound based on Assumption~\ref{assum:sdebounds}, the moment bound on $(X_t)$, and the Laplacian bound,
we obtain
\begin{equation}\label{eq:beta-bound_T} 
\|\beta_u(X_u)\|_{L^2}\le \frac{C}{\rho(u)^2}\sqrt d.
\end{equation}
Define
\[
\text{Bias}_k:=C\sqrt d\int_{t_k}^{t_{k+1}}\frac{t_{k+1}-u}{\rho(u)^2} du,
\qquad
\text{Var}_k:=\text{Bias}_k
+C\sqrt d\Bigl(\int_{t_k}^{t_{k+1}}\frac{(t_{k+1}-u)^2}{\rho(u)^2} du\Bigr)^{1/2},
\]
Then \eqref{eq:bias-bound_T}, \eqref{eq:var-bound_T}, \eqref{eq:beta-bound_T}, \eqref{eq:G-bound_T} give
$\|\E[V_k\mid \mathcal F_k]\|_{L^2}\le \text{Bias}_k$ and
$\|V_k-\E[V_k\mid \mathcal F_k]\|_{L^2}\le \text{Var}_k$.
Together with the stability \eqref{eq:stability_T}, Proposition~\ref{prop:biasvar_no_propagation} yields
\begin{equation}\label{eq:accumulation_T}
\|X_\tau-\bar X_N\|_{L^2}
\le \sum_{k=0}^{N-1}\Lambda_k \text{Bias}_k
+\Bigl(\sum_{k=0}^{N-1}\Lambda_k^2 \text{Var}_k^2\Bigr)^{1/2},
\qquad
\Lambda_k:=\exp \Bigl(\int_{t_{k+1}}^\tau L_s ds\Bigr).
\end{equation}
Since $\sup_z\int_z^\tau L_s ds\le A$, we have $\Lambda_k\le e^A$ for all $k$.

\paragraph{Step 4: geometric-grid estimates.}
We use the pointwise bound
\[
\frac{1}{\rho(u)^2}\le 1+\frac{1}{(T-u)^2},\qquad u\in[0,\tau],
\]
which follows from $\rho(u)=\min\{1,T-u\}$.
Therefore,
\[
\int_{t_k}^{t_{k+1}}\frac{t_{k+1}-u}{\rho(u)^2} du
\le \int_{t_k}^{t_{k+1}}(t_{k+1}-u) du
+\int_{t_k}^{t_{k+1}}\frac{t_{k+1}-u}{(T-u)^2} du
= \frac{h_k^2}{2}+\int_{t_k}^{t_{k+1}}\frac{t_{k+1}-u}{(T-u)^2} du.
\]
For the geometric grid $(\tau,T,N)$, one has $T-t_k=r^k T$, $T-t_{k+1}=r^{k+1}T$, and $h_k=r^k h_{\max}$, and the same computation as in the original proof gives
\begin{equation}\label{eq:Ik_T}
\int_{t_k}^{t_{k+1}}\frac{t_{k+1}-u}{(T-u)^2} du
=\log \frac1r+(r-1)
= -\log \Bigl(1-\frac{h_{\max}}{T}\Bigr)-\frac{h_{\max}}{T}
\le C\Bigl(\frac{h_{\max}}{T}\Bigr)^2,
\end{equation}
for $h_{\max}\le \bar h$ small enough (in particular ensuring $h_{\max}/T\le 1/2$).
Moreover, since $r^N=(T-\tau)/T$,
\begin{equation}\label{eq:N-bound_T}
N=\frac{\log \frac{T}{T-\tau}}{\log \frac1r}
=\frac{\log \frac{T}{T-\tau}}{-\log \bigl(1-\frac{h_{\max}}{T}\bigr)}
\le \frac{T}{h_{\max}}\log \frac{T}{T-\tau}.
\end{equation}
Using $\sum_{k=0}^{N-1} h_k=\tau$ and $h_k\le h_{\max}$,
\[
\sum_{k=0}^{N-1} h_k^2 \le h_{\max}\sum_{k=0}^{N-1} h_k = h_{\max}\tau \le h_{\max}T.
\]
Therefore,
\begin{equation}\label{eq:bias-sum_T}
\sum_{k=0}^{N-1}\text{Bias}_k
\le C\sqrt d\left(\sum_{k=0}^{N-1}h_k^2 + N\Bigl(\frac{h_{\max}}{T}\Bigr)^2\right)
\le C\sqrt d h_{\max}\left(T+\log \frac{T}{T-\tau}\right),
\end{equation}
where the last step uses \eqref{eq:N-bound_T} and $T>1$.

For the variance term, write $\text{Var}_k\le \text{Bias}_k+S_k$ where
\[
S_k:=C\sqrt d\Bigl(\int_{t_k}^{t_{k+1}}\frac{(t_{k+1}-u)^2}{\rho(u)^2} du\Bigr)^{1/2}.
\]
Then
\[
\sum_{k=0}^{N-1}\text{Var}_k^2\le 2\sum_{k=0}^{N-1}\text{Bias}_k^2+2\sum_{k=0}^{N-1}S_k^2.
\]
We again use $1/\rho(u)^2\le 1+1/(T-u)^2$ to obtain
\[
S_k^2\le C d \int_{t_k}^{t_{k+1}}(t_{k+1}-u)^2 du
+ C d\int_{t_k}^{t_{k+1}}\frac{(t_{k+1}-u)^2}{(T-u)^2} du.
\]
The first integral equals $h_k^3/3$, hence
\[
\sum_{k=0}^{N-1} h_k^3 \le h_{\max}^2\sum_{k=0}^{N-1} h_k = h_{\max}^2\tau \le h_{\max}^2 T.
\]
For the second integral, using $(t_{k+1}-u)^2\le h_k^2$ and $(T-u)\ge (T-t_{k+1})$ for $u\in[t_k,t_{k+1}]$,
\[
\int_{t_k}^{t_{k+1}}\frac{(t_{k+1}-u)^2}{(T-u)^2} du
\le \int_{t_k}^{t_{k+1}}\frac{h_k^2}{(T-t_{k+1})^2} du
= \frac{h_k^3}{(T-t_{k+1})^2}
= \frac{(r^k h_{\max})^3}{(r^{k+1}T)^2}
= r^{k-2} \frac{h_{\max}^3}{T^2}.
\]
If $h_{\max}/T\le 1/2$, then $r=1-h_{\max}/T\ge 1/2$ and thus
\[
\sum_{k=0}^{N-1} r^{k-2}\le C \frac{1}{1-r}=C \frac{T}{h_{\max}}.
\]
Therefore,
\[
\sum_{k=0}^{N-1}\int_{t_k}^{t_{k+1}}\frac{(t_{k+1}-u)^2}{(T-u)^2} du
\le C \frac{h_{\max}^3}{T^2}\cdot \frac{T}{h_{\max}}
= C \frac{h_{\max}^2}{T}.
\]
Combining the two contributions yields
\[
\sum_{k=0}^{N-1}S_k^2 \le C d h_{\max}^2\left(T+\frac{1}{T}\right),
\qquad\Rightarrow\qquad
\Bigl(\sum_{k=0}^{N-1}S_k^2\Bigr)^{1/2}\le C\sqrt d h_{\max}\sqrt{T}
\le C\sqrt d h_{\max}T,
\]
where we used $T>1$.
Also $\sqrt{\sum \text{Bias}_k^2}\le \sum \text{Bias}_k$, so using \eqref{eq:bias-sum_T} we obtain
\begin{equation}\label{eq:var-sqrt_T}
\Bigl(\sum_{k=0}^{N-1}\text{Var}_k^2\Bigr)^{1/2}
\le C\sqrt d h_{\max}\left(T+\log \frac{T}{T-\tau}\right).
\end{equation}

\paragraph{Conclusion.}
Plugging \eqref{eq:bias-sum_T} and \eqref{eq:var-sqrt_T} into \eqref{eq:accumulation_T}, using $\Lambda_k\le e^A$, we obtain
\[
\|X_\tau-\bar X_N\|_{L^2}
\le C\sqrt d h_{\max}\left(T+\log \frac{T}{T-\tau}\right).
\]
Combining with \eqref{eq:w2-strong_T} yields the claim.
\end{proof} 

\subsubsection{Proof of Proposition~\ref{prop:geomeshLIp}}\label{sec:prop:geomeshLIp} 
\begin{proof}
We write the proof in the reduced model
\[
X_t=f_tY+\bar g_t\xi,\qquad Y\sim p^\star,\ \xi\sim \mathcal N(0,I_d)\ \text{independent},
\]
where $\bar g_t=\sigma_t$ in the Lipman setting and $\bar g_t=(g_t^2+\sigma_t^2)^{1/2}$ in the
Stochastic-interpolant setting. The two cases are identical once written in this form. We verify the hypotheses of Proposition~\ref{prop:geom_mesh_general} on $[0,\tau]$.
\textbf{Step 1: Lipschitz bounds.} By Corollaries~\ref{coro:lipmafm}/\ref{cor:lipman-global}/\ref{coro:dtv_rate}
in the Lipman case, and by Corollaries \ref{coro:notlipmafm}/\ref{coro:global-Lipschitz-v}/\ref{coro:timeestimatesi} in the
stochastic-interpolant case, there exists a dimension-free constant $C>0$ such that for a.e.\ $t\in(0,\tau)$,
\[
\int_t^\tau \sup_{x\in\R^d}\lambda_{\max}(\nabla v_s(x))\,ds \le C,
\qquad
\sup_{x\in\R^d}\|\nabla v_t(x)\|_{\text{op}}\le \frac{C}{1-t},
\]
and
\[
\|\partial_t v_t(x)\|\le \frac{C}{(1-t)^2}\bigl(\sqrt d+\|x\|\bigr),
\qquad x\in\R^d.
\]
Thus Assumption~\ref{assum:sdebounds}-(A1),(A2),(A3) hold with
\[
L_t:=\sup_{x\in\R^d}\lambda_{\max}(\nabla v_t(x)),
\qquad
C_t:=\frac{C}{1-t},
\qquad
M_t:=\frac{C}{(1-t)^2}.
\]
\textbf{Step 2: linear-growth bound on $v_t$.}
Set
\[
\mu_t(x):=\E[Y\mid X_t=x],\qquad
a_t:=\frac{\bar g_t'}{\bar g_t},\qquad
c_t:=f_t'-a_tf_t,\qquad
D_t:=\alpha \bar g_t^2+f_t^2.
\]
Since
\[
\dot X_t=f_t'Y+\bar g_t'\xi
      =a_tX_t+c_tY,
\]
taking conditional expectation given $X_t=x$ yields
\[
v_t(x)=a_tx+c_t\mu_t(x).
\]
By Lemma~\ref{lemma:bounfmut},
\[
\|\mu_t(x)\|
\le
C\left(
\| \E_{\nu_{t,0}}[W]\|
+\sqrt{\frac d{\gamma_t}}
+\frac{f_t}{\bar g_t^2\gamma_t}\|x\|
\right),
\qquad
\gamma_t:=\alpha+\frac{f_t^2}{\bar g_t^2},
\]
hence, since $\bar g_t^2\gamma_t=D_t$,
\[
\|\mu_t(x)\|
\le
C\left(
\| \E_{\nu_{t,0}}[W]\|
+\sqrt{\frac d{\gamma_t}}
+\frac{f_t}{D_t}\|x\|
\right).
\]
Moreover Lemma~\ref{lem:mean_bound_nu_t0} gives
\[
\| \E_{\nu_{t,0}}[W]\|
\le
\|\argmin u\|+\sqrt{\frac d{\alpha+f_t^2/\bar g_t^2}}
\le C\sqrt d,
\]
and also $\gamma_t\ge \alpha$, so $\sqrt{d/\gamma_t}\le C\sqrt d$. Therefore,
\[
\|\mu_t(x)\|
\le
C\left(\sqrt d+\frac{f_t}{D_t}\|x\|\right).
\]
Consequently,
\[
\|v_t(x)\|
\le
|a_t|\,\|x\|+|c_t|\,\|\mu_t(x)\|
\le
\left(|a_t|+C|c_t|\frac{f_t}{D_t}\right)\|x\|+C|c_t|\sqrt d.
\]
We now bound the coefficients. On $[t_0,\tau]$, Assumption~\ref{assum:lipmanliptime} gives
\[
\bar g_t=(1-t)^p\ell(t),\qquad
0<A^{-1}\le \ell(t)\le A,\qquad
\|\ell'\|_\infty+\|f'\|_\infty\le A,
\]
hence
\[
|a_t|
=
\left|\frac{\bar g_t'}{\bar g_t}\right|
=
\left|-\frac{p}{1-t}+\frac{\ell'(t)}{\ell(t)}\right|
\le \frac{C}{1-t},
\qquad
|c_t|
\le |f_t'|+|a_t|\,|f_t|
\le \frac{C}{1-t}.
\]
Also $f$ is increasing and $f_1=1$, so $f_t\ge f_{t_0}>0$ on $[t_0,\tau]$, and therefore
\[
D_t=\alpha \bar g_t^2+f_t^2\ge f_{t_0}^2.
\]
It follows that
\[
|a_t|+|c_t|\frac{f_t}{D_t}\le \frac{C}{1-t},
\qquad
|c_t|\le \frac{C}{1-t},
\qquad t\in[t_0,\tau].
\]

On $(0,t_0]$, Assumption~\ref{assum:lipmanliptime} gives
\[
\bar g_t^{-1}+|\bar g_t'|+|f_t'|\le C,
\]
hence $|a_t|+|c_t|\le C$. Since $0\le f_t\le 1$ and $\bar g_t\ge C^{-1}$ on $(0,t_0]$,
\[
D_t=\alpha \bar g_t^2+f_t^2\ge c_0>0,
\]
so
\[
|a_t|+|c_t|\frac{f_t}{D_t}\le C,
\qquad
|c_t|\le C,
\qquad t\in(0,t_0].
\]
Combining the two regimes, we obtain for all $t\in(0,\tau]$,
\[
\|v_t(x)\|\le \frac{C}{1-t}\bigl(\sqrt d+\|x\|\bigr).
\]
Thus Assumption~\ref{assum:sdebounds}-(A4) holds with
\[
F_t:=\frac{C}{1-t}.
\]
\textbf{Step 3: moment bound for the Euler scheme.}
Let $\hat X_t$ be the continuous Euler interpolation:
for $t\in[t_k,t_{k+1})$,
\[
\dot{\hat X}_t=v_{t_k}(\hat X_t),
\qquad
\hat X_{t_k}=\bar X_k.
\]
Then $\hat X_{t_k}=\bar X_k$ for all $k$. First, by Proposition~\ref{prop:second-moment},
\[
\|Y\|_{L_2}\le C\sqrt d,
\]
while $\|\xi\|_{L_2}=\sqrt d$ and $f_t,\bar g_t$ are bounded on $[0,\tau]$, so
\[
\|X_t\|_{L_2}\le |f_t|\,\|Y\|_{L_2}+|\bar g_t|\,\|\xi\|_{L_2}\le C\sqrt d.
\]

Now define
\[
\Delta_t:=\hat X_t-X_t,
\qquad
\delta_t:=\|\Delta_t\|_{L_2},
\qquad
\underline t:=t_k \ \text{for } t\in[t_k,t_{k+1}).
\]
For a.e.\ $t\in[t_k,t_{k+1})$,
\[
\dot\Delta_t
=
v_{\underline t}(\hat X_t)-v_t(X_t)
=
\bigl(v_t(\hat X_t)-v_t(X_t)\bigr)
+\bigl(v_{\underline t}(\hat X_t)-v_t(\hat X_t)\bigr).
\]
Hence
\[
\frac12\frac{d}{dt}\|\Delta_t\|^2
=
\langle \Delta_t, v_t(\hat X_t)-v_t(X_t)\rangle
+
\langle \Delta_t, v_{\underline t}(\hat X_t)-v_t(\hat X_t)\rangle
\le
L_t\|\Delta_t\|^2
+
\|\Delta_t\|\,\|v_t(\hat X_t)-v_{\underline t}(\hat X_t)\|.
\]
Therefore, for a.e.\ $t$ such that $\delta_t>0$,
\[
\delta_t'
\le
L_t\delta_t+\|v_t(\hat X_t)-v_{\underline t}(\hat X_t)\|_{L_2}.
\]
Using the time-Lipschitz bound at the fixed space point $\hat X_t$,
\[
\|v_t(\hat X_t)-v_{\underline t}(\hat X_t)\|_{L_2}
\le
\int_{\underline t}^t \|\partial_s v_s(\hat X_t)\|_{L_2}\,ds
\le
\left(\int_{\underline t}^t M_s\,ds\right)\bigl(\sqrt d+\|\hat X_t\|_{L_2}\bigr).
\]
Since $\|\hat X_t\|_{L_2}\le \|X_t\|_{L_2}+\delta_t\le C\sqrt d+\delta_t$, we get
\[
\delta_t'
\le
\bigl(L_t+\eta_t\bigr)\delta_t+C\sqrt d\,\eta_t,
\qquad
\eta_t:=\int_{\underline t}^t M_s\,ds.
\]
By Grönwall's lemma and $\delta_0=0$,
\[
\delta_t
\le
C\sqrt d\int_0^t
\eta_s
\exp\!\left(\int_s^t (L_u+\eta_u)\,du\right)ds.
\]
Now it remains to bound $\int_0^\tau \eta_s\,ds$. On each interval $[t_k,t_{k+1})$,
\[
\int_{t_k}^{t_{k+1}} \eta_s\,ds
=
\int_{t_k}^{t_{k+1}}\int_{t_k}^s M_u\,du\,ds
=
\int_{t_k}^{t_{k+1}}(t_{k+1}-u)M_u\,du
\le
h_k\int_{t_k}^{t_{k+1}} \frac{C\,du}{(1-u)^2}
\le
C\frac{h_k^2}{(1-t_k)^2}.
\]
Since $h_k=(1-r)(1-t_k)=h_{\max}(1-t_k)$ on the geometric grid,
\[
\int_{t_k}^{t_{k+1}} \eta_s\,ds \le C h_{\max}^2.
\]
Summing over $k$ yields
\[
\int_0^\tau \eta_s\,ds \le C N h_{\max}^2.
\]
Moreover, if
\[
x:=\frac1N\log\!\left(\frac1{1-\tau}\right),
\qquad
r=e^{-x},
\qquad
h_{\max}=1-r,
\]
then $h_{\max}\le x$, hence
\[
N h_{\max}^2 \le h_{\max}\log\!\left(\frac1{1-\tau}\right).
\]
Therefore
\[
\int_0^\tau \eta_s\,ds
\le
C h_{\max}\log\!\left(\frac1{1-\tau}\right).
\]
If $h_{\max}\le \bar h \log\!\bigl(\frac1{1-\tau}\bigr)^{-1}$ with $\bar h$ small enough, then
\[
\int_0^\tau \eta_s\,ds\le C,
\]
and the previous Grönwall bound gives
\[
\sup_{t\in[0,\tau]}\|\hat X_t-X_t\|_{L_2}
\le
C\sqrt d\, h_{\max}\log\!\left(\frac1{1-\tau}\right)
\le C\sqrt d.
\]
In particular, at the grid points,
\[
\sup_{0\le k\le N}\|\bar X_k\|_{L_2}
=
\sup_{0\le k\le N}\|\hat X_{t_k}\|_{L_2}
\le
\sup_{t\in[0,\tau]}\|X_t\|_{L_2}
+
\sup_{t\in[0,\tau]}\|\hat X_t-X_t\|_{L_2}
\le
C\sqrt d.
\]

\medskip

\noindent\textbf{Conclusion.}
We have verified all hypotheses of Proposition~\ref{prop:geom_mesh_general} with
\[
\sup_{z\in[0,\tau]}\int_z^\tau L_t\,dt\le C,
\qquad
F_t\le \frac{C}{1-t},
\qquad
C_t\le \frac{C}{1-t},
\qquad
M_t\le \frac{C}{(1-t)^2},\qquad
\sup_{0\le k\le N}\|\bar X_k\|_{L_2}\le C\sqrt d.
\]
Therefore, Proposition~\ref{prop:geom_mesh_general} applies and yields
\[
W_2\bigl(\text{Law}(X_\tau),\text{Law}(\bar X_N)\bigr)
\le
C\sqrt d\, h_{\max}\log\!\left(\frac1{1-\tau}\right).
\]
\end{proof}

\subsubsection{Proof of Proposition~\ref{prop:biasvar_no_propagation}}\label{sec:prop:biasvar_no_propagation}

\begin{proof}
We work on a probability space carrying the same $d$-dimensional Brownian motion $(W_t)_{t\in[0,\tau]}$
and the same initial condition $X_0=\bar X_0$.
By definition of the $2$-Wasserstein distance,
\[
W_2 \big(\text{Law}(X_\tau),\text{Law}(\bar X_N)\big)
\le \|X_\tau-\bar X_N\|_{L^2}=\|\Delta_N\|_{L^2}.
\]
Subtracting the Euler step  and using the same Brownian increment yields
\begin{align}
\Delta_{k+1}
&= \Delta_k + \int_{t_k}^{t_{k+1}} a_t(X_t) dt - h_k a_{t_k}(\bar  X_k)\nonumber\\
&= \Delta_k + h_k\big(a_{t_k}(X_{t_k})-a_{t_k}(\bar X_k)\big) + V_k ,
\end{align}
where $V_k$ is defined in \eqref{eq:Vk_def}. Defining,
\[
A_k := \Delta_k + h_k\big(a_{t_k}(X_{t_k})-a_{t_k}(\bar X_k)\big),
\qquad
B_k:=\E[V_k\mid\mathcal F_k], \qquad M_{k+1}:=V_k-\E[V_k\mid\mathcal F_k].
\]
we have
\begin{equation}\label{eq:Delta_rec_bias_mg}
\Delta_{k+1} = A_k+B_k + M_{k+1}.
\end{equation}
Since $A_k+B_k$ is $\mathcal F_k$-measurable and $\E[M_{k+1}\mid\mathcal F_k]=0$,
\[
\E\big[\langle A_k+B_k, M_{k+1}\rangle\big]
=\E\Big[\E\big[\langle A_k+B_k, M_{k+1}\rangle\mid\mathcal F_k\big]\Big]
=\E\big[\langle A_k+B_k, \E[M_{k+1}\mid\mathcal F_k]\rangle\big]=0.
\]
Therefore,
\begin{equation}\label{eq:Pythagoras}
\|\Delta_{k+1}\|_{L^2}^2
= \|A_k+B_k\|_{L^2}^2 + \|M_{k+1}\|_{L^2}^2
= \|A_k+\E[V_k\mid\mathcal F_k]\|_{L^2}^2 + \|V_k-\E[V_k\mid\mathcal F_k]\|_{L^2}^2 .
\end{equation}

By assumptions \eqref{eq:stability_discrete}--\eqref{eq:biasvar_Vk}, we have
\begin{align*}
\|A_k+\E[V_k\mid\mathcal F_k]\|_{L^2}
&\le \|A_k\|_{L^2} + \|\E[V_k\mid\mathcal F_k]\|_{L^2}
\le e^{\int_{t_k}^{t_{k+1}}L_s ds} \|\Delta_k\|_{L^2} + \mathrm{Bias}_k,
\\
\|V_k-\E[V_k\mid\mathcal F_k]\|_{L^2}
&\le \mathrm{Var}_k,
\end{align*}
so
plugging into \eqref{eq:Pythagoras} gives
\begin{equation}\label{eq:one_step_ineq}
\|\Delta_{k+1}\|_{L^2}^2 \le \Big(\Lambda(t_k,t_{k+1}) \|\Delta_k\|_{L^2} + \mathrm{Bias}_k\Big)^2 + \mathrm{Var}_k^{ 2},
\qquad
\Lambda(s,t):=\exp\Big(\int_s^t L_u du\Big).
\end{equation}
For convenience set $\alpha_k:=\Lambda(t_k,t_{k+1})>0$.
Define two nonnegative sequences $(S_k)_{k=0}^N$ and $(T_k)_{k=0}^N$ by
\[
S_0:=0,\qquad T_0:=0,
\qquad
S_{k+1}:=\alpha_k S_k + \mathrm{Bias}_k,
\qquad
T_{k+1}:=\alpha_k^2 T_k + \mathrm{Var}_k^{ 2}.
\]
We claim that for all $k\in\{0,\dots,N\}$,
\begin{equation}\label{eq:induction_claim}
\|\Delta_k\|_{L^2} \le S_k + \sqrt{T_k}.
\end{equation}
This is true for $k=0$ since $\|X_0-\bar X_0\|_{L^2}=0$.
Assume \eqref{eq:induction_claim} holds at step $k$.
Using \eqref{eq:one_step_ineq},
\begin{align*}
\|\Delta_{k+1}\|_{L^2}&
\le \sqrt{\big(\alpha_k(S_k+\sqrt{T_k})+\mathrm{Bias}_k\big)^2 + \mathrm{Var}_k^{ 2}}\\
& \leq \alpha_k S_k+\mathrm{Bias}_k+\sqrt{(\alpha_k\sqrt{T_k})^2+\mathrm{Var}_k^2}\\
& \leq S_{k+1} + \sqrt{T_{k+1}},
\end{align*}
which proves \eqref{eq:induction_claim} by induction. Now, by iterating the recursions for $S_k$ and $T_k$, we obtain
\[
S_N = \sum_{k=0}^{N-1}\Big(\prod_{j=k+1}^{N-1}\alpha_j\Big) \mathrm{Bias}_k
= \sum_{k=0}^{N-1}\Lambda(t_{k+1},\tau) \mathrm{Bias}_k,
\]
and
\[
T_N = \sum_{k=0}^{N-1}\Big(\prod_{j=k+1}^{N-1}\alpha_j\Big)^2 \mathrm{Var}_k^{ 2}
= \sum_{k=0}^{N-1}\Lambda(t_{k+1},\tau)^2 \mathrm{Var}_k^{ 2}.
\]
Combining with \eqref{eq:induction_claim} at $k=N$ yields
\[
\|\Delta_{N}\|_{L^2} \le S_N + \sqrt{T_N}
= \sum_{k=0}^{N-1}\Lambda(t_{k+1},\tau) \mathrm{Bias}_k
 + 
\Big(\sum_{k=0}^{N-1}\Lambda(t_{k+1},\tau)^2 \mathrm{Var}_k^{ 2}\Big)^{1/2}.
\]
Together with $W_2(\text{Law}(X_\tau),\text{Law}(\bar X_N))\le \|\Delta_{N}\|_{L^2}$, this gives the result.
\end{proof}

\subsubsection{Proof of Proposition~\ref{theo:dicdiff}}\label{sec:theo:dicdiff}

\begin{proof}
We apply Theorem~\ref{theo:geom_mesh_generaldiff} with diffusion coefficient
$b_t \equiv \sqrt{2}$ and drift
\[
a_t(x)=x+2s_t(x),\qquad s_t(x):=\nabla \log p_t(x),\qquad p_t:=q_{T-t},\quad t\in[0,T].
\]
Recall $\rho(t):=\min\{1,T-t\}$ and set
\[
\theta(t):=e^{-(T-t)}\in(0,1).
\]
We will repeatedly use the elementary comparability (for $u\ge 0$)
\begin{equation}\label{eq:rho-compare}
1-e^{-u} \gtrsim \min\{1,u\},\qquad 1-e^{-2u} \gtrsim \min\{1,u\}.
\end{equation}
First,
Since $s_t=\nabla\log p_t$ is a gradient field, $\nabla a_t$
is symmetric. For the moment bound, we use that the reverse SDE in Proposition~\ref{theo:dicdiff} is the time-reversal
of the OU path, so $\text{Law}(X_t)=p_t$ for $t\in[0,\tau]$. Then,
\[
\mathbb{E}\|X_t\|^2
=\theta(t)^2 \mathbb{E}\|Y\|^2+(1-\theta(t)^2) \mathbb{E}\|\xi\|^2
\le \mathbb{E}\|Y\|^2+d.
\]
Using Proposition~\ref{prop:second-moment} we have $\mathbb{E}\|Y\|^2\lesssim d$ so
\[
\sup_{t\in[0,\tau]}\|X_t\|_{L_2} \lesssim \sqrt d.
\]
Let us now verify the assumptions (A1)-(A4)

\noindent\emph{(A1)}
Because $\nabla a_t$ is symmetric,
\[
L_t:=\sup_{x\in\R^d}\lambda_{\max}(\nabla a_t(x))
=\sup_{x\in\R^d}\lambda_{\max}(I_d+2\nabla s_t(x)),
\]
so
\[
\sup_{z\in[0,\tau]}\int_z^\tau L_s ds
=\sup_{z\in[0,\tau]}\int_{\theta(z)}^{\theta(\tau)}\frac{1}{\theta} 
\sup_x \lambda_{\max} \big(I_d+2\nabla s_{\theta}(x)\big) d\theta
\le 2\int_{\theta(z)}^{\theta(\tau)}\frac{1}{\theta} 
\sup_x \lambda_{\max} \big(I_d+\nabla s_{\theta}(x)\big) d\theta.
\]
By Corollary~\ref{coro:score_osl_selfcontained}, the last integral is bounded by a constant
depending only on $(\alpha,\beta,K)$, uniformly over $z$, hence
$\sup_{z\in[0,\tau]}\int_z^\tau L_s ds\leq C$.

\noindent\emph{(A2).}
By Corollary~\ref{coro:difflipbound}, we have
\[
C_t:=\sup_{x\in\R^d}\|\nabla a_t(x)\|_{\text{op}}
=\sup_x \|I_d+2\nabla s_t(x)\|_{\text{op}}
\lesssim \frac{1}{1-\theta(t)^2}
\lesssim \frac{1}{\rho(t)}.
\]

\noindent\emph{(A3).}
By Corollary~\ref{coro:timeregudiffu}, we have
\[
\|\partial_t a_t(x)\|=2\|\partial_t s_t(x)\|\lesssim \frac{\sqrt d+\|x\|}{\rho(t)^2},
\]

\noindent\emph{(A4)}
By Lemma~\ref{lemma:boundnormvt}, we get
\[
\|a_t(x)\|
\le \|x\|+2\|s_t(x)\|
\leq C \frac{\sqrt d+\|x\|}{\rho(t)}.
\]

Let us now give a bound on $\|\Delta a_t\|$. The Fokker-Planck equation gives
\begin{equation}\label{eq:FP-pt}
\partial_t p_t= -\nabla \cdot(xp_t)-\Delta p_t.
\end{equation}
Let $g_t:=\log p_t$, so $s_t=\nabla g_t$. Dividing \eqref{eq:FP-pt} by $p_t$ and using
$\Delta p_t/p_t=\Delta g_t+\|\nabla g_t\|^2$ and $\nabla \cdot(xp_t)/p_t=d+x\cdot\nabla g_t$ gives
\[
\partial_t g_t=-(d+x\cdot\nabla g_t)-\Delta g_t-\|\nabla g_t\|^2.
\]
Taking a gradient yields the score PDE
\begin{equation}\label{eq:score-PDE}
\partial_t s_t + s_t + (\nabla s_t)(x+2s_t) + \Delta s_t = 0,
\end{equation}
so
\[
\|\Delta a_t(x)\|
\le 2\|\partial_t s_t(x)\| + 2\|s_t(x)\| + 2\|\nabla s_t(x)\|_{\text{op}} \|x+2s_t(x)\|.
\]
Using the bounds established above:
\[
\|\partial_t s_t(x)\|\lesssim \rho(t)^{-2}(\sqrt d+\|x\|),\qquad
\|s_t(x)\|\lesssim \rho(t)^{-1}(\sqrt d+\|x\|),\qquad
\|\nabla s_t(x)\|_{\text{op}}\lesssim \rho(t)^{-1},
\]
and $\|x+2s_t(x)\|=\|a_t(x)\|\lesssim \rho(t)^{-1}(\sqrt d+\|x\|)$, we obtain
\[
\|\Delta a_t(x)\| \lesssim \frac{\sqrt d+\|x\|}{\rho(t)^2},
\]
which is exactly the Laplacian hypothesis of Theorem~\ref{theo:geom_mesh_generaldiff}.

All assumptions of Theorem~\ref{theo:geom_mesh_generaldiff} are satisfied
so for $h_{\max}$ small enough we have
\[
W_2 \big(\text{Law}(X_\tau),\text{Law}(\bar X_N)\big)
 \le C\sqrt d h_{\max}\Big(T+\log\Big(\frac{T}{T-\tau}\Big)\Big).
\]
Finally, using $T\le A\log((T-\tau)^{-1})$ and $\log T\le T$, we get
\[
T+\log\Big(\frac{T}{T-\tau}\Big)
= T+\log T+\log\Big(\frac{1}{T-\tau}\Big)
 \le (2A+1)\log\Big(\frac{1}{T-\tau}\Big).
\]
which gives the result.
\end{proof}

\subsection{Proofs of the sampling errors}
\subsubsection{Proof of Lemma~\ref{lemma:sampling_error_sde}}\label{sec:lemma:sampling_error_sde}
\begin{proof}First, we have 
$$W_2 \bigl(p^\star,\text{Law}(\bar X_N)\bigr)
\le
W_2 \bigl(p^\star,\text{Law}(X_\tau)\bigr)
+
W_2 \bigl(\text{Law}(X_\tau),\text{Law}(\bar X_N)\bigl)$$ 
and $p^\star=\text{Law}(X_T)$.
Now, let $\hat X$ be the (learned) solution of
\[
d\hat X_t = \hat a_t(\hat X_t) dt + b_t dW_t,\qquad t\in[0,\tau],
\]
with the same initial condition $X_0=\hat X_0$ in law. We decompose
\begin{equation}
\label{eq:tri}
W_2 \bigl(\text{Law}(X_\tau),\text{Law}(\bar X_N)\bigr)
\le
W_2 \bigl(\text{Law}(X_\tau),\text{Law}(\hat X_\tau)\bigr)
+
W_2 \bigl(\text{Law}(\hat X_\tau),\text{Law}(\bar X_N)\bigr).
\end{equation}
Couple $X$ and $\hat X$ by taking the same Brownian motion $W$ and the same initial condition. Let $$\Delta_t:=X_t-\hat X_t,$$ we have
$$W_2^2 \bigl(\text{Law}(X_\tau),\text{Law}(\hat X_\tau)\bigr)\leq \mathbb{E}\|\Delta_{\tau}\|^2.$$
Furthermore,
\begin{align*}
\frac{d}{dt}\E\|\Delta_t\|^2
&= 2 \E\bigl\langle \Delta_t, a_t(X_t)-\hat a_t(\hat X_t)\bigr\rangle\\
&=
2 \E\bigl\langle \Delta_t, \hat a_t(X_t)-\hat a_t(\hat X_t)\bigr\rangle 
+
2 \E\bigl\langle \Delta_t, a_t(X_t)-\hat a_t(X_t)\bigr\rangle\\
& \leq 2\hat  L_t \E\|\Delta_t\|^2+2\|\Delta_t\|_{L_2} \epsilon_{\mathrm{drift}}(t)
\end{align*}
using the one-sided Lipschitzness of $\hat{a}_t$
and  Cauchy-Schwarz.
Finally, applying the non-linear Grönwall's inequality we get 
\[
\left(\mathbb E\|\Delta_\tau\|_{L^2}\right)^{1/2} \le \int_0^\tau \exp \Bigl(\int_t^\tau \hat L_s ds\Bigr) 
\epsilon_{\mathrm{drift}}(t) dt.
\]
\end{proof}

\subsubsection{Proof of Corollary~\ref{cor:lipman_sampling_error_log2_over_N}}\label{sec:cor:lipman_sampling_error_log2_over_N}
\begin{proof} We only give the detail of the proof in the Lipman setting as the proof is the same in the Stochastic interpolant setting by replacing Corollaries~\ref{cor:lipman-global} and~\ref{coro:dtv_rate} by Corollaries~\ref{coro:global-Lipschitz-v} and~\ref{coro:timeestimatesi}.
Applying Lemma~\ref{lemma:sampling_error_sde}, we have
for $(\hat X_t)_{t\in[0,\tau]}$
\begin{align*}
W_2 \bigl(p^\star,\text{Law}(\bar X_N)\bigr)
\le& W_2 \bigl(\text{Law}(X_1),\text{Law}(X_\tau)\bigr)+
\int_0^\tau
\exp \Bigl(\int_t^\tau \hat L_s ds\Bigr) 
\epsilon_{\mathrm{drift}}(t) dt+ W_2\bigl(\text{Law}(\hat{X}_\tau),\text{Law}(\bar X_N)\bigr).
\end{align*}
for $\hat X_t$ the solution of
\[
\partial_t \hat X_t = \hat v_t(\hat X_t),\qquad \hat X_0 \stackrel{d}= X_0,
\]
\noindent\textbf{Early stopping error $W_2(\text{Law}(X_1),\text{Law}(X_\tau))$.}
In the Lipman reduced-model setting, the marginal at time $t$ can be represented as
\[
X_t \sim f_t Y + \sigma_t \xi,
\qquad Y\sim p^\star,\ \xi\sim\mathcal N(0,I_d),\ Y\perp \xi,
\]
so in particular $X_1\sim Y\sim p^\star$ and $X_\tau\sim f_\tau Y+\sigma_\tau \xi$.
Applying Lemma~\ref{lem:general-coupling} with $A=f_\tau \text{Id}$ and $\eta=\sigma_\tau\xi$ yields
\begin{equation}
\label{eq:lemma6_trunc}
W_2\bigl(\text{Law}(X_1),\text{Law}(X_\tau)\bigr)
\le
\Bigl((1-f_\tau)^2 \E\|Y\|^2 + \E\|\sigma_\tau\xi\|^2\Bigr)^{1/2}
=
\Bigl((1-f_\tau)^2 \E\|Y\|^2 + d \sigma_\tau^2\Bigr)^{1/2}.
\end{equation}
Under the assumptions $\|\argmin u\|\leq A\sqrt{d}$, one has $\E\|Y\|^2 \leq C d$ by Proposition~\ref{prop:second-moment},
so \eqref{eq:lemma6_trunc} gives
\begin{equation}
\label{eq:trunc_simple}
W_2\bigl(\text{Law}(X_1),\text{Law}(X_\tau)\bigr)
 \leq C 
\sqrt d\bigl(|1-f_\tau|+\sigma_\tau\bigr).
\end{equation}
Next, since $f'\in L^\infty(0,1)$ and $f_1=1$, we have
\[
|1-f_\tau| = \Bigl|\int_\tau^1 f'(s) ds\Bigr|\le \|f'\|_\infty (1-\tau).
\]
By the schedule assumption, $\sigma_\tau \le A(1-\tau)^{p}$ for $\tau$ close enough to $1$.
Set
\[
q:=p\wedge 1,
\qquad
\tau = 1-\Bigl(\frac{\log^2 N}{N}\Bigr)^{1/q}.
\]
Writing
\[
x_N:=\frac{\log^2 N}{N},
\]
we have $1-\tau=x_N^{1/q}$. Hence
\[
\sigma_\tau  \leq C  (1-\tau)^{p}
= C x_N^{p/q}
\le C x_N
= C\frac{\log^2 N}{N},
\]
since $p/q\ge 1$, and
\[
|1-f_\tau|  \leq C  (1-\tau)
= C x_N^{1/q}
\le C x_N
= C\frac{\log^2 N}{N},
\]
since $1/q\ge 1$.
Plugging into \eqref{eq:trunc_simple} yields
\begin{equation}
\label{eq:trunc_final}
W_2\bigl(\text{Law}(X_1),\text{Law}(X_\tau)\bigr)
 \leq C 
\sqrt d \frac{\log^2 N}{N}.
\end{equation}
\noindent\textbf{Discretization error $W_2(\text{Law}(\hat X_\tau),\text{Law}(\bar X_N))$.}
By the same arguments than in the proof of Proposition~\ref{prop:geomeshLIp} , $\hat v_t$ satisfies the hypotheses \eqref{eq:lipman_envelopes_general} of Proposition~\ref{prop:geom_mesh_general}. It remains to verify the moment assumption in Proposition~\ref{prop:geom_mesh_general} for the learned dynamics. Let $(\hat X_t)_{t\in[0,\tau]}$ solve
\[
\dot{\hat X}_t=\hat v_t(\hat X_t),\qquad \hat X_0=X_0.
\]
Coupling $(X_t,\hat X_t)$ with the same initial condition and setting $\Delta_t:=X_t-\hat X_t$, the same computation as in Lemma~\ref{lemma:sampling_error_sde} yields
\[
\frac{d}{dt}\|\Delta_t\|_{L^2}\le \hat L_t\|\Delta_t\|_{L^2}+\varepsilon_{\mathrm{drift}}(t),
\]
hence, by Gr\"onwall and the integrated one-sided Lipschitz bound on $\hat v$,
\[
\sup_{t\in[0,\tau]}\|X_t-\hat X_t\|_{L^2}\le C\int_0^\tau \varepsilon_{\mathrm{drift}}(s)\,ds\le C\sqrt d .
\]
On the other hand, by Proposition~\ref{prop:second-moment} and the explicit representation of the exact interpolation, we have
\[
\sup_{t\in[0,\tau]}\|X_t\|_{L^2}\le C\sqrt d .
\]
Therefore,
\[
\sup_{t\in[0,\tau]}\|\hat X_t\|_{L^2}\le \sup_{t\in[0,\tau]}\|X_t\|_{L^2}
+\sup_{t\in[0,\tau]}\|X_t-\hat X_t\|_{L^2}\le C\sqrt d,
\]
so that
\[
\sup_{t\in[0,\tau]}\E\|\hat X_t\|^2\le Cd.
\]
We may now repeat verbatim Step~3 in the proof of Proposition~\ref{prop:geomeshLIp} with $\hat X_t$ in place of $X_t$, which gives
\[
\sup_{k=0,\dots,N}\E\|\bar X_k\|^2\le Cd,
\]
and thus the hypotheses of Proposition~\ref{prop:geom_mesh_general} are satisfied for the learned drift $\hat v$. Therefore, Proposition~\ref{prop:geom_mesh_general}
yields for $h_{\max}$ small enough,
\begin{equation}
\label{eq:prop3}
W_2\bigl(\text{Law}(\hat X_\tau),\text{Law}(\bar X_N)\bigr)
 \leq C 
\sqrt d  h_{\max} \log\Bigl(\frac{1}{1-\tau}\Bigr).
\end{equation}
We now relate $h_{\max}$ and $N$ for the geometric grid:
since $r^N=1-\tau$ and $h_{\max}=1-r$, we can write
\[
r = \exp\Bigl(-\frac{1}{N}\log\frac{1}{1-\tau}\Bigr)
\quad\Longrightarrow\quad
h_{\max} = 1-e^{-x}\le x
\quad\text{with}\quad
x:=\frac{1}{N}\log\frac{1}{1-\tau}.
\]
Therefore,
\begin{equation}
\label{eq:hmax_bound}
h_{\max}\le \frac{1}{N}\log\Bigl(\frac{1}{1-\tau}\Bigr).
\end{equation}
Combining \eqref{eq:prop3} and \eqref{eq:hmax_bound} gives
\begin{equation}
\label{eq:disc_log2_over_N_general_tau}
W_2\bigl(\text{Law}(\hat X_\tau),\text{Law}(\bar X_N)\bigr)
 \leq C 
\sqrt d 
\frac{\log^2 \bigl(\frac{1}{1-\tau}\bigr)}{N}.
\end{equation}
With $\tau=1-(\log^2 N/N)^{1/q}$ we have
\[
\log\Bigl(\frac{1}{1-\tau}\Bigr)
=
\frac{1}{q}\log\Bigl(\frac{N}{\log^2 N}\Bigr)
\le \frac{1}{q}\log N
\qquad\text{for $N$ large enough.}
\]
hence \eqref{eq:disc_log2_over_N_general_tau} implies
\begin{equation}
\label{eq:disc_final}
W_2\bigl(\text{Law}(\hat X_\tau),\text{Law}(\bar X_N)\bigr)
 \leq C 
\sqrt d \frac{\log^2 N}{N}.
\end{equation}
\end{proof}

\subsubsection{Proof of Corollary~\ref{cor:diffusion_sde_sampling_logN}}\label{sec:cor:diffusion_sde_sampling_logN}

\begin{proof}
We decompose,
\[
W_2\big(p^\star,\text{Law}(\bar X_N)\big)
\le
W_2\big(\text{Law}(X_T),\text{Law}(X_\tau)\big)
+
W_2\big(\text{Law}(X_\tau),\text{Law}(\hat X_\tau)\big)
+
W_2\big(\text{Law}(\hat X_\tau),\text{Law}(\bar X_N)\big)
\]
and bound separately each term.

\smallskip
\noindent\textbf{1- Early stopping.}
Since $\text{Law}(X_T)=p^\star$ and $\text{Law}(X_\tau)=q_{T-\tau}=q_{1/N^2}$, and
$q_s=\text{Law}(e^{-s}Y+\sqrt{1-e^{-2s}} \xi)$ with $Y\sim p^\star$ and $\xi\sim\mathcal{N}(0,I_d)$ independent, we may couple
$Y$ with $e^{-s}Y+\sqrt{1-e^{-2s}} \xi$ and obtain
\[
W_2\big(p^\star,q_s\big)^2
\le
\E\Big\|Y-\Big(e^{-s}Y+\sqrt{1-e^{-2s}} \xi\Big)\Big\|^2
=
(1-e^{-s})^2 \E\|Y\|^2 + (1-e^{-2s}) \E\|\xi\|^2.
\]
With $s=T-\tau=1/N^2$ and $\E\|\xi\|^2=d$, using $1-e^{-s}\le s$ and $1-e^{-2s}\le 2s$ gives
\[
W_2\big(p^\star,\text{Law}(X_\tau)\big)
=
W_2\big(p^\star,q_{1/N^2}\big)
\le
\Big(\frac{\E\|Y\|^2}{N^4}+\frac{2d}{N^2}\Big)^{1/2}
\le
C \frac{\sqrt d}{N},
\]
by Proposition~\ref{prop:second-moment}.

\noindent\textbf{2- Learning error.}
Couple $(X_0,\hat X_0)$ optimally so that $\|X_0-\hat X_0\|_{L^2}=W_2(q_T,\gamma)$, and drive $(X_t,\hat X_t)$
with the same Brownian motion. With $\Delta_t=X_t-\hat X_t$, the same computation as above yields
\[
\frac{d}{dt}\|\Delta_t\|_{L^2}\le \hat L_t\|\Delta_t\|_{L^2}+\varepsilon_{\mathrm{drift}}(t),
\]
hence by Grönwall,
\[
W_2(\text{Law}(X_\tau),\text{Law}(\hat X_\tau))
\le \|\Delta_\tau\|_{L^2}
\le e^{\int_0^\tau \hat L_s ds} W_2(q_T,\gamma)
+\int_0^\tau e^{\int_t^\tau \hat L_s ds} \varepsilon_{\mathrm{drift}}(t) dt.
\]
Using the bound $e^{\sup_{z\in[0,\tau]}\int_z^\tau \hat L_s ds}\le C$, it remains to control $W_2(q_T,\gamma)$.
Couple $q_T$ with $\gamma_d$ by pairing $Z_T$ with the same $\xi$. Then
\[
W_2(q_T,\gamma)^2 \le \E\|Z_T-\xi\|^2
= \E\Big\|e^{-T}Y+\big(\sqrt{1-e^{-2T}}-1\big)\xi\Big\|^2.
\]
Using $\big(1-\sqrt{1-u}\big)=\frac{u}{1+\sqrt{1-u}}\le u$ for $u\in[0,1]$,
\[
\E\|Z_T-\xi\|^2
\le 2e^{-2T}\E\|Y\|^2 + 2\big(1-\sqrt{1-e^{-2T}}\big)^2 \E\|\xi\|^2
\le 2e^{-2T}\E\|Y\|^2 + 2e^{-4T} d.
\]
Hence
\[
W_2(q_T,\gamma) \le \sqrt{2} e^{-T}\sqrt{\E\|Y\|^2} + \sqrt{2d} e^{-2T}
\le C \frac{\sqrt d}{N}.
\]
Therefore,
\[
W_2(\text{Law}(X_\tau),\text{Law}(\hat X_\tau))
\le C\frac{\sqrt d}{N}+C\int_0^\tau \varepsilon_{\mathrm{drift}}(t) dt.
\]

\noindent\textbf{3- Time discretization.}
Set $\Delta_t:=X_t-\hat X_t$. By Grönwall's lemma,
for every $t\in[0,\tau]$,
\[
\|\Delta_t\|_{L_2}
\le e^{\int_0^t \hat L_s ds}\|\Delta_0\|_{L_2}
+ \int_0^t e^{\int_u^t \hat L_s ds} \varepsilon_{\mathrm{drift}}(u) du.
\]
Since $\|\Delta_0\|_{L_2}\leq \|X_0\|_{L_2}+\|\hat{X}_0\|_{L_2}\leq C\sqrt d$, this gives
\begin{align*}
\|\hat X_t\|_{L_2}&\le \|X_t\|_{L_2}+\|X_t-\hat X_t\|_{L_2}\\
& \le C\sqrt{d}+ C\sqrt{d}
+ C \int_0^\tau \varepsilon_{\mathrm{drift}}(u) du\\
& \le C\sqrt{d}.
\end{align*}
Then, by repeating verbatim the verification of the hypotheses of Theorem~\ref{theo:geom_mesh_generaldiff} carried out in the proof of Proposition~\ref{theo:dicdiff} with $a_t$ replaced by $\hat a_t = x + 2\hat s_t$ and using the assumed structural bounds on $\hat s_t$, we may apply Theorem~\ref{theo:geom_mesh_generaldiff} to $(\hat X_t)_{t\leq \tau}$ and its Euler scheme  to obtain
\[
W_2\big(\text{Law}(\hat X_\tau),\text{Law}(\bar X_N)\big)
\le
C \sqrt d h_{\max} \log \Big(\frac1{T-\tau}\Big).
\]

Since $r=((T-\tau)/T)^{1/N}=\exp \big(-\frac1N\log\frac{T}{T-\tau}\big)$, letting
$x:=\frac1N\log\frac{T}{T-\tau}\ge 0$ gives $r=e^{-x}$ and
\[
h_{\max}=T(1-r)=T(1-e^{-x})\le T x
=
\frac{T}{N}\log \Big(\frac{T}{T-\tau}\Big).
\]
With $T=\log N$ and $T-\tau=1/N^2$, this yields
\[
h_{\max}\log \Big(\frac1{T-\tau}\Big)
\le
\frac{\log N}{N} 
\log \big( (\log N) N^2\big) 
\log(N^2)
\;\lesssim\;
\frac{\log^3 N}{N},
\]
for $N$ large enough. Therefore
\[
W_2\big(\text{Law}(\hat X_\tau),\text{Law}(\bar X_N)\big)
\le
C \sqrt d \frac{\log^3 N}{N}.
\]
\end{proof}

\end{document}